\documentclass{article}

\usepackage[letterpaper,margin=1.2in]{geometry}

\usepackage{amsmath,amssymb,amsthm}
\usepackage{MnSymbol}
\usepackage{enumerate}
\usepackage{hyperref}
\usepackage{mathtools}
\usepackage{bbm}
\usepackage{xcolor}
\usepackage{tikz}
\usepackage{tikz-cd}

\usepackage{comment}

\newtheorem{introthm}{Theorem}

\newtheorem{introcor}[introthm]{Corollary}
\newtheorem{introprop}[introthm]{Proposition}

\newtheorem{theorem}{Theorem}[section]

\newtheorem{lemma}[theorem]{Lemma}
\newtheorem{proposition}[theorem]{Proposition}
\newtheorem{corollary}[theorem]{Corollary}

\newtheorem*{theorem*}{Theorem}

\numberwithin{equation}{section}
\numberwithin{figure}{section}

\theoremstyle{definition}
\newtheorem{definition}[theorem]{Definition}
\newtheorem{notation}[theorem]{Notation}

\theoremstyle{remark}
\newtheorem{remark}[theorem]{Remark}
\newtheorem{example}[theorem]{Example}

\newcommand\R{\mathbb{R}}
\newcommand\C{\mathbb{C}}
\newcommand\Q{\mathbb{Q}}

\newcommand{\M}{\mathbb{M}}
\newcommand\N{\mathbb{N}}

\newcommand{\cD}{\mathcal{D}}

\newcommand{\cF}{\mathcal{F}}

\newcommand{\cL}{\mathcal{L}}
\newcommand{\cM}{\mathcal{M}}
\newcommand{\cN}{\mathcal{N}}

\newcommand{\cQ}{\mathcal{Q}}

\newcommand{\cU}{\mathcal{U}}
\newcommand{\cV}{\mathcal{V}}

\newcommand{\rT}{\mathrm{T}}

\DeclareMathOperator{\id}{id}

\DeclareMathOperator{\re}{Re}
\DeclareMathOperator{\im}{Im}
\DeclareMathOperator{\Tr}{Tr}
\DeclareMathOperator{\tr}{tr}

\DeclareMathOperator{\sa}{sa}

\DeclareMathOperator{\qf}{qf}
\DeclareMathOperator{\tp}{tp}
\DeclareMathOperator{\Th}{Th}

\DeclareMathOperator{\Boch}{Boch}
\DeclareMathOperator{\chron}{chron}
\DeclareMathOperator{\filt}{filt}
\DeclareMathOperator{\proc}{proc}
\DeclareMathOperator{\stat}{stat}

\DeclareMathOperator{\acl}{acl}
\DeclareMathOperator{\vol}{vol}
\DeclareMathOperator{\Lip}{Lip}

\DeclarePairedDelimiter{\norm}{\lVert}{\rVert}
\DeclarePairedDelimiter{\ip}{\langle}{\rangle}

\title{Free information geometry and the model theory of noncommutative stochastic processes}
\author{David Jekel}

\begin{document}
	
	\maketitle
	
	\begin{abstract}
		We study entropy and optimal transport theory in the multivariate free probabilistic setting motivated by the large-$n$ theory of random tuples of matrices.  We define a new version of free entropy $\chi_{\chron}^{\cU}$, which is concave along geodesics in the corresponding Wasserstein space.  Moreover, the heat evolution satisfies the evolution variational inequality, which means that the heat evolution is the Wasserstein gradient flow for entropy in the metric sense.  It also has further desirable properties such as a chain rule for iterated conditioning, and an expression in terms of stochastic control problems.  This entropy is defined using microstate spaces of matrix approximations with respect to an expanded class of test functions called \emph{chronological formulas}, which are constructed so as to be closed under taking partial suprema and infima and application of a free heat semigroup.  These formulas are part of a novel framework for studying noncommutative filtrations and stochastic processes as metric structures in the sense of continuous model theory.
	\end{abstract}
	
	\section{Introduction} \label{sec: intro}
	
	\subsection{Motivation} \label{subsec: intro motivation}
	
	The goal of this work is to study Wasserstein information geometry, the theory of optimal transport and entropy, in the setting of free probability.  In classical Wasserstein geometry, one studies the space of probability distributions on $\R^m$, or more generally some manifold, equipped with a metric coming from optimal transport theory.  The $L^2$ Wasserstein distance of two probability measures $\mu, \nu \in \mathcal{P}(\C^m)$ is the infimum of $\norm{\mathbf{X} - \mathbf{Y}}_{L^2}$ over random variables $\mathbf{X} \sim \mu$ and $\mathbf{Y} \sim \nu$ in some diffuse probability space.  The Wasserstein space can also be viewed as an infinite-dimensional manifold \cite{Lafferty1988}.  The entropy $h(\mu) = \int -\rho \log \rho$ is function on this manifold.  The entropy (with the sign convention chosen here) is concave along geodesics \cite{McCann1997convexity} \cite[Corollary 17.19]{Villani2008}.  Moreover, the gradient flow of entropy on the Wasserstein manifold is precisely given by the heat evolution of the measure \cite{JKO1998}.  These fundamental results unify many functional inequalities relating entropy, Fisher information, and Wasserstein distance.  Wasserstein geometry has had many applications in probability, statistics, and data analysis.  It has also inspired many generalizations, such as the development of gradient flow in general metric spaces (see \cite{AGS2008gradient}), and the gradient flow of entropy in the quantum setting \cite{CarlenMaas2017gradient}.  In fact, the free probability setting is a quite suitable test case for the metric geometry perspective because, as explained below, it has very limited regularity compared to the classical setting.
	
	Free probability can be described as a theory of non-commuting random variables which bridges random matrix theory and self-adjoint operator algebras.  The space of random variables $L^\infty(\Omega,\mathcal{F},\mathbb{P})$ is replaced with a von Neumann algebra $\cM$, and the expectation is replaced by a tracial state $\tau: M \to \C$.  Independence is replaced by free independence, resulting in a free central limit theorem with the Gaussian measure replaced by Wigner's semicircular measure \cite{Voiculescu1985,Voiculescu1986}.  The objects in free probability also describe the large-$n$ bulk behavior of several random matrices in many cases.  A free semicircular system $(s_1,\dots,s_m)$ for instance models the large-$n$ behavior of random matrices $(S_1^{(n)},\dots,S_m^{(n)})$ from the Gaussian unitary ensemble (GUE) \cite{Voiculescu1991}.
	
	Analogs of entropy in free probability were introduced by Voiculescu \cite{VoiculescuFE2,VoiculescuFE5} and soon had many applications to the structure of free product von Neumann algebras (e.g.\ \cite{VoiculescuFE3,Ge1998Prime,Hayes2018,HJNS2021,Hayes2022PT,JKE2026upgraded}).  A free Wasserstein distance was defined by Biane and Voiculescu \cite{BV2001}, and many functional inequalities in information theory also have free analogs (e.g.\ \cite{VoiculescuFE5,BS1998,HPU2004,HU2006,Dabrowski2010,MM2014transportentropy,FathiNelson2017,ST2022fractional,Diez2024,LLX2020}).  In particular, the information geometry for free probability is well-developed in the setting of a single (self-adjoint) variable, which deals with probability measures on $\R$; for instance, the free heat flow is the Wasserstein gradient flow of free entropy by \cite[Theorem 1.2]{LLX2020}.  In the multivariate setting, Guionnet and Shlyakhtenko \cite{GS2014} introduced a noncommutative analog of optimal transport to show that von Neumann algebra associated to free Gibbs law for convex potentials (a noncommutative analog of strongly log-concave probability distributions) is isomorphic to that of a free semicircular family, and the behavior of entropy and Wasserstein distance for free Gibbs law from convex potentials is somewhat well understood (see \cite{Dabrowski2017Laplace,JekelEntropy,JekelExpectation,DGS2021,JLS2022}).
	
	However, free information geometry for several variables is beset with long-standing open problems that prevent a satisfying unified theory, in particular:
	\begin{enumerate}[(1)]
		\item Voiculescu defined two analogs $\chi$ and $\chi^*$ for entropy in \cite{VoiculescuFE2} and \cite{VoiculescuFE5}.  While $\chi$ is defined in terms of random matrix models, $\chi^*$ is defined by integrating the free Fisher information along the heat flow, and it is unknown in general $\chi$ and $\chi^*$ are equal outside the single-variable setting (see \cite{Voiculescu2002}).  More generally, it is unclear whether the heat flow is actually the gradient flow for free entropy, and whether the free Fisher information is actually the norm of the gradient squared for entropy.
		\item The large deviations theory for several GUE matrices $(S_1^{(n)},\dots,S_d^{(n)})$ has remained incomplete since \cite{BCG2003}.  The rate function for the large deviations principle should be the free entropy.  The paper \cite{BCG2003} proved an upper bound (implying $\chi \leq \chi^*$), and also gave a lower bound in the case of elements given as a solution to certain class of free SDE.
		\item We do not know in general whether the classical Wasserstein distance associated to natural matrix models converges to the free Wasserstein distance (and this is false for certain settings by \cite[\S 6.1]{GJNS2021} and \cite[\S 4.4]{Jekel2025InfoGeom}).
	\end{enumerate}
	A key piece of the puzzle that was not available until recently was a good free analog of Monge--Kantorovich duality, which would be crucial for analyzing optimal couplings.  A more systematic study of the free Wasserstein distance and Monge--Kantorovich duality in \cite{GJNS2021} also brought into focus some of the obstructions that make the noncommutative setting inherently harder than the classical case:
	\begin{enumerate}[(1)]
		\item There are many non-isomorphic tracial von Neumann algebras, including uncountably many that cannot embed into each other \cite{Ozawa2004}.  As a consequence, the space of noncommutative laws (or probability distributions) is not separable with respect to the Wasserstein distance \cite[\S 5.5]{GJNS2021}.
		\item For noncommutative probability distributions of bounded support, the Wasserstein distance generates a much stronger topology than the weak-$*$ topology \cite[\S 5.4]{GJNS2021}.
		\item Thanks to the far-reaching quantum complexity result MIP$^*=$RE \cite{JNVWY2020}, there are tracial von Neumann algebras which do not admit any finite-dimensional approximations.  Moreover, this also implies that there are some noncommutative laws which do have finite-dimensional approximations, but such that the optimal coupling in the sense of Biane--Voiculescu distance does not admit any finite-dimensional approximations \cite[Corollary 5.14]{GJNS2021}.
		\item It was also shown in \cite[\S 4.4]{Jekel2025InfoGeom} that for certain noncommutative laws, one cannot construct random matrix models that simultaneously have the correct behavior for entropy and for Wasserstein distance (even after modifying the distance to only look at couplings that admit finite-dimensional approximations).
	\end{enumerate}
	In particular, unlike the classical case, two measures with finite entropy in the free setting can generate very different von Neumann algebras, so that there is no hope of the optimal coupling being realized by transport functions.  However, the noncommutative Monge--Kantorovich duality shows that at least in the \emph{interior} of a geodesic, the associated von Neumann algebra is the same \cite[Theorem 1.5]{GJNS2021}.
	
	In \cite{JekelModelEntropy,JekelTypeCoupling,Jekel2025InfoGeom}, we began to approach the problem of constructing a free information geometry by expanding the class of test functions (hence also changing the notion of ``noncommutative probability distribution'').  Let us motivate this in terms of geodesic concavity.  In the classical setting, the geodesic concavity of entropy on $\R^m$ is not hard to prove from optimal transport theory.  Indeed, if $(\mathbf{X}_0,\mathbf{X}_1)$ are an optimal coupling of sufficiently regular probability measures, then there is a convex function $\varphi$ such that $\mathbf{X}_1 = \nabla \varphi(\mathbf{X}_0)$.  The corresponding Wasserstein geodesic is the law $\mu_t$ of $\mathbf{X}_t = (\id + t \nabla \varphi)(\mathbf{X}_0)$.  One can then compute $h(\mathbf{X}_t)$ using the change of variables formula for entropy:
	\[
	h((\id + t \nabla \varphi)(\mathbf{X}_0)) = h(\mathbf{X}_0) + \mathbb{E} \Tr \log(\id + t H \varphi(\mathbf{X}_0)),
	\]
	and the latter is easily seen to be concave.  One can also compute the derivative
	\begin{equation} \label{eq: classical change of variables}
		\frac{d}{dt}\biggr|_{t=0} h((\id + t \nabla \varphi)(\mathbf{X}_0)) = \mathbb{E} \Delta \varphi(\mathbf{X}_0).
	\end{equation}
	The problem in the free setting is that there is a large gap between the very \emph{low} regularity that we have for the functions in Monge--Kantorovich duality and the comparatively \emph{high} regularity that we need in order to use the change of variables for entropy.
	
	The most basic scalar-valued functions in free probability are those of the form $\varphi(\mathbf{x}) = \tr(p(\mathbf{x}))$ for some noncommutative polynomial $p$ in several variables.  By taking sums and products of such functions, one obtains the algebra of \emph{trace polynomials}.  The derivatives of $\tr(p(\mathbf{x}))$ (in particular, the gradient and Hessian) can be described in terms of Voiculescu's cyclic gradient and free difference quotient operators, and certainly one can make sense of the Laplacian and the trace of the log of the Hessian, and prove of change of variables for free microstate entropy; see \cite[\S 3]{VoiculescuFE2}, \cite[\S 7.2]{JLS2022}
	
	Unfortunately, this is not a large enough class of functions to realize Monge--Kantorovich duality.  For this, we need to obtain a pair of convex functions $\varphi$ and $\psi$ such that
	\[
	\varphi(\mathbf{x}) + \psi(\mathbf{y}) \geq \ip{\mathbf{x},\mathbf{y}}
	\]
	with equality in the case of an optimal coupling of $(\mu_0,\mu_1)$.  It is natural to try to construct $\varphi$ and $\psi$ via Legendre transforms of suitable other functions (see proof of Theorem \ref{thm: MK duality}).  Thus, we might set
	\[
	\varphi(\mathbf{x}) = \sup_{\mathbf{y} \in M_{\sa}^m} \left[ \ip{\mathbf{x},\mathbf{y}} - \eta(\mathbf{y}) \right]
	\]
	where $\eta$ is some trace polynomial (not necessarily convex).  However, we immediately run into the problem that the supremum might depend on how $\mathbf{x}$ sits inside the ambient von Neumann algebra $\cM$, so it does not only depend on the noncommutative law of $\mathbf{x}$.  Hence, $\varphi$ immediately falls out of the class of trace polynomials (and its completion).  This behavior is also strikingly different than the classical setting:  If I take a continuous function $f(x,y)$ on $\R^{2m}$ and a compact $K \subseteq \R^m$, then $g(x) = \sup_{y \in K} f(x,y)$ will be continuous, and hence approximated by polynomials.  In the case of classical probability, or equivalently commutative von Neumann algebras, the supremum also does not depend on the ambient algebra, provided the measure space is atomless.  In model-theoretic terms, this is because atomless classical probability spaces admit \emph{quantifier elimination} (see \cite[Example 4.3]{BYU2010} and \cite[Fact 2.10]{BY2012}), and noncommutative tracial von Neumann algebras usually do not \cite{Farah2024,FJP2026}.
	
	We can enlarge the space of test functions so that it is closed under partial suprema and infima over operator-norm balls; this produces the class of \emph{formulas} for tracial von Neumann algebras in the sense of continuous model theory.  Furthermore, in this larger class of test functions, we can find convex functions that witness Monge--Kantorovich duality \cite[Theorem 1.1]{JekelTypeCoupling}.  However, we can no longer apply the change of variables for entropy in this larger class of test functions, as it is not clear how to define the Laplacian or free difference quotients.  We would expect that the Laplacian should satisfy
	\[
	\frac{1}{2} \Delta \varphi(\mathbf{x}) = \lim_{t \to 0^+} \frac{\varphi(\mathbf{x} + \mathbf{z}_t) - \varphi(\mathbf{x})}{t},
	\]
	where $\mathbf{z}_t$ is a free Brownian motion freely independent of $\mathbf{x}$.  The value of $\varphi(\mathbf{x} + \mathbf{z}_t)$ should be the free heat semigroup applied to the function $\varphi$.  Note here that for scalar-valued functions in the free setting we do not need to take any expectation; for motivation, note that $\varphi(\mathbf{x} + \mathbf{z}_t)$ is supposed to be a large-$n$ model of $\varphi(\mathbf{X}^{(n)} + \mathbf{Z}_t^{(n)})$ where $\mathbf{Z}_t^{(n)}$ is a matrix Brownian motion independent of $\mathbf{X}^{(n)}$, and due to concentration of measure the limiting quantity will be deterministic.  In the free probabilistic limit, we want the free Brownian motion $\mathbf{z}_t$ to be freely independent of the base algebra $M_0$ where the point $\mathbf{x}$ lives.
	
	When we combine the partial supremum operation and the heat semigroup, we need to specify the domain where we take the supremum.  In the classical case, in general
	\[
	\mathbb{E}_{\mathbf{X}} \sup_{\mathbf{Y}} \varphi(\mathbf{X},\mathbf{Y}) \neq \sup_{\mathbf{Y}} \mathbb{E}_{\mathbf{X}} \varphi(\mathbf{X},\mathbf{Y})
	\]
	since in the first case $\mathbf{Y}$ is allowed to depend on $\mathbf{X}$ and in the second case it is not.  Analogously, in the free case, some of the variables in the suprema and infima might only range over the original algebra $M_0$, and some might range over the algebra $M_t$ that contains the Brownian motion $\mathbf{z}_t$ at time $t$.  For example, one might have a formula like
	\[
	\varphi(\mathbf{x}) = \inf_{\mathbf{v} \in M^m} \sup_{\mathbf{w} \in M_t^m} \varphi(\mathbf{x} + \mathbf{z}_t, \mathbf{v},\mathbf{w});
	\]
	here we first apply the $\sup$ over $\mathbf{w}$, then the heat semigroup in $\mathbf{x}$, and then the infimum over $\mathbf{v}$.  Thus, $\mathbf{w}$ should be allowed to depend on $\mathbf{z}_t$, while $\mathbf{v}$ does not, and so $\mathbf{w} \in \cM_t^m$ and $\mathbf{v} \in M_0^m$.
	
	Therefore, in order for our class of test functions to be closed under both partial suprema and infima and the heat semigroup, we want to study a filtration $(M_t)_{t \geq 0}$ of von Neumann algebras and a compatible Brownian motion $\mathbf{z}_t$ (this means that $M_t$ is a von Neumann subalgebra of $\cM$ and $M_s \subseteq M_t$ for $s \leq t$; $\mathbf{z}_t$ is in $M_t$ and $\mathbf{z}_t - \mathbf{z}_s$ is freely independent of $M_s$ for $s \leq t$).  Each supremum or infimum will have a specified domain which is an operator-norm ball one of the $M_t$'s.  Moreover, we will assume that all the $\sup$ and $\inf$ operations occur in chronological order, meaning that if $s \leq t$, a supremum over $M_s$ will always occur \emph{outside} of a infimum over $M_t$, and so forth.  Starting from scalar-valued trace polynomials and repeatedly applying continuous functions, substitution of the Brownian motion for some of the arguments, and partial suprema and infima in chronological order produces a class of test functions we call \emph{chronological formulas}.  The associated analog of probability distributions is called the \emph{chronological type}.
	
	Fortunately, the effort undertaken to define and study the space of chronological formulas (and its completion) will be rewarded.
	\begin{enumerate}[(1)]
		\item We will obtain both a Monge--Kantorovich duality and an infinitesimal change-of-variables formula for the associated free entropy $\chi_{\chron}^{\cU}$, and as a consequence show concavity of $\chi_{\chron}^{\cU}$ along Wasserstein geodesics.
		\item We will also show that the heat evolution of chronological types is the gradient flow for $\chi_{\chron}^{\cU}$ in the metric sense, that is, it satisfies the evolution variational inequality \cite{MurSav2020gradient}. 
		\item We can define conditional entropy and obtain a chain rule for entropy under conditioning.
		\item We will make some progress in understanding the large deviations problem for several random matrices; namely, we will prove a large deviations principle along each ultrafilter on $\N$, with the rate function given by a stochastic optimization problem, which can also be expressed as a limit of chronological formulas.
	\end{enumerate}
	Beyond the specific free probabilistic content, this article constitutes a case study of how Wasserstein geometry interacts with ultraproducts, and in this connection we remark that Warren \cite{Warren2023ultralimits} has studied the behavior of Wasserstein spaces under taking ultralimits of the underlying metric spaces.  The noncommutative setting is particularly challenging because of the model-theoretically wild behavior of the limiting objects, which requires the development of new spaces and new techniques.  Besides the techniques from Voiculescu's free entropy theory \cite{VoiculescuFE2}, this paper incorporates analysis of ultraproducts and randomness from \cite{BY2012,GaoJekel2025}, noncommutative Monge--Kantorovich duality from \cite{GJNS2021,JekelTypeCoupling}, matrix models for such optimal couplings from \cite{Jekel2025InfoGeom}, noncommutative stochastic optimization problems from \cite{GJNPviscositysolution,GJNPmatrixcontrols}
	
	\subsection{Main Results} \label{subsec: intro results}
	
	We use the following terminology and conventions.
	\begin{itemize}
		\item A \emph{tracial von Neumann algebra} is a von Neumann equipped with a specified faithful, normal, tracial state.
		\item A \emph{noncommutative filtration} is a family of tracial von Neumann algebra $(M_t)_{t \in [0,\infty]}$ with $M_s \subseteq M_t$ for $s \leq t$.
		\item We consider a free Brownian $\mathbf{z}_t = (z_{t,j})_{j \in \N}$ whose operator real and imaginary parts form a free semicircular family with common variance $t$, such that $z_{t,j} \in M_t$ and $\mathbf{z}_{s,t} = \mathbf{z}_t - \mathbf{z}_s$ is free from $M_s$ when $s \leq t$.
		\item We use boldface for tuples (despite that many works in continuous model theory use overline instead).  Lowercase letters will be used for elements of general von Neumann algebras and uppercase for matrices.  We also write $\mathbb{M}_n$ for the $n \times n$ complex matrices.
		\item We work with general non-self-adjoint tuples.
	\end{itemize}
	Each non-self-adjoint variable could be expressed equivalently using two self-adjoint variables by taking the operator real and imaginary parts.  However, the non-self-adjoint version is convenient for applying model theory of von Neumann algebras, since we take suprema and infima over operator norm balls rather than sets of self-adjoint operators in the operator norm balls.  Hence, the use of non-self-adjoint variables makes the notation somewhat simpler, although it does not materially affect the results.
	
	As described at the end of the previous section, we consider \emph{chronological formulas} which are generated recursively by the following operations:
	\begin{itemize}
		\item Starting with the real part of the trace of some noncommutative polynomial in the given variables and in the increments of the process $\mathbf{z}_{s,t}$.
		\item Applying a continuous function $f: \R^k \to \R$ to some existing formulas $\varphi_1$, \dots, $\varphi_k$.
		\item Applying the supremum or infimum over an operator-norm ball in $M_0$ in one of the free variables, which is thus no longer a free variable,
	\end{itemize}
	with the restriction that the suprema and infima occur in chronological order; see Definition \ref{def: chronological formulas}.  We also study the space of \emph{chronologically definable predicates} which is a natural completion of the space of chronological formulas (Definition \ref{def: chron def pred}).
	
	The space of chronological formulas in $m$ free variables is denoted $\mathcal{F}_{\chron,m}$.  The chronological type $\tp_{\chron}^{\cM}(\mathbf{x})$ of an $m$-tuple $\mathbf{x} \in M_0^m$ is the mapping $\mathcal{F}_{\chron,m} \to \R$ given by sending a chronological formula $\varphi$ to its evaluation $\varphi^{\cM}(\mathbf{x})$.  We will define a variant of free entropy $\chi_{\chron}^{\cU}(\mu)$ where $\mu$ is the chronological type of some $m$-tuple.  The entropy $\chi_{\chron}^{\cU}(\mu)$ is defined in terms of the Lebesgue measure or the Gaussian measure of certain spaces of matrix approximations, in the same way as Voiculescu's free entropy $\chi$ in \cite{VoiculescuFE2}; we take the log-volume of the space of matrix approximations, divide by $n^2$, and take the limit along an ultrafilter $\cU$ on $\N$.  However, the relationship between the matrix approximations and the tuple $\mathbf{x}$ in the filtered noncommutative probability space is considerably more complicated, as we explain in \S \ref{subsec: intro filtrations}.
	
	For the entropy for chronological types $\chi_{\chron}^{\cU}$, we are able to prove several foundational properties for information geometry.  This is the first time that such properties have been shown for any version of multivariate free entropy.  We first highlight the property of geodesic concavity with respect to the $L^2$-Wasserstein distance.  As one might expect the Wasserstein distance of two types $\mu_0$ and $\mu_1$ is defined as the minimal $L^2$-distance $\norm{\mathbf{x}_0 - \mathbf{x}_1}_2$ where $\mathbf{x}_0$ and $\mathbf{x}_1$ are $m$-tuples with types $\mu_0$ and $\mu_1$ respectively.
	
	\begin{introthm}[Geodesic concavity; see {Theorem \ref{thm: geodesic concavity}}] \label{introthm: geodesic concavity}
		Let $\mathbf{x}_0$ and $\mathbf{x}_1$ be $m$-tuples which form an optimal coupling of chronological types $\mu_0$ and $\mu_1$.  Let $\mu_t$ be the chronological type of $(1-t)\mathbf{x}_0 + t \mathbf{x}_1$.  Then $t \mapsto \chi_{\chron}^{\cU}(\mu_t)$ is concave.
	\end{introthm}
	
	There are several important ingredients for this geodesic concavity theorem:
	\begin{enumerate}[(1)]
		\item We prove an analog of Monge--Kantorovich duality for the space of chronological types in Theorem \ref{thm: MK duality}, in much the same way as \cite[Theorem 1.1]{JekelTypeCoupling}.
		\item We prove an infinitesimal change of variables for free entropy in Theorem \ref{thm: infinitesimal change of variables} (analogous to \eqref{eq: classical change of variables}) which works in general for chronologically definable predicates with limited regularity.
		\item Similar to \cite[Lemma 4.9]{Jekel2025InfoGeom}, we prove that there exist random matrix models $\mathbf{X}_0^{(n)}$ and $\mathbf{X}_1^{(n)}$ for $\mathbf{x}_0$ and $\mathbf{x}_1$, such that $\mathbf{X}_0^{(n)}$ and $\mathbf{X}_1^{(n)}$ form a classical optimal coupling and also the normalized classical entropy of $\mathbf{X}_0^{(n)}$ approximates $\chi_{\chron}^{\cU}(\mathbf{x}_0)$; see Theorem \ref{thm: optimal coupling lifts}.  This result also heavily relies on the Monge--Kantorovich duality in item (1).
	\end{enumerate}
	The ingredients (1) and (2) are in fact sufficient to obtain concavity in the interior of the geodesic.  For $0 < s < t < 1$, we express convex functions associated to $\mathbf{x}_s$ and $\mathbf{x}_t$ in the Monge--Kantorovich duality using Legendre transforms and inf-convolutions, show that the infinitesimal change of variables for entropy can be applied to them, and give bounds for the Laplacian of the inf-convolution in terms of that of the original function.  The behavior at the endpoints of the geodesic requires special attention, and this is handled using item (3) in the same way as in \cite[Lemma 4.10]{Jekel2025InfoGeom}.
	
	Item (3) is also the key ingredient for proving the evolution variational inequality (EVI), which is a metric version of the statement that the heat evolution gives the Wasserstein gradient flow of entropy (see \cite{AGS2008gradient,MurSav2020gradient}; similar inequalities also played a key role in the study of quantum Wasserstein gradient flow by Carlen and Maas in \cite[\S 8]{CarlenMaas2017gradient}).  Recall that the sign convention here is that $\chi_{\chron}^{\cU}$ should be geodesically \emph{concave} and we consider the \emph{upward} gradient flow.  The version of the EVI that we use here is the second integral characterization EVI$_0$'' from \cite[Theorem 3.3]{MurSav2020gradient} with $\lambda = 0$, applied to $-\chi_{\chron}^{\cU}$.  The factor of $1/2$ on the left-hand side is because we express the heat semigroup in terms of Brownian motion, which puts a factor of $1/2$ in front of the Laplacian in the heat equation.
	
	\begin{introthm}[Evolution variational inequality; see {Theorem \ref{thm: EVI}}] \label{introthm: EVI}
		Let $\nu_0$ be the chronological type of some $m$-tuple $\mathbf{x}_0$ in $\cQ^m$ and let $\nu_t$ be the chronological type of $\mathbf{x}_0 + \mathbf{z}_t$ where $\mathbf{z}_t$ is the Brownian motion in $\cQ$.  Let $\sigma$ be another chronological type.  Then we have the evolution variational inequality
		\[
		\frac{1}{2} \left[ d_W(\nu_t,\sigma)^2 - d_W(\nu_s,\sigma)^2 \right] \leq (t-s) \left[ \chi_{\chron}^{\cU}(\nu_t) - \chi_{\chron}^{\cU}(\sigma) \right] \text{ for } 0 \leq s \leq t < \infty.
		\]
	\end{introthm}
	
	In fact, the evolution variational inequality implies geodesic concavity by \cite[Theorem 3.2]{DanSav2008eulerian}, so in fact Theorem \ref{introthm: geodesic concavity} is a consequence of Theorem \ref{introthm: EVI}.  Our proof of Theorem \ref{introthm: geodesic concavity}, however, has independent interest, since it gives a more specific description of the derivative of entropy along the geodesic in Theorem \ref{thm: geodesic concavity} (2) and thereby demonstrates the viability of change of variables for free entropy in a much more general setting.
	
	Although ingredient (3) already holds in the setting of types in the language of tracial von Neumann algebras, without any filtration, this is not enough for Theorem \ref{introthm: EVI} to hold in this setting.  The problem is that we do not know whether the type of $\mathbf{x}_0 + \mathbf{z}_t$ is uniquely determined by the type of $\mathbf{x}_0$, so that $\nu_t$ is not well-defined to begin with.  If there were a way to obtain the type of $\mathbf{x}_0 + \mathbf{z}_t$ from the type of $\mathbf{x}_0$ (specifically for the random matrix filtration $Q_t$ described in Proposition \ref{introprop: quotient filtration} below), then many of the results in the paper could likely be proved in a simpler framework.  However, this question seems difficult and out of reach by current methods.
	
	In this work, we also study the conditional version of the entropy.  We define $\chi_{\chron}^{\cU}(\mathbf{x} \mid \mathbf{y})$, roughly speaking, by fixing deterministic matrix approximations $\mathbf{Y}^{(n)}$ for the tuple $\mathbf{y}$, and then studying the measure of the space of compatible matrix approximations for $\mathbf{x}$ (analogous to Shlyakhtenko's conditional microstate entropy in \cite{Shlyakhtenko2002}).  In fact, Theorems \ref{introthm: geodesic concavity} and \ref{introthm: EVI} hold in the more general setting of conditional entropy and Wasserstein distance.
	
	A key point is that for the setting of chronological types, the resulting conditional entropy is independent of the choice of deterministic matrix approximations $\mathbf{Y}^{(n)}$ (see Theorem \ref{thm: final formula for entropy}).  This is the first time that this property has been proved for any version of conditional free entropy outside of the very restricted cases (1) where $\mathbf{y}$ generates an amenable von Neumann algebra and (2) where $\mathbf{y}$ is freely independent of $\mathbf{x}$.  The intuition is that the chronological type has enough information to pin down the behavior of $\mathbf{y}$ exactly, so that two different approximations cannot be chosen which would exhibit different behaviors in any way that affects entropy (see the discussion around Theorem \ref{introthm: entropy in terms of pressure} below).  As a consequence, we can obtain a chain rule for entropy under iterated conditioning, while for previous versions of free entropy one can only get inequalities (see \cite[Proposition 2.5]{Shlyakhtenko2002}).
	
	\begin{introthm}[Chain rule for entropy; see Theorem {\ref{thm: chain rule for entropy}}] \label{introthm: chain rule}
		Let $\mathbf{x}$, $\mathbf{y}$, and $\mathbf{w}$ be tuples of length $m_1 \in \N$, $m_2 \in \N$, and $m_3 \in \N \cup \{0,\infty\}$.  Then
		\[
		\chi_{\chron}^{\cU}(\mathbf{x}, \mathbf{y} \mid \mathbf{w}) = \chi_{\chron}^{\cU}(\mathbf{x} \mid \mathbf{y}, \mathbf{w}) + \chi_{\chron}^{\cU}(\mathbf{y} \mid \mathbf{w}).
		\]
	\end{introthm}
	
	We also show another natural invariance property, namely, that if $\mathbf{y}$ and $\mathbf{y}'$ generate the same von Neumann algebra, then $\chi_{\chron}^{\cU}(\mathbf{x} \mid \mathbf{y}) = \chi_{\chron}^{\cU}(\mathbf{x} \mid \mathbf{y}')$, which was already known for other versions of conditional free entropy; see \cite[Proposition 2.15]{Shlyakhtenko2002} and \cite[Lemma 4.4]{JekelPi2024}.  In our setting, we can upgrade the von Neumann algebra generated by $\mathbf{y}$ to the \emph{algebraic closure} in the model-theoretic sense.  See \S \ref{subsec: invariance} and compare also \cite[\S 3]{JekelTypeCoupling} and \cite[\S 4.6]{JekelCoveringEntropy}.
	
	\begin{remark}
		We warn the reader that for the chronological entropy to be finite, the chronological type of $(\mathbf{x},\mathbf{y})$ has to admit some matrix approximations as $n \to \cU$, which in particular means that the structure $\cM$ has to be chronologically elementarily equivalent to the matrix ultraproduct $\cQ$ described in the next section (see Remark \ref{rem: m prime zero case}). However, this does not mean that $\chi_{\chron}^{\cU}$ (or, similarly, the entropy for full types in \cite{JekelModelEntropy}) is useless for studying other von Neumann algebras.  Rather, the notion of types and the corresponding free entropy quantities are naturally suited for studying von Neumann \emph{subalgebras} inside $\cQ$.  Any Connes-embeddable von Neumann algebra can be embedded in $\cQ$, usually in many distinct ways.  In general, spaces of types can also be viewed as spaces of embeddings into an ultraproduct up to automorphic equivalence as in \cite{AGKE2022}.  One can also use variational principles as in \cite[Lemma 4.3 and 4.7]{JekelModelEntropy} to relate the more complicated entropy for types with the original free entropy for noncommutative laws and entropy in the presence.  See \cite{JKE2026upgraded} for an application of these techniques.
	\end{remark}

	\subsection{Classical and noncommutative filtrations} \label{subsec: intro filtrations}
	
	In defining $\chi_{\chron}^{\cU}$, the main question is what to use for the space of matrix approximations for a chronological type.  In the original microstate version of free entropy \cite{VoiculescuFE2}, matrix approximations are defined in terms of moments; we look at matrix tuples $\mathbf{X}$ with $\tr_n(p(\mathbf{X})) \approx \tr^{\cM}(p(\mathbf{x}))$ for noncommutative polynomials $p$.  In free entropy for types from \cite{JekelModelEntropy}, we consider formulas in the language of tracial von Neumann algebras and demand that $\varphi^{\mathbb{M}_n}(\mathbf{X}) \approx \varphi^{\cM}(\mathbf{x})$.  In both cases, the test functions are interpreted as pointwise functions on $\mathbb{M}_n$, so that there is no classical randomness involved in the definition of the microstate space.  This fits well with the concentration of measure phenomenon that applies to many of the random matrix models that we care about.  However, for chronological types, this approach cannot be applied na{\"\i}vely since we need a filtration rather than a single von Neumann algebra in order to evaluate a formula.  Clearly, $\mathbb{M}_n$ does not have a suitable noncommutative filtration indexed by real $t \geq 0$, since any chain of nested subalgebras must have finite length!
	
	Instead, we work with the random matrix algebra $L^\infty(\Omega,\mathcal{G},\mathbb{P}) \otimes \mathbb{M}_n$ and consider a classical filtration $\mathcal{G}_t$ of $\sigma$-algebras.  We consider Brownian motions $Z_{t,j}^{(n)}$ on $\mathbb{M}_n$ for $t \geq 0$ and $j \in \N$, whose operator real and imaginary parts are GUE-Brownian motion.  The Brownian motion $(Z_{t,j}^{(n)})_{j \in \N}$ is assumed to be compatible with the filtration $\mathcal{G}_t$ (i.e. it is progressively measurable and the increment $(Z_{t,j}^{(n)} - Z_{s,j}^{(n)})_{j \in \N}$ is independent of $\mathcal{G}_s$ for $s < t$).  This Brownian motion should play an analogous role to the process $z_{t,j}$ in our chronological formulas.  We then define a pointwise function $\Lambda_\varphi^{(n)}: \mathbb{M}_n^m \to \R$ associated to each chronological formula $\varphi$ in $m$ variables.  The main difference between the construction of $\Lambda_{\varphi}^{(n)}$ and the evaluation of $\varphi$ for a noncommutative filtration is that in defining $\Lambda_{\varphi}^{(n)}$ we take the classical expectation at appropriate steps in the construction.  As an illustrative example, suppose that $\varphi$ is a chronological formula of the form $\varphi(\mathbf{x}) = \psi(\mathbf{x}, z_{0,t,1},\dots,z_{0,t,m})$, where $\psi$ is a formula where all the quantifiers occur for times $\geq t$.  Then we want to take
	\[
	\Lambda_\varphi^{(n)}(\mathbf{X}) = \mathbb{E} \Lambda_{\psi}^{(n)}(\mathbf{X}, Z_{0,t,1}^{(n)},\dots,Z_{0,t,m}^{(n)}).
	\]
	The precise construction of $\Lambda_{\varphi}^{(n)}$ is described in Proposition \ref{prop: approx by pointwise function}.  Importantly, we rely on concentration of measure for the Gaussian random matrices to relate the expected value of $\Lambda_{\psi}^{(n)}(\mathbf{X}, Z_{0,t,1}^{(n)},\dots,Z_{0,t,m}^{(n)})$ with the pointwise value (in fact, this is the reason why we need to use chronological formulas rather than general formulas for noncommutative filtrations as metric structures; because the quantifiers occur in chronological order, we are able to handle them one at a time and have Lipschitz dependence on the random increments $Z_{s,t,j}^{(n)}$ at each step, and thus apply concentration of measure).  These pointwise functions $\Lambda_{\varphi}^{(n)}$ are used in place of $\varphi$ itself to define the spaces of matrix approximations in the definition of $\chi_{\chron}^{\cU}$.
	
	We furthermore need to relate the classical random matrix models and the classical filtration with some noncommutative filtration.  From a free probability viewpoint, a natural candidate is the filtration $M_t$ generated by the free Brownian motion itself.  Unfortunately, it is not clear if this filtration adequately captures the large-$n$ behavior of $\Lambda_\varphi^{(n)}$ on matrices.  Indeed, even without considering filtrations, we do not know if the \emph{theory} of $\mathbb{M}_n$ (in the model-theoretic) sense converges as $n \to \infty$ (and in fact an analogous limit for permutation groups does not exist \cite{AlTh2024}).  Moreover, if the limit does exist, there is no reason it would agree with the theory of the von Neumann algebra generated by a free Brownian motion.
	
	Instead, we consider a filtration constructed from the random matrix algebra $L^\infty(\Omega,\mathcal{G},\mathbb{P}) \otimes \mathbb{M}_n$ using ultraproducts.  Take
	\begin{align} 
		M &= \prod_{n \to \cU} L^\infty(\Omega,\mathcal{G},\mathbb{P}) \otimes \mathbb{M}_n \\ \label{eq: random matrix ultraproduct intro} \\
		M_t &= \prod_{n \to \cU} L^\infty(\Omega,\mathcal{G}_t,\mathbb{P}) \otimes \mathbb{M}_n. \label{eq: random matrix ultraproduct intro 2}
	\end{align}
	We also take $z_{t,j}$ in $M$ to be the equivalence class of the sequence $[Z_{t,j}^{(n)}]_{n \in \N}$ in $M$ (technically, we have to apply an operator-norm cutoff function to make this well-defined).  After first taking the ultraproduct, we then must remove the classical randomness in a sense by taking a quotient, so that we can obtain a truly noncommutative filtration.  To see why a quotient is needed, consider for instance that the increment $z_{t,j} - z_{s,j}$ is not free from $M_s$ with respect to the trace since $z_{t,j} - z_{s,j}$ commutes with the center $Z(M_s) = \prod_{n \to \cU} L^\infty(\Omega,\mathcal{G}_t,\mathbb{P})$ (we have instead freeness with amalgamation over the center).  We therefore follow the method of \cite{GaoJekel2025} and take a quotient of $M$ by a maximal ideal to obtain a $\mathrm{II}_1$ factor.  Specifically, let $\cV$ be a pure state or character on the center $Z(M)$, and consider the ideal $I_{\cV} = \{x \in M: \cV \circ E_{Z(M)}[x^*x] = 0\}$.  Let $Q = M / I_{\cV}$ and $\pi: Q \to M$ be the quotient map, and $Q_t = \pi(M_t)$.
	
	\begin{introprop}[See Propositions {\ref{prop: approx by pointwise function}} and {\ref{prop: matrix quotient filtration}}] \label{introprop: quotient filtration}
		Let $M_t$ and $Q_t$ be as above.
		\begin{enumerate}[(1)]
			\item $(Q_t)_{t \geq 0}$ is filtration of von Neumann algebras such that for $s \leq t$, the inclusion $Q_s \subseteq Q_t$ is elementary, and $Q_t$ is elementarily equivalent to the matrix ultraproduct $\prod_{n \to \cU} \mathbb{M}_n$.
			\item  $(\pi(z_{t,j}))_{t \geq 0, j \in \N}$ is a free Brownian motion compatible with the filtration $(Q_t)_{t \geq 0}$.
			\item Let $\cQ$ be the filtration $(Q_t)_{t \geq 0}$ together with the Brownian motion described above.  Let $\mathbf{X}^{(n)}$ be an $m$-tuple of random matrices from $L^\infty(\Omega,\mathcal{G}_0,\mathbb{P})$ which is bounded in operator norm by some $R$.  Let $[\mathbf{X}^{(n)}]_{n \in \N}$ be the corresponding equivalence class in $M_0^m$.  For chronological formulas $\varphi$, we have
			\[
			\varphi^{\cQ}(\pi([\mathbf{X}^{(n)}]_{n \in \N})) = \cV \left( [\Lambda_{\varphi}^{(n)}(\mathbf{X}^{(n)})]_{n \in \N} \right);
			\]
			here note that $[\Lambda_{\varphi}^{(n)}(\mathbf{X}^{(n)})]_{n \in \N} \in \prod_{n \to \cU} L^\infty(\Omega,\mathcal{G},\mathbb{P}) = Z(M)$.
		\end{enumerate}
	\end{introprop}
	
	This proposition shows in particular that the properties that we care about are independent of the choice of the character $\cV$ used to define the quotient $Q$.  The properties above also make the filtration $(Q_t)_{t \geq 0}$ a suitable noncommutative filtration for studying the behavior of the matrix Brownian motion in the limit as $n \to \cU$.  The proof of Proposition \ref{introprop: quotient filtration} relies on the technology of \cite{GaoJekel2025} developed to study the relationship between direct integrals and model theory of von Neumann algebras (see also \cite{BY2012} and \cite{BYIT2024} for related more general results on metric structures).
	
	With the functions $\Lambda_{\varphi}^{(n)}$ and the filtration $(Q_t)_{t \geq 0}$ in hand, we can now describe why the chronological type leads to a better behaved notion of entropy than in previous settings (see comments before Theorem \ref{introthm: chain rule}).  For this section, it is convenient to work with
	\[
	\tilde{\chi}_{\chron}^{\cU}(\mathbf{x} \mid \mathbf{y}) = \frac{1}{2} \norm{\mathbf{x}}_2^2 - \chi_{\chron}^{\cU}(\mathbf{x} \mid \mathbf{y}),
	\]
	which corresponds to using the Gaussian measure rather than the Lebesgue measure of the spaces of matrix microstates, and making a sign change so that $\tilde{\chi}_{\chron}^{\cU}$ is nonnegative.  Fix a tuple of matrix approximations $\mathbf{Y}^{(n)}$ for $\mathbf{y}$.  In light of Varadhan's lemma and its converse in Bryc's theorem, we approximate the log of the Gaussian measure of the microstate spaces by quantities of the form
	\[
	\mathcal{P}^{(n)}(\varphi) = -\frac{1}{n^2} \log \int_{\mathbb{M}_n^m} \exp(-n^2 \Lambda_\varphi^{(n)}(\mathbf{X}))\,d\sigma^{(n)}(\mathbf{X},\mathbf{Y}^{(n)}),
	\]
	where $\varphi$ is a suitable chronological formula and $\sigma^{(n)}$ is the Gaussian measure on $\mathbb{M}_n^m$ described by an $m$-tuple of matrices whose real and imaginary parts are GUE (in the case where $\varphi$ is the trace of noncommutative polynomial, this is exactly Hiai's free pressure 
	\cite{Hiai2005}).  We then apply Borell's formula that
	\begin{equation} \label{eq: matrix control problem intro}
		\mathcal{P}^{(n)}(\varphi)(\mathbf{Y}^{(n)}) = \inf_{\alpha} \mathbb{E} \left[ \Lambda_\varphi^{(n)}\left(\mathbf{Z}_1^{(n)} + \int_0^1 \alpha_t\,dt,\mathbf{Y}^{(n)} \right) + \frac{1}{2} \int_0^1 \norm{\alpha_t}_2^2\,dt \right],
	\end{equation}
	where $\mathbf{Z}_t^{(n)} = (Z_{t,j}^{(n)})_{j=1}^m$ is a matrix Brownian motion in $\mathbb{M}_n^m$, and $\alpha$ ranges over progressively measurable processes $[0,1] \to \mathbb{M}_n^m$.  This is an application of the same technique as in \cite{GJNPmatrixcontrols}, and also related to Dabrowski's approach to large deviations in \cite{Dabrowski2017Laplace}.
	
	Equation \ref{eq: matrix control problem intro} is a stochastic control problem in $\mathbb{M}_n^m$ with continuous time.  However, we can make a discrete-time approximation with error bound independent of $n$, as well as arrange that $\alpha$ is bounded in operator norm (see \S \ref{subsec: stochastic variational}).  The discretized control problem is something that we can represent using chronological formulas, and hence apply Proposition \ref{introprop: quotient filtration} (3).  We thus arrive at the following result.
	
	\begin{introthm}[See {Theorem \ref{thm: final formula for entropy}}] \label{introthm: entropy in terms of pressure}
		With the notation above, for $\varphi \in \overline{\mathcal{F}}_{\chron,m}^0$ (see Definition \ref{def: uniform of completion of RCF}), define
		\begin{equation} \label{eq: variational problem intro}
		\mathcal{P}^{\cU}(\varphi)(\mathbf{y}) = \inf_{\alpha} \left[ \varphi^{\cQ}\left(\pi(\mathbf{z}_1) + \int_0^1 \alpha_t\,dt, \mathbf{y} \right) +  \frac{1}{2} \int_0^1 \norm{\alpha_t}_2^2\,dt \right],
		\end{equation}
		where $\mathbf{z}_t = (z_{t,j})_{j=1}^m$ is the free Brownian motion in $\cQ$ and $\alpha$ ranges over measurable maps $[0,1] \to \cQ_{\sa}^m$ which are bounded in operator norm and adapted to the noncommutative filtration $(\cQ_t)_{t \in [0,1]}$.  Then
		\[
		\lim_{n \to \cU} \mathcal{P}^{(n)}(\varphi)(\mathbf{Y}^{(n)}) = \mathcal{P}^{\cU}(\varphi)(\mathbf{y}),
		\]
		so in particular this is independent of the choice of matrix approximations $\mathbf{Y}^{(n)}$ for $\mathbf{y}$.  Moreover,
		\[
		\tilde{\chi}_{\chron}^{\cU}(\mathbf{x}) = \sup_{\varphi \in \overline{\mathcal{F}}_{\chron,m}^0} \left[ \mathcal{P}^{\cU}(\varphi)(\mathbf{y}) -  \varphi^{\cQ}(\mathbf{x},\mathbf{y}) \right].
		\]
	\end{introthm}
	
	The fact that $\mathcal{P}(\varphi)$ is a chronologically definable predicate is precisely why the entropy $\chi_{\chron}^{\cU}(\mathbf{x} \mid \mathbf{y})$ does not depend on the choice of matrix approximations for $\mathbf{y}$.
	
	To illustrate this result more concretely, let us return to the setting without conditioning on another tuple $\mathbf{y}$ and without sup and inf quantifiers.  Consider an $m$-tuple $\mathbf{x}$ and the ordinary microstates free entropy from \cite{VoiculescuFE2} with respect to its noncommutative $*$-distribution, or equivalently its quantifier-free type.  We denote the corresponding free entropy along the ultrafilter $\cU$ by $\tilde{\chi}_{\qf}^{\cU}$ to highlight the use of quantifier-free formulas.
	
	\begin{introcor} \label{introcor: LDP}
		Let $\mathcal{T}$ be the set of trace $*$-polynomials in $(x_j + i)^{-1}$ for $j = 1$, \dots, $m$.  Then the free entropy with respect to the Gaussian measure is given by
		\[
		\tilde{\chi}_{\qf}^{\cU}(\mathbf{x}) = \sup_{\varphi \in \mathcal{T}} \left[ \mathcal{P}^{\cU}(\varphi) -  \varphi^{\cQ}(\mathbf{x}) \right].
		\]
		Moreover, the noncommutative distribution of a Ginibre matrix tuple $Z_1^{(n)}$, \dots, $Z_m^{(n)}$ satisfies a large deviations principle as $n \to \infty$ (see \S \ref{subsec: LDP} for definitions) if and only if $\mathcal{P}^{\cU}(\varphi)$ is independent of the ultrafilter $\cU$ for all trace $*$-polynomials in $(x_j + i)^{-1}$ for $j = 1$, \dots, $m$, and in this case the right-hand side serves as the rate function.
	\end{introcor}
	
	\begin{remark} \label{rem: existential embedding}
	If one believes that $\mathcal{P}^{\cU}(\varphi)$ is independent of the choice of $\cU$ when $\varphi$ is quantifier-free, one natural candidate would be the value of the variational problem using only the filtration generated by a Brownian motion as in Example \ref{ex: filtration of free Brownian motion}.  More precisely, consider a free Brownian motion increments $\mathbf{z}_{s,t}$ indexed by the real line, and let $N_t$ be the von Neumann algebra generated by $\mathbf{z}_{s_1,s_2}$ for $s_1, s_2 \leq t$, and let $N = N_\infty$.  Let $\mathcal{P}^{\operatorname{free}}(\varphi)$ be defined by \eqref{eq: variational problem intro} with measurable $\alpha: [0,1] \to N_\infty$ such that $\alpha(t) \in N_t$ for every $t$.  Then does $\lim_{n \to \infty} \mathcal{P}^{(n)}(\varphi) = \mathcal{P}^{\operatorname{free}}(\varphi)$, or equivalently, does $\mathcal{P}^{\cU}(\varphi) = \mathcal{P}^{\operatorname{free}}(\varphi)$ for all ultrafilters $\cU$?
	
	We note that by the same reasoning as Proposition \ref{introprop: quotient filtration}, random matrix models produce a natural embedding of $N_t$ into $Q_t$ for $t \geq 0$.  From this, it is immediate that $\mathcal{P}^{\cU}(\varphi) \leq \mathcal{P}^{\operatorname{free}}(\varphi)$ since it is an infimum over a larger set.  Now let $\mathcal{N} = (N_t)_{t \in [0,\infty]}$ be the metric structure corresponding to the filtration.  The embedding of $\mathcal{N}$ into $\mathcal{Q}$ is said to be \emph{existential} if for every quantifier-free formula $\psi$ and $\mathbf{x} \in \mathcal{N}^m$,
	\[
	\inf_{y_1 \in D_{t_1,r_1}^{\cN}, \dots, y_{m'} \in D_{t_m,r_m}^{\cN}} \psi^{\mathcal{N}}(\mathbf{x},\mathbf{y}) = \inf_{y_1 \in D_{t_1,r_1}^{\cQ}, \dots, y_{m'} \in D_{t_m,r_m}^{\cQ}} \psi^{\mathcal{Q}}(\mathbf{x}),
	\]
	where $D_{t,r}^{\cN}$ is the $r$-ball in $N_t$ and likewise for $D_{t,r}^{\cQ}$.  If the embedding of $\cN$ into $\cQ$ is existential, then we would have $\mathcal{P}^{\cU}(\varphi) = \mathcal{P}^{\operatorname{free}}(\varphi)$, since the pressure can be approximated by inf-formulas using the discretization procedure in Lemmas \ref{lem: discrete approx of matrix pressure} and \ref{lem: discrete approximation of free pressure}.  Therefore, existentiality of the embedding of $\cN$ into $\cQ$ would imply the large deviations principle.  For related problems about random matrices and existential embeddings, see \cite{HJKE2026existential}.
	
	In the opposite direction, recent work of Peterson has shown that a $\mathrm{II}_1$ factor $M$ elementarily equivalent to the matrix ultraproduct cannot existentially embed into any $N \equiv L(\mathbb{F}_\infty)$ \cite[Theorem 7.10]{Peterson2026}, which immediately implies a filtration elemementarily equivalent to $\mathcal{Q}$ cannot existentially embed into a filtration elementarily equivalent to $\mathcal{N}$ as well.  In particular, $\mathcal{N}$ and $\mathcal{Q}$ are not elementarily equivalent, so there is no hope of $\mathcal{P}^{\operatorname{free}}(\varphi)$ being equal to $\mathcal{P}^{\cU}(\varphi)$ for general non-quantifier-free $\varphi$.
	\end{remark}
	
	\begin{remark}
	We do not know whether the ultraproduct factors $\mathcal{Q}$ are elementarily equivalent for different choices of the ultrafilter $\cU$.  In fact, in the somewhat analogous setting of ultraproducts of permutation groups as metric structures, different ultrafilters actually produce different elementarily equivalence classes \cite{AlTh2024}.  This suggests the strange possibility that the information geometry developed here might depend nontrivially on the choice of ultrafilter $\cU$, and it indicates how difficult the large deviations problem for random matrices really is.
	\end{remark}
	
	\subsection{Organization}
	
	\begin{itemize}
		\item \S \ref{sec: preliminaries} recalls relevant background and notation on von Neumann algebras, free probability, random matrices, and model theory for metric structures.
		\item \S \ref{sec: model theory and filtration} sets up the language for a filtration of von Neumann algebras and an associated stochastic process as metric structures.  In particular, it defines the chronological formulas and a model-theoretic version of stationarity.
		\item \S \ref{sec: matrix ultraproduct} applies the setup of \S \ref{sec: model theory and filtration} to quotients $\cQ$ of a random matrix ultraproduct $\cM$ as in \eqref{eq: random matrix ultraproduct intro}.  We use the techniques of \cite{GaoJekel2025} to relate chronological formulas evaluated in $\cQ$ to analogous objects in $\cM$. We thus define the finite-dimensional approximations $\Lambda_\varphi^{(n)}$ and prove Proposition \ref{introprop: quotient filtration}.
		\item \S \ref{sec: entropy} introduces the version of free entropy for chronological types, with respect to Gaussian measure.  We prove Theorem \ref{introthm: entropy in terms of pressure} to express the entropy in terms of stochastic control problems.  We then show Theorem \ref{introthm: chain rule} on the chain rule, establish invariance of $\chi_{\chron}^{\cU}(\mathbf{x} \mid \mathbf{y})$ with respect to $\mathbf{y}$, and conclude with Corollary \ref{introcor: LDP} on large deviations.
		\item \S \ref{sec: change of variables} introduces entropy for chronological types with respect to the Lebesgue measure.  After developing basic real analytic results for chronologically definable predicates and their finite-dimensional approximations, we define $\limsup$ and $\liminf$ versions of the Laplacian and prove the infinitesimal change of variables for entropy.
		\item \S \ref{sec: info geom} begins the proper study of information geometry.  We define a Wasserstein distance for a space of conditional chronological types.  We prove a version of Monge--Kantorovich duality.  We prove an extension theorem for random matrix models of optimal couplings.  We then conclude by proving Theorem \ref{introthm: geodesic concavity} on geodesic concavity and Theorem \ref{introthm: EVI} on the evolution variational inequality.
	\end{itemize}
	
	\subsection*{Acknowledgements}
	
	I thank the following colleagues for enlightening discussions that aided this work.  Jennifer Pi provided feedback on the exposition and the model-theoretic concepts. Charles-Philippe Diez explained many important aspects of information geometry and functional inequalities, especially during his visit to Copenhagen.  Kohei Suzuki highlighted the problem of proving evolution variational inequalities in discussions on March 22-24, 2026 at the conference ``Entropy and random matrices'' at CIRM in Marseille.
	
	\subsection*{Funding}
	
	The work was funded by an EU Horizon Marie Sk{\l}odowska Curie Action, FREEINFOGEOM, grant id: 101209517.  Views and opinions expressed are those of the author(s) only and do not necessarily reflect 
	those of the European Union or the Research Executive Agency. Neither the European Union nor the granting authority can be held responsible for them.
	
	\subsection*{AI Declaration}
	
	No large language models were used at any stage in the preparation of this article.
	
	\section{Preliminaries} \label{sec: preliminaries}
	
	\subsection{Von Neumann algebras} \label{subsec: von Neumann preliminaries}
	
	In this paper, a \emph{tracial von Neumann algebra}, or equivalently \emph{tracial $\mathrm{W}^*$-algebra} refers to a finite von Neumann algebras with a specified tracial state.  We recommend \cite{Ioana2023} for an introduction to the topic, as well as the following standard reference books \cite{KadisonRingroseI,Dixmier1969,Sakai1971,TakesakiI,Blackadar2006,Zhu1993}.
	
	\textbf{Tracial von Neumann algebras:} A \emph{von Neumann algebra} on a Hilbert space $H$ is a unital $*$-subalgebra of $B(H)$ which is closed in the weak operator topology (WOT) or equivalently closed in the strong operator topology (SOT).  Sakai \cite{Sakai1971} gave an abstract characterization of von Neumann algebras as $\mathrm{C}^*$-algebras which are also dual spaces as Banach spaces ($\mathrm{C}^*$-algebra in turn are abstractly characterized as Banach $*$-algebras satisfying $\norm{x^*x} = \norm{x}^2$).  A von Neumann algebra thus always has a weak-$*$ topology, and this coincides with the $\sigma$-WOT, which is the WOT for $M$ acting on $H^{\oplus \infty}$.
	
	A \emph{tracial von Neumann algebra} is a pair $\cM = (M,\tau)$ where $M$ is a von Neumann algebra and $\tau: M \to \C$ is a faithful normal tracial state, that is, a linear map satisfying
	\begin{itemize}
		\item \emph{positivity}: $\tau(x^*x) \geq 0$ for all $x \in M$
		\item \emph{unitality}: $\tau(1) = 1$
		\item \emph{traciality}: $\tau(xy) = \tau(yx)$ for $x, y \in A$
		\item \emph{faithfulness}:  $\tau(x^*x) = 0$ implies $x = 0$ for $x \in A$.
		\item \emph{weak-$*$ continuity}:  $\tau: M \to \C$ is weak-$*$ continuous.
	\end{itemize}
	We will often denote $\tau$ by $\tr^{\cM}$ for consistency with the model-theoretic notation introduced below.  We also denote by $D_r^{\cM}$ the operator-norm ball of radius $r$.
	
	We also use the following notations:
	\begin{itemize}
		\item $Z(M)$ denotes the center of $M$.
		\item $M_{\sa}$ denotes the set of self-adjoint elements.
		\item $D_r^M$ denotes the $r$-ball in $M$ with respect to operator norm.
	\end{itemize}

	\textbf{Real/imaginary and polar decompositions:} We will make frequent use of the operator real and imaginary parts:  For $x \in M$, write $\re(x) = (x + x^*)/2$ and $\im(x) = (x - x^*)/2i$.  Thus, $\re(x)$ and $\im(x)$ are self-adjoint and $x = \re(x) + i \im (x)$.  We also recall that $|x|$ is defined as $ (x^*x)^{1/2}$.  In any tracial von Neumann algebra, one can write express $x$ using a polar decomposition of the form $x = u|x|$ for some unitary $u$ (not necessarily unique).
	
	\textbf{noncommutative $L^p$ spaces:} Let $(M,\tau)$ be a tracial von Neumann algebra.  For $m \in \N$ and $\mathbf{x}, \mathbf{y} \in M^m$, denote
	\[
	\norm{\mathbf{x}}_p = \norm{\mathbf{x}}_{\tau,p} = \begin{cases} \left( \sum_{j=1}^m \tau(|x_j|^p) \right)^{1/p}, & p \in [1,\infty), \\
		\max_j \norm{x_j}, & p = \infty, \end{cases}
	\]
	where a plain $\norm{\cdot}$ always denotes the operator norm, and
	\[
	\ip{\mathbf{x},\mathbf{y}} = \ip{\mathbf{x},\mathbf{y}}_\tau = \sum_{j=1}^m \tau(x_j^* y_j).
	\]
	(In this paper, inner products are assumed to be linear in the \emph{right} coordinate.)  We recall that $M$ can be completed with respect to $\norm{\cdot}_p$ into a noncommutative $L^p$ space $L^p(M)$ (the dependence on $\tau$ is suppressed in the notation).  The adjoint operation $x \mapsto x^*$ always extends isometrically to $L^p(M)$.  The analogs of H{\"o}lder's inequality and $L^p$-$L^{p/(p-1)}$-duality also hold for tracial von Neumann algebra. In this paper, we will mostly use $L^2(M)$ and sometimes $L^1(M)$.
	
	\textbf{Embeddings and conditional expectations:} An \emph{embedding} or \emph{inclusion} of tracial von Neumann algebras always refers to a trace-preserving $*$-homomorphism, which is automatically injective because of faithfulness of the trace.  Given a tracial von Neumann algebra $(M,\tau)$ and a von Neumann subalgebra $(N,\tau|_{N})$, the orthogonal projection $L^2(M) \to L^2(N,\tau|_N)$ restricts to a map $E_N: M \to N$, and $E_N$ is contractive with respect to $\norm{\cdot}_p$ for all $p \in [1,\infty]$.  We refer to $E_N$ as the \emph{conditional expectation} from $M$ onto $N$.  The map $E_N$ is also a conditional expectation in the $\mathrm{C}^*$-algebraic sense; it is a completely positive $N$-$N$ bimodule map with $E_N|_N = \id_N$.  We point out that the choice of trace is necessary in order to uniquely determine $E_N$, even though the choice of trace is often suppressed in the notation.
	
	\textbf{Functional calculus:} Von Neumann algebras are always closed under application of Borel functional calculus and so in particular continuous functional calculus.  Thus, for every normal operator $x$ in a tracial von Neumann algebra $(M,\tau)$, there is an associated \emph{spectral measure} $\mu$ with respect to the trace, given by
	\[
	\int_{\C} f\,d\mu = \tau(f(x)) \text{ for } f \in C(\C);
	\]
	$\mu$ is a Borel probability measure supported on the spectrum of $x$, and in particular if $x$ is self-adjoint then $\mu \in \mathcal{P}(\R)$.  An important fact for us is that the application of a Lipschitz function on $\R$ by functional calculus will be Lipschitz on $L^2(M)_{\sa}$.  (Beware, however, that it is \emph{not} Lipschitz with respect to operator norm.)  This can be considered a folklore result in free probability or in operator theory, but we include the proof for the reader's convenience.
	
	\begin{lemma} \label{lem: Lipschitz functional calculus}
		Let $f: \R \to \R$ be an $L$-Lipschitz function.  Let $x$ and $y$ be self-adjoint elements in a tracial von Neumann algebra $(M,\tau)$.  Then
		\[
		\norm{f(x) - f(y)}_2 \leq L \norm{x - y}_2.
		\]
	\end{lemma}
	
	\begin{proof}
		Assume that $x$ and $y$ are bounded in operator norm by some constant $R > 0$.  Suppose that $f$ is polynomial.  Let $\partial f$ be the free difference quotient of $f$; $\partial$ is defined as the linear operator on $\C[t] \to \C[t] \otimes \C[t]$, such that when $\partial[t^k] = t^j \otimes t^{k-1-j}$.  Recall that there is a unique $\mathrm{C}^*$-tensor product $C([-R,R]) \otimes C([-R,R])$, or the maximal and minimal tensor products agree, and moreover, $C([-R,R]) \otimes C([-R,R]) \cong C([-R,R]^2)$ in the natural way.  Hence, $\partial[t^k]$ viewed as an element of $C([-R,R]) \otimes C([-R,R])$ corresponds to the two-variable function
		\[
		g(s,t) = \sum_{j=0}^{k-1} s^j t^{k-1-j} = \frac{s^k - t_k}{s-t}.
		\]
		Hence, for $f \in \C[t]$, the element $\partial f$ corresponds to $(f(s) - f(t)) / (s - t)$.  Hence,
		\[
		\norm{\partial f}_{C([-R,R]) \otimes C([-R,R])} = \norm{\frac{f(s) - f(t)}{s-t} }_{C([-R,R]^2)} = \norm{f}_{\Lip([-R,R])}.
		\]
		Let $\pi: C([-R,R]) \otimes C([-R,R]) \to B(L^2(M))$ be the $*$-representation $\pi(g \otimes h) z = g(x)zh(y)$.  By direct computation on monomials, we have
		\[
		f(x) - f(y) = \pi(\partial f)(x - y).
		\]
		Thus,
		\[
		\norm{f(x) - f(y)}_2 \leq \norm{\partial f}_{C([-R,R]) \otimes C([-R,R])} \norm{x - y}_2 \leq \norm{f}_{\Lip([-R,R])} \norm{x - y}_2.
		\]
		Now it is straightforward to show (using smooth curoffs and convolution) that if $f$ is $L$-Lipschitz on $\R$, then $f$ can be uniformly approximated on $[-R,R]$ by polynomials $f_k$ which are also $L$-Lipschitz on $[-R,R]$.  We have $f_k(x) \to f(x)$ and $f_k(y) \to f(y)$ in operator norm, hence also $\norm{\cdot}_2$, and the asserted inequality follows.
	\end{proof}
	
	\textbf{Truncation:} The following truncation trick will be useful for our study of suprema and infima over operator norm balls.
	
	\begin{lemma} \label{lem: projection onto R ball}
		Let $(M,\tau)$ be a tracial von Neumann algebra and $x \in M$. Let $R > 0$.  Let $F_R: \R \to \R$ be the function
		\begin{equation} \label{eq: function FR}
			F_R(t) = \min(\max(t,-R),R).
		\end{equation}
		\begin{enumerate}[(1)]
			\item There is a unique closest point to $x$ in $\{x': \norm{x'} \leq R\}$.
			\item Write a polar decomposition $x = u|x|$ where $u$ is unitary.  Then $x' = u F_R(|x|)$.
			\item $2R \norm{x - x'}_1 \leq \norm{x}_2^2 - \norm{x'}_2^2$.
		\end{enumerate}
	\end{lemma}
	
	\begin{proof}
		(1) Note that $\{x': \norm{x'} \leq R\}$ is a closed convex set in the Hilbert space $L^2(M,\tau)$.
		
		(2) Since the $R$-ball is invariant under left multiplication by unitaries, $u^* x'$ is the closest point to $|x|$.  Hence, it suffices to show that if $x \geq 0$, then $F_R(x)$ is the closest point to $x$.  Let $\norm{y} \leq R$.  Let $p = \mathbbm{1}_{(R,\infty)}(x)$ and note that $(x - F_R(x))p = x - F_R(x)$.  Then
		\begin{align*}
			\re \ip{F_R(x) - y, x - F_R(x)} &= \tau[(F_R(x) - \re(y))(x - F_R(x))] \\
			&= \tau[(F_R(x) - \re(y))p(x - F_R(x))p] \\
			&= \tau[p(F_R(x) - \re(y))p(x - F_R(x))].
		\end{align*}
		Now $p (F_R(x) - \re(y)) p = Rp - p \re(y) p \geq 0$ and $x - F_R(x) \geq 0$.  Hence, the trace of the product is nonegative, and so $\re \ip{F_R(x) - y, x - F_R(x)} \geq 0$ for all $y$ in the $R$-ball, which implies that $F_R(x)$ is the closest point.
		
		(3) Again, it suffices to consider the case when $x$ is positive.  Since $t - F_R(t)$ vanishes on $[0,R]$, we have
		\[
		t^2 - F_R(t)^2 = (t + F_R(t))(t - F_R(t)) \geq 2R(t - F_R(t)).
		\]
		Evaluating on $x$ and taking the trace, we see that
		\[
		\tau(x^2) - \tau(F_R(x)^2) \geq 2R \tau(x - F_R(x)) = 2R \norm{x - F_R(x)}_1.  \qedhere
		\]
	\end{proof}
	
	\textbf{Ultraproducts:}  We briefly recall the definition of ultrafilters and ultraproducts of tracial von Neumann algebras.  For background, see \cite[Appendix A]{BrownOzawa2008}, \cite[\S 2]{Capraro2010}, \cite[\S 5.4]{ADP}, \cite[\S 5.3]{GJNS2021}.
	
	An \emph{ultrafilter} $\cU$ on the natural numbers $\N$ is a collection of subsets $\cU$ which is closed under supersets and finite intersections, such that $\varnothing \not \in \cU$ for every $A \subseteq \N$, either $A$ or $A^c$ is in $\cU$.  For each $n \in \N$, there is an associated \emph{principal ultrafilter} $\{A \subseteq \N: n \in A\}$.  We are primarily concerned with the \emph{non-principal} or \emph{free} ultrafilters on $\N$.
	
	If $\cU$ is an ultrafilter on $\N$, we say that a set $A$ is \emph{$\cU$-large} if $A \in \cU$.  Similarly, a property is said to hold for \emph{$\cU$-many $n$} if the \emph{set} of $n$ that satisfies this property is an element of $\cU$.  Let $\cU$ be an ultrafilter on $\N$, and let $(x_n)_{n \in \N}$ be a sequence in some topological space.  We say that $\lim_{n \to \cU} x_n = x$ if for every neighborhood $O$ of $x$, we have $\{n: x_n \in O \} \in \cU$.  If $\cU$ is an ultrafilter on $\N$ and $x_n$ is a sequence in a compact Hausdorff space, then $\lim_{n \to \cU} x_n$ exists and is unique.
	
	We remark that ultrafilters on $\N$ are in bijection with characters on $\ell^\infty(\N)$; here ``character'' refers to a multiplicative linear functional $\ell^\infty(\N) \to \C$, or equivalently a pure state on $\ell^\infty(\N)$ as a $\mathrm{C}^*$-algebra, where the character $\varphi$ is given by $\varphi(f) = \lim_{n \to \cU} f(n)$.  The space of ultrafilters can thus be identified with the space of characters on $\ell^\infty(\N)$ or the Stone-{\v C}ech compactification of $\N$.
	
	Ultraproducts of tracial von Neumann algebras are defined as follows.  For $n \in \N$, let $\cM_n = (M_n,\tau_n)$ be a sequence of tracial von Neumann algebras.  Let $\prod_{n \in \N} M_n$ be the set of sequences $(x_n)_{n \in \N}$ such that $\sup_n \norm{x_n} < \infty$, which is a $\mathrm{C}^*$-algebra.  Let
	\[
	I_{\mathcal{U}} = \left\{(x_n)_{n \in \N} \in \prod_{n \in \N} M_n: \lim_{n \to {\mathcal{U}}} \norm{x_n}_{L^2(\cM_n)} = 0 \right\}.
	\]
	It turns out that $I_{\cU}$ is a closed two-sided ideal, so the quotient
	\[
	\prod_{n \to \cU} M_n := \prod_{n \in \N} M_n / I_{\cU}
	\]
	is a $\mathrm{C}^*$-algebra.  We denote by $[x_n]_{n \in \N}$ the equivalence class in $\prod_{n \in \N} M_n / I_{\mathcal{U}}$ of a sequence $(x_n)_{n \in \N}$. Furthermore, we define a trace $\tau_{\cU}$ on $\prod_{n \in \N} M_n / I_\mathcal{U}$ by
	\[
	\tau_{\cU}([x_n]_{n \in \N}) = \lim_{n \to \cU} \tau_n(x_n).
	\]
	the limit exists because of the boundedness of the sequence and it is independent of the particular representative of the equivalence class $[x_n]_{n \in \N}$ because $|\tau_n(x_n) - \tau(y_n)| \leq \norm{x_n - y_n}_{L^2(M_n,\tau_n)}$.  The pair $(\prod_{n \in \N} A_n / I_{\mathcal{U}},\tau_{\mathcal{U}})$ is automatically a tracial von Neumann algebra.
	
	\subsection{Free probability and random matrices} \label{subsec: free probability}
	
	Here we recall standard definitions and background on free probability and random matrices.  See also \cite{VDN1992,AGZ2009,MS2017}.
	
	\textbf{Free independence:} Let $(M,\tau)$ be a tracial von Neumann algebra.  We say that $*$-subalgebras $(A_i)_{i \in I}$ are \emph{freely independent} if whenever $i_1$, \dots, $i_k \in I$ with $i_j \leq i_{j+1}$, whenever $a_j \in A_{i_j}$ for $j = 1$, \dots, $k$, we have
	\[
	\tau[(a_1 - \tau(a_1)) \dots (a_k - \tau(a_k))] = 0.
	\]
	Similarly, elements are said to be freely independent if the $*$-algebras that they generate are freely independent.
	
	\textbf{Semicircular and circular operators:}  A \emph{standard semicircular operator} is a self-adjoint operator $x$ in a tracial von Neumann algebra whose spectral measure is the Wigner semicircular distribution $\frac{1}{2\pi} \sqrt{4 - t^2} \mathbbm{1}_{[-2,2]}(t)\,dt$.  A \emph{standard semicircular family} is a tuple $(x_1,\dots,x_m)$ where $x_1$, \dots, $x_m$ are freely independent standard semicircular elements.  A \emph{standard circular family} in this paper will be a family $(z_1,\dots,z_m)$ such that $(\re(z_1),\im(z_1),\dots,\re(z_m),\im(z_m))$ are a standard semicircular family.  Beware that many authors define a free circular operator as $(x + iy) / \sqrt{2}$ rather than $x + iy$, where $x$ and $y$ are free semicirculars.  We do not include the factor of $1/\sqrt{2}$ because we want a simpler conversion between the self-adjoint and non-self-adjoint cases, and because we want the resulting conventions around free entropy for non-self-adjoint operators to match those of \cite{JekelModelEntropy}.  Note that $\norm{x + iy}$ thus becomes $2 \sqrt{2}$ rather than $2$.
	
	\textbf{Filtrations and free Brownian motion:}  A \emph{noncommutative filtration} is a family $(M_t)_{t \geq 0}$ of tracial von Neumann subalgebras of some $(M,\tau)$ such that if $s \leq t$, then $M_s \subseteq M_t$.  We also adopt the convention that $M_\infty = M$.  An \emph{adapted process} $\alpha: [0,t_0] \to L^2(M,\tau)$ is a Bochner-measurable map such that $\alpha_t \in L^2(M_t,\tau)$ for each $t \geq 0$.
	
	For $m \in \N \cup \{\infty\}$, an \emph{$m$-variable (non-self-adjoint) free Brownian motion compatible with the filtration $(M_t)$} is a family of operators $(z_{t,j})_{t \geq 0, j \leq m}$ such that
	\begin{itemize}
		\item For each $t \geq 0$ and $j \in \N$, we have $z_{t,j} \in M_t$ (i.e.\ $t \mapsto z_{t,j}$ is adapted).
		\item For each $s < t$, the family $(z_{t,j} - z_{s,j})_{j \leq m}$ is freely independent of $M_s$.
		\item For each $s < t$, the family $(t-s)^{-1/2} (z_{t,j} - z_{s,j})_{j \leq m}$ is a standard circular family.
	\end{itemize}
	The \emph{filtration generated by the Brownian motion} is the filtration $(N_t)_{t \geq 0}$ where $N_t$ is the von Neumann algebra generated by $(z_{s,j})_{s \leq t, j \leq m}$.
	
	\textbf{Free Gaussian functor:}  One of the main ways to construct a free Brownian motion is Voiculescu's free Gaussian functor \cite[\S 3]{Voiculescu1985} \cite[\S 2.6]{VDN1992}.  Given a real Hilbert space $H_{\R}$ and its complexification $H$, the functor produces a von Neumann algebra $M(H)$ and a family of operators $X(h)$ for $h \in H_{\R}$ such that
	\begin{itemize}
		\item $x(h)$ is $\norm{h}$ times a standard semicircular.
		\item If $H_1$, \dots, $H_k$ are orthogonal subspaces of $H$, then the corresponding families of operators are freely independent.
	\end{itemize}
	Moreover, any isometry $H \to H'$ yields an trace-preserving $*$-homomorphism from $M(H)$ to $M(H')$; similarly, an autmorphism of $H$ produces an automorphism of $M(H)$.  See \cite[Theorem 2.6.2]{VDN1992}.
	
	In particular, by taking $H = L^2([0,\infty);\R^m)$ and $x_{t,j} = x(\mathbbm{1}_{[0,t]}e_j)$, we obtain a self-adjoint free Brownian motion.  We can adapt the construction to the non-self-adjoint case, by taking the free Gaussian functor applied to $H_{\R} \oplus H_{\R}$ and setting $z(h) = x(h \oplus 0) + i x(0 \oplus h)$ for $h \in H_{\R}$.  Then with $L^2([0,\infty);\R^m)$ we obtain a non-self-adjoint free Brownian motion.

	\textbf{Matrix algebras:} We denote by $\mathbb{M}_n$ the $n \times n$ complex matrices and $\tr_n$ the normalized trace $(1/n) \Tr_n$.  We equip $\M_n^m$ with the inner product
	\[
	\ip{\mathbf{X},\mathbf{Y}}_{\tr_n} = \sum_{j=1}^m \tr_n(X_j^*Y_j),
	\]
	where $\tr_n = (1/n) \Tr_n$ is the normalized trace; we also denote the corresponding normalized Hilbert-Schmidt norm by $\norm{\cdot}_{\tr_n}$.  As a complex inner-product space $\M_n^m$ may be transformed by a linear isometry to $\C^m$.  By the \emph{Lebesgue measure} on $\M_n^m$, we mean the Lebesgue measure on $\C^m$ transported by such an isometry, which is independent of the particular choice of isometry.  Note that many authors identify $\M_n$ with $\C^{n^2}$ entrywise to obtain Lebesgue measure, but our choice of Lebesgue measure differs by a factor depending on $n$ since we use $\tr_n$ rather than $\Tr_n$ to define the inner product.  See e.g.\ \cite[\S 3.3]{JekelPi2024} for further discussion of the normalization.
	
	\textbf{Gaussian random matrices and Brownian motion:} A \emph{random matrix tuple} is a random variable on some classical probability space $(\Omega,\mathcal{G},\mathbb{P})$ which takes values in $\mathbb{M}_n^m$.
	
	A \emph{GUE (Gaussian unitary ensemble)} random matrix is a random element of $(\mathbb{M}_n)_{\sa}$ with probability density $(2 \pi n^2)^{-n^2/2} \exp(-n^2 \norm{X}_2^2/2)$ with respect to the Lebesgue measure on $(\M_n)_{\sa}$ with respect to the normalized inner product $\ip{\cdot,\cdot}_{\tr_n}$.  A \emph{Ginibre random matrix} is a random element of $\M_n$ with probability density $(2 \pi n^2)^{-n^2} \exp(-n^2 \norm{X}_2^2/2)$ on $\M_n$ with respect to Lebesgue measure; note that the real and imaginary parts are independent GUE matrices.  Similar to the case of circular operators, our convention differs by a factor of $\sqrt{2}$ from that of many other works.
	
	Similarly, \emph{GUE} and \emph{Ginibre families} are $m$-tuples of independent GUE and Ginibre matrices respectively.  A \emph{Brownian motion on} $\M_n^m$ will refer to a family $(Z_{t,j})_{t \geq 0,j \leq m}$ with independent increments such that $(t-s)^{-1/2} (Z_{t,j} - Z_{s,j})_{j \leq m}$ is a Ginibre family for each $s < t$.  Such a Brownian motion will is said to be compatible with some filtration $\mathcal{G}_t$ of $\sigma$-algebras on the classical probability space $(\Omega,\mathcal{G},\mathbb{P})$ if $(Z_{s,j})_{s \leq t,j\leq m}$ is $\mathcal{G}_t$-measurable for each $t > 0$ (that is, $(Z_{t,j})_{t \geq 0}$ is \emph{progressively measurable}), and $(Z_{t,j} - Z_{s,j})_{j \leq m}$ is independent of $\mathcal{G}_s$ for each $s < t$.
	
	\textbf{Asymptotic freeness:}  Voiculescu showed that GUE random matrices are asymptotically freely independent of deterministic matrices.  The same applies of course to Ginibre matrices and Brownian motion on $\M_n^m$.  Various statements of the asymptotic freeness theorem can be found in \cite[Theorem 2.2]{Voiculescu1991}, \cite[Theorem 5.4.2]{AGZ2009}, \cite[Chapter 4, Theorem 5]{MS2017}, from which it is not hard to extract the following statement.
	
	\begin{theorem} \label{thm: asymptotic freeness}
		Let $m, m' \in \N \cup \{\infty\}$.  Let $\mathbf{Z}_t^{(n)} = (Z_{t,j}^{(n)})_{j \leq m}$ be a Brownian motion on $\M_n^m$.  Let $\mathbf{Y}^{(n)} = (Y_j^{(n)})_{j \leq m'}$ be a tuple of random matrices such that $\norm{Y_j^{(n)}} \leq r_j$ for some $r_j \in (0,\infty)$, and $\mathbf{Y}^{(n)}$ is independent of $(\mathbf{Z}_t^{(n)})_{t \geq 0}$.
		
		Let $\mathbf{z}_t = (z_{t,j})_{j \leq m}$ be a free circular Brownian motion in a tracial von Neumann algebra $(M,\tau)$.  Let $\mathbf{y} = (y_j)_{j \leq m'} \in M^{m'}$ be a tuple in $(M,\tau)$ freely independent of $(\mathbf{z}_t)_{t \geq 0}$.  Assume that for all $m'$-variable $*$-polynomials $f$, we have
		\[
		\lim_{n \to \infty} \tr_n(f(\mathbf{Y}^{(n)})) = \tau(f(\mathbf{y})) \text{ almost surely.}
		\]
		Then for all $k \in \N$ and $t_1$, \dots, $t_k \geq 0$ and $(m+km')$-variable $*$-polynomials $f$, we have
		\[
		\lim_{n \to \infty} \tr_n(f(\mathbf{Y}^{(n)},\mathbf{Z}_{t_1}^{(n)},\dots,\mathbf{Z}_{t_k}^{(n)})) = \tau(f(\mathbf{y},\mathbf{z}_{t_1},\dots,\mathbf{z}_{t_k})) \text{ almost surely.}
		\]
	\end{theorem}

	\textbf{Concentration of measure:}  Let $m \in \N$ and let $\mathbf{Z}^{(n)} = (Z_1^{(n)},\dots,Z_m^{(n)})$ be a Ginibre family.  Then $\mathbf{Z}^{(n)}$ satisfies strong concentration-of-measure properties which show that functions of $\mathbf{Z}^{(n)}$ are close to being constant with high probability.  Suppose that $f: \mathbb{M}_n^m \to \R$ is $L$-Lipschitz with respect to the normalized $2$-norm $\norm{\cdot}_2 = \norm{\cdot}_{2,\tr_n}$.  Then we have the \emph{Herbst concentration inequality}
	\begin{equation} \label{eq: Herbst concentration inequality}
		\mathbb{P}(|f(\mathbf{Z}^{(n)}) - \mathbb{E} f(\mathbf{Z}^{(n)})| \geq \delta) \leq 2 e^{-n^2 \delta^2 / 2L^2} \text{ for } \delta > 0,
	\end{equation}
	and the \emph{Poincar{\'e} inequality}
	\begin{equation} \label{eq: Poincare inequality}
		\mathbb{E} |f(\mathbf{Z}^{(n)}) - \mathbb{E} f(\mathbf{Z}^{(n)})|^2 \leq \frac{1}{n^2} \mathbb{E} \norm{\nabla f(\mathbf{Z}^{(n)})}_2^2 \leq \frac{L^2}{n^2}.
	\end{equation}
	Here $\nabla f$ is the gradient of a function on $\mathbb{M}_n^m \cong \R^{mn^2}$ with respect to the real inner product $\re \ip{\cdot,\cdot}_{\tr_n}$.  For proof, see e.g.\ \cite[\S 2.3 and \S 4.4]{AGZ2009}.  Both of these inequalities can also be applied to the rescaled Brownian motion increments $(t-s)^{-1/2}(\mathbf{Z}_t^{(n)} - \mathbf{Z}_s^{(n)})$ by suitable substitutions.
	
	\textbf{Truncation:}  It is important at various points to guarantee that the tuples of matrices are bounded in operator norm as $n \to \infty$.  We therefore recall the following well-known fact (see e.g.\ \cite[p. 24]{AGZ2009} or \cite{Ledoux2003remark}):  Let $S^{(n)}$ be a GUE random matrix.  Then
	\begin{equation} \label{eq: operator norm bound for GUE}
		\lim_{n \to \infty} \norm{S^{(n)}} = 2 \text{ almost surely.}
	\end{equation}
	For $R > 0$, we use the function $F_R(t) = \max(-R,\min(t,R))$ as in \eqref{eq: function FR}.  We then note the following facts:  Equation \eqref{eq: operator norm bound for GUE} implies that for $R > 2$,
	\begin{equation} \label{eq: truncated GUE matrix}
		\text{almost surely } F_R(S^{(n)}) = S^{(n)} \text{ for sufficiently large $n$.}
	\end{equation}
	Moreover, standard tail estimates for the GUE matrix imply that
	\begin{equation} \label{eq: truncated GUE matrix 2}
		\lim_{n \to \infty} \norm{S^{(n)} - F_R(S^{(n)})}_{L^2} = 0 \text{ for } R \geq 2,
	\end{equation}
	where $\norm{X}_{L^2} := (\mathbb{E} \norm{X}_2^2)^{1/2}$ for any random matrix $X$.
	
	The same statements can be applied to Ginibre matrices and tuples.  For ease of notation, if $\mathbf{Z}^{(n)} = (Z_j^{(n)})_{j \leq m}$ is matrix tuple (for instance, a Ginibre tuple) and $R > 0$, we set
	\begin{equation} \label{eq: truncation of tuple}
		F_R(\mathbf{Z}^{(n)}) = (F_R(\re Z_j^{(n)}) + i F_R(\im Z_j^{(n)}))_{j \leq m}.
	\end{equation}
	
	\subsection{Model theory} \label{subsec: model theory}
	
	Next, let us sketch the setup of continuous model theory, or model theory for metric structures \cite{BYBHU2008,BYU2010}.  We will follow the treatment in \cite{FHS2014} which introduces ``domains of quantification'' rather than ``sorts.''  Classical model theory deals with first-order statements:  one starts with basic predicates associated to the class of objects under consideration (for example, ``$xy = z$'' in the setting of groups), and then builds new statements from them by applying quantifiers ``for all'' and ``there exists'' and connectives like ``and'', ``or'', and ``if/then.''  In particular, each predicate takes the value true or false for any given inputs.  In continuous model theory, the predicates will take real values rather than true/false values.  For instance, the relation ``$=$'' which takes values true or false will be replaced by the metric $d$ which outputs a real-valued distance between two points.
	
	Just as in classical logic, a particular language comes with \emph{function symbols} and \emph{relation (or predicate) symbols}.  In the case of groups, for instance, there is a binary function symbol for multiplication and a $0$-ary function, i.e. a constant, for the identity.  In the case of $*$-algebras over $\mathbb{C}$, one needs symbols for addition, multiplication, adjoint, and scalar multiplication.  While the function symbols encode operations from $M^n$ into $M$ for each structure $M$, the relation symbols take inputs from $M$ and output true/false in the discrete setting, and real values in the metric setting.
	
	When the structures are metric spaces, the operations associated to the structures should also be assumed to be continuous and globally bounded.  In fact, there should be a uniform bound and modulus of continuity independent of the particular structure, so that we will be able to construct ultraproducts of the structures.  This leads to another annoyance in the setting of Banach spaces, $\mathrm{C}^*$-algebras, von Neumann algebras, and the like, namely, the operations are not always globally bounded or uniformly continuous on the whole space.  We therefore dissect the space into a family of \emph{domains of quantification} (often the family of balls of radius $r$ for nonnegative integers or rationals or reals $r$) on which the operations are bounded and uniformly continuous.  When we apply a supremum or infimum, it will always be over one of these domains, to ensure that the result remains finite; for instance, we may take $\sup_{y \in D_1} \varphi(x,y)$ where $D_1$ is the unit ball and $\varphi$ is some formula that is already defined.  If one wants to avoid using multiple domains, this can sometimes be done by suitably modifying the collection of function symbols in order to phrase everything in terms of the unit ball \cite[Proposition 29.4]{BYIT2024}; however, this is not convenient for our purposes here since we want to reason about globally defined functions.

	\textbf{Languages:} A \emph{metric language} $\mathcal{L}$ consists of:
	\begin{itemize}
		\item A set $\mathcal{D}$ whose elements are called \emph{domains of quantification}.
		\item For each $D, D' \in \mathcal{D}$ an assigned constant $C_{D,D'}$.
		\item A countably infinite set of \emph{variable symbols} for each sort $S$.  We denote the variables by $(x_i)_{i \in \N}$.
		\item A set of \emph{function symbols}, each with an assigned \emph{arity} $m \in \N$.
		\item For each function symbol $f$ and for every $\mathbf{D} = (D_1,\dots,D_m) \in \mathcal{D}^m$, an assigned $D_{f,\mathbf{D}} \in \mathcal{D}$ (representing a range bound), and assigned moduli of continuity $\omega_{f,\mathbf{D},1}$, \dots, $\omega_{f,\mathcal{D},m}$. (Here ``modulus of coninuity'' means a continuous increasing, zero-preserving function $[0,\infty) \to [0,\infty)$).
		\item A set of \emph{relation symbols}, each with an assigned arity $n \in \N$.
		\item A special binary relation symbol $d$ for the metric.
		\item For each relation symbol $R$ and for every $\mathbf{D} = (D_1,\dots,D_m) \in \mathcal{D}^m$, an assigned bound $N_{R,\mathbf{D}} \in [0,\infty)$ and assigned moduli of continuity $\omega_{R,\mathbf{D},1}$, \dots, $\omega_{R,\mathcal{D},m}$.
	\end{itemize}
	
	\textbf{Structures:} Given a language $\mathcal{L}$, an \emph{$\mathcal{L}$-structure} $\cM$ is a metric space $M$ with an object assigned to each symbol in $\mathcal{L}$, called the \emph{interpretation} of that symbol, in the following manner:
	\begin{itemize}
		\item Each domain of quantification $D \in \mathcal{D}_S$ is assigned a subset $D^{\cM} \subseteq M$, such that $D^{\cM}$ is complete for each $D$, and $\cM = \bigcup_{D \in \mathcal{D}_S} D^{\cM}$, and $\sup_{x \in D, y \in D'} d^{\cM}(x,y) \leq C_{D,D'}$.
		\item Each function symbol $f$ of arity $m$ is interpreted as a function $f^{\mathcal{M}}: M^m \to M$.  Moreover, for each $\mathbf{D} = (D_1,\dots,D_n) \in \mathcal{D}_{S_1} \times \dots \times \mathcal{D}_{S_n}$, the function $f^{\cM}$ maps $D_1^{\cM} \times \dots \times D_n^{\cM}$ into $D_{f,\mathbf{D}}^{\cM}$.  Finally, $f^{\cM}$ restricted to $D_1^{\cM} \times \dots \times D_n^{\cM}$ is uniformly continuous in the $i$th variable with modulus of continuity (less than or equal to) $\omega_{f,\mathbf{D},i}$.
		\item Each relation symbol $R$ of arity $m$ is interpreted as a function $R^{\cM}: M^m \to \R$.  Moreover, for each $\mathbf{D} = (D_1,\dots,D_m) \in \mathcal{D}^m$, $f^{\cM}$ is bounded by $N_{R,\mathbf{D}}$ on $D_1^{\cM} \times \dots \times D_m^{\cM}$ and uniformly continuous in the $i$th argument with modulus of continuity of $\omega_{R,\mathbf{D},i}$.
	\end{itemize}
	
	\textbf{Language for tracial von Neumann algebras:} Next, we recall the \emph{language $\mathcal{L}_{\tr}$ of tracial $\mathrm{W}^*$-algebras} together with how each symbol will be interpreted for a given tracial von Neumann algebra, as formulated in \cite{FHS2014}.
	\begin{itemize}
		\item Domains of quantification $\{D_r\}_{r \in (0,\infty)}$, to be interpreted as the operator norm balls of radius $r$ in $M$.
		\item The metric symbol $d$, to be interpreted as the metric induced by the noncommutative $L^2$-norm $\norm{\cdot}_{2,\tau}$.
		\item A binary function symbol $+$, to be interpreted as addition.
		\item A binary function symbol $\cdot$, to be interpreted as multiplication.
		\item A unary function symbol $*$, to be interpreted as the adjoint operation.
		\item For each $\lambda \in \C$, a unary function symbol, to be interpreted as multiplication by $\lambda$.
		\item Function symbols of arity $0$ (in other words constants) $0$ and $1$, to be interpreted as additive and multiplicative identity elements.
		\item Two unary relation symbols $\re \tr$ and $\im \tr$, to be interpreted the real and imaginary parts of the trace $\tau$.
		\item For technical reasons (see \cite[Proposition 3.2]{FHS2014} and \cite[p.\ 93]{Hart2023}), we also introduce for each $d$-variable noncommutative polynomial $p$ and $n \in \N$ a $d$-variable function symbol $t_p$ representing the evaluation of $p$, along with the appropriate range bounds $N_{t_p,\mathbf{r}}$ between $M_p$ and $M_p + 1/n$, where $M_p$ is the supremum of $\norm{p(X_1,\dots,X_d)}$ over all $(X_1,\dots,X_d)$ in a tracial $\mathrm{W}^*$-algebra $\cM$.
	\end{itemize}
	Each function and relation symbol is assigned range bounds and moduli of continuity that one would expect, e.g.\ multiplication is supposed to map $D_r \times D_{r'}$ into $D_{rr'}$ with $\omega_{(D_r,D_r'),1}^{\cdot}(t) = r't$ and $\omega_{(D_r,D_r'),2}^{\cdot} = rt$.  Although not every $\mathcal{L}_{\tr}$-structure comes from a tracial $\mathrm{W}^*$-algebra, one can formulate axioms in the language such that any structure satisfying these axioms comes from a tracial $\mathrm{W}^*$-algebra \cite[\S 3.2]{FHS2014}.
	
	\textbf{Syntax:} We now briefly recall \emph{terms}, \emph{formulas}, \emph{sentences}, \emph{types}, and \emph{theories}.  See \cite[\S 2.3-2.4]{JekelTypeCoupling} for more detail.  Fix a set of variables $X$.  In the following definitions, we write $\cL$-terms, $\cL$-formulas, etc., but we will also refer to these objects simply as terms, formulas, etc.\ if the language $\cL$ is clear from context.
	
	$\cL$-\emph{terms} in $X$ are expressions formed recursively by composition of the function symbols in $\cL$ applied to variables in $X$.  For each $\cL$-structure $\cM$, the interpretation of a term $t(x_1,\dots,x_m)$ is the function $M^m \to M$ obtained by composing the interpretation of the function symbols in $M$.  For instance, a term in the language of tracial von Neumann algebras is an expression in the $*$-algebra operations, such as $((x_1 + x_2) x_3 + x_4^*) + 3 x_1 (x_3 x_2)$ where $x_1$, $x_2$, \dots are variables in $X$.  Although the $*$-algebra axioms are not assumed to hold in the definition of terms, every term will coincide with a $*$-polynomial when evaluated in an actual $*$-algebra.
	
	$\cL$-\emph{formulas} in $X$ are expressions defined recursively as follows:
	\begin{itemize}
		\item \emph{Basic formulas} have the form $r(t_1(\mathbf{x}),\dots,t_k(\mathbf{x}))$ where $r$ is a $k$-ary relation symbol from $\cL$ and $t_1$, \dots, $t_k$ are terms in variables $\mathbf{x} = (x_1,\dots,x_m)$ from $X$.
		\item \emph{Connectives:}  If $\varphi_1$, \dots, $\varphi_k$ are formulas in $\mathbf{x} = (x_1,\dots,x_m)$ and $f: \R^k \to \R$ is continuous, then $f(\varphi_1,\dots,\varphi_k)$ is a formula.
		\item \emph{Quantifiers:} If $\varphi(\mathbf{x},y)$ is a formula where $x_1,\dots,x_m, y \in X$ are distinct, and if $D \in \mathcal{D}$, then $\psi_1(\mathbf{x}) = \sup_{y \in D} \varphi(\mathbf{x},y)$ and $\psi_2(\mathbf{x}) = \inf_{y \in D} \varphi(\mathbf{x},y)$ are formulas.
	\end{itemize}
	\emph{Quantifier-free formulas} are defined recursively using only the basic formulas and connectives.  For instance, for tracial von Neumann algebras, trace polynomials can be expressed using quantifier-free formulas.
	
	In a formula, the \emph{free variables} are the variables which are not bound to any quantifier, and we usually denote a formula as a function of the free variables.  In general, a variable could have one occurrence that is bound to a quantifier and one occurrence which is free, but by relabeling variables we can assume that each quantifier in the formula has a unique variable associated to it.  We also denote by $\mathcal{F}_m$ the set of formulas in $m$ free variables $x_1$, \dots, $x_m$.  Here not every free variable is required to appear in the formula in a nontrivial way; that is, every formula in a subset of the variables $x_1$, \dots, $x_m$ is also an element of $\mathcal{F}_m$.
	
	\textbf{Evaluation of formulas:} Given an $\cL$-structure $\cM$ and a formula $\varphi(\mathbf{x})$, the \emph{evaluation} or \emph{interpretation} of $\varphi$ in $\cM$ is the function $\varphi^{\cM}: M^m \to \R$ defined by plugging in inputs $x_1$, \dots, $x_m$ from $M$, evaluating each term and relation symbol in $M$, and evaluating each expression $\sup_{y \in D}$ as the actual supremum over $y \in D^{\cM}$.   We remark that if $\varphi$ is a formula in $m$ free variables and $\mathbf{D} = (D_1,\dots,D_m) \in \mathcal{D}^m$, then $\varphi^{\cM}$ is uniformly continuous on $D_1^{\cM} \times \dots \times D_m^{\cM}$ with a modulus of continuity that is independent of the particular $\cL$-structure $\cM$.  This is easily seen by induction on the complexity of the formulas \cite[Theorem 3.5]{BYBHU2008}.
	
	\textbf{Sentences, theories, and axiomatization:} A $\cL$-\emph{sentence} is a formula with no free variables.  Thus, the evaluation of a sentence in $\cM$ is a real number associated to $\cM$.  An $\cL$-\emph{theory} is a collection of $\cL$-sentences.  An $\cL$-structure $\cM$ is said to \emph{satisfy} a theory $\rT$, if $\varphi^{\cM} = 0$ for all $\varphi \in \rT$ (one also says that $\cM$ is a \emph{model} of $\rT$ or $\cM \models \rT$).  A class $\mathcal{C}$ of $\cL$-structures is said to be \emph{axiomatizable} if there is a theory $\rT$ such that $\mathcal{C}$ is the class of $\cL$-structures satisfying $\rT$.  For instance, the class of tracial von Neumann algebras and the class of $\mathrm{II}_1$ factors are axiomatizable in $\cL_{\tr}$ by \cite{FHS2014}.  Similarly, a property is said to be axiomatizable if the class of structures with that property is axiomatizable.
	
	\textbf{Elementary equivalence and embeddings:} Two $\cL$-structures $\cM_1$ and $\cM_2$ are said to be \emph{elementarily equivalent} if $\varphi^{\cM_1} = \varphi^{\cM_2}$ for all $\cL$-sentences $\varphi$.  An inclusion $\cM_1 \subseteq \cM_2$ is said to be \emph{elementary} if for every formula $\varphi$ in $m$ variables and $\mathbf{x} \in (\cM_1)^m$, we have $\varphi^{\cM_2}(\mathbf{x}) = \varphi^{\cM_1}(\mathbf{x})$.
	
	\textbf{Types and definable predicates:} Let $m \in \N \cup \{\infty\}$.  Given an $\cL$-structure $\cM$ and $\mathbf{x} \in M^m$, the \emph{full type} $\tp^{\cM}(\mathbf{x})$ of $\mathbf{x}$ is the map $\mathcal{F}_m \to \R$ sending $\varphi$ to $\varphi^{\cM}(\mathbf{x})$.  We denote by $\mathbb{S}_m(\rT)$ the set of types $\tp^{\cM}(\mathbf{x})$ of tuples from $\cL$-structures satisfying $\rT$.  If $\mathbf{D} = (D_1,\dots,D_m)$ is a tuple from $\cD$, then $\mathbb{S}_{\mathbf{D}}(\rT)$ will denote the set of types associated to tuples from $D_1^{\cM} \times \dots \times D_m^{\cM}$ and $\cL$-structures satisfying $\rT$.  We equip $\mathbb{S}_{\mathbf{D}}(\rT)$ with the weak-$*$ topology from the space of linear functionals on $\mathcal{F}_m$.  We equip $\mathbb{S}_m(\rT)$ with the topology such that $O$ is open in $\mathbb{S}_m(\rT)$ if and only if $O \cap \mathbb{S}_{\mathbf{D}}(\rT)$ is open in $\mathbb{S}_{\mathbf{D}}(\rT)$ for each $\mathbf{D}$.
	
	\emph{Definable predicates} are a certain completion of the space of formulas that are dual to full types.  As motivation, note that by definition of the weak-$*$ topology, every formula $\varphi$ defines a continuous function on $\mathbb{S}_{\mathbf{D}}(\rT)$, which is also a compact Hausdorff space (see \cite[Proposition 8.6]{BYBHU2008}).  The continuous functions given by formulas are a subalgebra of $C(\mathbb{S}_{\mathbf{D}}(\rT);\R)$ that separates points and contiains $1$, and hence they are dense in $C(\mathbb{S}_{\mathbf{D}}(\rT);\R)$.  The norm of $C(\mathbb{S}_{\mathbf{D}}(\rT);\R)$ corresponds to
	\[
	\norm{\varphi}_{\mathbf{D},\rT} = \sup_{\cM \models \rT} \sup_{\mathbf{x} \in D_1^{\cM} \times \dots \times D_m^{\cM}} |\varphi^{\cM}(\mathbf{x})|.
	\]
	It is therefore natural to study the completion of $\mathcal{F}_m$ with respect to the family of seminorms $\norm{\cdot}_{\mathbf{D},\rT}$ for $\mathbf{D} \in \cD^m$.  The elements in the completion will be called \emph{definable predicates} in $m$ variables relative to $\rT$.  By \cite[Lemma 2.16]{JekelTypeCoupling}, every element of $C(\mathbb{S}_{\mathbf{D}}(\rT);\R)$ agrees with some definable predicate.  In this paper, we will use $(\mu,\varphi)$ to denote the pairing of a type $\mu$ with a definable predicate $\varphi$, namely,
	\[
	(\mu,\varphi) = \varphi^{\cM}(\mathbf{x}) \text{ when } \tp^{\cM}(\mathbf{x}) = \mu.
	\]
	
	\textbf{Notation for the von Neumann algebra case:} For tracial von Neumann algebras, the domains are indexed by $r \geq 0$, and we therefore use the following notation:
	\begin{itemize}
		\item For $m \in \N \cup \{\infty\}$ and $\mathbf{r} \in (0,\infty)^m$ and a tracial von Neumann algebra $\cM$, we write $D_{\mathbf{r}}^{\cM} = \prod_{j \leq m} D_{r_j}^{\cM}$.
		\item We write $\mathbb{S}_{\mathbf{r}}(\rT)$ for the space of types associated to $(D_{r_j})_{j \leq m}$.
	\end{itemize}

	\section{Filtered NC probability spaces as metric stuctures} \label{sec: model theory and filtration}
	
	\subsection{Setup of the language} \label{subsec: language for filtration}
	
	The goal of this section is to set up the study of filtered NC probability spaces and adapted processes as \emph{metric structures} in the sense of \S \ref{subsec: model theory} (for terminology on NC filtrations, see \S \ref{subsec: free probability}).  Throughout, we consider a process $\mathbf{z}_t = (z_{t,j})_{j \in \N}$ for $t \geq 0$ which is not necessarily self-adjoint, and we usually work in terms the increments $z_{s,t,j} = z_{t,j} - z_{s,j}$ for $s < t$ and $j \in \N$.
	
	Due to the nature of the metric language for von Neumann algebras, we want to assume that $\mathbf{z}_{s,t,j}$ lies in an operator-norm ball $D_{r(s,t)}$ where $r(s,t)$ is independent of the particular structure.  Because we are concerned with circular Brownian motion, we take $r(s,t)$ be some number larger than $2^{3/2} (t-s)^{1/2}$.  For concreteness and for convenience in the random matrix arguments, we take $r(s,t) = 6(t-s)^{1/2}$.  Of course, if one would like to study a \emph{different} process which is still continuous in operator norm, one can adjust the constants accordingly, but we simply use $6(t-s)^{1/2}$ throughout this paper.
	
	\begin{definition}
		The language $\cL_{\proc}$ is defined as follows:
		\begin{itemize}
			\item For each $t \in [0,\infty]$ and $r \in [0,\infty)$, there is an associated domain $D_{t,r}$.
			\item There is a binary function symbol $+$ which maps $D_{t_1,r_1} \times D_{t_2,r_2}$ into $D_{\max(t_1,t_2),r_1+r_2}$ for each $t_1$, $t_2$, $r_1$, and $r_2$.
			\item There is a binary function symbol $\cdot$ which maps $D_{t_1,r_1} \times D_{t_2,r_2}$ into $D_{\max(t_1,t_2),r_1 r_2}$ for each $t_1$, $t_2$, $r_1$, and $r_2$.
			\item There is a unary function symbol $*$ which maps $D_{t,r}$ into $D_{t,r}$ for each $t$ and $r$.
			\item For each $\lambda \in \C$, there is a unary function symbol $\lambda$ that maps $D_{t,r}$ into $D_{t,|\lambda|r}$ for each $t$ and $r$.
			\item There is a constant symbol $1 \in D_{0,1}$.
			\item For $0 \leq s \leq t < \infty$ and $j \in \N$, there is a constant symbol $z_{s,t,j} \in D_{t,3(t-s)^{1/2}}$.
			\item There is a predicate symbol $d$ for the distance, which maps $D_{t_1,r_1} \times D_{t_2,r_2}$ into $[0,r_1+r_2]$.
			\item There are predicate symbols $\re \tr$ and $\im \tr$ for the real and imaginary parts of the trace, which map $D_{t,r}$ into $[-r,r]$.
		\end{itemize}
	\end{definition}
	
	We now consider how $\cL_{\proc}$-structures relate with $\cL_{\tr}$ structures and with tracial von Neumann algebras.
	
	\begin{definition}
		Let $\cM$ be an $\cL_{\proc}$-structure.  For $t \in [0,\infty]$, denote by $\cM_t$ the $\cL_{\tr}$-structure obtained by considering only the domains $(D_{t,r})_{r \in [0,\infty)}$ and ignoring the constant symbols $z_{s,t,j}$.  Moreover, let $M_t = \bigcup_{r > 0} D_{t,r}$ as a set.
	\end{definition}
	
	\begin{definition}
		Let $\rT_{\filt}$ be the set of axioms in $\cL_{\proc}$ described as follows:
		\begin{enumerate}[(1)]
			\item For each $t_1 \leq t_2$ and $r_1 \leq r_2$, we have
			\[
			\sup_{x \in D_{t_1,r_1}} \inf_{y \in D_{t_2,r_2}} d(x,y) = 0.
			\]
			Note that this axiom expresses that $D_{t_1,r_1}^{\cM} \subseteq D_{t_2,r_2}^{\cM}$.
			\item $\cM_\infty$ satisfies the theory $\rT_{\tr}$ used to axiomatize tracial von Neumann algebras in \cite{FHS2014}.
		\end{enumerate}
		We also define $\rT_{\proc}$ to be the $\rT_{\filt}$ together with the additional axioms:
		\begin{itemize}
			\item[(3)] For each $t_1 \leq t_2 \leq t_3 < \infty$, we have $d(z_{t_1,t_2,j}+z_{t_2,t_3,j}, z_{t_1,t_3,j}) = 0$, which expresses that $z_{t_1,t_2,j}+z_{t_2,t_3,j} = z_{t_1,t_3,j}$.
		\end{itemize}
	\end{definition}
	
	The following lemma is immediate from the fact $\rT_{\tr}$ axiomatizes tracial von Neumann algebras.
	
	\begin{lemma}
		If $\cM$ is an $\cL_{\proc}$-structure satisfying $\rT_{\filt}$, then $M_\infty$ is a tracial von Neumann algebra and $M_t$ is a von Neumann subalgebra for each $t \in [0,\infty]$, with $M_{t_1} \subseteq M_{t_2}$ for $t_1 < t_2$.
	\end{lemma}
	
	We close the section with important notation that will be used throughout the paper.
	
	\begin{notation}
		Let $m \in \N \cup \{\infty\}$, $\mathbf{r} \in (0,\infty)^m$, and $t \in [0,\infty]$.  For an $\cL_{\proc}$-structure $\cM$, we write
		\[
		D_{t,\mathbf{r}}^{\cM} = \prod_{j \leq m} D_{t,r_j}^{\cM}.
		\]
		Let $\rT$ be a theory in $\cL_{\proc}$.  For a formula $\varphi$ in $m$-free variables and $\mathbf{r} \in (0,\infty)^m$, we denote
		\[
		\norm{\varphi^{\cM}}_{t,\mathbf{r}} = \sup_{\mathbf{x} \in D_{t,\mathbf{r}}^{\cM}} |\varphi^{\cM}(\mathbf{x})|
		\]
		and
		\[
		\norm{\varphi}_{t,\mathbf{r},\rT} = \sup_{\cM \models \rT} \norm{\varphi^{\cM}}_{t,\mathbf{r}} = \sup_{\cM \models \rT} \sup_{\mathbf{x} \in D_{t,\mathbf{r}}^{\cM}} |\varphi^{\cM}(\mathbf{x})|.
		\]
		The notation could also be extended to consider a tuple of different times $\mathbf{t} = (t_j)_{j \leq m}$, although this will not be needed in the present work.  We remark that although the class of $\cM$ satisfying $\rT$ is not a set, the above expression can be equivalently expressed in terms of sets using standard techniques (see Remark \ref{rem: extending definable functions} below).  Also, there is no difficulty in applying the definition when $m = \infty$ since each formula only depends on a finite subset of the variables.
	\end{notation}

	\subsection{Chronological formulas and stationarity} \label{subsec: chronological formulas}
	
	For our results in \S \ref{subsec: concentration}, we must restrict to formulas where the quantifiers appear in chronological order, meaning that if $s < t$, each $\sup$ or $\inf$ over $D_{s,r}$ occurs to the left (or outside) of any $\sup$ or $\inf$ over $D_{t,r'}$ in the formula (for $r, r' \in (0,\infty)$).  Hence, in the remainder of the paper, we will mostly work with these \emph{chronological formulas}.
	
	\begin{definition}[Chronological formulas] \label{def: chronological formulas}
		For $t \in [0,\infty)$, define a set of $\mathcal{L}_{\chron}$-formulas $\mathcal{F}_{\chron}(t)$ recursively as follows:
		\begin{enumerate}[(1)]
			\item Every basic formula is in $\mathcal{F}_{\chron}(t)$ for all $t$.
			\item If $\varphi_1$, \dots, $\varphi_k \in \mathcal{F}_{\chron}(t)$ and $f: \R^k \to \R$ is continuous, then $f(\varphi_1,\dots,\varphi_k) \in \mathcal{F}_{\chron}(t)$.
			\item If $\varphi(\mathbf{x},y) \in \mathcal{F}_{\chron}(t)$ and $s \leq t$ and $r > 0$, then $\sup_{y \in D_{r,s}} \varphi(\mathbf{x},y)$ and $\inf_{y \in D_{r,s}} \varphi(\mathbf{x},y)$ are in $\mathcal{F}_{\chron}(s)$.
		\end{enumerate}
		We then set $\mathcal{F}_{\chron} = \bigcup_{t \in [0,\infty)} \mathcal{F}_{\chron}(t)$, and refer to this as the set of chronological formulas.  Moreover, let $\mathcal{F}_{\chron,m}$ be the set of chronological formulas with $m$ free variables $x_1$, \dots, $x_m$.
	\end{definition}
	
	\begin{definition}[Shift on formulas] \label{def: shift on formulas}
		For $t_0 \in [0,\infty)$, we define a shift operation $S_{t_0}$ on $\cL_{\proc}$-formulas recursively:
		\begin{enumerate}[(1)]
			\item For a basic formula $\varphi$, we obtain $S_{t_0} \varphi$ by replacing each occurrence of a constant symbol $z_{t,t',j}$ by $z_{t_0+t,t_0+t',j}$.
			\item For a continuous connective $f: \R^k \to \R$, we set $S_{t_0}[f(\varphi_1,\dots,\varphi_k)] = f(S_{t_0}\varphi_1,\dots,S_{t_0} \varphi_k)$.
			\item For a quantifier $\sup_{y \in D_{r,t}}$, we set
			\[
			S_{t_0} \left[\sup_{y \in D_{r,t}} \varphi(x,y) \right] = \sup_{y \in D_{r,t_0+t}} S_{t_0} \varphi(x,y),
			\]
			and analogously for $\inf_{y \in D_{r,t}}$.
		\end{enumerate}
	\end{definition}
	
	\begin{lemma}
		The shift operators satisfy $S_{t_0} S_{t_1} = S_{t_0+t_1}$.
	\end{lemma}
	
	\begin{proof}
		We proceed by induction on the complexity of the formulas.  Consider the case of a basic formula $\varphi$.  To obtain $S_{t_1} \varphi$, we replace $z_{t,t',j}$ by $z_{t_1+t,t_1+t',j}$.  Then applying $S_{t_0}$ will replace $z_{t_1+t,t_1+t',j}$ with $z_{t_0+t_1+t,t_0+t_1+t',j}$.  Hence, $S_{t_0} S_{t_1} \varphi$ agrees with $S_{t_0+t_1} \varphi$.
		
		The case of a continuous connective is immediate.
		
		For the case of a sup quantifier, observe that if $\varphi$ satisfies the induction hypothesis, then
		\begin{align*}
			S_{t_0} S_{t_1} \left[\sup_{y \in D_{r,t}} \varphi(x,y) \right] &= S_{t_0} \left[ \sup_{y \in D_{r,t_1+t}} S_{t_1} \varphi(x,y) \right] \\
			&= \sup_{y \in D_{r,t_0+t_1+t}} S_{t_0} S_{t_1} \varphi(x,y) \\
			&= \sup_{y \in D_{r,t_0+t_1+t}} S_{t_0+t_1} \varphi(x,y) \\
			&= S_{t_0+t_1} \left[\sup_{y \in D_{r,t}} \varphi(x,y) \right],
		\end{align*}
		which completes the proof.
	\end{proof}
	
	We can view the shift operators in another way by defining a dual shift operation on $\cL_{\proc}$-structures.  This will allow us to verify properties involving the shift operators by means of manipulating particular structures.
	
	\begin{definition}[Shift on structures] \label{def: shifted structure}
		For an $\cL_{\proc}$-structure $\cM$, define the shifted structure $S_{t_0}\cM$ as follows:
		\begin{itemize}
			\item $D_{t,r}^{S_{t_0}\cM} = D_{t+t_0,r}^{\cM}$ for $r \in [0,\infty)$ and $t \in [0,\infty]$.
			\item $z_{s,t,j}^{S_{t_0} \cM} = z_{s+t_0,t+t_0,j}^{\cM}$ for $0 \leq s < t < \infty$.
			\item The other function and relation symbols are the same in $S_{t_0} \cM$ as in $\cM$.
		\end{itemize}
	\end{definition}
	
	The following observation is immediate from the definitions, using induction on the definition of formulas.
	
	\begin{lemma} \label{lem: shift on formulas and structures}
		For an $\cL_{\proc}$-structure $\cM$, an $\cL_{\proc}$-formula $\varphi$ in $m$ free variables, $t_0 > 0$, and $\mathbf{x} \in M_0^m$, we have
		\[
		(S_{t_0} \varphi)^{\cM}(\mathbf{x}) = \varphi^{S_{t_0} \cM}(\mathbf{x}).
		\]
	\end{lemma}
	
	With the aid of the shift operation, we can give an alternative procedure for generating chronological formulas, which is useful for inductive arguments later in the paper.
	
	\begin{definition}
		If $\rT$ is a theory in some metric language (e.g.\ $\cL_{\proc}$) and $\varphi$ and $\psi$ are formulas, we say that $\varphi$ and $\psi$ are \emph{equivalent modulo $\rT$} if $\varphi^{\cM} = \psi^{\cM}$ for every $\cM$ satisfying $\rT$.  Similarly, we say that two collections of formulas are equivalent modulo $\rT$ if every formula in one of the collections is equivalent to a formula in the other collection.
	\end{definition}
	
	\begin{lemma} \label{lem: chronological formulas alternate}
		Let $\mathcal{S}$ be the set of formulas in $\mathcal{L}_{\proc}$ generated recursively as follows:
		\begin{enumerate}[(1)]
			\item Every basic formula without the symbols $z_{s,t}$ (i.e.\ basic formula in $\mathcal{L}_{\filt}$) is in $\mathcal{S}$.
			\item If $\varphi_1$, \dots, $\varphi_k \in \mathcal{S}$ and $f: \R^k \to \R$ is continuous, then $f(\varphi_1,\dots,\varphi_k) \in \mathcal{S}$.
			\item If $\varphi(\mathbf{x},y)$ is in $\mathcal{S}$ and $r > 0$, then $\sup_{y \in D_{0,r}} \varphi(\mathbf{x},y)$ and $\inf_{y \in D_{0,r}} \varphi(\mathbf{x},y)$ are in $\mathcal{S}$.
			\item If $\varphi(\mathbf{x},y_1,\dots,y_k) \in \mathcal{S}$  and $j_1$, \dots, $j_k \in \N$ and $t > 0$, then $(S_t \varphi)(\mathbf{x},z_{0,t,j_1},\dots,z_{0,t,j_k})$ is in $\mathcal{S}$.
		\end{enumerate}
		Then $\mathcal{S}$ is equivalent to $\cF_{\chron}$ modulo $\rT_{\proc}$.
	\end{lemma}
	
	\begin{proof}
		First, we prove that $\mathcal{S} \subseteq \mathcal{F}_{\chron}$.  It suffices to show that $\cL_{\chron}$ is closed under the operations (1) - (4).  In fact, closure under (1) - (3) is clear from Definition \ref{def: chronological formulas}.  For (4), it suffices to show that $\mathcal{L}_{\chron}$ is closed under shifts and under substituting one of the constant symbols $z_{0,t}$ for one of the free variables.  The closure of $\mathcal{L}_{\chron}$ under shifts follows from an inductive argument based on Definition \ref{def: chronological formulas}; in fact, one can show that if $\varphi \in \mathcal{F}_{\chron}(s)$, then $S_t \varphi \in \mathcal{F}_{\chron}(s+t)$.  Similarly, an inductive argument on Definition \ref{def: chronological formulas} shows that $\mathcal{F}_{\chron}(t)$ is closed under the operation of substituting one of the constant symbols for one of the free variables.
		
		Second, we prove that $\mathcal{F}_{\chron} \subseteq \mathcal{S}$ modulo $\rT_{\proc}$.  In fact, we prove the following claim:  If $\varphi(\mathbf{x}) \in \mathcal{L}_{\chron}(t)$, then
		\begin{equation} \label{eq: expressing chron formula using shift}
			\varphi(\mathbf{x}) = (S_t \psi)(\mathbf{x}, z_{s_1,t_1,i_1}, \dots, z_{s_n,t_n,i_n})
		\end{equation}
		for some $n \in \N_0$ and $\psi(\mathbf{x},y_1,\dots,y_n) \in \mathcal{S}$ and $0 \leq s_j \leq t_j \leq t$ for $j = 1, \dots, n$.  We handle each of the cases in Definition \ref{def: chronological formulas}:
		
		(1) A basic formula in $\cL_{\proc}$ is a composition of predicate symbols with terms.  By an induction argument, each $\cL_{\proc}$-term can be expressed as an $\cL_{\filt}$ term composed with the constant symbols $(z_{s,t})_{0 \leq s \leq t}$.  Hence, every basic formula has the form $\varphi(\mathbf{x}) = \psi(\mathbf{x},z_{s_1,t_1,i_1},\dots,z_{s_n,t_n,i_n})$ where $\varphi$ is a basic formula in $\cL_{\filt}$.
		
		Now we can show that $\varphi \in \mathcal{S}$.  Let $\{\tau_1,\dots,\tau_k\} = \{s_1,\dots,s_n,t_1,\dots,t_n\}$ as sets with $\tau_1 < \dots < \tau_k$.  We proceed by induction on $k$.  The case $k = 1$ is degenerate since $z_{t,t} = 0$ by $\rT_{\proc}$.  So we take as the base case $k = 2$ where $s_j = \tau_1$ and $t_j = \tau_2$ for all $j$.  Then
		\[
		\phi(\mathbf{x}) = S_{\tau_1} \left[ (S_{\tau_2-\tau_1} \psi)(\mathbf{x},z_{0,\tau_2-\tau_1,i_1},\dots,z_{0,\tau_2-\tau_1,i_n}) \right],
		\]
		which is in $\mathcal{S}$.  Similarly, for the inductive step, assume without loss of generality that $t_j > s_j$, and let
		\[
		\varphi'(\mathbf{x},y_1,\dots,y_n) = \psi(\mathbf{x},y_1+z_{\tau_2 \vee s_1-\tau_2,t_1-\tau_2,i_1},\dots,y_n+z_{\tau_2 \vee s_n-\tau_2,t_n-\tau_2,i_n}).
		\]
		Then $\varphi' \in \mathcal{S}$ by inductive hypothesis, and
		\[
		\varphi(\mathbf{x}) = S_{\tau_1}[(S_{\tau_2-\tau_1}\varphi')(\mathbf{x},z_{s_1 \wedge \tau_2 -\tau_1,\tau_2-\tau_1,i_1},\dots,z_{s_n \wedge \tau_2 -\tau_1,\tau_2-\tau_1,i_n})].
		\]
		Here $s_1 \wedge \tau_2$ is either $\tau_2$ or $\tau_1$, and so each of the constant symbols is either $z_{\tau_2-\tau_1,\tau_2-\tau_1}$, which is zero under $\rT_{\proc}$, or $z_{0,\tau_2-\tau_1}$.  Hence $\varphi \in \mathcal{S}$.
		
		Thus, we have shown that $\varphi \in \mathcal{S}$, but it remains to show that it has the form \eqref{eq: expressing chron formula using shift} for each $t > 0$ (since $\varphi$ can be viewed as an element of $\mathcal{F}_{\chron}(t)$ for every $t > 0$).  Invoking the additivity axiom in $\rT_{\proc}$, $z_{s_j,t_j,i_j}$ can be replaced by $z_{s_j\wedge t,t_j'\wedge t,i_j} + z_{t_j' \vee t,t_j \vee t,i_j}$.  Let
		\[
		\varphi'(\mathbf{x},y_1,\dots,y_n) = \psi(\mathbf{x},y_1+z_{s_1 \vee t - t,t_1 \vee t - t,i_1},\dots,y_n+z_{s_n \vee t - t,t_n \vee t - t,i_n}).
		\]
		By the foregoing argument, $\varphi' \in \mathcal{S}$, and
		\[
		\varphi(\mathbf{x}) = (S_t\varphi')(\mathbf{x},z_{s_1 \wedge t - t,t_1 \wedge t - t,i_1},\dots, z_{s_n \wedge t - t,t_n \wedge t - t,i_n}),
		\]
		which is an expression of the desired form \eqref{eq: expressing chron formula using shift}.
		
		(2) It is straightforward to see that formulas of the form \eqref{eq: expressing chron formula using shift} are closed under application of continuous connectives.
		
		(3) Suppose that $0 \leq s \leq t$ and $\varphi(\mathbf{x},y) \in \mathcal{F}_{\chron}(t)$ and $\varphi'(\mathbf{x}) = \sup_{y \in D_{r,s}} \varphi(\mathbf{x},y)$ (the case of an $\inf$ rather than a $\sup$ is, of course, symmetrical).  Suppose that $\varphi$ satisfies the inductive hypothesis, and so
		\[
		\varphi(\mathbf{x},y) = (S_t \psi)(\mathbf{x}, y, z_{s_1,t_1,i_1}, \dots, z_{s_n,t_n,i_n}),
		\]
		where $0 \leq s_j \leq t_j \leq t$ and $\psi \in \mathcal{S}$.  Similar to case (1), by splitting each interval $[s_j,t_j]$ at the point $s$ when applicable, we can write
		\[
		(S_t \psi)(\mathbf{x}, y, z_{s_1,t_1,i_1}, \dots, z_{s_n,t_n,i_n}) = (S_s \psi')(\mathbf{x},y,z_{s_1 \wedge s,t_1 \wedge s,i_1},\dots,z_{s_n \wedge s, t_n \wedge s, i_n}),
		\]
		where $\psi' \in \mathcal{S}$.  Then by definition of $S_s$,
		\[
		\sup_{y \in D_{s,r}} (S_s \psi')(\mathbf{x},y,z_{s_1 \wedge s,t_1 \wedge s,i_1},\dots,z_{s_n \wedge s, t_n \wedge s, i_n}) = \left( S_s \left[ \sup_{y \in D_{0,r}} \psi' \right] \right)(\mathbf{x},z_{s_1,t_1,i_1}, \dots, z_{s_n,t_n,i_n}),
		\]
		which is an expression of the desired form \eqref{eq: expressing chron formula using shift} for $\varphi$ with respect to the parameter $s$.  This completes the inductive proof of \eqref{eq: expressing chron formula using shift} and hence the claim that $\mathcal{F}_{\chron} \subseteq \mathcal{S}$ up to equivalence modulo $\rT_{\proc}$.
	\end{proof}
	
	\begin{definition} \label{def: stationary}
		We say that $\mathcal{M}$ is \emph{stationary} if for every $m \in \N_0$, for all formulas $\varphi \in \mathcal{F}_{\chron,m}$ and $t_0 > 0$ and $\mathbf{r} \in (0,\infty)^m$, we have
		\begin{equation} \label{eq: stationary def}
			\sup_{\mathbf{x} \in D_{0,\mathbf{r}}^{\mathcal{M}}} |S_{t_0} \varphi^{\mathcal{M}}(\mathbf{x}) - \varphi^{\mathcal{M}}(\mathbf{x})| = 0.
		\end{equation}
	\end{definition}
	
	\begin{definition}
		We denote by $\rT_{\stat}$ the theory consisting of $\rT_{\proc}$ together with the axioms \eqref{eq: stationary def} for chronological formulas $\varphi$.
	\end{definition}
	
	It is immediate from the definition that stationarity is an axiomatizable property using the collection of axioms \eqref{eq: stationary def}.  Hence also $\rT_{\stat}$ axiomatizes the class of noncommutative filtrations which are also stationary.
	
	\begin{lemma}
		Suppose that $\mathcal{M}$ is stationary.  Then for every formula $\varphi \in \mathcal{F}_{\operatorname{chron},m}$ and $t_0 > 0$ and $\mathbf{r} \in (0,\infty)^m$, we have
		\[
		\norm{S_{t_0} \varphi^{\cM}}_{t_0,\mathbf{r}} = \norm{\varphi^{\cM}}_{0,\mathbf{r}}.
		\]
		Hence, also for $t_0, t_1 > 0$,
		\begin{equation} \label{eq: stationarity for later times}
			\sup_{\mathbf{x} \in D_{t_0,\mathbf{r}}^{\mathcal{M}}} |S_{t_0+t_1} \varphi^{\mathcal{M}}(\mathbf{x}) - S_{t_0} \varphi^{\mathcal{M}}(\mathbf{x})| = 0.
		\end{equation}
	\end{lemma}
	
	\begin{proof}
		Fix $\varphi$ and let $\psi$ be the sentence
		\[
		\psi = \sup_{\mathbf{x} \in (D_{0,\mathbf{x}})^m} |\varphi(\mathbf{x})|.
		\]
		Then
		\[
		S_{t_0} \psi = \sup_{\mathbf{x} \in D_{t_0,\mathbf{r}}} |S_{t_0} \varphi(\mathbf{x})|.
		\]
		Hence, $\norm{\varphi^{\cM}}_{0,\mathbf{r}} = \psi^{\cM}$ and $\norm{S_{t_0} \varphi^{\cM}}_{t_0,\mathbf{r}} = (S_{t_0} \psi)^{\cM}$.  By stationarity, $(S_{t_0} \psi)^{\cM} = \psi^{\cM}$, so the first claim follows.
		
		For the second claim, let
		\[
		\eta = \sup_{\mathbf{x} \in D_{0,\mathbf{r}}^{\mathcal{M}}} |S_{t_1} \varphi^{\mathcal{M}}(\mathbf{x}) - \varphi^{\mathcal{M}}(\mathbf{x})|,
		\]
		so that
		\[
		S_{t_0} \eta = \sup_{\mathbf{x} \in D_{t_0,\mathbf{r}}^{\mathcal{M}}} |S_{t_0+t_1} \varphi^{\mathcal{M}}(\mathbf{x}) - S_{t_0} \varphi^{\mathcal{M}}(\mathbf{x})|.
		\]
		Stationarity implies that $\eta^{\cM} = 0$ and also that $S_{t_0} \eta^{\cM} = \eta^{\cM}$, hence $S_{t_0} \eta^{\cM} = 0$.
	\end{proof}
	
	\begin{lemma} \label{lem: stationary implies elementary}
		Suppose that $\mathcal{M}$ is stationary.  Let $0 \leq t_0 < t_1$.  Then the inclusion $\mathcal{M}_{t_0} \subseteq \mathcal{M}_{t_1}$ is an elementary inclusion of tracial von Neumann algebras.  Moreover, if $M_\infty = \bigvee_{t > 0} M_t$, then $\cM_{t_0} \subseteq \cM$ is an elementary embedding as well.
	\end{lemma}
	
	\begin{proof}
		Let $\varphi$ be a formula in $\mathcal{L}_{\tr}$ with $m$ free variables, and let $\varphi_0$ be the $\mathcal{L}_{\operatorname{proc}}$ formula obtained by replacing each occurrence of $D_r$ in $\mathcal{L}_{\tr}$ with $D_{0,r}$ in $\mathcal{L}_{\operatorname{proc}}$; clearly $\varphi_0$ is a chronological formula.  Then for $t > 0$, and $\mathbf{x} \in \mathcal{M}_t^m$, we have
		\[
		\varphi^{\mathcal{M}_t}(\mathbf{x}) = (S_t \varphi_0)^{\cM}(\mathbf{x}).
		\]
		Therefore, for $t_0 < t_1$, and $\mathbf{x} \in \mathcal{M}_{t_0}^m$,
		\[
		\varphi^{\mathcal{M}_{t_0}}(\mathbf{x}) = (S_{t_0} \varphi_0)^{\cM}(\mathbf{x}) = (S_{t_1} \varphi_0)^{\cM}(\mathbf{x}) = \varphi^{\cM_{t_1}}(\mathbf{x}).
		\]
		Since $\varphi$ was arbitrary, we have $\cM_{t_0} \preceq \cM_{t_1}$.  The second claim about $\cM_{t_0} \preceq \cM$ follows because $\cM$ is the inductive limit of a chain of elementary inclusions $\cM_{t_j} \to \cM_{t_{j+1}}$, and hence elementary.
	\end{proof}
	
	\begin{example}[Filtration generated by free Brownian motion] \label{ex: filtration of free Brownian motion}
		Let $M$ be the tracial von Neumann algebra generated by a family of semicircular operators associated to $(L^2(\R;\R) \oplus L^2(\R;\R)) \otimes \ell^2(\N)$.  Let
		\[
		\mathbf{z}_{s,t,j} = S((\mathbf{1}_{[s,t]} \oplus 0) \otimes e_j) + i S( (0 \oplus \mathbf{1}_{[s,t]}) \otimes e_j).
		\]
		Note that $(\mathbf{z}_{0,t,j})_{t \geq 0, j \in \N}$ is a non-self-adjoint free Brownian motion, but the increments are also extended consistently to $\R$ instead of $[0,\infty)$.  For $t \in \R$, let $M_t$ be the von Neumann subalgebra generated by $(z_{t_1,t_2,j}: t_1 \leq t_2 \leq t, j \in \N)$.  Consider the $\cL_{\proc}$-structure given by $(M_t)_{t \in [0,\infty]}$ with the process $(z_{s,t,j})_{-\infty < s \leq t < \infty, j \in \N}$.
		
		We claim that $\cM$ is stationary.  In fact, we claim that for every $\cL_{\proc}$-formula $\varphi$ (not only the chronological formulas), for every $t_0 > 0$, and every $\mathbf{x} \in M_0^m$  we have $\varphi^{\cM}(\mathbf{x}) = (S_{t_0} \varphi)^{\cM}(\mathbf{x})$.  Because formulas are uniformly continuous in $L^2$ on each operator-norm ball, it suffices to prove the claim for $\mathbf{x}$ from a dense subset of $M_0$.  Denote by $M_{s,t}$ the von Neumann algebra generated by $(z_{t_1,t_2,j}: s \leq t_1 \leq t_2 \leq t, j \in \N)$.  Thus, $M_0$ is the inductive limit of $M_{-t,0}$ as $t \to \infty$.  Hence, assume without loss of generality that $\mathbf{x} \in M_{-t_1,0}$ for some $t_1 > 0$.
		
		Now we define an automorphism $\alpha$ of $M$ as follows.  Consider the Lebesgue-measure-preserving transformation of $\R$ given by
		\[
		f(t) = \begin{cases}
			t + t_0, & t \in \R \setminus (-t_0-t_1,0] \\
			t + t_0 + t_1, & t \in (-t_0-t_1,-t_1] \\
			t, & t \in (-t_1,0].
		\end{cases}
		\]
		This induces an automorphism of $L^2(\R)$ given by $h \mapsto h \circ f^{-1}$.  In turn, by construction of the free Gaussian functor, this induces an automorphism $\alpha$ of $M$, which is easily verified to have the following properties:
		\begin{itemize}
			\item $\alpha|_{M_{-t_1,0}} = \id$.
			\item $\alpha(z_{s,t,j}) = z_{s+t_0,t+t_0,j}$ for $0 \leq s \leq t < \infty$.
			\item $\alpha(M_t) = M_{t+t_0}$ for $t \geq 0$.
		\end{itemize}
		In particular, $\alpha$ induces an isomorphism between the $\cL_{\proc}$-structure $\cM$ and the $\cL_{\proc}$-structure $S_{t_0} \cM$, and also $\alpha(\mathbf{x}) = \mathbf{x}$.  We therefore have that
		\[
		\varphi^{\cM}(\mathbf{x}) = \varphi^{S_{t_0} \cM}(\alpha(\mathbf{x})) = \varphi^{S_{t_0} \cM}(\mathbf{x}) = (S_{t_0} \varphi)^{\cM}(\mathbf{x})
		\]
		as desired.
	\end{example}
	
	\begin{remark}
		This example has the following application for the model theory of von Neumann algebras.  Let $L(\mathbb{F}_\infty)$ be the free group von Neumann algebra on infinitely many generators, which is well known to be isomorphic to the von Neumann algebra $M$ generated by a free Brownian motion above.  Then the first-factor embedding of $L(\mathbf{F}_\infty)$ into $L(\mathbb{F}_\infty) * L(\mathbb{F}_\infty)$ is an elementary embedding.  Indeed, this inclusion is isomorphic to the inclusion of $M_0$ into $M_{t_0}$ discussed above, which is elementary as a special case of stationarity (Lemma \ref{lem: stationary implies elementary}).
	\end{remark}
	
	\begin{remark} \label{rem: properties of increments}
		Various natural properties of a NC filtration and stochastic process are \emph{chronologically axiomatizable} (i.e.\ axiomatizable using chronological formulas).  In the following, we already restrict our attention to structures $\cM$ satisfying $\cL_{\proc}$.
		\begin{itemize}
			\item We say that $\cM$ has the \emph{martingale property} if for each $s \leq t$ and $j \in \N$, we have $E_{M_s}[z_{s,t,j}] = 0$.  This property is axiomatizable using the conditions
			\[
			\sup_{x \in D_{s,1}} |\tr(x z_{s,t,j})| = 0
			\]
			for each $s \leq t$ and $j \in \N$.  Note that the above expression can be written in terms of the real and imaginary parts of the trace.
			\item We say that $\cM$ has \emph{free increments} if for each $s \leq t$, the tuple $\mathbf{z}_{s,t} = (z_{s,t,j})_{j \in \N}$ is free from $M_s$.   This property is axiomatizable using the conditions
			\[
			\sup_{x_0, x_1, \dots, x_m \in D_{s,1}} |\tr[x_0[p_1(\mathbf{z}_{s,t}) - \tr(p_1(\mathbf{z}_{s,t})][x_1 - \tr(x_1)] \dots [p_k(\mathbf{z}_{s,t}) - \tr(p_m(\mathbf{z}_{s,t})] x_m ] |
			\]
			for noncommutative polynomials $p_1$, \dots, $p_m$ and $m \in \N$.  The above expression, again, can be expanded and expressed in terms of the real and imaginary parts of traces of polynomials in $\mathbf{z}_{s,t}$ and $x_1$, \dots, $x_m$.
		\end{itemize}
		Similarly, one can axiomatize (with quantifier-free sentences) the property that the increments $\mathbf{z}_{s,t}$ have real and imaginary parts which form a free semicircular family with variance $t-s$.
	\end{remark}
	
	\subsection{The restriction to $M_0$; definable predicates and functions} \label{subsec: M zero definable things}
	
	In this subsection, we consider the restriction of chronological formulas to the initial element of the filtration $M_0$ as defining a language $\cL_{\chron,0}$ which extends $\mathcal{L}_{\tr}$.  While this might appear to be merely a semantic exercise, it is useful because it will allow us to apply results about types, definable predicates, and definable functions over metric structures directly to the chronological versions, rather than repeating the proofs.
	
	If we only have access to $M_0$, then we would only be able to apply suprema and infima associated to $D_{0,r}$ for $r \geq 0$, while the suprema and infima associated to $D_{t,r}$ for $t > 0$ cannot be expressed in terms of $M_0$.  We therefore have to artificially add in predicates in $\cL_{\chron,0}$ which represent general chronological formulas.  This is essentially the same as the procedure of \emph{Morleyization} in model theory, in which a new language is defined with relation symbols for every formula in the original language; in particular, the formulas in the original language are all quantifier-free formulas in the new language.  We will then discuss several theories and axioms in the new language $\cL_{\chron,0}$ related to chronological theories in $\cL_{\proc}$.
	
	\begin{definition}
		We define $\mathcal{L}_{\operatorname{chron},0}$ as follows:
		\begin{itemize}
			\item The domains are $(D_r)_{r > 0}$ as in $\mathcal{L}_{\tr}$ the language of tracial von Neumann algebras.
			\item The function symbols describe the $*$-algebra operations, the same as in $\mathcal{L}_{\tr}$.
			\item The predicate symbols are indexed by the chronological formulas $\varphi$ in $\mathcal{L}_{\proc}$; for clarity, we denote the corresponding predicate in $\mathcal{L}_{\chron,0}$ by $\Theta(\varphi)$.
		\end{itemize}
	\end{definition}
	
	Note that a more common notation for $\Theta(\varphi)$ in the model theory literature is $P_\varphi$.
	
	\begin{definition} \label{def: theory on M zero}
		Given an $\mathcal{L}_{\chron}$-theory $\mathrm{T}$, we define an $\mathcal{L}_{\chron,0}$-theory $\Theta_0(\rT)$ as follows:
		\begin{enumerate}[(1)]
			\item For each $\mathcal{L}_{\chron}$-sentence $\varphi \in \rT$, we have $\Theta(\varphi) \in \Theta(\rT)$.
			\item For each $\mathcal{L}_{\chron}$-formula $\varphi$ in $k$ free variables and for terms $t_1$, \dots, $t_k$ in variables $(x_1,\dots,x_m)$, we have
			\[
			\sup_{x_1 \in D_{r_1}} \dots \sup_{x_m \in D_{r_m}} |\Theta(\varphi)(t_1(\mathbf{x}),\dots,t_k(\mathbf{x})) - \Theta(\varphi(t_1,\dots,t_m))(\mathbf{x})| = 0.
			\]
			\item For each $\mathcal{L}_{\chron}$-formula $\varphi(\mathbf{x},y)$ in $m+1$ variables, letting $\psi(\mathbf{x}) = \sup_{y \in D_{0,r}} \varphi(\mathbf{x},y)$, we have
			\[
			\sup_{x_1 \in D_{r_1}} \dots \sup_{x_m \in D_{r_m}} |\sup_{y \in D_{r'}} \Theta(\varphi)(\mathbf{x},y) - \Theta(\psi)(\mathbf{x})| \in \Theta(\rT),
			\]
			and the same with $\sup_{y \in D_{0,r}}$ and $\sup_{y \in D_r}$ replaced by $\inf_{y \in D_{0,r}}$ and $\inf_{y \in D_r}$ respectively.
			\item For each $f: \R^k \to \R$ is continuous and $\varphi_1$, \dots, $\varphi_k$ are $\mathcal{L}_{\chron}$-formulas in $m$ free variables, we have
			\[
			\sup_{x_1 \in D_{r_1}} \dots \sup_{x_m \in D_{r_m}} |f(\Theta(\varphi_1),\dots,\Theta(\varphi_k)) - \Theta(f(\varphi_1,\dots,\varphi_k))| \in \Theta(\rT).
			\]
		\end{enumerate}
	\end{definition}
	
	\begin{remark}
		Intuitively, these axioms are saying that the map $\varphi \mapsto \Theta(\varphi)$ respects composition with terms, connectives, and the supremum operation if we understand the underlying space for the $\mathcal{L}_{\chron,0}$-structure to be $0$th level of the filtration in the $\mathcal{L}_{\proc}$-structure.  Hence, the next lemma should not be surprising.
	\end{remark}
	
	\begin{lemma}
		Let $\rT$ be an $\mathcal{L}_{\chron}$ theory.  Then every formula in $\mathcal{L}_{\chron,0}$ is equivalent modulo $\Theta_0(\rT)$ to a formula of the form $\Theta(\varphi)$ for some $\mathcal{L}_{\chron}$-formula $\varphi$.
	\end{lemma}
	
	\begin{proof}
		We proceed by induction on the complexity of formulas.
		\begin{enumerate}[(1)]
			\item Suppose that $\varphi$ is a basic formula in $\mathcal{L}_{\chron,0}$.  Then $\varphi$ has the form $\Theta(\psi)(t_1,\dots,t_k)$ for some predicate symbol $\Theta(\psi)$ and terms $t_1$, \dots, $t_k$ in light of Definition \ref{def: theory on M zero} (1).  By Definition \ref{def: theory on M zero} (2), $\varphi$ is equivalent to $\Theta(\psi(t_1,\dots,t_k))$.
			\item Suppose that $\varphi(\mathbf{x}) = \sup_{y \in D_r} \psi(\mathbf{x},y)$.  By inductive hypothesis, assume $\psi$ is equivalent to $\Theta(\psi')$ for some $\psi' \in \mathcal{L}_{\chron}$.  Let $\varphi' = \sup_{y \in D_{0,r}} \psi'(\mathbf{x},y)$, which is in $\mathcal{L}_{\chron}$. Then using Definition \ref{def: theory on M zero} (3), $\varphi$ is equivalent to $\Theta(\varphi')$.
			\item Suppose that $\varphi = f(\varphi_1,\dots,\varphi_k)$ where $f$ is a continuous connective.  Then applying the inductive hypothesis to $\varphi_1$, \dots, $\varphi_k$ and using Definition \ref{def: theory on M zero}, we obtain that $\varphi$ is equivalent to $\Theta(\varphi')$ for some $\varphi'$.  \qedhere
		\end{enumerate}
	\end{proof}
	
	\begin{lemma} \label{lem: restricted structure}
		Suppose that $\mathcal{M}$ is an $\cL_{\filt}$-structure satisfying a chronological theory $\rT$.  Let $\cM_0$ be the $\mathcal{L}_{\chron,0}$-structure obtained as follows:
		\begin{itemize}
			\item $D_r$ in $\mathcal{L}_{\chron,0}$ is interpreted as $D_{0,r}^{\cM}$.
			\item The function symbols in $\cM_0$ are the restrictions of the function symbols from $\cM$.
			\item For $\varphi \in \mathcal{L}_{\chron}$, $\Theta(\varphi)$ is interpreted as $\varphi^{\cM}$ restricted to inputs in $\bigcup_{r > 0} D_r^{\cM_0}$.
		\end{itemize}
		Then $\cM_0$ satisfies $\Theta_0(\rT)$.
	\end{lemma}
	
	\begin{proof}
		The result follows by straightforwardly checking that $\cM_0$ satisfies each of the kinds of sentences in $\Theta_0(\rT)$.  For instance the sentences of type (1) hold because for each sentence $\varphi \in \rT$, we have that $\Theta(\varphi)^{\cM_0} = \varphi^{\cM}$ by definition.  For type (2), this is because the function symbols in $\cM_0$ are the restrictions of those in $\cM$.  Cases (3) and (4) follow similarly.
	\end{proof}
	
	\begin{definition}
		Let $\rT$ be a $\cL_{\chron}$-theory.  Define $\Theta_1(\rT)$ as the set of $\cL_{\chron,0}$-sentences satisfied by $\cM_0$ by $\cM_0$ for all $\cL_{\filt}$-structures satisfying $\rT$.
	\end{definition}
	
	\begin{remark}
		Note it is immediate from Lemma \ref{lem: restricted structure} that $\Theta_1(\rT)$ implies $\Theta_0(\rT)$, and hence also that every $\cL_{\chron,0}$-formula is equivalent to a formula of the form $\Theta(\varphi)$.  Such a formula is, by construction, considered a quantifier-free formula in $\cL_{\chron,0}$.
	\end{remark}
	
	\begin{remark} \label{rem: restriction of theories}
		$\Theta_1(\rT)$ can also be described as follows.  If we view complete theories as linear functionals as in \cite[Lemma 3.8ff]{GaoJekel2025}, the mapping that sends $\Th(\cM)$ to $\Th(\cM_0)$ for $\cL_{\filt}$-structures $\cM$ is a restriction mapping and hence is continuous.  Thus, the image of the restriction map gives a closed set of complete $\cL_{\chron,0}$-theories, which corresponds to an axiomatizable class of $\cL_{\chron,0}$-structures.  Then $\Theta_1(\rT)$ is the complete theory of this axiomatizable class.  From this reasoning, it follows that every consistant complete theory extending $\Theta(\rT)$ can be realized by the restriction $\cM_0$ for some $\cM$ satisfying $\rT$.
	\end{remark}
	
	With these preparations in hand, we set up the definitions and basic facts for chronological definable predicates, definable functions, and types.
	
	\begin{definition}[Chronologically definable predicates] \label{def: chron def pred}
		Let $\rT$ be an $\cL_{\chron}$-theory. The space of \emph{chronologically definable predicates} in $m$ free variables in $M_0$ is the completion of $\mathcal{F}_{\chron,m}$ with respect to the family of seminorms $\norm{\cdot}_{0,r,\rT}$ for $r \in (0,\infty)$.
	\end{definition}
	
	The following fact is now straightforward to verify from the setup.
	
	\begin{lemma} \label{lem: isomorphism of definable predicates}
		For $\cL_{\chron}$ formulas in $m$-free variables, we have $\norm{\varphi}_{0,r,\rT} = \norm{\Theta(\varphi)}_{r,\Theta_1(\rT)}$.  Hence, the space of chronologically definable predicates with inputs in $\cM_0$ is isomorphic as a topological vector space to the space of definable predicates in $m$ free variables in $\cL_{\chron,0}$ with respect to $\Theta_1(\rT)$.
	\end{lemma}
	
	\begin{definition}[Chronological types] ~
		\begin{enumerate}[(1)]
			\item Let $\cM$ be an $\cL_{\filt}$-structure and $\mathbf{x} \in \cM_0^m$.  The \emph{chronological type} of $\mathbf{x}$ is the map $\tp_{\chron}^{\cM}: \cF_{\chron,m} \to \R$, $\varphi \mapsto \varphi^{\cM}(\mathbf{x})$.
			\item For $m \in \N \cup \{\infty\}$, define $\mathbb{S}_{\chron,\mathbf{r}}(\rT)$ as the set of chronological types $\tp_{\chron}^{\cM}(\mathbf{x})$ for $\cM$ satisfying $\rT$ and $\mathbf{x} \in D_{0,\mathbf{r}}^{\cM}$.  We equip $\mathbb{S}_{\chron,\mathbf{r}}(\rT)$ with the weak-$*$ topology.
			\item The \emph{space of chronological types} is $\mathbb{S}_{\chron,m}(\rT) = \bigcup_{\mathbf{r} \in (0,\infty)^m} \mathbb{S}_{\chron,\mathbf{r}}(\rT)$ equipped with the inductive limit topology.
		\end{enumerate}
	\end{definition}
	
	\begin{definition}
		For an $\cL_{\chron,0}$-theory $\rT$, and $\mathbf{r} \in (0,\infty)^m$, let $\mathbb{S}_{\mathbf{r}}(\rT)$ denote the space of types of $m$-tuples in $D_{0,\mathbf{r}}$, and similarly $\mathbb{S}_m(\rT)$ the space of types of all $m$-tuples.
	\end{definition}
	
	\begin{lemma} \label{lem: isomorphism of types}
		Let $\rT$ be an $\cL_{\chron}$-theory.  The map $\Theta$ from $\mathcal{F}_{\chron,m}$ to $\mathcal{L}_{\chron,0}$-formulas in $m$ variables induces homeomorphisms $\mathbb{S}_{\chron,\mathbf{r}}(\rT) \to \mathbb{S}_{\mathbf{r}}(\Theta_1(\rT))$ and $\mathbb{S}_{\chron,m}(\rT) \to \mathbb{S}_m(\Theta_1(\rT))$.
	\end{lemma}
	
	\begin{corollary}
		There is an isomorphism of topological vector spaces from the space of chronologically definable predicates with respect to $\rT$ to $C(\mathbb{S}_m(\rT))$, such that if $f_\varphi$ is the function corresponding to the definable predicate $\varphi$, then $\norm{\varphi}_{0,r,\rT} = \norm{f_\varphi|_{\mathbb{S}_m(\rT)}}_u$.
	\end{corollary}
	
	\begin{proof}
		It is standard that the space of definable predicates with respect to $\cL_{\chron,0}$ and $\Theta_1(\rT)$ is isomorphic to the continuous functions on $\mathbb{S}_m(\Theta_1(\rT))$ (see \cite[Proposition 3.9]{JekelCoveringEntropy} for the setting with domains of quantification used here).  We then apply Lemmas \ref{lem: isomorphism of definable predicates} and \ref{lem: isomorphism of types} to transform this into a statement about the spaces of chronological definable predicates and types.
	\end{proof}
	
	\begin{definition}[Chronologically definable functions] \label{def: chr def funct}
		Let $\rT$ be an $\cL_{\chron}$-theory.  A \emph{chronologically definable function} in $m$ variables is a collection of functions $f^{\cM}: (\cM_0)^m \to \cM_0$ for models $\cM$ of $\rT$ such that
		\begin{enumerate}[(1)]
			\item For every $r > 0$, there exists $r' > 0$ such that $f^{\cM}$ maps $D_{0,r}^{\cM}$ into $D_{0,r'}^{\cM}$ for every $\cM$ satisfying $\rT$.
			\item There exists a chronologically definable predicate $\varphi$ relative to $\rT$ such that $d^{\cM}(f^{\cM}(\mathbf{x}),y) = \varphi^{\cM}(\mathbf{x},y)$ for all $\cM$ satisfying $\rT$ and all $\mathbf{x} \in \cM_0^m$ and $y \in \cM_0$.
		\end{enumerate}
	\end{definition}
	
	\begin{remark} \label{rem: extending definable functions}
		Although the class of models of $\rT$ is not a set, it is sufficient for $f^{\cM}$ to be given for a single $\cM$ in each elementary equivalence class for $\rT$ that satisfies (2) for the same $\varphi$.  Indeed, it is clear that for each $\cM$, $f^{\cM}$ is definable over $\cM$ as in \cite[Definition 2.22]{BYBHU2008} and thus by \cite[Definition 2.25]{BYBHU2008}, $f^{\cM}$ admits a canonical extension in any elementary extension $\cN$ of $\cM$, and of course the same definable predicate $\varphi$ satisfies (2) in $\cN$.  Any $\cM' \equiv \cM$ can be elementarily embedded into some elementary extension $\cN$ of $\cM$ (this is an easier form of the Keisler--Shelah theorem; see \cite[Theorem 5.7 and Proposition 7.10]{BYBHU2008}), and $f^{\cN}$ will leave an elementary submodel invariant, and so we obtain $f^{\cM'}$ as the restriction.  But then we note that by the downward L{\"o}wenheim--Skolem theorem \cite[Proposition 7.3]{BYBHU2008}, each elementary equivalence class has a representative with a fixed bound on its cardinality that only depends on the language. In conclusion, by specifying $f^{\cM}$ only for set of equivalence class representatives, we see that the class of definable functions is indeed a set.
	\end{remark}
	
	\begin{lemma}
		Let $\rT$ be an $\cL_{\chron}$-theory and fix $m \in \N$.  There is a bijection $\Theta$ from chronologically definable functions in $m$-variables with respect to $\rT$ and definable functions with respect to $\Theta_1(\rT)$ such that for each $\cM$ satisfying $\rT$, we have $\Theta(f)^{\cM_0} = f^{\cM}$ as functions on $\cM_0$.
	\end{lemma}
	
	\begin{proof}
		For each chronologically definable $f$, suppose $\varphi$ is a corresponding chronologically definable predicate in Definition \ref{def: chr def funct} (2).  Now $\Theta(f)^{\cM_0}$ can defined as $f^{\cM}$ when $\cM_0$ is the restriction of an $\cL_{\filt}$-structure $\cM$ satisfying $\rT$.  By Remark \ref{rem: restriction of theories}, every $\cN$ satisfying $\Theta_1(\rT)$ is elementarily equivalent to such an $\cM_0$, so, as in Remark \ref{rem: extending definable functions}, $\Theta(f)^{\cN}$ can be canonically defined for any $\cN$ satisfying $\Theta_1(\rT)$.
	\end{proof}
	
	\begin{corollary}
		Let $f = (f_1,\dots,f_{m_2})$ be a tuple of chronologically definable functions in $m_1$ variables, and let $g = (g_1,\dots,g_{m_3})$ be a tuple of chronologically definable functions in $m_2$ variables.  Then $g \circ f$ is a tuple of chronologically definable functions.
		
		Similarly, if $\varphi$ is a chronologically definable predicate in $m_2$ variables, then $\varphi \circ f$ is a chronologically definable predicate.
	\end{corollary}
	
	\begin{proof}
		This follows by rewriting the statement in terms of $\Theta_1(\rT)$ and $\cL_{\chron,0}$ and then applying the known results about definable functions and definable predicates.  See e.g.\ \cite[Definition 3.16, Proposition 3.17, Observation 3.21]{JekelCoveringEntropy}.
	\end{proof}
	
	We conclude the section by showing that the space of chronologically definable predicates is separable.
	
	\begin{lemma}[Separability of chronologically definable predicates] \label{lem: separability}
		Fix $m \in \N \cup \{\infty\}$.  The space of chronological definable predicates with repsect to $\rT_{\stat}$ is separable with respect to the family of seminorms $\norm{\cdot}_{0,\mathbf{r},\rT_{\stat}}$ for $\mathbf{r} \in (0,\infty)^m$.  In particular, $\mathbb{S}_{\chron,\mathbf{r}}(\rT_{\stat})$ is metrizable in the weak-$*$ topology for each $\mathbf{r}$.
	\end{lemma}
	
	\begin{proof}
		Consider the countable subset $\mathcal{S}_0$ of chronological formulas generated as in Lemma \ref{lem: chronological formulas alternate} with the following modifications:
		\begin{enumerate}[(1)]
			\item The terms in the basic formulas are defined using scalar multiplication by $\lambda \in \Q[i]$ rather than $\C$.
			\item The connectives are chosen from a fixed countable dense subset of $C(\R^k;\R)$ for each $k$ (for instance, polynomials with rational coefficients or certain sets of piecewise affine functions).
			\item The quantifiers $\sup_{y \in D_{0,r}}$ and $\inf_{y \in D_{0,r}}$ are only applied for rational $r > 0$.
			\item We construct formulas $(S_t \varphi)(\mathbf{x},z_{0,t,j_1},\dots,z_{0,t,j_k})$ only for rational values of $t \geq 0$.
		\end{enumerate}
		To prove that every chronological formula can be approximated by elements of $\mathcal{S}_0$ with respect to $\norm{\cdot}_{0,\mathbf{r}}$, we proceed by induction using the method of generating chronological formulas given in Lemma \ref{lem: chronological formulas alternate}.  For the base case, it is an exercise to show that a basic formula can always be approximated in $\norm{\cdot}_{0,\mathbf{r}}$ by basic formulas where the scalar multiplications are taken only from $\Q[i]$.  Similarly, the case of connectives is immediate since we chose a dense set of connectives in $C(\R^k;\R)$.
		
		Now suppose that $\varphi(\mathbf{x}) = \sup_{y \in D_{0,r}} \psi(\mathbf{x},y)$ where $\psi$ can be approximated by elements of $\mathcal{S}_0$.  Let $\psi'(\mathbf{x},y) \in \mathcal{S}_0$ approximate $\psi$ within $\varepsilon/2$ in $\norm{\cdot}_{0,(\mathbf{r},r),\rT_{\stat}}$.  Since $\psi'$ is uniformly continuous on $D_{0,(\mathbf{r},r)}$, there exists $\delta > 0$ such that for $\cM \models \rT_{\stat}$ and $\mathbf{x} \in D_{0,\mathbf{r}}$ and $y, y' \in D_{0,r}$, we have
		\[
		\norm{y - y'}_2 < \delta \implies |(\psi')^{\cM}(\mathbf{x},y) - (\psi')^{\cM}(\mathbf{x},y')| \leq \frac{\varepsilon}{2}.
		\]
		Let $r'$ be a rational number with $|r - r'| < \delta$.  Since every element of the $r$-ball is within distance $\delta$ of some element of the $r'$-ball and vice versa,
		\[
		\left| \sup_{y \in D_{0,r'}} (\psi')^{\cM}(\mathbf{x},y) - \sup_{y \in D_{0,r}} (\psi')^{\cM}(\mathbf{x},y) \right| \leq \frac{\varepsilon}{2}.
		\]
		Then setting $\varphi'(\mathbf{x}) = \sup_{y \in D_{0,r'}} \psi'(\mathbf{x},y) \in \mathcal{S}_0$, we have $\norm{\varphi - \varphi'}_{0,\mathbf{r},\rT_{\stat}} \leq \varepsilon$.
		
		Finally, suppose that $t > 0$ and $j_1$, \dots, $j_k \in \N$ and $\varphi(\mathbf{x}) = (S_t \psi)(\mathbf{x},z_{0,t,j_1}, \dots, z_{0,t,j_k})$.  Let $\mathbf{r}'$ be the tuple with entries all equal to $3 \sqrt{t}$.  Choose $\psi'$ that approximates $\psi$ within $\varepsilon/2$ with respect to $\norm{\cdot}_{0,(\mathbf{r},\mathbf{r}'),\rT_{\stat}}$.  By uniform continuity, fix $\delta > 0$ such that for $\cM \models \rT_{\stat}$ and $\mathbf{x} \in D_{0,\mathbf{r}}$ and $\mathbf{y}, \mathbf{y}' \in D_{0,\mathbf{r}'}$, we have
		\[
		\max_j \norm{y_j - y_j'}_2 < \delta \implies |(\psi')^{\cM}(\mathbf{x},\mathbf{y}) - (\psi')^{\cM}(\mathbf{x},\mathbf{y}')| \leq \frac{\varepsilon}{2}.
		\]
		Let $t' \in [0,t]$ be rational such that $3\sqrt{t - t'} < \delta$, and hence $\norm{z_{0,t,j}^{\cM} - z_{0,t',j}^{\cM}}_2 < \delta$ for $j \in \N$, which implies that
		\[
		|(S_t \psi)^{\cM}(\mathbf{x},z_{0,t,j_1}^{\cM}, \dots, z_{0,t,j_k}^{\cM}) - (S_t \psi)^{\cM}(\mathbf{x},z_{0,t',j_1}^{\cM}, \dots, z_{0,t',j_k}^{\cM})| \leq \frac{\varepsilon}{2}.
		\]
		Moreover, using condition \eqref{eq: stationarity for later times} from stationarity, we have
		\[
		(S_t \psi)^{\cM}(\mathbf{x},z_{0,t',j_1}^{\cM}, \dots, z_{0,t',j_k}^{\cM}) = (S_{t'} \psi)^{\cM}(\mathbf{x},z_{0,t',j_1}^{\cM}, \dots, z_{0,t',j_k}^{\cM}).
		\]
		Therefore, letting $\varphi'(\mathbf{x}) = (S_{t'} \psi')(\mathbf{x},z_{0,t',j_1}, \dots, z_{0,t',j_k})$, we have that $\norm{\varphi' - \varphi}_{0,\mathbf{r},\rT_{\stat}} \leq \varepsilon$.
	\end{proof}

	\section{Free Brownian motion in the matrix ultraproduct} \label{sec: matrix ultraproduct}
	
	In this section, we begin to relate chronological formulas with random matrices.  In particular, we want to obtain a filtration of von Neumann algebras and a free Brownian motion as an ultraproduct of the Brownian motions in $n \times n$ matrices (compare \cite{Dabrowski2017Laplace}).
	
	We consider a classical filtration of $\sigma$-algebras $\mathcal{G}_t$ on a probability space $(\Omega,\mathcal{G},\mathbb{P})$.   For concreteness, one can assume that $\mathcal{G}_0$ is a $\sigma$-algebra associated to an atomless (or diffuse) probability space measure space and $\mathcal{G}_t$ is the $\sigma$-algebra generated by $\mathcal{G}_0$ and an independent Hermitian Brownian motion $(Z_{0,t,j}^{(n)})$ for $t \geq 0$.  Let $Z_{s,t,j}^{(n)} = Z_{0,t,j}^{(n)} - Z_{0,s,j}^{(n)}$.  We then study the random matrix ultraproduct given by
	\begin{align}
		M &:= \prod_{n \to \cU} (L^\infty(\Omega,\mathcal{G},\mathbb{P}) \otimes \mathbb{M}_n), \label{eq: random matrix ultraproduct} \\
		M_t &:= \prod_{n \to \cU} (L^\infty(\Omega,\mathcal{G}_t,\mathbb{P}) \otimes \mathbb{M}_n). \label{eq: random matrix ultraproduct 2}
	\end{align}

	Then $(L^\infty(\Omega,\mathcal{G}_t,\mathbb{P}) \otimes \mathbb{M}_n)_{t \in [0,\infty]})$ is a filtration of von Neumann algebras with the trace given by $\mathbb{E} \otimes \tr_n$.  However, as usual in the study of asymptotic freeness, we want to study almost sure convergence of traces rather than merely convergence of the expected trace.  We thus want to be able to interpret formulas as \emph{pointwise} functions on $\mathbb{M}_n$.
	
	While \cite{Jekel2025InfoGeom} simply evaluated $\cL_{\tr}$-formulas on $\mathbb{M}_n$ pointwise, this is not an option in our setting because our filtration is not defined on $\mathbb{M}_n$ itself, but rather inherently uses the classical randomness.  Instead, we use the idea of center-valued evaluation of formulas from \cite{GaoJekel2025}, where formulas are interpreted as functions with inputs in the von Neumann algebra and outputs in the center, rather than the real numbers.  In particular for $L^\infty(\Omega,\mathcal{G},\mathbb{P}) \otimes \mathbb{M}_n$, this produces a function in $L^\infty(\Omega,\mathcal{G},\mathbb{P})$.
	
	We study the behavior of these center-valued formulas in the 
	We take a quotient $\pi: \mathcal{M} \to \mathcal{Q}$ by a maximal ideal, which produces a $\mathrm{II}_1$-factor $\mathcal{Q}$ elementarily equivalent to the matrix ultraproduct $\prod_{n \to \cU} \mathbb{M}_n$.  Following \cite[\S 6]{GaoJekel2025}, this can be viewed as a kind of ``derandomization''; rather than selecting a point $\omega_n$ in $\Omega$ for each $n$, we perform the ultraproduct first and then take the quotient that removes the center.  The quotients $\mathcal{Q}_t = \pi(\mathcal{M}_t)$ provide a filtration of $\mathrm{II}_1$ factors with an associated free Brownian motion $z_{s,t,j}$ obtained as the ultraproduct of the matrix Brownian motions.
	
	A key point in making this construction work is to show that the center-valued evaluation of a chronological formula $\varphi$ in the random matrix ultraproduct can be approximated asymptotically by a pointwise function $\Lambda_{\varphi}^{(n)}: (\mathbb{M}_n)^m \to \mathbb{M}_n$ (see Proposition \ref{prop: approx by pointwise function}).  The functions $\Lambda_{\varphi}^{(n)}$ will also be crucial for our definition of entropy for chronological types in \S \ref{subsec: entropy definition}, and allows us to connect classical stochastic optimization problems with their noncommutative counterparts in \S \ref{subsec: stochastic variational}.
	
	\subsection{Operator-valued interpretation of formulas} \label{subsec: center valued formulas}
	
	The following definition is the analog of \cite[Definition 3.11]{GaoJekel2025} but for $\cL_{\filt}$ rather than $\cL_{\tr}$.  For (2) below, we recall that multivariable continuous function calculus is well-defined on any \emph{commutative} $\mathrm{C}^*$-algebra.  For (3), we recall that every bounded family of self-adjoint elements in a \emph{commutative} von Neumann algebra has a supremum.
	
	\begin{definition}[$N$-valued interpretation of formulas]
		Let $\cM = (M_t)_{t \in [0,\infty]}$ be a filtration of von Neumann algebras and let $\mathbf{z} = (z_{s,t,j})_{0\leq s \leq t < \infty, j\in \N}$ be elements $z_{s,t,j} \in D_{0,6 \sqrt{t-s}}^{\cM}$.  Let $N \subseteq Z(M_\infty)$ and let $E: M_\infty \to N$ be a normal, faithful, tracial conditional expectation.  For formulas $\varphi$ in the language of tracial von Neumann algebras in $n$-variables, we define the \textit{$N$-valued interpretation} $\varphi_E^{\cM}(x_1,\dots,x_m)$ as follows, by recursion on the complexity of $\varphi$.
		\begin{enumerate}[(1)]
			\item If $\varphi = \re \tr(p(\mathbf{x},z_{s_1,t_1,j_1},\dots,z_{s_n,t_n,j_n})$ for some $*$-polynomial $p$, then
			\[
			\varphi^{\cM}_E(\mathbf{x}) = \re(E(p(\mathbf{x},z_{s_1,t_1,j_1},\dots,z_{s_n,t_n,j_n}))
			\]
			for $\mathbf{x} \in M_\infty^m$.  We make the same definition for $\im \tr_n$ instead of $\re \tr_n$.  For a formula $\varphi(\mathbf{x}) = d(p(\mathbf{x},z_{s_1,t_1,j_1},\dots,z_{s_n,t_n,j_n}),q(\mathbf{x},z_{s_1,t_1,j_1},\dots,z_{s_n,t_n,j_n}))$, we set
			\[
			\varphi^{\cM}_E(\mathbf{x}) = \sqrt{E[| (p-q)(\mathbf{x},z_{s_1,t_1,j_1},\dots,z_{s_n,t_n,j_n})|^2]}.
			\]
			\item If $\varphi = f(\varphi_1, \cdots, \varphi_k)$ where $f: \mathbb{R}^k \to \mathbb{R}$ is continuous, then $\varphi^{\cM}_E(\mathbf{x}) = f((\varphi_1)^{\cM}_E(\mathbf{x}), \cdots, (\varphi_n)^{\cM}_E(\mathbf{x}))$, where $f$ is applied in the sense of continuous functional calculus.
			\item If $\varphi = \sup_{y \in D_{t,r}} \phi(\mathbf{x}, y)$, then
			\[
			\varphi^{\cM}_E(\mathbf{x}) = \sup_{y \in D_{t,r}^{\cM}} \psi^{\cM}_E(\mathbf{x}, y)
			\]
			The $\inf$ case is defined similarly.
		\end{enumerate}
	\end{definition}
	
	First, in order to take the suprema and infima in the above definition, we note that each formula will be bounded on each operator norm ball.  The proof proceeds the same way as in \cite{GaoJekel2025}, so we omit the details.
	
	\begin{lemma}[{See \cite[Lemma 3.12]{GaoJekel2025}}]
		Let $\cM = (M_t)_{t \in [0,\infty]}$ be a filtration of von Neumann algebras and let $\mathbf{z} = (z_{s,t,j})_{0\leq s \leq t < \infty, j\in \N}$ be elements $z_{s,t,j} \in D_{0,6 \sqrt{t-s}}^{\cM}$.  Let $N \subseteq Z(M_\infty)$ and let $E: M_\infty \to N$ be a normal, faithful, tracial conditional expectation.  Let $\varphi$ be an $\cL_{\proc}$-formula with $m$ free variables. Then there exists $C > 0$ s.t. $\norm{\varphi^{\cM}_E(\mathbf{x})} \leq C$ for all $\mathbf{x} \in (D_{\infty,r}^{\cM})^m$. Furthermore, $C$ can be chosen independent of $\cM$, $N$, and $E$.
	\end{lemma}
	
	We now want to adapt the results of \cite[\S 3]{GaoJekel2025} for the setting of $\cL_{\chron}$ rather than $\cL_{\tr}$.  Most of the proofs work in the exact same way, but they require an important hypothesis.
	
	\begin{definition}[Chronological invariance property]
		Let $(M_t)_{t \in [0,\infty]}$ be a filtration of von Neumann algebras and let $\mathbf{z} = (z_{s,t,j})_{0\leq s \leq t < \infty, j\in \N}$ be elements $z_{s,t,j} \in D_{0,6 \sqrt{t-s}}^{\cM}$.  Let $N \subseteq Z(M_\infty)$ and let $E: \cM \to \cN$ be a normal, faithful, tracial conditional expectation.  We say that $\cM = ((M_t)_{t \in [0,\infty]}, \mathbf{z},N,E)$ has the \emph{chronological invariance property} if whenever $\varphi \in \mathcal{F}_{\chron}(t)$ is a formula with $m$ free variables, we have
		\[
		x_1, \dots, x_m \in M_t \implies \varphi_E^{\cM}(x_1,\dots,x_m) \in M_t \cap N.
		\]
	\end{definition}
	
	The chronological invariance property is somewhat similar to various notions of adaptedness in classical and noncommutative probability (compare for instance \cite[Definition 6.2(iv)-(v)]{JKN2026martingale}).  We will show in the next section that the random matrix ultraproducts \eqref{eq: random matrix ultraproduct} satisfy the chronological invariance property.  It is natural to ask for general criteria to test the chronological invariance property, but we leave this question for future research.
	
	The key point in \cite[\S 3]{GaoJekel2025} is the behavior of the $N$-valued interpretation of formulas with respect to partitioning by projection-valued measures.  Recall that a \emph{projection-valued measure (PVM)} over $\cM$ is a family $(e_j)_{j \in J}$ of projections in $\cM$ such that $\sum_{j \in J} e_j = 1$ in the strong operator topology.  (Of course, projection-valued measures in the sense of the spectral theorem are generally defined over a measure space, rather than a discrete index set.  The terminology here is borrowed from the quantum setting, but can be regarded as a special case of projection-valued measures in the functional-analytic sense.)  In adapting \cite[Lemma 3.14]{GaoJekel2025}, we need to assume that the PVM takes values in $\cN \cap \cM_t$ when the formula is in $\mathcal{F}_{\chron}(t)$, which is what will lead us to the chronological invariance property as a hypothesis in the later statements (see Proposition \ref{prop: maximizer}).
	
	\begin{lemma}[{Compare \cite[Lemma 3.14]{GaoJekel2025}}]\label{lem: formula partition}
		Let $(M_t)_{t \in [0,\infty]}$ be a filtration of von Neumann algebras and let $\mathbf{z} = (z_{s,t,j})_{0\leq s \leq t < \infty, j\in \N}$ be elements $z_{s,t,j} \in D_{0,6 \sqrt{t-s}}^{\cM}$.  Let $N \subseteq Z(M_\infty)$ and let $E: M_\infty \to N$ be a normal, faithful, tracial conditional expectation.  Let $\varphi \in \cL_{\chron}(t)$ be a formula with $m$ free variables $\mathbf{x} = (x_1, \cdots, x_m)$. Let $\{e_j\}_{j \in J}$ be a PVM over $M_t \cap N$ and $x_{i,j} \in M_t$ for $i = 1, \dots, n$ and $j \in J$ with $\sup_j \norm{x_{i,j}} < \infty$ for each $i$. Then
		\[
		\sum_{j \in J} e_j\varphi^{\cM}_E(x_{1,j},\dots,x_{n,j}) = \varphi^{\cM}_E\left( \sum_{j \in J} e_j x_{1,j},\dots, \sum_{j \in J} e_j x_{n,j} \right).
		\]
	\end{lemma}
	
	\begin{proof}
		The proof proceeds by induction on the complexity of the formula, using Definition \ref{def: chronological formulas}.  The base case of basic formulas and the inductive case of continuous connectives are proved in the exact same way as \cite[Lemma 3.14]{GaoJekel2025}, so we omit the details.
		
		Now let $s \leq t$ and suppose that $\varphi(\mathbf{x}) = \sup_{y \in D_{s,r}} \psi(\mathbf{x},y)$ where $\psi \in \cL_{\chron}(t)$ satisfies the induction hypothesis.  Let $(e_j)_{j \in J}$ be a PVM over $M_s \cap N$.  Let $x_{i,j} \in M_s$ with $\sup_j \norm{x_{i,j}} < \infty$, and let $x_i = \sum_j e_j x_{i,j}$.  Because we assumed that $e_j \in N \cap M_s$, we know that $x_i \in \cM_s$.  As noted in \cite{GaoJekel2025}, if $(z_i)_{i \in I}$ is a family of self-adjoint elements of $\cN$ and $e$ is a projection in $\cN$, then $e \sup_i z_i = \sup_i e z_i$.  Therefore,
		\begin{align*}
			e_j \varphi_E^{\cM}(x_{1,j},\dots,x_{m,j}) &= e_j \sup_{y \in D_{s,r}^{\cM}} \psi_E^{\cM}(x_{1,j},\dots,x_{m,j},y) \\
			&= \sup_{y \in D_{s,r}^{\cM}} e_j \psi_E^{\cM}(x_{1,j},\dots,x_{n,j},y) \\
			&= \sup_{y \in D_{s,r}^{\cM}} e_j \psi_E^{\cM}(e_j x_{1,j},\dots,e_jx_{m,j}, e_jy) \\
			&= \sup_{y \in D_{s,r}^{\cM}} e_j \psi_E^{\cM}(e_j x_1,\dots,e_j x_m, e_j y).
		\end{align*}
		Here the third equality follows from applying the inductive hypothesis to $\psi$ with the families of operators $(\mathbbm{1}_{j=j'} x_{i,j'})_{i,j}$ and $(\mathbbm{1}_{j=j'} y)_{j \in J}$ and the family of projections $(e_j)_{j \in J}$; importantly, the inductive hypothesis can be applied here because $e_j \in M_s \cap N \subseteq M_t \cap N$.  By the symmetrical manipulations,
		\[
		e_j \varphi_E^{\cM}(x_1,\dots,x_m) = \sup_{y \in D_{s,r}^{\cM}} e_j \psi_E^{\cM}(e_j x_1,\dots,e_j x_m, e_j y).
		\]
		Hence, for each $j \in J$,
		\[
		e_j \varphi_E^{\cM}(x_1,\dots,x_m) = e_j \varphi_E^{\cM}(x_{1,j},\dots,x_{m,j}).
		\]
		Summing over $j$ shows that
		\[
		\varphi_E^{\cM}(x_1,\dots,x_m) = \sum_{j \in J} e_j \varphi_E^{\cM}(x_{1,j},\dots,x_{m,j}), 
		\]
		which completes the proof.  The case of an infimum is symmetrical.
	\end{proof}
	
	\begin{proposition}[{Compare \cite[Proposition 3.13]{GaoJekel2025}}] \label{prop: maximizer}
		Let $\cM = (M_t)_{t \in [0,\infty]}$ be a filtration of von Neumann algebras and let $\mathbf{z} = (z_{s,t,j})_{0\leq s \leq t < \infty, j\in \N}$ be elements $z_{s,t,j} \in D_{0,6 \sqrt{t-s}}^{\cM}$.  Let $N \subseteq Z(M_\infty)$ and let $E: M_\infty \to N$ be a normal, faithful, tracial conditional expectation.  Assume the chronological invariance property.  Let $0 \leq s \leq t$ and $r > 0$, and let $\psi \in \mathcal{F}_{\chron}(t)$ be a formula with $m+1$ free variables $\mathbf{x} = (x_1, \cdots, x_m,y)$, and let
		\[
		\varphi(\mathbf{x}) = \sup_{y \in D_{r,s}} \psi(\mathbf{x},y).
		\]
		Then for every $(x_1,\dots,x_m) \in \cM_s^m$ and $\epsilon > 0$, there exists $y \in D_{s,r}^{\cM}$ such that 
		\[
		\psi^{\cM}_E(\mathbf{x}, y) \geq \varphi^{\cM}_E(\mathbf{x}) - \epsilon \text{ in } N_{\sa}.
		\]
	\end{proposition}
	
	\begin{proof}
		Note that $\psi \in \mathcal{F}_{\chron}(s)$.  Thus, by the chronological invariance property, we have $\psi_E^{\cM}(\mathbf{x},y) \in M_s \cap N$ for all $y \in D_{s,r}^{\cM}$.  Using classical facts about commutative von Neumann algebras or about Stone spaces (see \cite[Lemma 3.15]{GaoJekel2025}), there exists a partition of unity $(e_j)_{j \in J}$ in $M_s \cap N$ and $y_j \in D_{s,r}^{\cM}$ such that
		\[
		\sum_{j \in J} e_j \psi_E^{\cM}(\mathbf{x},y_j) \geq \varphi_E^{\cM}(\mathbf{x}) - \varepsilon.
		\]
		Let $y = \sum_{j \in J} e_j y_j$.  Then by Lemma \ref{lem: formula partition},
		\[
		\psi_E^{\cM}(\mathbf{x},y) = \sum_{j \in J} e_j \psi_E^{\cM}(\mathbf{x},y_j) \geq \varphi_E^{\cM}(\mathbf{x}) - \varepsilon
		\]
		as desired.
	\end{proof}
	
	We now recall the ``ultrafiber'' construction from \cite{GaoJekel2025} (taken essentially from \cite{BY2012}).  Given $N \subseteq Z(M)$, we take a quotient of $M$ by a maximal ideal corresponding to a character (or equivalently, a pure state) on $\cN$.  We can then relate the $N$-valued interpretations of formulas in $\cM$ to the plain evaluation of formulas on the quotient by \cite[Theorem 3.19]{BY2012} or \cite[Theorem 3.18]{GaoJekel2025}, which is an analog of {\L}o{\'s}'s theorem for ultraproducts.  Our present goal is to adapt this result to the setting of filtrations.  First, we recall the definition of the ultrafiber.
	
	\begin{definition}[{\cite[Definition 2.1]{GaoJekel2025}}]
		Let $M$ be a finite von Neumann algebra, $N \subseteq Z(M)$ be a subalgebra of its center, and $E: M \to N$ be a normal, faithful, tracial conditional expectation.  Let $\cV$ be a character on $\cN$, and let
		\[
		I_{E,\cV} = \{ x \in \cM: \cV \circ E[x^*x] = 0 \}.
		\]
		Define $M^{/E,\cV} = M / I_{E,\cV}$ and let $\pi$ be the quotient map.  Let $\tau_{E,\cV} = \cV \circ E$ viewed as a state on $M^{/E,\cV}$.
	\end{definition}
	
	It is shown in \cite[Proposition 2.2]{GaoJekel2025} that $M^{/E,\cV}$ is a von Neumann algebra and $\tau_{E,\cV}$ is a faithful normal tracial state.  We then have the following result that adapts \cite[Theorem 3.18]{GaoJekel2025} to the setting of filtrations.
	
	\begin{proposition} \label{prop: Los theorem}
		Let $(M_t)_{t \in [0,\infty]}$ be a filtration of von Neumann algebras and let $\mathbf{z} = (z_{s,t,j})_{0\leq s \leq t < \infty, j\in \N}$ be elements $z_{s,t,j} \in D_{0,6 \sqrt{t-s}}^{\cM}$.  Let $N \subseteq Z(M_\infty)$ and let $E: M_\infty \to N$ be a normal, faithful, tracial conditional expectation.  Assume the chronological invariance property.
		
		Let $\pi: M_\infty \to Q_\infty := M_\infty^{/E,\cV}$ be the quotient map, and let $Q_t = \pi(M_t)$ for $t \in [0,\infty)$.  Then $(Q_t)_{t \in [0,\infty]}$ is a filtration of von Neumann algebras.  Viewing $\cQ = ((Q_t)_{t \in [0,\infty]}, (\pi(z_{s,t,j}))_{0\leq s\leq t,j\in \N}$ is an $\cL_{\proc}$-structure, for every $m$-variable formula $\varphi \in \cL_{\chron}(t)$ and $x_1$, \dots, $x_m \in M_t$, we have
		\begin{equation} \label{eq: Los theorem identity}
			\varphi^{\cQ}(\pi(x_1),\dots,\pi(x_m)) = \cV \circ \varphi_E^{\cM}(x_1,\dots,x_m).
		\end{equation}
	\end{proposition}
	
	\begin{proof}
		Note that the chronological invariance property applied to the basic formulas $\re \tr$ and $\im \tr$ asserts that $E$ maps $M_t$ into $M_t \cap N$.  Let $N_t = N \cap M_t$ and $E_t = E|_{M_t}: M_t \to N_t$ and $\cV_t = \cV|_{N_t}$.  Note that $I_{E_t,\cV_t} = I_{E,\cV} \cap M_t$.  From this we easily obtain a natural isomorphism $Q_t \cong M_t^{/E_t,\cV_t}$, which implies that $Q_t$ is a tracial von Neumann algebra.  Hence, $(Q_t)_{t \in [0,\infty]}$ is a filtration of von Neumann algebras as desired.
		
		The identity \eqref{eq: Los theorem identity} is proved by induction on Definition \ref{def: chronological formulas}.  The proof is the same as \cite[Theorem 3.19]{GaoJekel2025} using Proposition \ref{prop: maximizer} in place of \cite[Proposition 3.13]{GaoJekel2025}.
	\end{proof}
	
	\begin{corollary} \label{cor: center-valued definable predicates}
		Let $\cM$, $M_t$, $E$, and $N$ be as in Proposition \ref{prop: Los theorem} satisfying the chronological invariance property.  Let $\rT$ be an $\cL_{\chron}$-theory such that for all characters $\cV$ on $\cN$, $\cM^{/E,\cV}$ satisfies $\rT$.  Let $m \in \N \cup \{0,\infty\}$.  Then for any chronological formula $\varphi$ in $m$ variables and $\mathbf{r} \in (0,\infty)^m$ and $\mathbf{x} \in \prod_{j \leq m} D_{0,r_j}^{\cM}$, we have
		\begin{equation} \label{eq: norm of center-valued formula}
			\norm{\varphi_E^{\cM}(\mathbf{x})}_{N} \leq \norm{\varphi}_{0,\mathbf{r},\rT}.
		\end{equation}
		In particular, the center-valued evaluation $\varphi \mapsto \varphi_E^{\cM}$ extends to the space of chronologically definable predicates with respect to $\rT$.
	\end{corollary}
	
	\begin{proof}
		We note that for every $\cV$, letting $\pi_{E,\cV}$ be the associated quotient map,
		\[
		|\cV[\varphi_E^{\cM}(\mathbf{x})]| = |\varphi^{\cM^{/E,\cV}}(\pi_{E,\cV}(\mathbf{x}))| \leq \norm{\varphi}_{0,\mathbf{r},\rT}.
		\]
		Since $\norm{\varphi_E^{\cM}(\mathbf{x})}_{\cN}$ is the supremum of $|\cV[\varphi_E^{\cM}(\mathbf{x})]|$ over characters of $\cV$ on $\cN$, we conclude \eqref{eq: norm of center-valued formula}.  The extension to definable predicates follows from this norm estimate.
	\end{proof}
	
	\subsection{Formulas in the random matrix setting} \label{subsec: concentration}
	
	We define a subset of the chronological formulas which will be convenient to use in the convergence arguments.
	
	\begin{definition} \label{def: restricted chronological formulas}
		\emph{Restricted chronological formulas} $\mathcal{F}_{\operatorname{chron}}^0$ in $\mathcal{L}_{\operatorname{proc}}$ in variables $\mathbf{x} = (x_j)_{j \in \N}$ are defined as follows:
		\begin{enumerate}
			\item A \emph{basic restricted chronological formula} is an expression of the form
			\[
			\varphi(\mathbf{x}) = \re \tr[p((\re(x_j) + i)^{-1},(\im(x_j) + i)^{-1}: j \in \N )],
			\]
			where $p$ is a $*$-polynomial.
			\item If $\varphi_1$, \dots, $\varphi_k$ are in $\mathcal{F}_{\operatorname{chron}}^0$ and $f: \R^k \to \R$ is Lipschitz, then $f(\varphi_1,\dots,\varphi_k)$ is in $\mathcal{F}_{\operatorname{chron}}^0$.
			\item If $\varphi(\mathbf{x},y)$ is a formula in $\mathcal{F}_{\operatorname{chron}}^0$ evaluated on $(\mathbf{x},y)$ rather than $\mathbf{x}$, and $r > 0$, then
			\[
			\psi_1(\mathbf{x}) = \sup_{y \in D_{0,r}} \varphi(\mathbf{x},y), \qquad \psi_2(\mathbf{x}) = \inf_{y \in D_{0,r}} \varphi(\mathbf{x},y)
			\]
			are in $\mathcal{F}_{\operatorname{chron}}^0$.
			\item If $\varphi(\mathbf{x},y)$ is a formula in $\mathcal{F}_{\operatorname{chron}}^0$ evaluated on $(\mathbf{x},y)$ rather than $\mathbf{x}$, and if $t_0 > 0$, then
			\[
			\psi(\mathbf{x}) = (S_{t_0}\varphi)(\mathbf{x}, z_{0,t_0})
			\quad \text{is in } \mathcal{F}_{\operatorname{chron}}^0.
			\]
		\end{enumerate}
	\end{definition}
	
	\begin{remark}
		Technically, $\mathcal{F}_{\chron}^0$ are formulas in $(\re(x_j) + i)^{-1}$ and $(\im(x_j) + i)^{-1}$ rather than in $x_j$.  They will be chronologically definable predicates in $x_j$ with respect to $\rT_{\filt}$ since $(\re(x_j) + i)^{-1}$ and $(\im(x_j)+i)^{-1}$ can be approximated by polynomials in $\re(x_j)$ on each operator norm ball.
	\end{remark}
	
	\begin{remark}
		In item (4), we allow the trivial case that $\varphi$ is independent of $y$, or $\varphi(\mathbf{x},y) = \varphi_0(\mathbf{x})$.  Hence, in particular, if $\varphi \in \mathcal{F}_{\operatorname{chron}}^0$ and $t_0 > 0$, then $S_{t_0} \varphi \in \mathcal{F}_{\operatorname{chron}}^0$.
	\end{remark}
	
	\begin{lemma} \label{lem: approximation by restricted chronological formulas}
		Let $\varphi \in \cF_{\chron,m}$ and $R > 0$ and $\varepsilon > 0$.  Then there exists $\psi \in \cF_{\chron,m}^0$ such that whenever $\cM$ is a tuple consisting of an NC filtration $(M_t)_{t \in [0,\infty]}$ and conditional expectation $E: \cM \to \cN \subseteq Z(\cM)$ and a process $z_{s,t,j} \in D_{\sqrt{t-s},t}^{\cM}$, we have
		\[
		\sup_{x_1, \dots, x_m \in D_{0,r}^{\cM}} \norm{\varphi_E^{\cM}(\mathbf{x}) - \psi_E^{\cM}(\mathbf{x})} \leq \varepsilon.
		\]
	\end{lemma}
	
	\begin{proof}
		The proof proceeds by induction on the construction of formulas in Definition \ref{def: restricted chronological formulas} and Lemma \ref{lem: chronological formulas alternate}.  For the base case, one uses the fact that in a von Neumann algebra $x_j$ can be approximated on any operator norm ball by $*$-polynomials in $(\re(x_j) + i)^{-1}$ and $(\im(x_j) + i)^{-1}$ with the error uniformly small over all von Neumann algebras $\cM$.  The inductive steps are straightforward and left to the reader.
	\end{proof}

	The next proposition describes the center-valued evaluation (in the sense of \cite{GaoJekel2025}) for the chronological formulas on $L^\infty(\Omega,\mathcal{F},\mathbb{P}) \otimes \mathbb{M}_n$.
	
	\begin{proposition} \label{prop: approx by pointwise function}
		Let $\cM$ and $\cM_t$ be the random matrix ultraproducts in \eqref{eq: random matrix ultraproduct} and \eqref{eq: random matrix ultraproduct 2} and let $E$ be expectation onto $\cN = Z(\cM)$.  Let $Z_{s,t,j}^{(n)}$ be the increments of independent Hermitian Brownian motions in $\mathbb{M}_n$, let $F_R$ be as in \eqref{eq: function FR}, and let
		\begin{equation} \label{eq: Brownian motion in random matrix ultraproduct}
			z_{s,t,j} = [F_{3 \sqrt{t-s}}(Z_{s,t,j}^{(n)})]_{n \in \N}.
		\end{equation}
		For every $\varphi \in \mathcal{F}_{\operatorname{chron}}^0$ with free variables among $x_1$, \dots, $x_m$, and for every $n \in \N$, there exists a function $\Lambda_{\varphi}^{(n)}: \mathbb{M}_n^m \to \R$ which is bounded and $\norm{\cdot}_1$-Lipschitz and satisfies
		\begin{equation} \label{eq: pointwise formula approximation}
			(S_{t_0}\varphi)^{\cM}_E(x_1,\dots,x_m) = [\Lambda_{\varphi}^{(n)}(X_1^{(n)},\dots,X_m^{(n)})]_{n \in \N},
		\end{equation}
		whenever $t_0 \geq 0$ and
		\[
		x_j = [X_j^{(n)}]_{n \in \N} \in \cM_{t_0}.
		\]
		Moreover, $\Lambda_{\varphi}^{(n)}$ are chosen independently of the ultrafilter $\cU$ in the random matrix ultraproduct as well as the original choice of classical filtration $(\mathcal{G}_t)_{t \geq 0}$.
	\end{proposition}
	
	\begin{proof}
		We proceed by induction on the complexity of the formula.  For simplicity, we focus first on the case where $t_0 = 0$.
		
		\textbf{Base case:} Suppose that $\varphi$ is a basic restricted chronological formula.  For $X_1, \dots, X_m \in \mathbb{M}_n$, let $\Lambda_\varphi^{(n)}(X_1,\dots,X_m) = \varphi^{\mathbb{M}_n}(X_1,\dots,X_m)$, where $\varphi^{\mathbb{M}_n}$ is the evaluation of $\varphi$ viewed as an $\mathcal{L}_{\tr}$-formula.  It follows from the resolvent identity that $(\re(x_j) + i)^{-1}$ is bounded in operator norm and Lipschitz in $1$-norm, uniformly for all tracial von Neumann algebras.  Hence, using the noncommutative H{\"o}lder inequality, the same holds for polynomials in these resolvents, and therefore also for their traces.  Thus, we obtain the desired uniform boundedness and Lipschitzness properties.  Moreover, if $x_j = [X_j^{(n)}]_{j \in \N}$ is in $\cM_0$, then
		\[
		\varphi_E^{\cM}(x_1,\dots,x_n) = E[[\Lambda_\varphi^{(n)}(X_1^{(n)},\dots,X_m^{(n)})]_{n \in \N}] = [\mathbb{E} \varphi^{\mathbb{M}_n}(X_1^{(n)},\dots,X_m^{(n)})]_{n \in \N},
		\]
		so \eqref{eq: pointwise formula approximation} holds.
		
		\textbf{Inductive case 1:}  Suppose $\varphi = f(\varphi_1,\dots,\varphi_k)$ where $f: \R^k \to \R$ is Lipschitz and $\varphi_1$, \dots, $\varphi_k$ are formulas in $x_1$, \dots, $x_m$ satisfying the inductive hypothesis with respect to some $\Lambda_{\varphi_j}^{(n)}$.  Let
		\[
		\Lambda_\varphi^{(n)} = f(\Lambda_{\varphi_1}^{(n)},\dots,\Lambda_{\varphi_k}^{(n)}),
		\]
		which is clearly bounded and $\norm{\cdot}_1$-Lipschitz as a composition of bounded $\norm{\cdot}_1$-Lipschitz functions (with constants uniform in $n$).  To show \eqref{eq: pointwise formula approximation}, observe that if $x_j = [X_j^{(n)}]_{n \in \N} \in \cM_0$, then 
		\begin{align*}
			\varphi_E^{\cM}(x_1,\dots,x_m) &= f((\varphi_1)_E^{\cM}(x_1,\dots,x_m),\dots,(\varphi_k)_E^{\cM}(x_1,\dots,x_m)) \\
			&= f([\Lambda_{\varphi_1}^{(n)}(X_1^{(n)},\dots,X_m^{(n)})]_{n \in \N},\dots, [\Lambda_{\varphi_k}^{(n)}(X_1^{(n)},\dots,X_m^{(n)})]_{n \in \N}) \\
			&= [f(\Lambda_{\varphi_1}^{(n)},\dots,\Lambda_{\varphi_k}^{(n)})(X_1^{(n)},\dots,X_m^{(n)})]_{n \in \N} \\
			&= [\Lambda_{\varphi}^{(n)}(X_1^{(n)},\dots,X_m^{(n)})]_{n \in \N},
		\end{align*}
		where we use that multivariable continuous functional calculus respects $*$-homomorphisms of commutative $\mathrm{C}^*$-algebras.
		
		\textbf{Inductive case 2:} Suppose that $\varphi$ is a formula in $(x_1,\dots,x_m,y)$ satisfying the inductive hypotheses, and let $\psi(\mathbf{x}) = \sup_{y \in D_{0,r}} \varphi(\mathbf{x},y)$.  Let
		\[
		\Lambda_\psi^{(n)}(x_1,\dots,x_m) = \sup_{y \in D_r^{\mathbb{M}_n}} \Lambda_\varphi^{(n)}(x_1,\dots,x_m,y).
		\]
		This is clearly a bounded $\norm{\cdot}_1$-Lipschitz function since it is a supremum of bounded $\norm{\cdot}_1$-Lipschitz functions in $x_1$, \dots, $x_m$ (with constants that are uniform in $n$).  Now suppose that $x_j = [X_j^{(n)}]_{n \in \N} \in \cM_0$.  Suppose $y = [Y^{(n)}]_{n \in \N}$ is in $\cM_0$ with $\norm{Y^{(n)}} \leq r$.  Then
		\[
		\Lambda_{\varphi}^{(n)}(X_1^{(n)},\dots,X_m^{(n)},Y^{(n)}) \leq \Lambda_{\psi}^{(n)}(X_1^{(n)},\dots,X_m^{(n)}),
		\]
		and hence
		\[
		\varphi_E^{\cM}(x_1,\dots,x_m,y) = [\Lambda_{\varphi}^{(n)}(X_1^{(n)},\dots,X_m^{(n)},Y)]_{n \in \N} \leq [\Lambda_{\psi}^{(n)}(X_1^{(n)},\dots,X_m^{(n)})]_{n \in \N}.
		\]
		Since $y$ was arbitrary,
		\[
		\psi_E^{\cM}(x_1,\dots,x_m) \leq [\Lambda_{\psi}^{(n)}(X_1^{(n)},\dots,X_m^{(n)})]_{n \in \N}.
		\]
		On the other hand, it is an elementary exercise in measure theory to show that there exists $Y^{(n)}$ in the $r$-ball of $L^\infty(\Omega,\mathcal{F}_0,\mathbb{P}) \otimes \mathbb{M}_n$ such that
		\[
		\Lambda_\varphi^{(n)}(X_1^{(n)},\dots,X_m^{(n)},Y^{(n)}) \geq \Lambda_\psi^{(n)}(X_1^{(n)},\dots,X_m^{(n)}) - \frac{1}{n}
		\]
		almost everywhere.  Therefore,
		\[
		\varphi_E^{\cM}(x_1,\dots,x_m,y) = [\Lambda_\varphi^{(n)}(X_1^{(n)},\dots,X_m^{(n)},Y^{(n)})]_{n \in \N} \geq [\Lambda_\psi^{(n)}(X_1^{(n)},\dots,X_m^{(n)})]_{n \in \N}.
		\]
		Hence,
		\[
		\psi_E^{\cM}(x_1,\dots,x_m) \geq \varphi_E^{\cM}(x_1,\dots,x_m,y) \geq [\Lambda_\psi^{(n)}(X_1^{(n)},\dots,X_m^{(n)})]_{n \in \N}.
		\]
		Therefore, \eqref{eq: pointwise formula approximation} holds since we proved inequality in both directions.
		
		\textbf{Inductive case 3:}  Fix a formula $\varphi(x_1,\dots,x_m,y)$, and let $\psi(x_1,\dots,x_m) = (S_{t_0} \varphi)(x_1,\dots,x_m,\mathbf{z}_{0,t_0})$.  Define
		\begin{equation} \label{eq: defining Lambda in semicircular case}
			\Lambda_{\psi}^{(n)}(x_1,\dots,x_m) = \mathbb{E} \Lambda_{\varphi}^{(n)}(x_1,\dots,x_m, F_{3\sqrt{t_0}}(\mathbf{Z}_{0,t_0}^{(n)})).
		\end{equation}
		Here only finitely many coordinates of the Brownian motion appear in the formula, even though we denote it as a function of the entire tuple.  Note that this is an average of uniformly bounded and $\norm{\cdot}_1$-Lipschitz functions in $(x_1,\dots,x_m)$ and hence is $\norm{\cdot}_1$-Lipschitz and bounded (with constants independent of $n$).  In particular, note that it is also Lipschitz with respect to $\norm{\cdot}_2$.  Now since $\Lambda_{\varphi}^{(n)}$ and $F_{3\sqrt{t_0}}$ are Lipschitz with respect to $\norm{\cdot}_2$, the Poincar{\'e} inequality \eqref{eq: Poincare inequality} implies that for $X_1$, \dots, $X_m \in \mathbb{M}_n$,
		\begin{equation} \label{eq: estimate from Poincare}
			\norm{\Lambda_{\varphi}^{(n)}(X_1,\dots,X_m,F_{3\sqrt{t_0}}(Z_{0,t_0}^{(n)})) - \Lambda_{\psi}^{(n)}(X_1,\dots,X_m)}_{L^2} \leq \frac{C t_0^{1/2}}{n},
		\end{equation}
		where $C$ is a Lipschitz constant of the function we take the expectation of.  Suppose that $x_j = [X_j^{(n)}]_{n \in \N} \in \cM_0$.  Using \eqref{eq: estimate from Poincare} together with conditioning, we conclude that
		\begin{equation} \label{eq: estimate from Poincare 2}
			\norm{\Lambda_{\varphi}^{(n)}(X_1^{(n)},\dots,X_m^{(n)},F_{3\sqrt{t_0}}(Z_{0,t_0}^{(n)})) - \Lambda_{\psi}^{(n)}(X_1,\dots,X_m)}_{L^2} \leq \frac{Ct_0^{1/2}}{n}.
		\end{equation}
		Meanwhile, we note that the statements of the proposition can be applied to $(\mathcal{G}_{t_0+t})_{t \geq 0}$ just as well as $(\mathcal{G}_t)_{t \geq 0}$.  Hence, using the inductive hypothesis on $\mathcal{G}_{t_0}$ and $\varphi$ and $(X_1^{(n)},\dots,X_m^{(n)},F_{3\sqrt{t_0}}(\mathbf{Z}_{0,t_0}^{(n)}))$, we get
		\[
		(S_{t_0}\varphi)_E^{\cM}(x_1,\dots,x_m,z_{0,t_0}) = [\Lambda_{\varphi}^{(n)}(X_1^{(n)},\dots,X_m^{(n)},F_{3\sqrt{t_0}}(\mathbf{Z}_{0,t_0}^{(n)})]_{n \in \N}.
		\]
		Then by applying \eqref{eq: estimate from Poincare},
		\begin{align*}
			\psi_E^{\cM}(x_1,\dots,x_m) &= (S_{t_0}\varphi)_E^{\cM}(x_1,\dots,x_m,z_{0,t_0}) \\
			&= [\Lambda_{\varphi}^{(n)}(X_1^{(n)},\dots,X_m^{(n)},F_R(Z_{0,t_0}^{(n)})]_{n \in \N} \\
			&= [\Lambda_\psi^{(n)}(X_1^{(n)},\dots,X_m^{(n)})]_{n \in \N},
		\end{align*}
		which proves \eqref{eq: pointwise formula approximation} in the case $t_0 = 0$.
		
		The case for general $t_0 > 0$ follows for the same reason as mentioned in Case 3, namely that the statement can be applied to the filtration $(\mathcal{G}_{t_0+t})_{t \geq 0}$.
	\end{proof}

	\begin{remark} \label{rem: independent of choices}
		The sequence of recursive steps taken to construct a formula in $\mathcal{F}_{\operatorname{chron}}^0$ are not quite unique since the application of connectives commutes with the operation of shifting and substituting $z_{0,t_0}$.  But for the finite-dimensional approximations, performing these operations in a different order results in a slightly different choice of $\Lambda_\varphi^{(n)}$ since in the one version we take the expectation over $F_{3 \sqrt{t_0}}(\mathbf{Z}_{0,t_0}^{(n)})$ before applying the continuous connective, and in the other we take the expectation afterward.  However, it follows from the Poincar{\'e} inequality \eqref{eq: Poincare inequality} that these two choices of $\Lambda_\varphi^{(n)}$ are within a distance of $O(1/n)$ uniformly, and so they are asymptotically the same.  In the sequel, we fix a choice of $\Lambda_{\varphi}^{(n)}$ for each $\varphi$ in $\mathcal{F}_{\chron}^0$.
	\end{remark}
	
	\begin{remark} \label{rem: bounded and Lipschitz}
		Similar reasoning as in the proof shows that every $\varphi \in \mathcal{F}_{\chron,m}^0$ is globally bounded and is $\norm{\cdot}_1$-Lipschitz, with a uniform bound for all $\cL_{\proc}$-structures satisfying $\rT_{\filt}$ (that is, all filtrations $(\cN_t)_{t \in [0,\infty]}$ of tracial von Neumann algebras with associated constants $z_{s,t,j}$).
	\end{remark}
	
	\begin{corollary}
		Let $\cM_t$ and $z_{s,t,j}$ be as above.  Then $((\cM_t)_{t \in [0,\infty]},(z_{s,t,j})_{0\leq s\leq t,j\in \N})$ has the chronological invariance property.
	\end{corollary}
	
	\begin{proof}
		Let $t > 0$ and let $x_j = [X_j^{(n)}]_{n \in \N} \in D_{t_0,\mathbf{r}}^{\cM}$ for some $\mathbf{r} \in (0,\infty)^m$.  Let $\varphi \in \mathcal{L}_{\chron}(t)$.  As in the proof of Lemma \ref{lem: chronological formulas alternate}, we can write
		\[
		\varphi(\mathbf{x}) = (S_t\psi)(\mathbf{x},z_{s_1,t_1,j_1},\dots,z_{s_k,t_k,j_k})
		\]
		for some $\psi \in \cF_{\chron,m_k}(0)$ and $s_i \leq t_i \leq t$.  By Lemma \ref{lem: approximation by restricted chronological formulas}, we can approximate $\psi$ by some $\psi' \in \cF_{\chron,m+k}^0$ within any given error $\varepsilon$ on any of the domains $D_{0,\mathbf{r}}$.  By the previous proposition,
		\[
		(S_{t_0}\varphi)^{\cM}_E(x_1,\dots,x_m,z_{s_1,t_1,j_1},\dots,z_{s_k,t_k,j_k}) = [\Lambda_{\varphi}^{(n)}(X_1^{(n)},\dots,X_m^{(n)},Z_{s_1,t_1,j_1}^{(n)},\dots,Z_{s_k,t_k,j_k}^{(n)})]_{n \in \N},
		\]
		which is manifestly in
		\[
		\prod_{n \to \cU} L^\infty(\Omega,\mathcal{G}_t,\mathbb{P}) \subseteq Z(M_\infty) \cap M_t.
		\]
		Since $\varepsilon$ was arbitrary, $\varphi_E^{\cM}(x_1,\dots,x_m) \in Z(\cM) \cap \cM_t$ as desired.
	\end{proof}
	
	We also immediately obtain the following form of stationarity for the random matrix ultraproduct.
	
	\begin{corollary} \label{cor: random matrix ultraproduct stationarity}
		With the same notation as above, we have that
		\[
		x_1, \dots, x_n \in \mathcal{M}_0 \implies (S_{t_0}\varphi)_{E}^{\cM}(x_1,\dots,x_n) = \varphi_E^{\cM}(x_1,\dots,x_n).
		\]
	\end{corollary}
	
	\begin{proof}
		This follows from two applications of \eqref{eq: pointwise formula approximation}, viewing $x_j \in \cM_{t_0}$ as well as $x_j \in \cM_0$.
	\end{proof}
	
	\subsection{Ultrafiber selection} \label{subsec: matrix quotient}
	
	With all the pieces in place, we now conclude the application of the ideas of \cite[\S 6]{GaoJekel2025} by studying the quotient of the random matrix ultraproduct by a maximal ideal.  Again, let $\cM$ and $M_t$ be as in \eqref{eq: random matrix ultraproduct} and \eqref{eq: random matrix ultraproduct 2} with $z_{s,t,j}$ as in \eqref{eq: Brownian motion in random matrix ultraproduct}.  Let $\cV$ be a character on $Z(M_\infty)$, and let
	\begin{equation} \label{eq: random matrix quotient}
		\pi: M_\infty \to Q_\infty := M_\infty^{/E,\cV} 
	\end{equation}
	be the quotient, and let
	\begin{equation} \label{eq: random matrix quotient 2}
		Q_t := \pi(M_t).
	\end{equation}
	Note that by \cite[Proposition 6.1]{GaoJekel2025}, $Q_t$ is elementarily equivalent to $\prod_{n \to \cU} \mathbb{M}_n$ as tracial von Neumann algebras.  The following result is also a natural generalization of several of the statements in \cite[Proposition 6.1]{GaoJekel2025} to the setting of filtrations.
	
	\begin{proposition} \label{prop: matrix quotient filtration}
		Let $\cQ = ((Q_t)_{t \in [0,\infty]}, (\pi(z_{s,t,j}))_{0 \leq s \leq t, j \in \N}$ where $Q_t$ and $\pi$ are as above.  Then $\cQ$ is a stationary $\cL_{\proc}$-structure such that $\pi(z_{s,t,j})$ are the increments of a circular Brownian motion (as defined in \S \ref{subsec: free probability}).  Moreover, for $x_j = [X_j^{(n)}]_{n \in \N} \in \cM_0$ and $\varphi \in \cF_{\chron,m}^0$,
		\begin{equation} \label{eq: evaluation of formula in quotient}
			\varphi^{\cQ}(\pi(x_1),\dots,\pi(x_m)) = \cV\left( [\Lambda_\varphi^{(n)}(X_1^{(n)},\dots,X_m^{(n)})]_{n \in \N} \right).
		\end{equation}
	\end{proposition}
	
	\begin{proof}
		Suppose that $x_j = [X_j^{(n)}]_{n \in \N} \in \cM_0$ and $\varphi \in \mathcal{F}_{\chron,m}^0$.  By Proposition \ref{prop: Los theorem} and \ref{prop: approx by pointwise function},
		\begin{align*}
			\varphi^{\cQ}(\pi(x_1),\dots,\pi(x_m)) &= \cV\left( \varphi_E^{\cM}(x_1,\dots,x_m) \right) \\
			&= \cV\left( [\Lambda_\varphi^{(n)}(X_1^{(n)},\dots,X_m^{(n)})]_{n \in \N} \right).
		\end{align*}
		Since the same applies to $S_t \varphi$, Corollary \ref{cor: random matrix ultraproduct stationarity} implies stationarity.
		
		Finally, let us show that $(t-s)^{-1/2}(\pi(z_{s,t,j}))_{j \in \N}$ is a free circular family freely independent of $\cQ_s$.  Let $\pi(x_1), \dots, \pi(x_m) \in \cQ_s$ with $x_j = [X_j^{(n)}]_{n \in \N}$ where $X_j^{(n)} \in L^\infty(\Omega,\mathcal{F}_s,\mathbb{P})$.  Consider a formula of the form
		\begin{multline*}
			\varphi(\mathbf{x}) = \re \tr\Bigl[ f_1(\mathbf{x})[p_1(z_{s,t,j}: j \in \N) - \tr(p_1(z_{s,t,j}: j \in \N)] [f_2(\mathbf{x}) - \tr(f_2(\mathbf{x}))] \dots \\
			\dots [p_{k-1}(z_{s,t,j}: j \in \N) - \tr(p_{k-1}(z_{s,t,j}: j \in \N)][f_k(\mathbf{x}) - \tr(f_k(\mathbf{x}))] p_k(z_{s,t,j}: j \in \N) \Bigr],
		\end{multline*}
		where $p_i$ is a noncommutative polynomial and
		\[
		f_i(\mathbf{x}) = q_i((\re x_j + i)^{-1}, (\im x_j + i)^{-1}: j = 1, \dots, m)
		\]
		for some noncommutative polynomials $q_j$.  By Voiculescu's asymptotic freeness theorem (see Theorem \ref{thm: asymptotic freeness}) and independence of $Z_{s,t,j}^{(n)}$ from $\mathcal{G}_s$, we have
		\[
		\Lambda_\varphi^{(n)}(X_1^{(n)},\dots,X_m^{(n)}) = \mathbb{E} [\varphi^{L^\infty \otimes \mathbb{M}_n}(X_1^{(n)},\dots,X_m^{(n)}) \mid \mathcal{G}_s ] \to 0
		\]
		almost surely as $n \to \infty$.  Therefore, by Proposition \ref{prop: approx by pointwise function}
		\[
		(S_s\varphi)_E^{\cM}(x_1,\dots,x_m) = 0
		\]
		and by Proposition \ref{prop: Los theorem},
		\[
		(S_s\varphi)^{\cQ}(\pi(x_1),\dots,\pi(x_m)) = 0. \qedhere
		\]
	\end{proof}
	
	\begin{proposition} \label{prop: convergence of type}
		Let $X_1^{(n)}$, \dots, $X_m^{(n)}$ be random matrices in $L^\infty(\Omega,\mathcal{F}_0,\mathbb{P}) \otimes \mathbb{M}_n$ with operator norm bounded by $R$, and let $x_j = [X_j^{(n)}]_{n \in \N} \in \cM_0$.  Let $\mu \in \mathbb{S}_{\chron,0}(\rT_{\filt})$.  Then the following are equivalent:
		\begin{enumerate}[(1)]
			\item For every $\varphi \in \mathcal{F}_{\chron,m}^0$, we have
			\[
			\lim_{n \to \cU} \Lambda_\varphi^{(n)}(X_1^{(n)},\dots,X_m^{(n)}) = (\mu,\varphi) \text{ in probability.}
			\]
			\item For every character $\cV$ on $Z(M)$, we have
			\[
			\tp_{\chron}^{\cQ}(\pi(x_1),\dots,\pi(x_m)) = \mu
			\]
			in the quotient $\cL_{\proc}$-structure $\cQ$ associated to $\cV$.
		\end{enumerate}
	\end{proposition}
	
	\begin{proof}
		Note that since $\Lambda_\varphi^{(n)}$ is a bounded sequence of random variables, convergence in probability to a constant is equivalent to convergence in $L^2$ to that constant.  Hence, (1) is equivalent to
		\[
		\forall \varphi \in \mathcal{F}_{\chron,m}^0, \quad [\Lambda_\varphi^{(n)}(X_1^{(n)},\dots,X_m^{(n)})]_{n \in \N} = (\mu,\varphi) 1 \text{ in } \prod_{n \to \cU} L^\infty(\Omega,\mathcal{G},\mathbb{P}) = Z(M).
		\]
		Then using \eqref{eq: pointwise formula approximation} with $t_0 = 0$, this is equivalent to
		\[
		\forall \varphi \in \mathcal{F}_{\chron,m}^0, \quad \varphi_E^{\cM}(\mathbf{x}) = (\mu,\varphi) 1 \in Z(M).
		\]
		Since $Z(M)$ is a commutative $\mathrm{C}^*$-algebra, two elements $a$ and $b$ are equal if and only if $\cV(a) = \cV(b)$ for all characters $\cV$.  Thus, using Proposition \ref{prop: Los theorem}, the above condition is equivalent to
		\[
		\forall \varphi \in \mathcal{F}_{\chron,m}^0, \forall \cV, \quad \varphi^{\cQ}(\pi(x_1),\dots,\pi(x_m)) = (\mu,\varphi).
		\]
		After writing $\forall \cV \forall \varphi$ instead of $\forall \varphi \forall \cV$, we see this is equivalent to condition (2).
	\end{proof}
	
	\section{Entropy for chronological types} \label{sec: entropy}
	
	\subsection{Definition of entropy} \label{subsec: entropy definition}
	
	Next, we define the version of free entropy using chronological formulas.  First, note that unlike \cite{JekelModelEntropy}, we cannot directly phrase the definition in terms of open neighborhoods in the type space.  Indeed, the chronological formulas do not have a relevant interpretation in $\mathbb{M}_n$ because our filtration is on $L^\infty(\Omega,\mathcal{F},\mathbb{P}) \otimes \mathbb{M}_n$ rather than $\mathbb{M}_n$ itself. We will in fact define the microstate spaces using the functions $\Lambda_{\varphi}^{(n)}$ from Proposition \ref{prop: approx by pointwise function} (and verify below that the definition is independent of the different choices of $\Lambda_{\varphi}^{(n)}$ as in Remark \ref{rem: independent of choices}).
	
	For convenience, we define the entropy using the Gaussian Ginibre probability measure on $\mathbb{M}_n^m$, which we denote by $\sigma^{(n)}$, rather than Lebesgue measure as the background measure similar to \cite{BCG2003}, and the sign is changed so that the resulting quantity is nonnegative.  We include a tilde to distinguish this from the version with Lebesgue measure.  We also add the subscript $\chron$ to indicate the use of chronological formulas.  We write the definition in terms of some $\cL_{\proc}$-structure $\cM$ satisfying $\rT_{\proc}$, but as we will see in Remark \ref{rem: m prime zero case} below, it is only nontrivial in the case that $\cM$ is chronologically elementarily equivalent to the random matrix ultraproduct quotient structure $\cQ$ discussed in \S \ref{subsec: matrix quotient}.  Hence, for most purposes, the reader can always assume that tuples $\mathbf{x}$ and $\mathbf{y}$ are from $\cQ$.
	
	\begin{definition}[Free entropy for chronological types] \label{def: chronological entropy}
		Fix a choice of functions $\Lambda_{\varphi}^{(n)}$ as in Proposition \ref{prop: approx by pointwise function}.  Consider an $\mathcal{L}_{\operatorname{proc}}$-structure $\cM$ satisfying $\rT_{\stat}$ and $\mathbf{x} \in M_0^m$.  Let $\Phi$ be a finite subset of $\mathcal{F}_{\chron,m}^0$ and $\varepsilon > 0$. Define
		\[
		\Gamma^{(n)}(\mathbf{x}; \Phi, \varepsilon) := \left\{ \mathbf{X} \in \mathbb{M}_n^m: \max_{\varphi \in \Phi} |\Lambda_{\varphi}^{(n)}(\mathbf{X}) - \varphi^{\cM}(\mathbf{x})| < \varepsilon \right\}.
		\]
		Let $\sigma^{(n)} \in \mathcal{P}(\mathbb{M}_n^m)$ be the probability distribution of a Ginibre $m$-tuple as defined in \S \ref{subsec: free probability}.  Define
		\[
		\tilde{\chi}_{\chron}^{\cU}(\mathbf{x}; \Phi, \varepsilon) := \lim_{n \to \cU} -\frac{1}{n^2} \log \sigma^{(n)}(\Gamma^{(n)}(\mathbf{x}; \Phi, \varepsilon))
		\]
		and
		\[
		\tilde{\chi}_{\chron}^{\cU}(\mathbf{x}) = \sup_{\Phi, \varepsilon} \tilde{\chi}^{\cU}(\mathbf{x}; \Phi, \varepsilon)
		\]
	\end{definition}
	
	\begin{definition}[Conditional free entropy for chronological types] \label{def: conditional chronological entropy}
		Fix a choice of functions $\Lambda_{\varphi}^{(n)}$ as in Proposition \ref{prop: approx by pointwise function}.  Let $m \in \N$ and $m' \in \N \cup \{0,\infty\}$.  Consider an $\mathcal{L}_{\operatorname{proc}}$-structure $\cM$ satisfying $\rT_{\stat}$ and $\mathbf{x} \in M_0^m$ and $\mathbf{y} \in M_0^{m'}$.  Let $\Phi$ be a finite subset of $\mathcal{F}_{\chron,m+m'}^0$ and $\varepsilon > 0$.  Let $\mathbf{Y}^{(n)} \in (\mathbb{M}_n)_{\sa}^{m'}$.  Define
		\[
		\Gamma^{(n)}(\mathbf{x} \mid \mathbf{Y}^{(n)} \rightsquigarrow \mathbf{y}; \Phi, \varepsilon) := \left\{ \mathbf{X} \in \mathbb{M}_n^m: \max_{\varphi \in \Phi} |\Lambda_{\varphi}^{(n)}(\mathbf{X},\mathbf{Y}^{(n)}) - \varphi^{\cM}(\mathbf{x},\mathbf{y})| < \varepsilon \right\}.
		\]
		Here $\mathbf{y}$ is a possibly infinite tuple, but each formula $\varphi$ depends on only finitely many of its coordinates.  Define
		\[
		\tilde{\chi}_{\chron}^{\cU}(\mathbf{x} \mid \mathbf{Y}^{(n)} \rightsquigarrow \mathbf{y}; \Phi, \varepsilon) := \lim_{n \to \cU} -\frac{1}{n^2} \log \sigma_m^{(n)}(\Gamma^{(n)}(\mathbf{x} \mid \mathbf{Y}^{(n)} \rightsquigarrow \mathbf{y}; \Phi, \varepsilon))
		\]
		and
		\[
		\tilde{\chi}_{\chron}^{\cU}(\mathbf{x} \mid \mathbf{Y}^{(n)} \rightsquigarrow \mathbf{y}) = \sup_{\Phi, \varepsilon} \tilde{\chi}^{\cU}(\mathbf{x} \mid \mathbf{Y}^{(n)} \rightsquigarrow \mathbf{y}; \Phi, \varepsilon)
		\]
	\end{definition}
	
	\begin{remark}
		Definition \ref{def: chronological entropy} can be regarded as a special case of Definition \ref{def: conditional chronological entropy} with $m' = 0$.  Hence, we will prove most of our results in the more general setting of Definition \ref{def: conditional chronological entropy}, and freely use them in the setting of Definition \ref{def: chronological entropy}.
	\end{remark}
	
	\begin{remark}
		As noted in Remark \ref{rem: independent of choices}, there are multiple choices of $\Lambda_{\varphi}^{(n)}$, but any two choices are within a distance of $O(1/n)$ asymptotically.  Suppose $\Lambda_{\varphi}^{(n)}$ and $\widehat{\Lambda}_{\varphi}^{(n)}$ are two possible choices, and denote the microstate space associated to $\tilde{\Lambda}_{\varphi}^{(n)}$ with $\widehat{\Gamma}$ instead of $\Gamma$.  If $\varepsilon' < \varepsilon$, then for sufficiently large $n$, the $O(1/n)$ error in $|\Lambda_{\varphi}^{(n)}$ and $\tilde{\Lambda}_{\varphi}^{(n)}|$ is smaller than $\varepsilon - \varepsilon'$.  Therefore,
		\begin{align*}
			\Gamma^{(n)}(\mathbf{x} \mid \mathbf{Y}^{(n)} \rightsquigarrow \mathbf{y}; \Phi, \varepsilon') &\subseteq \widehat{\Gamma}^{(n)}(\mathbf{x} \mid \mathbf{Y}^{(n)} \rightsquigarrow \mathbf{y}; \Phi, \varepsilon) \\
			\widehat{\Gamma}^{(n)}(\mathbf{x} \mid \mathbf{Y}^{(n)} \rightsquigarrow \mathbf{y}; \Phi, \varepsilon') &\subseteq \Gamma^{(n)}(\mathbf{x} \mid \mathbf{Y}^{(n)} \rightsquigarrow \mathbf{y}; \Phi, \varepsilon).
		\end{align*}
		This results in a corresponding inequality between the Gaussian measures of the two microstate spaces.  Since in Definition \ref{def: conditional chronological entropy}, we take the supremum over $\Phi$ and $\varepsilon$ at the end, we see that $\Lambda_{\varphi}^{(n)}$ and $\widehat{\Lambda}_{\varphi}^{(n)}$ both lead to the same quantity $\tilde{\chi}_{\chron}^{\cU}(\mathbf{x} \mid \mathbf{Y}^{(n)} \rightsquigarrow \mathbf{y})$.
	\end{remark}
	
	We will soon show that $\tilde{\chi}_{\chron}^{\cU}(\mathbf{x} \mid \mathbf{Y}^{(n)} \rightsquigarrow \mathbf{y})$ is actually independent of the particular choice of $\mathbf{Y}^{(n)}$ and only depends on the chronological type of $(\mathbf{x},\mathbf{y})$, so long as $\Lambda_{\varphi}^{(n)}(\mathbf{Y}^{(n)}) \to \varphi(\mathbf{y})$ for $\varphi \in \mathcal{F}_{\chron,m'}^0$ (see Theorem \ref{thm: final formula for entropy}).
	
	\subsection{Stochastic variational formula for entropy} \label{subsec: stochastic variational}
	
	In order to evaluate the entropy in terms of chronological formulas, we first express the entropy in an equivalent form using an analog of the pressure functional on $\mathcal{F}_{\chron,m+m'}^0$.  The manipulations in the proof of the lemma are fairly standard in large deviations theory (compare Varadhan's lemma and Bryc's theorem), and analogous to Hiai's expression of free microstate entropy in terms of modified free pressure in \cite[\S 6]{Hiai2005}.
	
	\begin{lemma} \label{lem: microstates versus pressure}
		Let $m \in \N$ and $m' \in \N \cup \{0,\infty\}$.  Fix tuples $\mathbf{x}$ and $\mathbf{y}$ from $M_0$.  For each $\varphi \in \mathcal{F}_{\chron,m+m'}^0$, let
		\begin{equation} \label{eq: pressure functional}
			\mathcal{P}_{m,m'}^{(n)}(\varphi)(\mathbf{Y}^{(n)}) = -\frac{1}{n^2} \log \int_{\mathbb{M}_n^m} \exp( -n^2 \Lambda_\varphi^{(n)}(\mathbf{z},\mathbf{Y}^{(n)}))\,d\sigma^{(n)}(\mathbf{z}),
		\end{equation}
		where $\sigma^{(n)}$ is the Ginibre measure.  Then
		\begin{equation} \label{eq: entropy as Laplace transform}
			\tilde{\chi}_{\chron}^{\cU}(\mathbf{x} \mid \mathbf{Y}^{(n)} \rightsquigarrow \mathbf{y}) = \sup_{\varphi \in \mathcal{F}_{\chron,m+m'}^0} \left[ \lim_{n \to \cU} \mathcal{P}_{m,m'}^{(n)}(\varphi)(\mathbf{Y}^{(n)}) - \varphi^{\cM}(\mathbf{x},\mathbf{y}) \right].
		\end{equation}
	\end{lemma}
	
	\begin{proof}
		($\geq$)  Let $\varphi \in \mathcal{F}_{\chron,m+m'}^0$.  We can examine the microstate space $\Gamma^{(n)}(\mathbf{x} \mid \mathbf{Y}^{(n)} \rightsquigarrow \mathbf{y}; \varphi, \varepsilon)$ where we take $\Phi = \{\varphi\}$.  Then
		\begin{align*}
			\int_{\mathbb{M}_n^m} \exp( -n^2 \Lambda_\varphi^{(n)}(\mathbf{z},\mathbf{Y}^{(n)}))\,d\sigma^{(n)}(\mathbf{z}) &\geq \int_{\Gamma^{(n)}(\mathbf{x} \mid \mathbf{Y}^{(n)} \rightsquigarrow \mathbf{y}; \varphi, \varepsilon)} \exp( -n^2 \Lambda_\varphi^{(n)}(\mathbf{z},\mathbf{Y}^{(n)}))\,d\sigma^{(n)}(\mathbf{z}) \\
			&\geq \int_{\Gamma^{(n)}(\mathbf{x} \mid \mathbf{Y}^{(n)} \rightsquigarrow \mathbf{y}; \varphi, \varepsilon)} \exp( -n^2 (\varphi^{\cM}(\mathbf{x},\mathbf{y}) + \varepsilon))\,d\sigma^{(n)}(\mathbf{z}) \\
			&\geq \sigma^{(n)}(\Gamma^{(n)}(\mathbf{x} \mid \mathbf{Y}^{(n)} \rightsquigarrow \mathbf{y}; \varphi, \varepsilon)) \exp(-n^2 (\varphi^{\cM}(\mathbf{x},\mathbf{y}) + \varepsilon)),
		\end{align*}
		which results in
		\begin{equation} \label{eq: upper bound for pressure using measure}
			\mathcal{P}_{m,m'}^{(n)}(\varphi)(\mathbf{Y}^{(n)}) \leq -\frac{1}{n^2} \log \sigma^{(n)}(\Gamma^{(n)}(\mathbf{x} \mid \mathbf{Y}^{(n)} \rightsquigarrow \mathbf{y}; \varphi, \varepsilon)) + \varphi^{\cM}(\mathbf{x},\mathbf{y}) + \varepsilon,
		\end{equation}
		and so
		\[
		-\frac{1}{n^2} \log \sigma^{(n)}(\Gamma^{(n)}(\mathbf{x} \mid \mathbf{Y}^{(n)} \rightsquigarrow \mathbf{y}; \varphi, \varepsilon)) + \varphi^{\cM}(\mathbf{x},\mathbf{y})  \geq \mathcal{P}_{m,m'}^{(n)}(\varphi)(\mathbf{Y}^{(n)}) - \varphi^{\cM}(\mathbf{x},\mathbf{y}) - \varepsilon.
		\]
		Therefore,
		\begin{align*}
			\tilde{\chi}_{\chron}^{\cU}(\mathbf{y} \mid \mathbf{Y}^{(n)} \rightsquigarrow \mathbf{y}) &\geq
			\tilde{\chi}_{\chron}^{\cU}(\mathbf{x} \mid \mathbf{Y}^{(n)} \rightsquigarrow \mathbf{y}; \varphi,\varepsilon) \\
			&\geq \lim_{n \to \cU} \mathcal{P}_{m,m'}^{(n)}(\varphi)(\mathbf{Y}^{(n)}) - \varphi^{\cM}(\mathbf{x},\mathbf{y}) - \varepsilon.
		\end{align*}
		Since $\varepsilon$ was arbitrary, the proof of ($\geq$) is complete.
		
		($\leq$) Fix a finite subset $\Phi \subseteq \mathcal{F}_{\chron,m+m'}^0$.  Let
		\[
		\psi = \max_{\varphi \in \Phi} |\varphi - \varphi^{\cM}(\mathbf{x},\mathbf{y})| \geq 0,
		\]
		where $(\mathbf{x},\mathbf{y})$ is the tuple fixed in the statement.  Let $a > 0$ and $\varepsilon > 0$.  Note that $\Lambda_{a\psi}^{(n)}$ can be chosen nonnegative.
		\begin{align*}
			\int_{\mathbb{M}_n^m} \exp( -n^2 \Lambda_{a\psi}^{(n)}(\mathbf{z},\mathbf{Y}^{(n)}))\,d\sigma^{(n)}(\mathbf{z}) &= \int_{\Gamma^{(n)}(\mathbf{x} \mid \mathbf{Y}^{(n)} \rightsquigarrow \mathbf{y}; \Phi, \varepsilon)} \exp( -n^2 \Lambda_{a\psi}^{(n)}(\mathbf{z},\mathbf{Y}^{(n)}))\,d\sigma^{(n)}(\mathbf{z}) \\
			& \quad + \int_{\mathbb{M}_n^m \setminus \Gamma^{(n)}(\mathbf{x} \mid \mathbf{Y}^{(n)} \rightsquigarrow \mathbf{y}; \Phi, \varepsilon)} \exp( -n^2 \Lambda_{a\psi}^{(n)}(\mathbf{z},\mathbf{Y}^{(n)}))\,d\sigma^{(n)}(\mathbf{z}) \\
			&\leq \sigma^{(n)}(\Gamma^{(n)}(\mathbf{x} \mid \mathbf{Y}^{(n)} \rightsquigarrow \mathbf{y}; \Phi, \varepsilon)) \\
			&\quad + \int_{\mathbb{M}_n^m \setminus \Gamma^{(n)}(\mathbf{x} \mid \mathbf{Y}^{(n)} \rightsquigarrow \mathbf{y}; \Phi, \varepsilon)} \exp( -n^2 a \varepsilon) \,d\sigma^{(n)}(\mathbf{z}) \\
			&\leq \sigma^{(n)}(\Gamma^{(n)}(\mathbf{x} \mid \mathbf{Y}^{(n)} \rightsquigarrow \mathbf{y}; \Phi, \varepsilon)) + \exp(-n^2 a \varepsilon) \\
			&\leq 2\max \left( \sigma^{(n)}(\Gamma^{(n)}(\mathbf{x} \mid \mathbf{Y}^{(n)} \rightsquigarrow \mathbf{y}; \Phi, \varepsilon)), \exp(-n^2 a \varepsilon) \right).
		\end{align*}
		Therefore,
		\begin{equation} \label{eq: lower bound for pressure using measure}
			\mathcal{P}_{m,m'}^{(n)}(a\psi)(\mathbf{Y}^{(n)}) \geq \min\left( -\frac{1}{n^2} \log \sigma^{(n)}(\Gamma^{(n)}(\mathbf{x} \mid \mathbf{Y}^{(n)} \rightsquigarrow \mathbf{y}; \Phi, \varepsilon)), a \varepsilon \right) - \frac{1}{n^2} \log 2.
		\end{equation}
		Hence, taking $n \to \cU$,
		\begin{align*}
			\min\left( \chi_{\chron}^{\cU}(\mathbf{x} \mid \mathbf{Y}^{(n)} \rightsquigarrow \mathbf{y}; \Phi, \varepsilon), a \varepsilon \right) &\leq \lim_{n \to \cU} \mathcal{P}_{m,m'}^{(n)}(a\psi)(\mathbf{Y}^{(n)}) \\
			&\leq \sup_{\varphi \in \mathcal{F}_{\chron,m+m'}^0} \left[\lim_{n \to \cU} \mathcal{P}_{m,m'}^{(n)}(\varphi)(\mathbf{Y}^{(n)}) - \varphi^{\cM}(\mathbf{x},\mathbf{y}) \right],
		\end{align*}
		where the last statement follows since $a \psi^{\cM}(\mathbf{x},\mathbf{y}) = 0$.  Then on the left hand side, we first take $a \to \infty$, and then take the supremum over $(\Phi,\varepsilon)$ to obtain the desired inequality ($\leq$) in \eqref{eq: entropy as Laplace transform}.
	\end{proof}
	
	We next will evaluate the pressure functional using the Bou{\'e}--Dupuis formula \cite{BoueDupuis1998variational}.
	
	\begin{theorem}[{Bou{\'e}--Dupuis \cite{BoueDupuis1998variational}}]
		Let $B_t$ be a standard Brownian motion on $\R^d$.  Let $\mathcal{G}_t$ be the associated augmented filtration \cite[\S 2]{BoueDupuis1998variational}.  Let $\varphi$ be a globally bounded Borel function $C([0,1];\R^d) \to \R$.  For $\alpha: [0,1] \to \R^d$ progressively measurable, let $\int_0^{\cdot} \alpha$ denote the process given as the antiderivative of $\alpha$.  Then
		\[
		-\log \mathbb{E} \exp(-\varphi(B)) = \inf_\alpha \mathbb{E} \left[ \varphi\left(B + \int_0^{\cdot} \alpha \right) + \frac{1}{2} \int_0^1 |\alpha|^2\,dt \right],
		\]
		where $\alpha$ ranges over the progressively measurable processes.
	\end{theorem}
	
	For the purposes of this work, we only consider the case where $\varphi(f)$ is a fixed function of $f(1)$ for $f \in C([0,1],\R^d)$ (in this case, the Bou{\'e}--Dupuis formula specializes to Borell's formula).  Since $(\mathbb{M}_n)_{\sa}^m \cong \R^{mn^2}$ as an inner-product space, we can apply the lemma to $(\mathbb{M}_n)_{\sa}^m$ with the Hermitian matrix Brownian motion, which is the standard Brownian motion in $(\mathbb{M}_n)_{\sa}^m$ up to renormalization.  This yields the following result, the proof of which can be found in \cite[Lemma 6.2]{GJNPmatrixcontrols}.
	
	\begin{corollary} \label{cor: application of Boue Dupuis}
		Consider the augmented filtration $\mathcal{G}_t$ generated by a Hermitian Brownian motion $(Z_{t,j}^{(n)})_{t \in [0,1],j \in [m]}$.  Let $\varphi: \mathbb{M}_n^m \to \R$ be bounded and Borel.  Then
		\[
		-\frac{1}{n^2} \log \int_{\mathbb{M}_n^m} \exp(-n^2 \varphi) \,d\sigma^{(n)} = \inf_{\mathbf{\alpha}} \mathbb{E} \left[ \varphi\left(\mathbf{Z}_t + \int_0^1 \mathbf{\alpha} \right) + \frac{1}{2} \int_0^1 \norm{\mathbf{\alpha}}_2^2 \,dt \right],
		\]
		where $\mathbf{Z}_t = (Z_{t,1}, \dots, Z_{t,m})$ and $\mathbf{\alpha} = (\alpha_1,\dots,\alpha_m): \Omega \times [0,1] \to \mathbb{M}_n^m$ is progressively measurable.
	\end{corollary}
	
	For $\varphi \in \mathcal{F}_{\chron,m+m'}^0$, we will approximate
	\[
	\inf_{\mathbf{\alpha}} \mathbb{E} \left[ \Lambda_\varphi^{(n)}\left(\mathbf{Z}_t + \int_0^1 \mathbf{\alpha}, \mathbf{Y}^{(n)} \right) + \frac{1}{2} \int_0^1 \norm{\mathbf{\alpha}}_2^2 \,dt \right]
	\]
	by $\Lambda_\psi^{(n)}(\mathbf{Y}^{(n)})$ for some other restricted chronological formulas $\psi$.  The idea is that if we discretize time in the optimization problem and make an operator norm cutoff, then the problem can be expressed in terms of chronological formulas.  The discretization and cutoff argument proceeds largely in a similar way to \cite[\S 4]{GJNPmatrixcontrols} with a couple of differences:  We have a more explicit problem without any classical $1$-dimensional Brownian motion.  The ``running cost'' term from \cite{GJNPmatrixcontrols} is here simply quadratic.  On the other hand, we \emph{do not} have any convexity hypothesis on $\Lambda_{\varphi}^{(n)}$, so we use only the boundedness and the $\norm{\cdot}_1$-Lipschitzness.
	
	\begin{lemma} \label{lem: discrete approx of matrix pressure}
		Let $m \in \N$ and $m' \in \N \cup \{0,\infty\}$.  For $\varphi \in \mathcal{F}_{\chron,m+m'}^0$, let
		\begin{equation} \label{eq: discretized HJB operator}
			(T_{t,m,m',r} \varphi)(\mathbf{x},\mathbf{y}) = \inf_{\gamma_1,\dots,\gamma_m \in D_{0,r}} \left[(S_t \varphi)((x_j + z_{0,t,j} + t \re(\gamma_j))_{j=1}^m, \mathbf{y}) + \frac{1}{2} t \norm{\re \gamma_j}_2^2 \right].
		\end{equation}
		Let $L$ be a $\norm{\cdot}_1$-Lipschitz constant for $\Lambda_{\varphi}^{(n)}$ in each matrix argument for all $n$.  Then for $\mathbf{Y} \in (\mathbb{M}_n)_{\sa}^m$,
		\begin{equation} \label{eq: discrete approx of matrix pressure}
			|\Lambda_{(T_{1/k,m,m',r})^k \varphi}^{(n)}(0,\mathbf{Y}) - \mathcal{P}^{(n)}(\varphi)(\mathbf{Y})| \leq Lm^{1/2} \norm{F_3(Z_{0,1,1}^{(n)}) - Z_{0,1,1}^{(n)}}_{L^2} + \frac{L^2}{k} \text{ for } r \geq L.
		\end{equation}
	\end{lemma}
	
	\begin{proof}
		\textbf{Step 1:} Assume that $\Lambda$ is constructed as in the proof of Proposition \ref{prop: approx by pointwise function}.  Then
		\[
		\Lambda_{T_{t,m,m',r} \varphi}^{(n)}(\mathbf{X},\mathbf{Y}) = \inf_{\substack{\mathbf{\beta} \in \mathbb{M}_n^m \\ \norm{\beta_j} \leq r}} \mathbb{E} \left[ \Lambda_\varphi^{(n)}(\mathbf{X} + (F_{3\sqrt{t}}(Z_{0,t,j}^{(n)}))_{j \in [m]} + t \mathbf{\beta}, \mathbf{Y}) + \frac{1}{2} t \norm{\mathbf{\beta}}_2^2 \right],
		\]
		where we replaced $\re(\gamma_j)$ with $\beta_j$.  Then a straightforward induction using conditioning shows that
		\[
		\Lambda_{(T_{1/k,m,m',r})^k \varphi}^{(n)}(\mathbf{X},\mathbf{Y}) = \inf_{\mathbf{\beta}_1, \dots, \mathbf{\beta}_k} \mathbb{E} \left[ \Lambda_\varphi^{(n)}\left(\mathbf{X} + \sum_{i=1}^k (F_{3\sqrt{1/k}}(Z_{(i-1)/k,i/k,j}^{(n)}))_{j \in [m]} + \frac{1}{k} \sum_{i=1}^k \mathbf{\beta}_i, \mathbf{Y} \right) + \frac{1}{2k} \sum_{i=1}^k \norm{\mathbf{\beta}_i}_2^2 \right],
		\]
		where $\mathbf{\beta}_i$ ranges over $\mathcal{F}_{(i-1)/k}$-measurable functions into $\mathbb{M}_n^m$ with $\norm{\beta_{i,j}} \leq r$ for $j \in [m]$.
		
		\textbf{Step 2:} Our next goal is to remove the cutoff function $F_{3\sqrt{1/k}}$ applied to the matrix Brownian motion. Since $F_{3\sqrt{1/k}}(Z_{(i-1)/k,i/k,j}^{(n)}) - Z_{(i-1)/k,i/k,j}^{(n)}$ are i.i.d.\ for $i = 1$, \dots, $k$ with mean zero, we have
		\begin{align*}
			\norm{\sum_{i=1}^k F_{3\sqrt{1/k}}(Z_{(i-1)/k,i/k,j}^{(n)}) - Z_{0,1,j}^{(n)}}_{L^2}^2 &= \norm{\sum_{i=1}^k (F_{\sqrt{1/k}}(Z_{(i-1)/k,i/k,j}^{(n)}) - Z_{(i-1)/k,i/k,j}^{(n)})}_{L^2}^2 \\
			&= k \norm{F_{3\sqrt{1/k}}(Z_{0,1/k,j}^{(n)}) - Z_{0,1/k,j}^{(n)}}_{L^2}^2 \\
			&= \norm{F_3(Z_{0,1,1}^{(n)}) - Z_{0,1,1}^{(n)}}_{L^2}^2
		\end{align*}
		because $F_{3\sqrt{t}}(x) = \sqrt{t}F_3(x/\sqrt{t})$ and $Z_{0,t,j}^{(n)} \sim t^{1/2} Z_{0,1,1}^{(n)}$ in distribution.  Let
		\[
		\Phi_{k,r}^{(n)}(\mathbf{X},\mathbf{Y}) = \inf_{\mathbf{\beta}_1, \dots, \mathbf{\beta}_k} \mathbb{E} \left[ \Lambda_\varphi^{(n)}\left(\mathbf{X} + (Z_{0,1,j}^{(n)})_{j \in [m]} + \frac{1}{k} \sum_{i=1}^k \mathbf{\beta}_i, \mathbf{Y} \right) + \frac{1}{2k} \sum_{i=1}^k \norm{\mathbf{\beta}_i}_2^2 \right].
		\]
		Using the Lipschitz property of $\Lambda_\varphi^{(n)}$ in $\norm{\cdot}_1$ and the fact that $\norm{\cdot}_1 \leq m^{1/2} \norm{\cdot}_2$ for $m$-tuples, we obtain that for each $\mathbf{\beta}_1$, \dots, $\mathbf{\beta}_k$,
		\begin{multline*}
			\biggl|\mathbb{E} \left[ \Lambda_\varphi^{(n)}\left(\mathbf{X} + (Z_{0,1,j}^{(n)})_{j \in [m]} + \frac{1}{k} \sum_{i=1}^k \mathbf{\beta}_i, \mathbf{Y} \right) + \frac{1}{2k} \sum_{i=1}^k \norm{\mathbf{\beta}_i}_2^2 \right] \\
			- \mathbb{E} \left[ \Lambda_\varphi^{(n)}\left(\mathbf{X} + \sum_{i=1}^k (F_{\sqrt{1/k}}(Z_{(i-1)/k,i/k,j}^{(n)}))_{j \in [m]} + \frac{1}{k} \sum_{i=1}^k \mathbf{\beta}_i, \mathbf{Y} \right) + \frac{1}{2k} \sum_{i=1}^k \norm{\mathbf{\beta}_i}_2^2 \right] \biggr| \\
			\leq Lm^{1/2} \norm{F_1(Z_{0,1,1}^{(n)}) - Z_{0,1,1}^{(n)}}_{L^2}.
		\end{multline*}
		Therefore, the same holds for the infimum of these expressions over $\mathbf{\beta}_1$, \dots, $\mathbf{\beta}_k$, that is,
		\begin{equation} \label{eq: substitution with Phi}
			|\Phi_{k,r}^{(n)}(\mathbf{X},\mathbf{Y}) - \Lambda_{(T_{1/k,m,m',r})^k \varphi}^{(n)}(\mathbf{X},\mathbf{Y})| \leq Lm^{1/2} \norm{F_3(Z_{0,1,1}^{(n)}) - Z_{0,1,1}^{(n)}}_{L^2}.
		\end{equation}
		
		\textbf{Step 3:} Our next goal is to remove the operator norm cutoff by comparing $\Phi_{k,r}^{(n)}$ with $\Phi_{k,\infty}^{(n)}$.  It is immediate that
		\[
		\Phi_{k,\infty}^{(n)}(\mathbf{X},\mathbf{Y}) \leq \Phi_{k,r}^{(n)}(\mathbf{X},\mathbf{Y})
		\]
		since the infimum is taken over a larger set, so we focus on the reverse inequality.  Consider a candidate $\beta_i$ for the infimum in $\Phi_{k,\infty}^{(n)}(\mathbf{X},\mathbf{Y})$.  For each $i$, $j$, let $\beta_{i,j}'$ be the projection onto the operator norm ball of radius $r$ as in Lemma \ref{lem: projection onto R ball}, i.e.\ $\beta_{i,j}'$ is the closest point in the $r$-ball to $X$ with respect to $\norm{\cdot}_2$, and let $\mathbf{\beta}_i = (\beta_{i,j})_{j=1}^m$.  Then by Lemma \ref{lem: projection onto R ball},
		\begin{align*}
			\norm*{\frac{1}{k} \sum_{i=1}^k \beta_{i,j} - \frac{1}{k} \sum_{i=1}^k \beta_{i,j}'}_1 &\leq \frac{1}{k} \sum_{i=1}^k \norm{\beta_{i,j} - \beta_{i,j}}_1 \\
			&\leq \frac{1}{2rk} \sum_{i=1}^k (\norm{\beta_{i,j}}_2^2 - \norm{\beta_{i,j}'}_2^2).
		\end{align*}
		Then
		\begin{align*}
			\biggl|\mathbb{E} \Lambda_\varphi^{(n)}\left(\mathbf{X} + (Z_{0,1,j}^{(n)})_{j \in [m]} + \frac{1}{k} \sum_{i=1}^k \mathbf{\beta}_i', \mathbf{Y} \right)
			& - \mathbb{E} \Lambda_\varphi^{(n)}\left(\mathbf{X} + (Z_{0,1,j}^{(n)})_{j \in [m]} + \frac{1}{k} \sum_{i=1}^k \mathbf{\beta}_i, \mathbf{Y} \right) \biggr| \\
			&\leq \frac{L}{k} \sum_{i=1}^k \sum_{j=1}^m \mathbb{E} \norm{\beta_{i,j} - \beta_{i,j}'}_1 \\
			&\leq \frac{L}{2kr} \sum_{i=1}^k \mathbb{E} (\norm{\mathbf{\beta}_i}_2^2 - \norm{\beta_i'}_2^2).
		\end{align*}
		When $r \geq L$, the right-hand side is bounded by $\frac{1}{2k} \sum_{i=1}^k \mathbb{E} (\norm{\mathbf{\beta}_i}_2^2 - \norm{\beta_i'}_2^2)$.  Thus, after rearranging,
		\begin{multline*}
			\mathbb{E} \left[ \Lambda_\varphi^{(n)}\left(\mathbf{X} + (Z_{0,1,j}^{(n)})_{j \in [m]} + \frac{1}{k} \sum_{i=1}^k \mathbf{\beta}_i', \mathbf{Y} \right) + \frac{1}{2k} \sum_{i=1}^k \norm{\mathbf{\beta}_i'}_2^2 \right] \\
			\leq \mathbb{E} \left[ \Lambda_\varphi^{(n)}\left(\mathbf{X} + (Z_{0,1,j}^{(n)})_{j \in [m]} + \frac{1}{k} \sum_{i=1}^k \mathbf{\beta}_i, \mathbf{Y} \right) + \frac{1}{2k} \sum_{i=1}^k \norm{\mathbf{\beta}_i}_2^2 \right].
		\end{multline*}
		Intuitively, the largest possible amount that $\Lambda_\varphi^{(n)}$ could grow when we replace $\beta_i$ by $\beta_i'$ is \emph{smaller} than the reduction in $L^2$-norm that get by replacing $\beta_i$ with $\beta_i'$.  We therefore have
		\begin{equation} \label{eq: removing operator norm cutoff}
			\Phi_{k,\infty}^{(n)}(\mathbf{X},\mathbf{Y}) = \Phi_{k,r}^{(n)}(\mathbf{X},\mathbf{Y}) \text{ for } r \geq L.
		\end{equation}
		
		\textbf{Step 4:} Finally, we pass from the discrete-time optimization problem for $\Phi_{k,\infty}^{(n)}$ to a continuous-time optimization problem.  Let
		\[
		\Phi_{\infty,\infty}^{(n)}(\mathbf{X},\mathbf{Y}) = \inf_{\mathbf{\alpha}} \mathbb{E} \left[ \varphi\left(\mathbf{X} + (Z_{0,t,j}^{(n)})_{j=1}^m + \int_0^1 \mathbf{\alpha}\,dt \right) + \frac{1}{2} \int_0^1 \norm{\mathbf{\alpha}}_2^2 \,dt \right],
		\]
		where $\mathbf{\alpha} = (\alpha_1,\dots,\alpha_m): \Omega \times [0,1] \to (\mathbb{M}_n)_{\sa}^m$ is progressively measurable.  We first observe that
		\[
		\Phi_{\infty,\infty}^{(n)}(\mathbf{X},\mathbf{Y}) \leq \Phi_{k,\infty}^{(n)}(\mathbf{X},\mathbf{Y})
		\]
		because if $\mathbf{\beta}_i$ is a $\mathcal{F}_{(i-1)/k}$-measurable function into $(\mathbb{M}_n)_{\sa}^m$, then we can take
		\[
		\mathbf{\alpha}_t = \sum_{k=1}^m \mathbbm{1}_{[(i-1)/k,i/k)}(t) \mathbf{\beta}_i
		\]
		as a candidate in the infimum for $\Phi_{\infty,\infty}^{(n)}(\mathbf{X},\mathbf{Y})$, noting that
		\[
		\int_0^1 \mathbf{\alpha}_t\,dt = \frac{1}{k} \sum_{k=1}^m \mathbf{\beta}_i
		\]
		and
		\[
		\int_0^1 \norm{\mathbf{\alpha}_t}_2^2\,dt = \frac{1}{k} \sum_{k=1}^m \norm{\mathbf{\beta}_i}_2^2.
		\]
		To obtain a corresponding upper bound on $\Phi_{k,\infty}^{(n)}(\mathbf{X},\mathbf{Y})$, suppose that $\mathbf{\alpha}$ is a candidate for the infimum in $\Phi_{\infty,\infty}^{(n)}$.  By Step 3, we can restrict our attention to candidates $\alpha$ which are bounded in operator norm by $L$.  Let
		\begin{align*}
			\mathbf{\beta}_1 &= 0; \\
			\mathbf{\beta}_i &= \int_{(i-2)/k}^{(i-1)/k} \mathbf{\alpha}_t\,dt \text{ for } 2 \leq i \leq k.
		\end{align*}
		Note that $\mathbf{\beta}_i$ is $\mathcal{F}_{i-1}$-measurable.  We also have
		\begin{align*}
			\frac{1}{k} \sum_{i=1}^k \norm{\mathbf{\beta}_i}_2^2 &= \frac{1}{k} \sum_{i=1}^{k-1} \norm*{\int_{(i-1)/k}^{i/k} \mathbf{\alpha}_t\,dt}_2^2 \\
			&\leq \frac{1}{k} \sum_{i=1}^{k-1} \int_{(i-1)/k}^{i/k} \norm{ \mathbf{\alpha}_t }_2^2\,dt \\
			&\leq \int_0^1 \norm{ \mathbf{\alpha}_t }_2^2\,dt.
		\end{align*}
		Meanwhile,
		\[
		\norm*{\frac{1}{k} \sum_{i=1}^k \mathbf{\beta}_i - \int_0^1 \mathbf{\alpha}_t\,dt}_1 = \norm*{\int_{(k-1)/k}^1 \mathbf{\alpha_t}\,dt }_1 \leq \int_{1-1/k}^1 \norm{\mathbf{\alpha_t} }_2 \,dt \leq \frac{L}{k}
		\]
		Then because $\Lambda_\varphi^{(n)}$ is $L$-Lipschitz, we have
		\begin{multline*}
			\mathbb{E} \left[ \Lambda_\varphi^{(n)}\left(\mathbf{X} + (Z_{0,1,j}^{(n)})_{j \in [m]} + \frac{1}{k} \sum_{i=1}^k \mathbf{\beta}_i, \mathbf{Y} \right) + \frac{1}{2k} \sum_{i=1}^k \norm{\mathbf{\beta}_i}_2^2 \right] \\
			\leq \mathbb{E} \left[ \varphi\left(\mathbf{X} + (Z_{0,t,j}^{(n)})_{j=1}^m + \int_0^1 \mathbf{\alpha}\,dt \right) + \frac{1}{2} \int_0^1 \norm{\mathbf{\alpha}}_2^2 \,dt \right] + 2L \sqrt{\frac{M}{k}}. 
		\end{multline*}
		We therefore conclude that
		\begin{equation} \label{eq: discrete to continuous error}
			\Phi_{\infty,\infty}^{(n)}(\mathbf{X},\mathbf{Y}) \leq \Phi_{k,\infty}^{(n)}(\mathbf{X},\mathbf{Y}) \leq \Phi_{\infty,\infty}^{(n)}(\mathbf{X},\mathbf{Y}) + \frac{L^2}{k}.
		\end{equation}
		
		\textbf{Conclusion:} Combining the results \eqref{eq: substitution with Phi}, \eqref{eq: removing operator norm cutoff}, and \eqref{eq: discrete to continuous error} from the previous steps,
		\begin{equation} \label{eq: overall control problem error}
			|\Lambda_{(T_{1/k,m,m',r})^k \varphi}^{(n)}(\mathbf{X},\mathbf{Y}) - \Phi_{\infty,\infty}^{(n)}(\mathbf{X},\mathbf{Y})| \leq Lm^{1/2} \norm{F_1(Z_{0,1,1}^{(n)}) - Z_{0,1,1}^{(n)}}_{L^2} + \frac{L^2}{k} \text{ for } r \geq L.
		\end{equation}
		By Corollary \ref{cor: application of Boue Dupuis}, we have
		\begin{align*}
			\mathcal{P}_{m,m'}^{(n)}(\varphi)(\mathbf{Y}) &= -\frac{1}{n^2} \log \int_{(\mathbb{M}_n)_{\sa}^m} \exp(-n^2 \varphi(\mathbf{Z},\mathbf{Y})) \,d\sigma^{(n)}(\mathbf{Z}) \\
			&= \inf_{\mathbf{\alpha}} \mathbb{E} \left[ \varphi\left( (Z_{0,t,j}^{(n)})_{j=1}^m + \int_0^1 \mathbf{\alpha}\,dt \right) + \frac{1}{2} \int_0^1 \norm{\mathbf{\alpha}}_2^2 \,dt \right] \\
			&= \Phi_{\infty,\infty}^{(n)}(0,\mathbf{Y}).
		\end{align*}
		Therefore, by taking $\mathbf{X} = 0$ in \eqref{eq: overall control problem error}, we obtain \eqref{eq: discrete approx of matrix pressure}.
	\end{proof}
	
	We next prove an analogous but easier estimate for the infinite-dimensional setting of $\cL_{\proc}$ structures.  First, we need to set up the corresponding optimization problem.  Suppose that $(M,\tau)$ is a tracial von Neumann algebra.  A \emph{simple function} $[0,T] \to L^2(M,\tau)$ is a function of the form
	\[
	f(t) = \sum_{j=1}^J \mathbbm{1}_{A_j}(t) x_j
	\]
	where $x_j \in L^2(M,\tau)$ and the $A_j$'s are disjoint and measurable.  The $L^2$-norm is then defined as
	\[
	\norm{f}_{L_{\Boch}^2([0,t_*],L^2(M))}^2 = \sum_j \operatorname{Leb}(A_j) \norm{x_j}_2^2,
	\]
	where $\operatorname{Leb}$ denotes the Lebesgue measure on $[0,T]$.
	The \emph{Bochner $L^2$-space} $L_{\Boch}^2([0,t_*],L^2(M,\tau))$ is the completion of the vector space of simple functions with respect to this norm.
	
	Let $\cM$ be a tracial von Neumann algebra $(M,\tau)$ equipped with a filtration $M_t$.  A function $\alpha \in L_{\Boch}^2([0,t_*],L^2(M,\tau))$ is \emph{adapted} if for almost every $t$ we have $\alpha(t) \in L^2(M_t,\tau)$.
	
	\begin{lemma} \label{lem: discrete approximation of free pressure}
		For $\cM = ((M_t)_{t \in [0,\infty]}, (z_{s,t,j})_{0\leq s \leq t,j\in \N})$ a model of $\rT_{\proc}$ and $\varphi \in \mathcal{F}_{\chron,m+m'}^0$, define
		\begin{equation} \label{eq: definition of P}
			\mathcal{P}_{m,m'}(\varphi)^{\cM}(\mathbf{y}) = \inf_{\mathbf{\alpha}} \left[ \varphi^{\cM}\left((z_{0,t,j})_{j=1}^m + \int_0^1 \mathbf{\alpha}_t\,dt, \mathbf{y} \right) + \frac{1}{2} \int_0^1 \norm{\alpha_t}_2^2 \,dt \right],
		\end{equation}
		where $\mathbf{\alpha} \in L_{\Boch}^2([0,1];L^2(M_\infty))$ is adapted and bounded in operator norm (note that the bound on operator norm is allowed to depend on $\mathbf{\alpha}$; we merely require each $\mathbf{\alpha}$ to be bounded).  Suppose that $\varphi$ is $L$-Lipschitz with respect to $\norm{\cdot}_1$ (see Remark \ref{rem: bounded and Lipschitz}).  Then
		\begin{equation} \label{eq: approx of free pressure}
			|\mathcal{P}_{m,m'}(\varphi)^{\cM}(\mathbf{y}) - (T_{1/k,m,m',r}^k \varphi)(0,\mathbf{y})| \leq \frac{L^2}{k} \text{ for } r \geq L.
		\end{equation}
		In particular, $\mathcal{P}_{m,m'}(\varphi)$ is a chronologically definable predicate with respect to $\rT_{\proc}$.
	\end{lemma}
	
	\begin{proof}
		The argument proceeds similarly to Lemma \ref{lem: discrete approx of matrix pressure}, so we will merely sketch the proof; note that a similar argument is given in detail in \cite[Proposition 4.17]{GJNPmatrixcontrols} as well.
		
		Fix $\cM$.  For $k \in \N$ and $r \in [0,\infty]$, let
		\[
		\Phi_{k,r}(\mathbf{x},\mathbf{y}) = \inf_{\mathbf{\beta}_1, \dots, \mathbf{\beta}_k} \left[ \varphi\left( \mathbf{x} + (z_{0,1,j})_{j=1}^m + \frac{1}{k} \sum_{i=1}^k \mathbf{\beta}_i \right) + \frac{1}{2k} \sum_{i=1}^k \norm{\mathbf{\beta}_i}_2^2 \right],
		\]
		where $\mathbf{\beta}_i \in (\cM_{(i-1)/k})_{\sa}^m$ with $\norm{\mathbf{\beta}_i} \leq r$.  Similar to Step 1 in Lemma \ref{lem: discrete approx of matrix pressure}, we observe that
		\[
		(T_{1/k,m,m',r}^k \varphi)(\mathbf{x},\mathbf{y}) = \Phi_{k,r}(\mathbf{x},\mathbf{y}).
		\]
		Here we rely on the fact that $z_{0,1,j} = \sum_{i=1}^k z_{(i-1)/k,i/k,j}$ under $\rT_{\proc}$.  Step 2 is not relevant here since we do not have a cutoff function applied to the process.  By similar reasoning as Step 3 of Lemma \ref{lem: discrete approx of matrix pressure}, we obtain
		\[
		\Phi_{k,r}(\mathbf{x},\mathbf{y}) = \Phi_{k,\infty}(\mathbf{x},\mathbf{y}) \text{ for } r \geq L.
		\]
		Then as in Step 4, we set
		\[
		\Phi_{\infty,\infty}(\mathbf{x},\mathbf{y}) = \inf_{\mathbf{\alpha}} \varphi^{\cM}\left((z_{0,t,j})_{j=1}^m + \int_0^1 \mathbf{\alpha}_t\,dt, \mathbf{y} \right),
		\]
		where $\mathbf{\alpha} \in L_{\Boch}^2([0,1];L^2(\cM))$ is adapted and bounded in operator norm.  We obtain by similar reasoning that
		\[
		|\Phi_{k,\infty}(\mathbf{x},\mathbf{y}) - \Phi_{\infty,\infty}(\mathbf{x},\mathbf{y})| \leq \frac{L^2}{k}.
		\]
		We observe that
		\[
		\mathcal{P}_{m,m'}(\varphi)^{\cM}(\mathbf{y}) = \Phi_{\infty,\infty}(0,\mathbf{y}),
		\]
		and hence obtain the asserted result for $\cM$.  Since the error bound is uniform for all $\cM$ modeling $\rT_{\proc}$, we see that $\mathcal{P}_{m,m'}(\varphi)$ is a chronologically definable predicate as desired.
	\end{proof}
	
	\begin{theorem} \label{thm: final formula for entropy}
		Let $\cM$ be an $\cL_{\proc}$-structure, let $m \in \N$ and $m' \in \N \cup \{0,\infty\}$, and let $(\mathbf{x},\mathbf{y}) \in M_0^{m+m'}$.  Suppose that $\mathbf{Y}^{(n)}$ is a matrix tuple with $\norm{Y_j^{(n)}} \leq R$ and
		\begin{equation} \label{eq: convergence hypothesis}
			\lim_{n \to \cU} \Lambda_\psi^{(n)}(\mathbf{Y}^{(n)}) = \psi^{\cM}(\mathbf{y}) \text{ for } \psi \in \mathcal{F}_{\chron,m'}^0.
		\end{equation}
		Then
		\begin{equation} \label{eq: convergence of matrix pressure}
			\lim_{n \to \cU} \mathcal{P}^{(n)}(\varphi)(\mathbf{Y}^{(n)}) = \mathcal{P}(\varphi)^{\cM}(\mathbf{y}) \text{ for } \varphi \in \mathcal{F}_{\chron,m+m'}^0.
		\end{equation}
		Consequently,
		\begin{equation} \label{eq: final pressure formula for entropy}
			\tilde{\chi}_{\chron}^{\cU}(\mathbf{x} \mid \mathbf{Y}^{(n)} \rightsquigarrow \mathbf{y}) = \sup_{\varphi \in \mathcal{F}_{\chron,m+m'}^0} \left[ \mathcal{P}(\varphi)^{\cM}(\mathbf{y}) - \varphi^{\cM}(\mathbf{x},\mathbf{y}) \right],
		\end{equation}
		which only depends on the type of $(\mathbf{x},\mathbf{y})$ and not on $\mathbf{Y}^{(n)}$.
	\end{theorem}
	
	\begin{remark} \label{rem: m prime zero case}
		Note in the case that $m' = 0$, the hypothesis \eqref{eq: convergence hypothesis} says that for restricted chronological formulas with no free variables (i.e.\ sentences),
		\[
		\varphi^{\cM} = \lim_{n \to \cU} \Lambda_\varphi^{(n)}.
		\]
		By Proposition \ref{prop: matrix quotient filtration}, this is the same as $\varphi^{\cQ}$ where $\cQ$ is $\cL_{\proc}$-structure obtain as the matrix ultraproduct quotient.  Hence, even if $m' = 0$, we need to assume that $\cM$ is chronologically elementarily equivalent to $\cQ$ (the values of all chronological sentences in the two structures agree).
	\end{remark}
	
	\begin{proof}[Proof of {Theorem \ref{thm: final formula for entropy}}]
		Let $\varphi \in \mathcal{F}_{\chron,m+m'}^0$, and let $M$ be a bound for $\varphi$ and $\Lambda_{\varphi}^{(n)}$, and let $L$ be a $\norm{\cdot}_1$-Lipschitz constant for $\varphi$ and $\Lambda_{\varphi}^{(n)}$.  For $k \in \N$ and $r > 0$, we apply \eqref{eq: approx of free pressure}:
		\[
		|\mathcal{P}_{m,m'}^{(n)}(\varphi)(\mathbf{Y}^{(n)}) - \Lambda_{T_{1/k,m,m',r}^k \varphi}(0,\mathbf{Y}^{(n)})| \leq Lm \norm{Z_{0,1}^{(n)} - F_3(Z_{0,1}^{(n)})}_{L^2} + \frac{L^2}{k} \text{ for } r \geq L.
		\]
		By \eqref{eq: truncated GUE matrix 2}, we have $\norm{Z_{0,1}^{(n)} - F_3(Z_{0,1}^{(n)})}_{L^2} \to 0$ as $n \to \infty$.  Moreover, by assumption \eqref{eq: convergence hypothesis}, $\Lambda_{T_{1/k,m,m',r}^k \varphi}(0,\mathbf{Y}^{(n)}) \to (T_{1/k,m,m',r}^k \varphi)^{\cM}(0,\mathbf{y})$ as $n \to \cU$, and therefore
		\[
		\left| \lim_{n \to \cU} \mathcal{P}_{m,m'}^{(n)}(\varphi)(\mathbf{Y}^{(n)}) - (T_{1/k,m,m',r}^k \varphi)^{\cM}(0,\mathbf{y}) \right| \leq \frac{L^2}{k} \text{ for } \geq L.
		\]
		Furthermore, applying \eqref{eq: approx of free pressure} with the triangle inequality,
		\[
		\left| \lim_{n \to \cU} \mathcal{P}_{m,m'}^{(n)}(\varphi)(\mathbf{Y}^{(n)}) - \mathcal{P}_{m,m'}(\varphi)^{\cM}(\mathbf{y}) \right| \leq \frac{2L^2}{k} \text{ for } r \geq L.
		\]
		Since $r$ and $k$ were arbitrary, we obtain \eqref{eq: approx of free pressure}.  Then by substituting \eqref{eq: approx of free pressure} into \eqref{eq: entropy as Laplace transform}, we obtain \eqref{eq: final pressure formula for entropy}.
	\end{proof}
	
	Now that we know that $\tilde{\chi}_{\chron}^{\cU}(\mathbf{x} \mid \mathbf{Y}^{(n)} \rightsquigarrow \mathbf{y})$ is independent of $\mathbf{Y}^{(n)}$ so long as \eqref{eq: convergence hypothesis} holds, we can make the following definition.
	
	\begin{definition}[Conditional chronological entropy: final version] \label{def: chronological entropy final}
		For a tuple $(\mathbf{x},\mathbf{y})$ in $M_0$ for an $\cL_{\proc}$-structure $\cM$ satisfying $\rT_{\proc}$, if there exists $\mathbf{Y}^{(n)}$ satisfying \eqref{eq: convergence hypothesis}, we define
		\[
		\tilde{\chi}_{\chron}^{\cU}(\mathbf{x} \mid \mathbf{y}) = \tilde{\chi}_{\chron}^{\cU}(\mathbf{x} \mid \mathbf{Y}^{(n)} \rightsquigarrow \mathbf{y}),
		\]
		or equivalently, define $\tilde{\chi}_{\chron}^{\cU}$ by the right-hand side of \eqref{eq: final pressure formula for entropy}.  In case that no $\mathbf{Y}^{(n)}$ exists satisfying \eqref{eq: convergence hypothesis}, we set $\tilde{\chi}_{\chron}^{\cU}(\mathbf{x} \mid \mathbf{y}) = +\infty$.
	\end{definition}
	
	\begin{remark}
		We remark that if $\mathbf{Y}^{(n)}$ is some matrix tuple that does not satisfy \eqref{eq: convergence hypothesis}, then $\tilde{\chi}_{\chron}^{\cU}(\mathbf{x} \mid \mathbf{Y}^{(n)} \rightsquigarrow \mathbf{y})$ will be infinite, because for a sufficiently precise choice of $\Phi$ and $\varepsilon$, the microstate space $\Gamma^{(n)}(\mathbf{x} \mid \mathbf{Y}^{(n)} \rightsquigarrow \mathbf{y}; \Phi, \varepsilon)$ will be empty.  Hence, it is natural to define the entropy to be infinite in this case.  Similarly, note that if $(\mathbf{x},\mathbf{y})$ does not admit any matrix approximations, then $\tilde{\chi}_{\chron}^{\cU}(\mathbf{x} \mid \mathbf{y})$ will always be $+\infty$.
	\end{remark}
	
	Theorem \ref{thm: final formula for entropy} also gives a convenient proof of lower semicontinuity of $\tilde{\chi}_{\chron}^{\cU}$ with respect to the weak-$*$ topology on types (in fact, the weak-$*$ topology induced by $\mathcal{F}_{\chron,m}^0$).
	
	\begin{corollary}[Lower semicontinuity] \label{cor: semicontinuity of entropy}
		Let $m \in \N$ and $m' \in \N \cup \{\infty\}$.  For $k \in \N \cup \{\infty\}$, let $\cM^{(k)}$ be an $\cL_{\proc}$-structure satisfying $\rT_{\stat}$.  Let $(\mathbf{x}^{(k)},\mathbf{y}^{(k)})$ be $(m+m')$-tuples from $M_0^{(k)}$, and assume that
		\begin{equation} \label{eq: limit for LSC}
			\forall \varphi \in \mathcal{F}_{\chron,m+m'}^0, \qquad \lim_{k \to \infty}  \varphi^{\cM^{(k)}}(\mathbf{x}^{(k)},\mathbf{y}^{(k)}) = \varphi^{\cM^{(\infty)}}(\mathbf{x}^{(\infty)},\mathbf{y}^{(\infty)}).
		\end{equation}
		Then
		\begin{equation} \label{eq: LSC}
			\liminf_{k \to \infty} \tilde{\chi}_{\chron}^{\cU}(\mathbf{x}^{(k)} \mid \mathbf{y}^{(k)}) \geq \tilde{\chi}_{\chron}^{\cU}(\mathbf{x}^{(\infty)} \mid \mathbf{y}^{(\infty)}). 
		\end{equation}
	\end{corollary}
	
	\begin{proof}
		First, note that \eqref{eq: limit for LSC} extends to $\varphi$ in $\overline{\cF}_{\chron,m+m'}^0$ by density.  For each $\varphi$, the quantity $\mathcal{P}(\varphi)(\mathbf{y}) - \varphi(\mathbf{x},\mathbf{y})$ depends continuously on the chronological type of $(\mathbf{x},\mathbf{y})$ with respect to the topology given by \eqref{eq: limit for LSC}.  Hence, by \eqref{eq: final pressure formula for entropy}, $\tilde{\chi}_{\chron}^{\cU}(\cdot \mid \cdot)$ is the supremum of a family of continuous functions, hence lower semicontinuous.
	\end{proof}
	
	\subsection{Chain rule} \label{subsec: chain rule}
	
	In this section, we prove a chain rule for $\tilde{\chi}_{\chron}^{\cU}$ under iterated conditioning.  The corresponding result for classical entropy of course is well-known.  The main difficulty in establishing such a result in the free probability setting was that there was no analog of Theorem \ref{thm: final formula for entropy} which would show that the entropy of $\mathbf{x}$ conditioned on $\mathbf{y}$ is independent of the choice of matrix approximations $\mathbf{Y}^{(n)}$ for $\mathbf{y}$.
	
	\begin{theorem}[Chain rule for entropy under conditioning] \label{thm: chain rule for entropy}
		Let $\mathbf{x}$, $\mathbf{y}$, $\mathbf{z}$ be $m_1$-, $m_2$-, and $m_3$-tuples from an $\mathcal{L}_{\operatorname{proc}}$-structure $M$, where $m_1, m_2 \in \N$ and $m_3 \in \N \cup \{0,\infty\}$.  Then
		\[
		\tilde{\chi}_{\chron}^{\cU}(\mathbf{x}, \mathbf{y} \mid \mathbf{z}) = \tilde{\chi}_{\chron}^{\cU}(\mathbf{x} \mid \mathbf{y}, \mathbf{z}) + \tilde{\chi}_{\chron}^{\cU}(\mathbf{y} \mid \mathbf{z}).
		\]
	\end{theorem}
	
	The formula for entropy in terms of pressure in Theorem \ref{thm: final formula for entropy} puts us in a good position to prove this, after a few additional lemmas.
	
	\begin{definition} \label{def: uniform of completion of RCF}
		Let $\overline{\mathcal{F}}_{\chron,m}^0$ be the completion of $\mathcal{F}_{\chron,m}^0$ with respect to the uniform norm
		\begin{equation} \label{eq: chain rule for entropy}
			\norm{\varphi}_u = \sup_{\cM \models \rT_{\proc}} \sup_{\mathbf{x} \in M_0^m} |\varphi^{\cM}(\mathbf{x})|.
		\end{equation}
	\end{definition}
	
	\begin{lemma}
		Let $m \in \N$ and $m' \in \N \cup \{0,\infty\}$.  For $\varphi, \psi \in \overline{\mathcal{F}}_{\chron,m+m'}^0$, let $\mathcal{P}_{m,m'}(\varphi)$ be given by \eqref{eq: definition of P}.  Then $\mathcal{P}_{m,m'}(\varphi) \in \overline{\mathcal{F}}_{\chron,m'}^0$ and
		\begin{equation} \label{eq: P contraction}
			\norm{\mathcal{P}_{m,m'}(\varphi) - \mathcal{P}_{m,m'}(\psi)}_u \leq \norm{\varphi - \psi}_u.
		\end{equation}
	\end{lemma}
	
	\begin{proof}
		First, one can show that \eqref{eq: P contraction} holds for $\varphi, \psi \in \mathcal{F}_{\chron,m+m'}^0$ by taking every candidate $\mathbf{\alpha}$ in the infimum \eqref{eq: definition of P} in the definition of $\varphi$ as a candidate in the infimum in the definition of $\psi$ and vice versa.  Since $\mathcal{P}_{m,m'}(\varphi) \in \overline{\mathcal{F}}_{\chron,m'}^0$ for $\varphi \in \mathcal{F}_{\chron,m+m'}^0$ by Lemma \ref{lem: discrete approx of matrix pressure}, the same is true when $\varphi$ is in the completion $\overline{\mathcal{F}}_{\chron,m+m'}^0$, and the inequality \eqref{eq: P contraction} extends to the completion.
	\end{proof}
	
	\begin{lemma}
		Let $(\mathbf{x},\mathbf{y})$ be a tuple in a model of $\cL_{\proc}$ which admits matrix approximations in the sense of \eqref{eq: convergence hypothesis}.  Then
		\begin{equation} \label{eq: final pressure formula for entropy 2}
			\tilde{\chi}_{\chron}^{\cU}(\mathbf{x} \mid \mathbf{y}) = \sup_{\varphi \in \overline{\mathcal{F}}_{\chron,m+m'}^0} \left[ \mathcal{P}(\varphi)^{\cM}(\mathbf{y}) - \varphi^{\cM}(\mathbf{x},\mathbf{y}) \right],
		\end{equation}
		that is, we can replace $\mathcal{F}_{\chron,m+m'}^0$ by its completion in \eqref{eq: final pressure formula for entropy}.
	\end{lemma}
	
	\begin{proof}
		Using \eqref{eq: P contraction}, we get the same result by taking the supremum over $\overline{\mathcal{F}}_{\chron,m+m'}^0$ or over the dense subset $\mathcal{F}_{\chron,m+m'}^0$.
	\end{proof}
	
	\begin{lemma} \label{lem: chain rule for pressure}
		Let $(\mathbf{x},\mathbf{y},\mathbf{w})$ be an $m_1 + m_2 + m_3$-tuple where $m_1, m_2 \in \N$ and $m_3 \in \N \cup \{0,\infty\}$, and assume that there exist matrix tuples with $(\mathbf{X}^{(n)},\mathbf{Y}^{(n)},\mathbf{W}^{(n)})$, with operator norm bounded by some constant $R$, such that
		\begin{equation} \label{eq: matrix approximation of triple}
			\lim_{n \to \cU} \Lambda_{\varphi}^{(n)}(\mathbf{X}^{(n)},\mathbf{Y}^{(n)},\mathbf{W}^{(n)}) = \varphi^{\cM}(\mathbf{x},\mathbf{y},\mathbf{w}).
		\end{equation}
		Let $\varphi \in \overline{\mathcal{F}}_{m_1+m_2+m_3}^0$.  Then
		\begin{equation} \label{eq: chain rule for pressure}
			\mathcal{P}_{m_1+m_2,m_3}(\varphi)^{\cM}(\mathbf{w}) = \mathcal{P}_{m_2,m_3}(\mathcal{P}_{m_1,m_2+m_3}(\varphi))^{\cM}(\mathbf{w}).
		\end{equation}
	\end{lemma}
	
	\begin{proof}
		First, assume that $\varphi \in \mathcal{F}_{m_1+m_2+m_3}^0$.  Let $\varepsilon > 0$.  By Lemma \ref{lem: discrete approx of matrix pressure} and \ref{lem: discrete approximation of free pressure}, choose $k \in \N$ and $r > 0$ such that for sufficiently large $n$, we have
		\[
		|\Lambda_{(T_{1/k,m_1,m_2+m_3,r})^k \varphi}^{(n)}(0,\mathbf{Y},\mathbf{W}) - \mathcal{P}^{(n)}(\varphi)(\mathbf{Y},\mathbf{W})| \leq \varepsilon
		\]
		and
		\begin{equation} \label{eq: approximation error for psi}
			\norm{(T_{1/k,m_1,m_2+m_3,r})^k \varphi(0,\cdot) - \mathcal{P}_{m_1,m_2+m_3}(\varphi)}_u \leq \varepsilon,
		\end{equation}
		and let $\psi_{k,r} = (T_{1/k,m_1,m_2+m_3,r})^k \varphi(0,\cdot)$.  Now observe that
		\begin{align*}
			\mathcal{P}_{m_1+m_2,m_3}^{(n)}(\varphi)(\mathbf{W}^{(n)}) &= -\frac{1}{n^2} \log \int_{(\mathbb{M}_n)_{\sa}^{m_2}} \int_{\mathbb{M}_n^{m_1}} \exp(-n^2 \varphi(\mathbf{X},\mathbf{Y},\mathbf{W}^{(n)})) d\sigma_{m_1}^{(n)}(\mathbf{X}) d\sigma_{m_2}^{(n)}(\mathbf{Y}) \\
			&= -\frac{1}{n^2} \log \int_{\mathbb{M}_n^{m_1+m_2}} \exp(-n^2 \mathcal{P}_{m_1,m_2+m_3}^{(n)}(\varphi)^{(n)}(\mathbf{Y},\mathbf{W}^{(n)}))\,d\sigma_{m_2}^{(n)}(\mathbf{Y}) \\
			&\leq -\frac{1}{n^2} \log \int_{\mathbb{M}_n^{m_1+m_2}} \exp(-n^2 [\Lambda_{\psi_{k,r}}^{(n)}(\mathbf{Y},\mathbf{W}^{(n)}) + \varepsilon])\,d\sigma_{m_2}^{(n)}(\mathbf{Y}) \\
			&= \mathcal{P}_{m_2,m_3}^{(n)}(\psi_{k,r})(\mathbf{W}^{(n)}) + \varepsilon.
		\end{align*}
		We also have $\mathcal{P}_{m_1+m_2,m_3}^{(n)}(\varphi)(\mathbf{W}^{(n)}) \geq \mathcal{P}_{m_2,m_3}^{(n)}(\psi_{k,r})(\mathbf{W}^{(n)}) - \varepsilon$ by similar reasoning, and so
		\[
		|\mathcal{P}_{m_1+m_2,m_3}^{(n)}(\varphi)(\mathbf{W}^{(n)}) - \mathcal{P}_{m_1,m_2+m_3}^{(n)}(\psi_{k,r})(\mathbf{W}^{(n)})| \leq \varepsilon.
		\]
		Therefore, taking $n \to \cU$ and applying \eqref{eq: convergence of matrix pressure},
		\[
		|\mathcal{P}_{m_1+m_2,m_3}(\varphi)^{\cM}(\mathbf{w}) - \mathcal{P}_{m_2,m_3}(\psi_{k,r})^{\cM}(\mathbf{w})| \leq \varepsilon.
		\]
		Combining this with \eqref{eq: approximation error for psi} and \eqref{eq: P contraction}, we get
		\[
		|\mathcal{P}_{m_1+m_2,m_3}(\varphi)^{\cM}(\mathbf{w}) - \mathcal{P}_{m_2,m_3}(\mathcal{P}_{m_1,m_2+m_3}(\varphi))^{\cM}(\mathbf{w})| \leq 2 \varepsilon.
		\]
		Since $\varepsilon$ was arbitrary, we have \eqref{eq: chain rule for pressure} for $\varphi \in \mathcal{F}_{\chron,m_1+m_2+m_3}^0$.  The formula then extends to $\varphi \in \overline{\mathcal{F}}_{\chron,m_1+m_2+m_3}^0$ by density using \eqref{eq: P contraction}.
	\end{proof}
	
	\begin{proof}[Proof of Theorem \ref{thm: chain rule for entropy}]
		If $(\mathbf{x},\mathbf{y},\mathbf{w})$ does not admit matrix approximations in the sense of \eqref{eq: matrix approximation of triple}, then both sides of \eqref{eq: chain rule for entropy} are $+\infty$.  So assume without loss of generality that such matrix approximations exist. We will use the characterization of entropy in \eqref{eq: final pressure formula for entropy 2}.  Let $\varphi \in \overline{\mathcal{F}}_{\chron,m_1+m_2+m_3}^0$.  Using \eqref{eq: chain rule for pressure},
		\begin{align*}
			\mathcal{P}_{m_1+m_2,m_3}(\varphi)^{\cM}(\mathbf{w}) - \varphi^{\cM}(\mathbf{x},\mathbf{y},\mathbf{w}) &= \mathcal{P}_{m_2,m_3}(\mathcal{P}_{m_1,m_2+m_3}(\varphi))^{\cM}(\mathbf{w}) - \mathcal{P}_{m_1,m_2+m_3}(\varphi)^{\cM}(\mathbf{y},\mathbf{w}) \\
			& \quad + \mathcal{P}_{m_1,m_2+m_3}(\varphi)^{\cM}(\mathbf{y},\mathbf{w}) - \varphi^{\cM}(\mathbf{x},\mathbf{y},\mathbf{w}) \\
			&\leq \tilde{\chi}_{\chron}^{\cU}(\mathbf{y} \mid \mathbf{w}) + \tilde{\chi}_{\chron}^{\cU}(\mathbf{x} \mid \mathbf{y}, \mathbf{w}).
		\end{align*}
		Therefore, taking the supremum over $\varphi$ on the left-hand side,
		\[
		\tilde{\chi}_{\chron}^{\cU}(\mathbf{x}, \mathbf{y} \mid \mathbf{w}) \leq \tilde{\chi}_{\chron}^{\cU}(\mathbf{y} \mid \mathbf{w}) + \tilde{\chi}_{\chron}^{\cU}(\mathbf{x} \mid \mathbf{y}, \mathbf{w}).
		\]
		On the other hand, let $\psi \in \overline{\mathcal{F}}_{\chron,m_2+m_3}^0$ and $\varphi \in \overline{\mathcal{F}}_{\chron,m_1+m_2+m_3}^0$.  Let $\psi' = \psi - \mathcal{P}_{m_1,m_2+m_3}(\varphi)$, and let
		\[
		(\varphi')^{\cM}(\mathbf{x},\mathbf{y},\mathbf{w}) = (\psi')^{\cM}(\mathbf{y},\mathbf{w}) + \varphi^{\cM}(\mathbf{x},\mathbf{y},\mathbf{w}).
		\]
		Note that since $\psi'$ is independent of $\mathbf{x}$, we have
		\[
		\mathcal{P}_{m_1,m_2+m_2}(\varphi') = \psi' + \mathcal{P}_{m_1,m_2+m_3}(\varphi);
		\]
		this follows immediately after expressing both sides using \eqref{eq: definition of P}.  Then
		\begin{align*}
			& \quad \mathcal{P}_{m_2,m_3}(\psi)^{\cM}(\mathbf{w}) - \psi^{\cM}(\mathbf{y},\mathbf{w}) + \mathcal{P}_{m_1,m_2+m_3}(\varphi)^{\cM}(\mathbf{y},\mathbf{w}) - \varphi^{\cM}(\mathbf{x},\mathbf{y},\mathbf{w}) \\
			&= \mathcal{P}_{m_2,m_3}(\psi' + \mathcal{P}_{m_1,m_2+m_3}(\varphi))^{\cM}(\mathbf{w}) - (\psi')^{\cM}(\mathbf{y},\mathbf{w}) - \varphi^{\cM}(\mathbf{x},\mathbf{y},\mathbf{w}) \\
			&= \mathcal{P}_{m_2,m_3}(\mathcal{P}_{m_1,m_2+m_3}(\varphi'))^{\cM}(\mathbf{w}) - (\varphi')^{\cM}(\mathbf{x}, \mathbf{y}, \mathbf{w}) \\
			&= \mathcal{P}_{m_1+m_2,m_3}(\varphi')^{\cM}(\mathbf{w}) - (\varphi')^{\cM}(\mathbf{x}, \mathbf{y}, \mathbf{w}) \\
			&\leq \mathbf{\chi}_{\chron}^{\cU}(\mathbf{x}, \mathbf{y} \mid \mathbf{w}).
		\end{align*}
		Now taking the supremum over $\varphi$ and $\psi$ in the first line, we obtain
		\[
		\tilde{\chi}_{\chron}^{\cU}(\mathbf{y} \mid \mathbf{w}) + \tilde{\chi}_{\chron}^{\cU}(\mathbf{x} \mid \mathbf{y}, \mathbf{w}) \leq \mathbf{\chi}_{\chron}^{\cU}(\mathbf{x}, \mathbf{y} \mid \mathbf{w}),
		\]
		which completes the proof.
	\end{proof}
	
	\subsection{Invariance in the second coordinate} \label{subsec: invariance}
	
	In this section, we consider an invariance property for $\chi^{\cU}(\mathbf{x} \mid \mathbf{y})$ with respect to $\mathbf{y}$.  We show that $\mathbf{y}$ and $\mathbf{w}$ generate the same chronological algebraic closure, then $\chi^{\cU}(\mathbf{x} \mid \mathbf{y}) = \chi^{\cU}(\mathbf{x} \mid \mathbf{w})$.  Here the \emph{chronological algebraic closure} is a version of the algebraic closure in model theory \cite[\S 10]{BY2012}.  We frame the definition here in a self-contained way.
	
	\begin{definition} \label{def: acl}
		Let $\cM$ be an $\cL_{\proc}$-structure satisfying $\rT_{\stat}$, let $m \in \N \cup \{\infty\}$ and $z \in M_0$.  We say that $z$ is in the \emph{chronological algebraic closure of} $\mathbf{y} \in M_0^m$, or $z \in \acl_{\chron}^{\cM}(\mathbf{y})$, if the following conditions hold:
		\begin{enumerate}[(1)]
			\item The set $K = \{ z' \in M_0: \tp_{\chron}^{\cM}(z',\mathbf{y}) = \tp_{\chron}^{\cM}(z,\mathbf{y}) \}$ is compact.
			\item There exists a chronologically definable predicate $\varphi$ with respect to $\rT_{\stat}$ such that
			\[
			d(z',K) = \psi^{\cM}(z,\mathbf{y}) \text{ for } z \in M_0.
			\]
		\end{enumerate}
	\end{definition}
	
	We remark that since every $\cL_{\tr}$ formula evaluated in $M_0$ is also a chronological formula, the chronological algebraic closure of $\mathbf{y}$ contains the algebraic closure with respect to $\cL_{\tr}$ and $\rT_{\tr}$ (see \cite[\S 3]{JekelTypeCoupling}), and hence in particular contains the von Neumann algebra generated by $\mathbf{y}$.  Our goal is to prove the following monotonicity result.  Note that in the case that $\mathbf{y}_0$ and $\mathbf{w}_0$ generate the same chronological algebraic closure, then we have inequalities in both directions, hence equality.
	
	\begin{proposition} \label{prop: entropy and acl}
		Let $\cM$ be an $\cL_{\proc}$-structure satisfying $\rT_{\stat}$, let $m \in \N$ and $m_1, m_2 \in \N \cup \{\infty\}$.  Let $\mathbf{x}_0 \in M_0^m$ and $\mathbf{y}_0 \in M_0^{m_1}$ and $\mathbf{w}_0 \in M_0^{m_2}$.  Suppose that for each $j \in \N$, $y_j' \in \acl_{\chron}^{\cM}(\mathbf{y})$.  Then
		\[
		\tilde{\chi}_{\chron}^{\cU}(\mathbf{x}_0 \mid \mathbf{w}_0) \geq \tilde{\chi}_{\chron}^{\cU}(\mathbf{x}_0 \mid \mathbf{y}_0).
		\]
	\end{proposition}
	
	We first prepare two technical lemmas for the proof.  The first shows that \emph{tuples} from the algebraic closure also live in a compact set $K$, and that finitely many balls will cover a neighborhood of $K$ defined by some restricted chronological formula.  By expressing this condition in terms of restricted chronological formulas, we make it easier to translate it to a condition on the microstate spaces to prove Proposition \ref{prop: entropy and acl}.
	
	\begin{lemma} \label{lem: closeness to compact set}
		Let $\cM$ be an $\cL_{\proc}$-structure satisfying $\rT_{\stat}$, let $m \in \N$ and $m_1, m_2 \in \N \cup \{\infty\}$.  $\mathbf{y}_0 \in M_0^{m_1}$ and $\mathbf{w}_0 \in M_0^{m_2}$.  Suppose that $w_{0,j} \in \acl_{\chron}^{\cM}(\mathbf{y})$ for each $j \leq m_2$.  Then
		\begin{enumerate}[(1)]
			\item $K = \{\mathbf{w} \in M_0^{m_2}: \tp_{\chron}^{\cU}(\mathbf{y}_0,\mathbf{w}) = \tp_{\chron}^{\cU}(\mathbf{y}_0,\mathbf{w}_0) \}$ is compact with respect to the product topology on $M_0^{m_2}$.
			\item Let $\mathbf{r}_2 \in (0,\infty)^{m_2}$ such that $\mathbf{w}_0 \in D_{0,\mathbf{r}_2}^{\cM}$.  Let $\varepsilon > 0$ and $F \subseteq \{j \leq m_2\}$.  Then there exists $k \in \N$, tuples $\mathbf{w}_1$, \dots, $\mathbf{w}_k \in K$, a restricted chronological formula $\psi$, and $\delta > 0$, such that $\psi \geq 0$ and $\psi^{\cM}(\mathbf{y}_0,\mathbf{w}_0) = 0$ and
			\begin{equation} \label{eq: closeness to compact set}
				\forall \mathbf{w} \in D_{0,\mathbf{r}}^{\cM}, \quad \psi^{\cM}(\mathbf{y}_0,\mathbf{w}) < \delta \implies \min_{i \in [k]} \max_{j \in F} \norm{w_j - w_{i,j}}_2 < \varepsilon,
			\end{equation}
			where $\mathbf{w}_i = (w_{i,j})_{j \leq m_2}$.
		\end{enumerate}
	\end{lemma}
	
	\begin{proof}
		(1) For $j \leq m_2$,
		\[
		K_j = \{w \in M_0: \tp_{\chron}^{\cM}(\mathbf{y}_0,w) = \tp_{\chron}^{\cM}(\mathbf{y}_0,w_{0,j}) \},
		\]
		which is compact by Definition \ref{def: acl}.  Note also that if $w \in K_j$, then $w \in D_{0,r_{2,j}}^{\cM}$ because $w$ has the same chronological type as $w_{0,j}$.  Note that $K \subseteq \prod_{j \leq m_2} K_j$, which is compact by Tychonoff's theorem.  Moreover, $K$ is closed because $K$ is contained in $D_{0,\mathbf{r}_2}^{\cM}$ and each chronological formula is continuous on $D_{0,\mathbf{r}_2}^{\cM}$.  Hence, $K$ is a closed subset of a compact set and so compact.
		
		(2) Since $K$ is compact, there exists $k$ and $\mathbf{w}_1$, \dots, $\mathbf{w}_k \in K$ such that
		\[
		\forall \mathbf{w} \in K, \qquad \min_{i \in [k]} \max_{j \in F} \norm{w_j - w_{i,j}}_2 < \frac{\varepsilon}{2}.
		\]
		Let
		\[
		K' = \left\{ \mathbf{w} \in \prod_{j \leq m_2} K_j: \min_{i \in [k]} \max_{j \in F} \norm{w_j - w_{i,j}}_2 \geq \frac{\varepsilon}{2}  \right\}.
		\]
		For $\varphi \in \mathcal{F}_{\chron,m_1+m_2}^0$ and $\eta > 0$, let
		\[
		O_{\varphi,\eta} = \{ \mathbf{w} \in \prod_{j \leq m_2} K_j: |\varphi^{\cM}(\mathbf{y}_0,\mathbf{w}) - \varphi^{\cM}(\mathbf{y}_0,\mathbf{w}_0)| > \eta \}.
		\]
		Note that if $\mathbf{w} \in \prod_{j \leq m_2} K_j \setminus O_{\varphi,\eta}$ for all $\varphi$ and $\eta$, then $\mathbf{w}$ must be in $K$ since the chronological type is determined by the evaluation on all restricted chronological formulas.  Hence,
		\[
		\bigcup_{\varphi, \eta} O_{\varphi,\eta} = \prod_{j \leq m_2} K_j \setminus K \supseteq K'.
		\]
		Since $K'$ is closed, hence compact, it is contained in a finite union of $O_{\varphi_s,\eta_s}$ for $s = 1$, \dots, $\ell$.  Hence, setting
		\[
		\varphi = \max_{s \in [\ell]} |\varphi_s - \varphi_s^{\cM}(\mathbf{x}_0,\mathbf{y}_0)|, \qquad \eta = \min_{s \in [\ell]},
		\]
		we have
		\begin{align} 
			&\forall \mathbf{w} \in K, & & \varphi^{\cM}(\mathbf{y}_0,\mathbf{w}) = 0, \label{eq: conclusion of first compactness argument} \\
			&\forall \mathbf{w} \in \prod_{j \leq m_2} K_j, & &  \varphi^{\cM}(\mathbf{y}_0,\mathbf{w}) \leq \eta \implies \min_{i \in [k]} \max_{j \in F} \norm{w_j - w_{i,j}}_2 < \frac{\varepsilon}{2}. \label{eq: conclusion of first compactness argument 2}
		\end{align}
		Let $F'$ be the set of $\mathbf{w}$-coordinates which $\varphi$ depends on.  By uniform continuity of the restricted chronological formulas, let $\gamma > 0$ such that
		\begin{equation} \label{eq: compactness continuity of phi}
			\forall \mathbf{w}, \mathbf{w}' \in D_{0,\mathbf{r}_2}^{\cM}, \qquad \max_{j \in F'} \norm{w_j - w_j'} < \gamma \implies |\varphi^{\cM}(\mathbf{y}_0,\mathbf{w}) - \varphi^{\cM}(\mathbf{y}_0,\mathbf{w}')| < \frac{\eta}{2}.
		\end{equation}
		By Definition \ref{def: acl}, let $\psi_j$ be a chronologically definable predicate such that
		\[
		d(w,K_j) = \psi_j(\mathbf{y}_0,w).
		\]
		Recall that we can choose $\psi_j' \in \mathcal{F}_{\chron,m}^0$ to approximate $\psi_j$ within any given error on $D_{0,r_{2,j}}^{\cM}$, and by considering $\max(\psi_j' - a, 0)$ for a sufficiently small constant $a$, we can also arrange that $\psi_j' \geq 0$ and vanishes on $K_j$.  Hence, we can arrange $\delta > 0$ and $\psi_j' \in \mathcal{F}_{\chron,m_1+m_2}^0$ for each $j \in F \cup F'$, such that
		\begin{equation} \label{eq: compactness def of psi j}
			\forall w \in D_{0,r_{2,j}}^{\cM}, \qquad \psi_j^{\cM}(\mathbf{y}_0,w) < \delta \implies d(w, K_j) < \min(\gamma, \varepsilon / 2).
		\end{equation}
		Finally, let
		\begin{equation} \label{eq: compactness def of psi}
			\psi(\mathbf{y},\mathbf{w}) = \max_{j \in F \cup F'} \psi_j(\mathbf{y},w_j) + \varphi(\mathbf{y},\mathbf{w}).
		\end{equation}
		We claim that
		\begin{equation} \label{eq: main conclusion of compactness}
			\forall \mathbf{w} \in D_{0,\mathbf{r}_2}^{\cM}, \qquad \psi^{\cM}(\mathbf{y}_0,\mathbf{w}) < \min(\delta, \eta/2) \implies \min_{i \in [k]} \max_{j \in F} \norm{w_j - w_{i,j}}_2 < \varepsilon.
		\end{equation}
		Indeed, suppose that $\psi^{\cM}(\mathbf{y}_0,\mathbf{w}) < \min(\delta, \eta/2)$.  Then by \eqref{eq: compactness def of psi} and \eqref{eq: compactness def of psi j},
		\[
		\max_{j \in F \cup F'} d(w_j, K_j) < \min(\gamma,\varepsilon/2).
		\]
		Hence, choose $\mathbf{w}' \in \prod_{j \leq m_2} K_j$ such that
		\begin{equation} \label{eq: compactness half epsilon}
			\max_{j \in F \cup F'} \norm{w_j - w_j'}_2 < \min(\gamma,\varepsilon/2).
		\end{equation}
		Thus, by \eqref{eq: compactness half epsilon} and \eqref{eq: compactness continuity of phi},
		\begin{equation} \label{eq: unnamed equation 1}
			|\varphi^{\cM}(\mathbf{y}_0,\mathbf{w}) - \varphi^{\cM}(\mathbf{y}_0,\mathbf{w}')| < \frac{\eta}{2}.
		\end{equation}
		Since $\psi^{\cM}(\mathbf{y}_0,\mathbf{w}) < \min(\delta, \eta/2)$, using \eqref{eq: compactness def of psi},
		\begin{equation} \label{eq: unnamed equation 2}
			|\varphi^{\cM}(\mathbf{y}_0,\mathbf{w}) - \varphi^{\cM}(\mathbf{y}_0,\mathbf{w}_0)| < \frac{\eta}{2}.
		\end{equation}
		Then \eqref{eq: unnamed equation 1} and \eqref{eq: unnamed equation 2} imply
		\begin{equation} \label{eq: unnamed equation 3}
			|\varphi^{\cM}(\mathbf{y}_0,\mathbf{w}') - \varphi^{\cM}(\mathbf{y}_0,\mathbf{w}_0)| < \eta.
		\end{equation}
		Then \eqref{eq: unnamed equation 3} and \eqref{eq: conclusion of first compactness argument 2} imply that
		\begin{equation} \label{eq: unnamed equation 4}
			\min_{i \in [k]} \max_{j \in F} \norm{w_j' - w_{i,j}}_2 < \frac{\varepsilon}{2}.
		\end{equation}
		Combining \eqref{eq: unnamed equation 4} with \eqref{eq: compactness half epsilon} and the triangle inequality yields
		\[
		\min_{i \in [k]} \max_{j \in F} \norm{w_j' - w_{i,j}}_2 < \varepsilon,
		\]
		which completes the proof of \eqref{eq: main conclusion of compactness}.  It is also straightforward to check from the definitions that $\psi \geq 0$ and $\psi^{\cM}(\mathbf{y}_0,\mathbf{w}) = 0$ for $\mathbf{w} \in K$.
	\end{proof}
	
	The next technical lemma shows that deterministic matrix models for $\mathbf{y}_0$ can be completed to deterministic matrix models for $(\mathbf{y}_0,\mathbf{w}_0)$, which can be proved using a diagonalization argument.  Later, Theorem \ref{thm: optimal coupling lifts} will prove a much stronger result about completing matrix models of optimal couplings.  Hence, for the sake of time, we will skip giving a self-contained proof for Lemma \ref{lem: lifting of deterministic matrices} and simply invoke Theorem \ref{thm: optimal coupling lifts}.  Needless to say, Theorem \ref{thm: optimal coupling lifts} does not depend on the results of this section in any way.
	
	\begin{lemma} \label{lem: lifting of deterministic matrices}
		Let $\cM$ be an $\cL_{\proc}$-structure satisfying $\rT_{\stat}$, let $m_1, m_2 \in \N \cup \{\infty\}$, let $\mathbf{r}_1 \in (0,\infty)^{m_1}$ and $\mathbf{r}_2 \in (0,\infty)^{m_2}$, and let $\mathbf{y}_0 \in D_{0,\mathbf{r}_1}^{\cM}$ and $\mathbf{w}_0 \in D_{0,\mathbf{r}_2}^{\cM}$.  Let $\mathbf{Y}^{(n)} \in D_{\mathbf{r}_1}^{\mathbb{M}_n}$ be matrix tuples satisfying
		\[
		\forall \varphi \in \mathcal{F}_{\chron,m_1}^0, \quad \Lambda_{\varphi}^{(n)}(\mathbf{Y}^{(n)}) = \varphi^{\cM}(\mathbf{y}_0).
		\]
		Then there exist matrix tuples $\mathbf{W}^{(n)} \in D_{\mathbf{r}_2}^{\mathbb{M}_n}$ such that
		\begin{equation} \label{eq: yet another convergence of matrix models}
			\forall \varphi \in \mathcal{F}_{\chron,m_1+m_2}^0, \quad \Lambda_{\varphi}^{(n)}(\mathbf{Y}^{(n)},\mathbf{W}^{(n)}) = \varphi^{\cM}(\mathbf{y}_0,\mathbf{z}_0).
		\end{equation}
	\end{lemma}
	
	\begin{proof}
		First, consider the case that $m_1 < \infty$.  We apply Theorem \ref{thm: optimal coupling lifts} with $\mathbf{x}_0 = \mathbf{0}$ and $\mathbf{x}_1 = \mathbf{z}_0$.  Indeed, it is easy to see that $(\mathbf{0}, \mathbf{w}_0)$ is an optimal coupling in the sense of Definition \ref{def: Wasserstein distance}.  Thus, Theorem \ref{thm: optimal coupling lifts} yields matrix models $\mathbf{W}^{(n)}$ for $\mathbf{w}_0$, which are also deterministic since they are given as a Borel function of $\mathbf{0}$.
		
		Now consider the case where $m_1 = \infty$.  We construct the matrix approximation $W_j^{(n)}$ by induction on $j$ such that $(\mathbf{Y}^{(n)},W_1^{(n)},\dots,W_j^{(n)})$ is a matrix model for $(\mathbf{y}_0,w_{0,1},\dots,w_{0,j})$ in the sense of \eqref{eq: yet another convergence of matrix models}.  The base case $j = 0$ is trivial, and the induction step follows from the $m_1 < \infty$ case above.  Then the tuple $(\mathbf{Y}^{(n)},\mathbf{Z}^{(n)})$ satisfies \eqref{eq: yet another convergence of matrix models} because each restricted chronological formula only uses finitely many coordinates.
	\end{proof}

	\begin{proof}[Proof of Proposition \ref{prop: entropy and acl}]
		Fix a restricted chronological formula $\varphi(\mathbf{x},\mathbf{w})$ and $\varepsilon > 0$.  Recall that $\varphi$ only depends on finite set $F$ of the coordinates of $\mathbf{w}$, and it is Lipschitz with respect to $\norm{\mathbf{w}}_1$ in these coordinates, and so also Lipschitz with respect to $\sum_{j \in F} \norm{w_j}_2$ for some Lipschitz constant $L$ (and the same holds for the matrix approximation $\Lambda_\varphi^{(n)}$ from Proposition \ref{prop: approx by pointwise function}).  Let $\mathbf{w}_1$, \dots, $\mathbf{w}_k$ and $\psi$ and $\delta$ be as in Lemma \ref{lem: closeness to compact set}, and consider the restricted chronological formula
		\begin{equation} \label{eq: definition of omega y w}
			\omega(\mathbf{y}',\mathbf{w}_1',\dots,\mathbf{w}_k') = \max\left( \delta/2 - \psi^{\cM}(\mathbf{y},\mathbf{w}), \, \min_{i \in [k]} \max_{j \in F} \norm{w_j - w_{i,j}}_2 - 2 \varepsilon \right).
		\end{equation}
		Note that \eqref{eq: closeness to compact set} implies that
		\[
		\omega^{\cM}(\mathbf{y}_0,\mathbf{w}_1,\dots,\mathbf{w}_k) \leq \max(-\delta/2, - \varepsilon) < 0.
		\]
		Define also the restricted chronological formula
		\begin{equation} \label{eq: definition of phi prime}
			\varphi'(\mathbf{x},\mathbf{y}) = \inf_{\mathbf{w} \in D_{0,\mathbf{r}_2}} \left[ |\varphi(\mathbf{x},\mathbf{w}) - \varphi^{\cM}(\mathbf{x}_0,\mathbf{w}_0)|  + \psi(\mathbf{y},\mathbf{w}) \right]
		\end{equation}
		which makes sense because $\varphi$ again only depends on finitely many coordinates of $\mathbf{w}$.  Note also that
		\[
		(\varphi')^{\cM}(\mathbf{x}_0,\mathbf{y}_0) = 0
		\]
		since $\mathbf{w}_0$ is a candidate in the infimum.
		
		Using Lemma \ref{lem: lifting of deterministic matrices}, let $\mathbf{W}_1^{(n)}$, \dots, $\mathbf{W}_k^{(n)}$ be deterministic matrix tuples in $D_{\mathbf{r}_2}^{\mathbb{M}_n}$ satisfying \eqref{eq: yet another convergence of matrix models}.  In particular, for sufficiently large $n$ with respect to the $\cU$, we have
		\[
		\Lambda_\omega^{(n)}(\mathbf{Y}^{(n)}, \mathbf{W}_1^{(n)}, \dots, \mathbf{W}_k^{(n)}) < 0.
		\]
		Recalling the definition of $\omega$ \eqref{eq: definition of omega y w} and assuming that $\Lambda_\omega^{(n)}$ is constructed from $\Lambda_\varphi^{(n)}$ and $\Lambda_\psi^{(n)}$ in the same way as $\omega$ is constructed from $\varphi$ and $\psi$ in Proposition \ref{prop: approx by pointwise function}, this means that
		\begin{equation} \label{eq: matrix implication for acl}
			\forall \mathbf{W} \in D_{\mathbf{r}_2}^{\mathbb{M}_n}, \quad \Lambda_\psi^{(n)}(\mathbf{Y}^{(n)},\mathbf{W}) < \frac{\delta}{2} \implies \min_{i \in [k]} \max_{j \in F}  \norm{W_j - W_{i,j}^{(n)}}_2 \leq 2 \varepsilon.
		\end{equation}
		Now we claim that
		\begin{equation} \label{eq: microstate inclusion for acl}
			\Gamma^{(n)}(\mathbf{x}_0 \mid \mathbf{W}^{(n)} \rightsquigarrow \mathbf{w}_0; \{\varphi'\}, \min(\delta/2,\varepsilon)) \subseteq \bigcup_{i=1}^k \Gamma^{(n)}(\mathbf{x}_0 \mid \mathbf{W}_i^{(n)} \rightsquigarrow \mathbf{w}_0; \{\varphi\}, (1 + 2L)\varepsilon).
		\end{equation}
		Indeed, suppose that $\mathbf{X}$ is in the microstate space on the left-hand side of \eqref{eq: microstate inclusion for acl}.  Then since $(\varphi')^{\cM}(\mathbf{x}_0,\mathbf{y}_0) = 0$, we have
		\[
		\Lambda_{\varphi'}^{(n)}(\mathbf{X},\mathbf{Y}^{(n)}) < \min(\delta/2, \varepsilon).
		\]
		By construction of $\varphi'$ \eqref{eq: definition of phi prime} and the corresponding construction of $\Lambda_{\varphi'}^{(n)}$, this implies that there exists $\mathbf{W} \in D_{\mathbf{r}_2}^{\mathbb{M}_n}$ such that
		\[
		|\Lambda_\varphi^{(n)}(\mathbf{X},\mathbf{W}) - \varphi^{\cM}(\mathbf{x}_0,\mathbf{w}_0)| + \Lambda_\psi^{(n)}(\mathbf{Y}^{(n)},\mathbf{W}) < \min(\delta/2,\varepsilon),
		\]
		and therefore
		\begin{align}
			|\Lambda_\varphi^{(n)}(\mathbf{X},\mathbf{W}) - \varphi^{\cM}(\mathbf{x}_0,\mathbf{w}_0)| &\leq \varepsilon \label{eq: nameless acl equation 1} \\
			\Lambda_\psi^{(n)}(\mathbf{Y}^{(n)},\mathbf{W}) &\leq \frac{\delta}{2} \label{eq: nameless acl equation 2}.
		\end{align}
		Now \eqref{eq: nameless acl equation 2} and \eqref{eq: matrix implication for acl} imply that there exists $i \in [k]$ such that
		\begin{equation} \label{eq: nameless acl equation 3}
			\max_{j \in F} \norm{W_j - W_{i,j}^{(n)}}_2 \leq 2 \varepsilon.
		\end{equation}
		Then using the $L$-Lipschitz property of $\Lambda_\varphi^{(n)}$ and \eqref{eq: nameless acl equation 1} with the triangle inequality, we get
		\begin{equation} \label{eq: nameless acl equation 4}
			|\Lambda_\varphi^{(n)}(\mathbf{X},\mathbf{W}_i^{(n)}) - \varphi^{\cM}(\mathbf{x}_0,\mathbf{w}_0)| \leq \varepsilon + L(2 \varepsilon),
		\end{equation}
		and therefore $\mathbf{X}$ is an element in the right-hand side of \eqref{eq: microstate inclusion for acl} as desired.
		
		Next, from \eqref{eq: microstate inclusion for acl}, we deduce that
		\begin{align*}
			\sigma^{(n)}(\Gamma^{(n)}(\mathbf{x}_0 \mid \mathbf{Y}^{(n)} \rightsquigarrow \mathbf{y}_0; \{\varphi'\}, \min(\delta/2,\varepsilon))) &\leq \sum_{i=1}^k \sigma^{(n)}(\Gamma^{(n)}(\mathbf{x}_0 \mid \mathbf{W}_i^{(n)} \rightsquigarrow \mathbf{w}_0; \{\varphi\}, (1 + 2L)\varepsilon)) \\
			&\leq k \max_{i \in [k]} \sigma^{(n)}(\Gamma^{(n)}(\mathbf{x}_0 \mid \mathbf{W}_i^{(n)} \rightsquigarrow \mathbf{w}_0; \{\varphi\}, (1 + 2L)\varepsilon)).
		\end{align*}
		Using \eqref{eq: upper bound for pressure using measure},
		\begin{align*}
			-\frac{1}{n^2} \log & \sigma^{(n)}(\Gamma^{(n)}(\mathbf{x}_0 \mid \mathbf{Y}^{(n)} \rightsquigarrow \mathbf{y}_0; \{\varphi'\}, \min(\delta/2,\varepsilon))) \\
			&\geq \min_{i \in [k]} -\frac{1}{n^2} \sigma^{(n)}(\Gamma^{(n)}(\mathbf{x}_0 \mid \mathbf{W}_i^{(n)} \rightsquigarrow \mathbf{w}_0; \{\varphi\}, (1 + 2L)\varepsilon)) - \frac{1}{n^2} \log k \\
			&\geq \min_{i \in k} \mathcal{P}_{m,m_2}^{(n)}(\varphi)(\mathbf{W}_i^{(n)}) - \varphi^{\cM}(\mathbf{x}_0,\mathbf{w}_0) - (1 + 2L) \varepsilon - \frac{1}{n^2} \log k.
		\end{align*}
		Therefore, by applying Definitions \ref{def: chronological entropy} and \ref{def: chronological entropy final} to $(\mathbf{x}_0,\mathbf{y}_0)$ and applying \eqref{eq: convergence of matrix pressure} to $(\mathbf{x}_0,\mathbf{w}_0)$ for each matrix approximation $\mathbf{W}_i^{(n)}$,
		\begin{align*}
			\chi_{\chron}^{\cU}(\mathbf{x}_0 \mid \mathbf{y}_0) &\geq \lim_{n \to \cU} -\frac{1}{n^2} \log \sigma^{(n)}(\Gamma^{(n)}(\mathbf{x}_0 \mid \mathbf{Y}^{(n)} \rightsquigarrow \mathbf{y}_0; \{\varphi'\}, \min(\delta/2,\varepsilon))) \\
			&\geq \lim_{n \to \cU} \left[ \min_{i \in k} \mathcal{P}_{m,m_2}^{(n)}(\varphi)(\mathbf{W}_i^{(n)}) - \varphi^{\cM}(\mathbf{x}_0,\mathbf{w}_0) - (1 + 2L) \varepsilon - \frac{1}{n^2} \log k \right] \\
			&= \mathcal{P}_{m,m_2}(\varphi)(\mathbf{w}_0) - \varphi^{\cM}(\mathbf{x}_0,\mathbf{w}_0) - (1 + 2L) \varepsilon.
		\end{align*}
		Hence, for every $\varphi$ and $\varepsilon$, we have
		\[
		\mathcal{P}_{m,m_2}(\varphi)(\mathbf{w}_0) - \varphi^{\cM}(\mathbf{x}_0,\mathbf{w}_0) \leq \chi_{\chron}^{\cU}(\mathbf{x}_0 \mid \mathbf{y}_0) + (1 + 2L) \varepsilon.
		\]
		Therefore, by \eqref{eq: final pressure formula for entropy}, we get $\chi_{\chron}^{\cU}(\mathbf{x}_0,\mathbf{w}_0) \leq \chi_{\chron}^{\cU}(\mathbf{x}_0,\mathbf{y}_0)$ as desired.
	\end{proof}

	\subsection{Large deviations} \label{subsec: LDP}
	
	We conclude the section with a discussion of large deviations theory.  Here it is convenient to work with a certain compactification of $\mathbb{S}_{\chron,m}(\rT_{\stat})$.  For similar compactification constructions in the context of free entropy, see \cite{BCG2003} and \cite[\S 7]{JLS2022}.  Note $\mathbb{S}_{\chron,m}(\rT_{\stat})$ is not itself compact without imposing a certain bound on the operator norm, and the large deviations principle can assign weight to regions that do not necessarily satisfy a fixed operator norm bound.  Recall we defined $\overline{\cF}_{\chron,m}^0$ as the completion of $\cF_{\chron,m}^0$ with respect to the uniform norm associated to $\rT_{\stat}$.  It is not hard to see that $\overline{\cF}_{\chron,m}^0$ is a commutative real $\mathrm{C}^*$-algebra, and therefore $\overline{\cF}_{\chron,m}^0$ is isomorphic to $C(\Sigma_{\chron,m};\R)$, where $\Sigma_{\chron,m}$ is its character space or Gelfand spectrum.  The type space $\mathbb{S}_{\chron,m}(\rT_{\stat})$ naturally embeds into $\Sigma_{\chron,m}$ by viewing each type as a character on $\overline{\cF}_{\chron,m}^0$ through the natural dual pairing.
	
	We specialize the results of \S \ref{subsec: stochastic variational} to the case $m' = 0$.  Thus, for each $\varphi \in \overline{\mathcal{F}}_{\chron,m}^0$, $\mathcal{P}(\varphi)$ is a definable predicate with no free variables.  Thus, the value of $\mathcal{P}(\varphi)$ only depends on the chronological theory of the ambient $\cL_{\proc}$-structure. The definition of entropy $\chi_{\chron}^{\cU}$ can naturally be extended to $\Sigma_{\chron,m}$ by letting
	\[
	\tilde{\chi}_{\chron}^{\cU}(\mu) = \sup_{\varphi \in \overline{\mathcal{F}}_{\chron,m}^0} \left[ \mathcal{P}^{\cU}(\varphi) - (\mu,\varphi) \right],
	\]
	where $\mathcal{P}(\varphi)^{\cU}$ denotes the value of $\mathcal{P}(\varphi)$ in the matrix ultraproduct quotients $\cQ$ associated to the given ultrafilter $\cU$.  Note that the same reasoning as Corollary \ref{cor: semicontinuity of entropy} applies to show that $\tilde{\chi}_{\chron}^{\cU}$ is lower semicontinuous on $\Sigma_{\chron,m}$.
	
	The statement \eqref{eq: convergence of matrix pressure} for the convergence of matrix pressure thus gives an analogous statement to Varadhan's lemma with $\tilde{\chi}_{\chron}^{\cU}(\mu)$ as the rate function, with two caveats.  First, the limit is along the ultrafilter $\cU$ rather than as $n \to \infty$, and second, it requires a choice of the finite-dimensional approximations $\Lambda_\varphi^{(n)}$.  The second issue makes it harder to state an analogous large deviations principle; indeed, the evaluation of $\Lambda_\varphi^{(n)}$ on a matrix tuple does \emph{not} produce a character on $\overline{\mathcal{F}}_{\chron,m}^0$ exactly, only approximately.  We will not attempt to address this issue here.
	
	Rather, let us focus on the quantifier-free setting advertised in Corollary \ref{introcor: LDP}.  Let $\mathcal{F}_{\qf,m}^0$ be the set of restricted quantifier-free formulas in $m$ variables; that is, $\mathcal{F}_{\qf,m}$ is generated by Definition \ref{def: restricted chronological formulas} without using the suprema and infima.  Hence, elements in $\mathcal{F}_{\qf,m}^0$ all take the form
	\[
	f(\tr(p_1(\mathbf{x})),\dots,\tr(p_k(\mathbf{x}))),
	\]
	where $f: \C^k \to \R$ is continuous and $p_1$, \dots, $p_k$ are $*$-polynomials in $(\re(x_j) + i)^{-1}$ and $(\im(x_j) + i)^{-1}$.  In that case that $f$ is polynomial, the resulting formula is a \emph{trace polynomial} in these resolvents (see \cite{DHK2013}).  In general, these formulas are the cylindrical test functions from \cite[\S B.4]{GJNPviscositysolution}.  Unlike the case of chronological formulas, there is a canonical choice of $\Lambda_\varphi^{(n)}$ as simply $\varphi^{\mathbb{M}_n}$.  (And more generally, formulas in $\cL_{\tr}$ that do not involve the Brownian motion have a canonical evaluation in $\mathbb{M}_n$.)  Thus, a matrix tuple $\mathbf{X}$ does produce a character on $\cF_{\qf,m}^0$, which we will denote by $\tp_{\qf}^{\M_n}(\mathbf{X})$.
	
	Let $\overline{\cF}_{\qf,m}^0$ be the completion of $\cF_{\qf,m}^0$ with respect to the uniform norm, and let $\Sigma_{\qf,m}$ be its Gelfand spectrum.  We then have the microstate spaces
	\[
	\Gamma^{(n)}(\mu; \Phi, \varepsilon) = \{ \mathbf{X} \in \mathbb{M}_n^m: \max_{\varphi \in \Phi} |\varphi^{\mathbb{M}_n}(\mathbf{X}) - (\mu,\varphi)| < \varepsilon\} 
	\]
	for finite $\Phi \subseteq \overline{\cF}_{\qf,m}^0$, and one can define $\tilde{\chi}_{\qf}^{\cU}(\mu)$ analogously to Definition \ref{def: chronological entropy} using these microstate spaces.  This $\tilde{\chi}_{\qf}^{\cU}(\mu)$ is exactly the free entropy with respect to Gaussian measure studied in \cite{BCG2003}.  The same proof as Lemma \ref{lem: microstates versus pressure} applies, and together with \eqref{eq: convergence of matrix pressure}, this yields
	\begin{equation} \label{eq: qf variational formula for entropy}
		\tilde{\chi}_{\qf}^{\cU}(\mu) = \sup_{\varphi \in \mathcal{F}_{\qf,m}^0} \left[ \mathcal{P}^{\cU}(\varphi) - (\mu,\varphi) \right].
	\end{equation}
	Moreover, standard large deviations techniques show that for $\mathcal{E} \subseteq \Sigma_{\qf,m}$, we have
	\begin{equation} \label{eq: LDP for law}
		\inf_{\mu \in \overline{\mathcal{E}}} \tilde{\chi}_{\qf}^{\cU}(\mu) \leq \lim_{n \to \cU} -\frac{1}{n^2} \log \sigma^{(n)}(\{\mathbf{X}: \tp_{\qf}^{\M_n}(\mathbf{X}) \in \mathcal{E}) \leq \inf_{\mu \in \mathcal{E}^\circ} \tilde{\chi}_{\qf}^{\cU}(\mu).
	\end{equation}
	For instance, the proof of Bryc's theorem (the converse of Varadhan's lemma; see \cite{Rassoul2015course}) works for ultralimits as well since it is simply based on comparing the measures of sets and the values of pressure functional as in Lemma \ref{lem: microstates versus pressure} together with topological arguments using compactness and lower semicontinuity.  See also \cite[Proposition 3.7]{JekelModelEntropy} for a similar variational principle based on a compactness argument.
	
	In order to obtain a true large deviations principle for the space $\Sigma_{\qf,m}$ and the Gaussian measures $\sigma^{(n)}$, we would need to understand whether the value of $\mathcal{P}^{\cU}(\varphi)$ is independent of the choice of ultrafilter.  Although this may seem like merely restating the problem in other terms, Proposition \ref{prop: convergence of type} gives another way of evaluating $\mathcal{P}^{\cU}(\varphi)$ in terms of a noncommutative filtration, rather than merely an ultralimit of classical quantities.  Namely,
	\[
	\mathcal{P}^{\cU}(\varphi) = \mathcal{P}(\varphi)^{\cQ} = \inf_{\mathbf{\alpha}} \left[ \varphi^{\cM}\left((z_{0,t,j})_{j=1}^m + \int_0^1 \mathbf{\alpha}_t\,dt, \mathbf{y} \right) + \frac{1}{2} \int_0^1 \norm{\alpha_t}_2^2 \,dt \right],
	\]
	where $\cQ$ is a quotient obtained from the random matrix ultraproduct with respect to $\cU$ \eqref{eq: random matrix ultraproduct}.  One can than hope to apply noncommutative stochastic analytical tools to the filtration $\cQ$.  Although it is no doubt a hard problem to determine whether $\mathcal{P}^{\cU}(\varphi)$ is independent of the ultrafilter $\cU$, it may very well be easier when $\varphi$ is quantifier-free and so $\mathcal{P}^{\cU}(\varphi)$ only involves inf-quantifiers (see also Remark \ref{rem: existential embedding}).  Let us conclude by finishing the proof of Corollary \ref{introcor: LDP}.
	
	\begin{proof}[Proof of Corollary {\ref{introcor: LDP}}]
		We already explained how the first statement \eqref{eq: qf variational formula for entropy} follows by the analog of Lemma \ref{lem: microstates versus pressure}, together with \eqref{eq: convergence of matrix pressure}.
		
		By Varadhan's lemma and Bryc's theorem, the Gaussian measures $\sigma^{(n)}$ and the space $\Sigma_{\qf,m}$ satisfy a large deviation principle with a good rate function if and only if $\lim_{n \to \infty} \mathcal{P}^{(n)}(\varphi)$ exists for all $\varphi \in \overline{\mathcal{F}}_{\qf,m}^0$.  This is in turn equivalent to the limit as $n \to \cU$ being the same for all non-principal ultrafilters on $\N$.
	\end{proof}
	
	\section{Change of variables for entropy} \label{sec: change of variables}
	
	In this section, we examine the change of variables formula for the free entropy of chronological types.  While we studied entropy $\tilde{\chi}_{\chron}^{\cU}(\mathbf{x} \mid \mathbf{y})$ with respect to the Gaussian measure in the previous section, we introduce the corresponding entropy $\chi_{\chron}^{\cU}(\mathbf{x} \mid \mathbf{y})$ with respect to Lebesgue measure here, including showing that in Definition \ref{def: chronological entropy} one gets the same result by restricting to microstate spaces with an operator-norm bound.
	
	Motivated by the application of geodesic concavity, we focus on extending the change of variables formula to gradients of functions of significantly less smoothness than has been considered in past work.  Indeed, relating $\chi_{\chron}^{\cU}(f(\mathbf{x}) \mid \mathbf{y})$ to $\chi_{\chron}^{\cU}(\mathbf{x} \mid \mathbf{y})$ is clearly possible when $f$ is a power series or a noncommutative smooth function along similar lines as \cite[Proposition 3.5]{VoiculescuFE2}.  However, in studying transport along a Wasserstein geodesic, it is necessary to study the derivative of $\chi_{\chron}^{\cU}((\id + t \nabla \varphi)\mathbf{x} \mid \mathbf{y})$ when $\varphi$ is merely a chronologically definable predicate, which we may assume to be semiconvex and semiconcave.  This requires defining a Laplacian $\Delta \varphi$ in much more general setting.  As explained in the introduction, since $\varphi$ can now be a formula involving suprema and infima, there is no way to give an explicit formula for the Laplacian as one can do in the setting of trace polynomials.  Rather the Laplacian has to be defined abstractly by the action of the process $z_{s,t}$ in our $\cL_{\proc}$-structure.
	
	\subsection{Entropy in terms of Lebesgue measure} \label{subsec: Leb entropy}
	
	\begin{definition}
		For $m \in \N$ and $m' \in \N \cup \{0,\infty\}$.  Let $\cM$ satisfy $\rT_{\stat}$.  Let $(\mathbf{x},\mathbf{y})$ be an $m + m'$-tuple from $M_0$.  Then define
		\[
		\chi_{\chron}^{\cU}(\mathbf{x} \mid \mathbf{y}) = m \log(2\pi) + \frac{1}{2} \norm{\mathbf{x}}_2^2 - \tilde{\chi}_{\chron}^{\cU}(\mathbf{x} \mid \mathbf{y}).
		\]
	\end{definition}
	
	We claim that $\chi_{\chron}^{\cU}$ can be evaluated in terms of the Lebesgue measure for microstate spaces, just as $\tilde{\chi}_{\chron}^{\cU}$ was evaluated using Gaussian measure.  In the Lebesgue measure case, we will also impose an operator-norm cutoff.  This result is a variant of standard results that were proved for the original free entropy \cite[Proposition 2.4]{VoiculescuFE2}, \cite[Lemma 7.1]{BCG2003}, \cite{BelBer2003property}.
	
	\begin{proposition} \label{prop: entropy with norm cutoff}
		Let $m \in \N$ and $m' \in \N \cup \{0,\infty\}$.  Let $R \geq \sqrt{2}$.  Let $\mathbf{x}$ be a self-adjoint $m$-tuple with $\norm{x_j} < R$ and $\mathbf{y}$ a self-adjoint $m'$-tuple.  Let $\mathbf{Y}^{(n)}$ be a tuple of deterministic matrices with $\sup_n \norm{Y_j^{(n)}} < \infty$ and $\lim_{n \to \cU} \Lambda_{\varphi}^{(n)}(\mathbf{Y}^{(n)}) = \varphi^{\cM}(\mathbf{y})$ for restricted chronological formulas $\varphi$.  Let
		\[
		\Gamma_R^{(n)}(\mathbf{x} \mid \mathbf{Y}^{(n)} \rightsquigarrow \mathbf{y}; \Phi, \varepsilon) := \left\{ \mathbf{X} \in (\mathbb{M}_n)_{\sa}^m: \norm{X_j} \leq R, \quad \max_{\varphi \in \Phi} |\Lambda_{\varphi}^{(n)}(\mathbf{X},\mathbf{Y}^{(n)}) - \varphi^{\cM}(\mathbf{x},\mathbf{y})| < \varepsilon \right\}.
		\]
		Then
		\begin{equation} \label{eq: Gaussian entropy with bounded microstates}
			\tilde{\chi}_{\chron}^{\cU}(\mathbf{x} \mid \mathbf{Y}^{(n)} \rightsquigarrow \mathbf{y}) = \sup_{\Phi,\varepsilon} \lim_{n \to \cU} -\frac{1}{n^2} \log \sigma^{(n)}(\Gamma_{4R}^{(n)}(\mathbf{x} \mid \mathbf{Y}^{(n)} \rightsquigarrow \mathbf{y}; \Phi, \varepsilon)).
		\end{equation}
	\end{proposition}
	
	The strategy of the proof is the same as \cite{BelBer2003property}; we map the unbounded microstate space into a microstate space with an operator norm cutoff using functional calculus.  The main differences are that we use the Gaussian measure and that our functions for our microstate spaces do not include the traces of polynomials, and only functions that are Lipschitz in $\norm{\cdot}_1$.  In particular, the proof of \cite[Lemma 2.3]{BelBer2003property} for the inclusion of microstate spaces will be replaced by an argument more suited to our test functions.  Moreover, we cannot estimate the Jacobian of the transformation in terms of the $2p$ moment as in \cite{BelBer2003property}.  We therefore use the following estimate for the Jacobian of the transformation.
	
	\begin{lemma} \label{lem: smooth cutoff}
		Let $R > 0$.  Let $g: \R \to \R$ be given by
		\[
		g(t) = \begin{cases} -2R - R^2/t, & t \in (-\infty,-R] \\ t, & t \in [-R,R] \\ 2R - R^2/t, & t \in [R,\infty), \end{cases}
		\]
		Let $G: \mathbb{M}_n^m \to \mathbb{M}_n^m$ be given by
		\[
		G(\mathbf{X}) = (g(\re X_j) + i g(\im X_j))_{j=1}^m.
		\]
		Moreover, let
		\begin{equation} \label{eq: delta R}
			\delta_R(\mathbf{X}) = \sum_{j=1}^m [2 -  \tr_n(\mathbbm{1}_{[-R,R]}(\re(X_j))) - \tr_n(\mathbbm{1}_{[-R,R]}(\re(X_j)))].
		\end{equation}
		Then
		\begin{equation} \label{eq: smooth cutoff Jacobian estimate}
			0 \geq \frac{1}{n^2} \log \det |DG(\mathbf{X})| \geq - \frac{2}{R} \delta_R(\mathbf{X})^{1/2} \left( \norm{\mathbf{X}}_2^2 - \norm{G(\mathbf{X})}_2^2 \right)^{1/2}.
		\end{equation}
	\end{lemma}
	
	\begin{proof}
		First, let $G_0$ be the map $(\mathbb{M}_n)_{\sa} \to (\mathbb{M}_n)_{\sa}$ given by application of $g$ through functional calculus. As in the proof of \cite[Proposition 2.4]{VoiculescuFE2} and \cite{BelBer2003property}, the Jacobian is expressed in terms of the eigenvalues $(\lambda_1,\dots,\lambda_n)$ of $X$ as
		\[
		\det |DG_0(X)| = \prod_{i \neq j} \frac{g(\lambda_i) - g(\lambda_j)}{\lambda_i - \lambda_j} \prod_i g'(\lambda_i).
		\]
		Note that for the diagonal terms
		\[
		g'(t) = \frac{R^2}{\max(|t|,R)^2}.
		\]
		We claim that
		\begin{equation} \label{eq: lower bound for difference quotient}
			\frac{g(s) - g(t)}{s-t} \geq \frac{R^2}{\max(|s|,R) \max(|t|,R)} = g'(s)^{1/2} g'(t)^{1/2}.
		\end{equation}
		We remark that \cite{BelBer2003property} show easily a lower bound of $\min(g'(s),g'(t)) \geq g'(s) g'(t)$ from the fact that $g'$ has no local minimum, but the sharper bound \eqref{eq: lower bound for difference quotient} can be deduced by direct casework as follows:
		\begin{enumerate}[(1)]
			\item If $s, t \in [-R,R]$, both sides of \eqref{eq: lower bound for difference quotient} are $1$.
			\item If $s, t \in [R,\infty)$.  Then
			\[
			\frac{g(s) - g(t)}{s-t} = \frac{R/t - R/s}{s - t} = \frac{R^2}{st} = \frac{R^2}{\max(|s|,R) \max(|t|,R)}.
			\]
			\item Suppose $s \in (-R,R)$ and $t \in [R,\infty)$.  Then
			\begin{align*}
				\frac{g(t) - g(s)}{t - s} &= \frac{R-s}{t-s} \frac{g(R) - g(s)}{R - s} + \frac{t-R}{t-s} \frac{g(t) - g(R)}{t - R} \\
				&\geq \frac{R-s}{t-s} \cdot 1 + \frac{t-R}{t-s} \frac{R^2}{R \max(|t|,R)} \\
				&\geq \frac{R-s}{t-s} \frac{R^2}{R \max(|t|,R)} + \frac{t-R}{t-s} \frac{R^2}{R \max(|t|,R)} \\
				&= \frac{R^2}{R \max(|t|,R)}.
			\end{align*}
			\item Similarly, suppose $s \in (-\infty,-R]$ and $t \in [R,\infty)$.  Then
			\begin{align*}
				\frac{g(t) - g(s)}{t - s} &= \frac{(-R)-s}{t-s} \frac{g(-R) - g(s)}{-R - s} + \frac{R - (-R)}{t-s} \frac{g(R) - g(-R)}{R - (-R)} + \frac{t-R}{t-s} \frac{g(t) - g(R)}{t - R} \\
				&\geq \frac{(-R)-s}{t-s} \frac{R^2}{R \max(|s|,R)} + \frac{R - (-R)}{t-s} \cdot 1 + \frac{t-R}{t-s} \frac{R^2}{R \max(|t|,R)} \\
				&\geq  \frac{R^2}{\max(|s|,R) \max(|t|,R)}
			\end{align*}
			\item The remaining cases are symmetrical.
		\end{enumerate}
		Therefore,
		\begin{align*}
			\log |\det DG_1(X)| &\geq \sum_{i,j=1}^n \log \frac{R^2}{\max(|\lambda_i|,R) \max(|\lambda_j|,R)} \\
			&= -2n \sum_{j=1}^n \log \frac{\max(|\lambda_j|,R)}{R} \\
			&\geq -2n \sum_{j=1}^n \left( \frac{\max(|\lambda_j|,R)}{R} - 1\right) \\
			&= -2n \sum_{j=1}^n \frac{(|\lambda_j| - R)^+}{R}.
		\end{align*}
		Hence,
		\begin{equation} \label{eq: lower bound for cutoff Jacobian}
			\frac{1}{n^2} \log |\det DG_0(X)| \geq -\frac{2}{R} \norm{(|X|-R)_+}_1.
		\end{equation}
		Furthermore, using the noncommutative H{\"o}lder's inequality,
		\[
		\norm{(|X| - R)_+}_1 = \norm{(1 - \mathbbm{1}_{[-R,R]}(X))(|X| - R)_+}_1 \leq (1 - \tr_n(\mathbbm{1}_{[-R,R]}(X)))^{1/2} \norm{(|X| - R)_+}_2.
		\]
		We next claim that
		\[
		t^2 - g(t)^2 \geq [(|t| - R)^+]^2.
		\]
		When $t \in [-R,R]$, both sides are zero, and the cases for $(-\infty,-R]$ and $[R,\infty)$ are symmetrical, so it suffices to consider $t \geq R$ and observe that
		\[
		t^2 - g(t)^2 = (t + g(t))(t - g(t)) \geq t(t - 2R + R^2 / t) = (t - R)^2.
		\]
		It follows that for $X \in (\mathbb{M}_n)_{\sa}$,
		\begin{equation} \label{eq: smooth cutoff L1 estimate 2}
			\norm{X}_2^2 - \norm{G_0(X)}_2^2 \geq \norm{(|X| - R)^+}_2^2.
		\end{equation}
		By combining \eqref{eq: smooth cutoff L1 estimate 2} and \eqref{eq: lower bound for cutoff Jacobian}, we see that
		\begin{align*}
			\frac{1}{n^2} \log |\det DG_0(X)| &\geq -\frac{2}{R} (1 - \tr_n(\mathbbm{1}_{[-R,R]}(X)))^{1/2} \norm{(|X| - R)_+}_2 \\
			&\geq  -\frac{2}{R} (1 - \tr_n(\mathbbm{1}_{[-R,R]}(X)))^{1/2} \left( \norm{X}_2^2 - \norm{G_0(X)}_2^2 \right)^{1/2}.
		\end{align*}
		Then summing over the real and imaginary parts of the $m$ coordinates and using the Cauchy--Schwarz inequality produces \eqref{eq: smooth cutoff Jacobian estimate}.
	\end{proof}
	
	\begin{proof}[Proof of Proposition \ref{prop: entropy with norm cutoff}]
		The inequality $\geq$ is true because $\Gamma_{4R}^{(n)}(\mathbf{x} \mid \mathbf{Y}^{(n)} \rightsquigarrow \mathbf{y}; \Phi, \varepsilon)$ is a subset of $\Gamma^{(n)}(\mathbf{x} \mid \mathbf{Y}^{(n)} \rightsquigarrow \mathbf{y}; \Phi, \varepsilon)$.  To show the reverse inequality $\leq$, we need to get a lower bound on the Gaussian measure of $\Gamma_{4R}^{(n)}(\mathbf{x} \mid \mathbf{Y}^{(n)} \rightsquigarrow \mathbf{y}; \Phi, \varepsilon)$.  Let $g$ and $G$ be as in Lemma \ref{lem: smooth cutoff}.  Note that since $(g(t) + i)^{-1} + i$ vanishes at $\infty$, it can be uniformly approximated by polynomials in $(t + i)^{-1}$ by the Stone Weierstrass theorem.  Hence, for every $\delta > 0$, there exists a polynomial $p$ with $|(g(t) + i)^{-1} - p((t+i)^{-1})| \leq \delta$.  For a restricted chronological formula $\varphi$, let $\tilde{\varphi}$ be the formula obtained by replacing each occurrence of $(\re x_j + i)^{-1}$ by $p((\re x_j+i)^{-1})$ and similarly for the imaginary part.  Because $\varphi$ is an ordinary chronological formula $\psi$ evaluated on $(\re x_j + i)^{-1}$ and $(\im x_j + i)^{-1}$ and similarly the resolvents of $\re y_j$ and $\im y_j$, we can apply uniform continuity of $\psi$ on an operator-norm ball to conclude that for a sufficiently small choice of $\delta$, and the resulting choice of $p$, we have
		\[
		|\varphi^{\cM}(G(\mathbf{x}),\mathbf{y}) - \tilde{\varphi}(\mathbf{x},\mathbf{y})| \leq \frac{\varepsilon}{3} \text{ for } \varphi \in \Phi;
		\]
		here also note that $G(\mathbf{x}) = \mathbf{x}$.  By carrying out the analogous substitutions and estimates for the finite-dimensional approximations $\Lambda_{\varphi}^{(n)}$, we likewise see that for sufficiently small $\delta$, for all $\mathbf{X} \in \mathbb{M}_n^m$,
		\[
		|\varphi^{\cM}(G(\mathbf{X}),\mathbf{Y}^{(n)}) - \tilde{\varphi}(\mathbf{X},\mathbf{Y}^{(n)})| \leq \frac{\varepsilon}{3} \text{ for } \varphi \in \Phi.
		\]
		Let $\tilde{\Phi} = \{ \tilde{\varphi}: \varphi \in \Phi\}$.  Using the triangle inequality,
		\[
		\mathbf{X} \in \Gamma^{(n)}(\mathbf{x} \mid \mathbf{Y}^{(n)} \rightsquigarrow \mathbf{y}; \tilde{\Phi}, \varepsilon / 3) \implies G(\mathbf{X}) \in \Gamma_{4R}^{(n)}(\mathbf{x} \mid \mathbf{Y}^{(n)} \rightsquigarrow \mathbf{y}; \Phi, \varepsilon),
		\]
		where of course we use that $\norm{G(\mathbf{x})}_\infty \leq 4R$ since the real and imaginary parts of each coordinate are bounded by $2R$ in operator norm.
		
		To compute the measure of the image, we will use change of coordinates with Lemma \ref{lem: smooth cutoff}, and hence we want a bound for $\delta_R(\mathbf{X})$.  Since $\norm{\re(x_j)} < R$ and $\norm{\im(x_j)} < R$, there exists a continuous function $h: \R \to [0,1]$ supported in $[-R,R]$ such that $\tr^{\cM}(h(\re x_j)) = 1$ and $\tr^{\cM}(h(\im x_j)) = 1$.  We can approximate $h(t)$ uniformly within $\varepsilon / 3$ by a polynomial in $(t + i)^{-1}$.  Hence, given $\delta > 0$, there exists a restricted chronological formula $\eta$ (in fact, a basic restricted formula) such that for all structures $\cN$
		\[
		\left| \eta^{\cN}(\mathbf{x}') - \sum_{j=1}^m \tr^{\cN}(h(\re x_j) + h(\im x_j)) \right| \leq \frac{\varepsilon}{3}
		\]
		and
		\[
		\left|\Lambda_\eta^{(n)}(\mathbf{X}) - \sum_{j=1}^m \tr^{\cM}(h(\re x_j) + h(\im x_j)) \right| \leq \frac{\varepsilon}{3},
		\]
		and therefore, since $h \leq \mathbbm{1}_{[-R,R]}$,
		\[
		\mathbf{X} \in \Gamma^{(n)}(\mathbf{x} \mid \mathbf{Y}^{(n)} \rightsquigarrow \mathbf{y}; \{\eta\},\delta) \implies \delta_R(\mathbf{X}) \leq \varepsilon,
		\]
		where $\delta_R(\mathbf{X})$ is given by \eqref{eq: delta R}.  Hence, by Lemma \ref{lem: smooth cutoff},
		\begin{align*}
			-\frac{1}{n^2} \log \det |DG(\mathbf{X})| &\leq \frac{2}{R} \varepsilon^{1/2} \left( \norm{\mathbf{X}}_2^2 - \norm{G_0(\mathbf{X})}_2^2 \right)^{1/2} \\
			&\leq \frac{4 \varepsilon^2}{2R^2} + \frac{1}{2} \left( \norm{\mathbf{X}}_2^2 - \norm{G(\mathbf{X})}_2^2 \right).
		\end{align*}
		We now compute the Gaussian measure of $G(\Gamma^{(n)}(\mathbf{x} \mid \mathbf{Y}^{(n)} \rightsquigarrow \mathbf{y}; \tilde{\Phi} \cup \{\psi\}, \varepsilon / 3))$ using change of coordinates.  For convenience, write $S^{(n)} = \Gamma^{(n)}(\mathbf{x} \mid \mathbf{Y}^{(n)} \rightsquigarrow \mathbf{y}; \tilde{\Phi} \cup \{\psi\}, \varepsilon / 3)$.  Using \eqref{eq: smooth cutoff Jacobian estimate},
		\begin{align*}
			\int_{G(S^{(n)})} \exp\left(-\frac{n^2}{2} \norm{\mathbf{X}}_2^2 \right)\,d\mathbf{X} &= \int_{S^{(n)}} |\det DG(\mathbf{X})| \exp\left(-\frac{n^2}{2} \norm{G(\mathbf{X})}_2^2 \right)\,d\mathbf{X} \\
			&\geq \int_{S^{(n)}} \exp\left(-2n^2 \varepsilon R^{-2} - \frac{n^2}{2} \left( \norm{\mathbf{X}}_2^2 - \norm{G(\mathbf{X})}_2^2 \right) \right) \exp\left(-\frac{n^2}{2} \norm{\mathbf{X}}_2^2 \right) \,d\mathbf{X} \\
			&\geq \exp(-2n^2 \varepsilon R^{-2}) \int_{S^{(n)}} \exp\left(-\frac{n^2}{2} \norm{\mathbf{X}}_2^2 \right) \,d\mathbf{X}.
		\end{align*}
		Thus,
		\[
		\sigma^{(n)}(\Gamma_{4R}^{(n)}(\mathbf{x} \mid \mathbf{Y}^{(n)} \rightsquigarrow \mathbf{y}; \Phi, \varepsilon)) \geq \exp(-2n^2 \varepsilon R^{-2}) \sigma^{(n)}(\Gamma^{(n)}(\mathbf{x} \mid \mathbf{Y}^{(n)} \rightsquigarrow \mathbf{y}; \tilde{\Phi}, \varepsilon / 3)).
		\]
		Since $\Phi$ and $\varepsilon$ were arbitrary, we have the desired lower bound on the Gaussian measure.
	\end{proof}
	
	\begin{remark}
		It is likely that Proposition \ref{prop: entropy with norm cutoff} holds with $R$ rather than $4R$.  For the self-adjoint setting, this would follow by \cite[Proposition 2.4]{VoiculescuFE2}.  For the non-self-adjoint setting, one would like fix $R' > R$ and apply a transformation of the form $X \mapsto X f(|X|)$ where $f = 1$ on $[-R',R']$ and $f(t) = R'/|t|$ for $|t| \geq R'$, and compute the Jacobian acting on non-self-adjoint matrices.
	\end{remark}
	
	We can now state the characterization of $\chi_{\chron}^{\cU}$ analogous to Voiculescu's original definition of free entropy $\chi$ \cite{VoiculescuFE2}.
	
	\begin{proposition}
		Let $m \in \N$ and $m' \in \N \cup \{0,\infty\}$.  Let $\cM$ satisfy $\rT_{\stat}$ and let $(\mathbf{x},\mathbf{y}) \in \cM^{m+m'}$.  Let $R \geq \sqrt{2}$ and assume $\norm{x_j} \leq R$.   Let $\mathbf{Y}^{(n)}$ be a tuple of deterministic matrices with $\sup_n \norm{Y_j^{(n)}} < \infty$ and $\lim_{n \to \cU} \Lambda_{\varphi}^{(n)}(\mathbf{Y}^{(n)}) = \varphi^{\cM}(\mathbf{y})$ for restricted chronological formulas $\varphi$.  Then
		\[
		\chi_{\chron}^{\cU}(\mathbf{x} \mid \mathbf{y}) = \inf_{\Phi, \varepsilon} \lim_{n \to \cU} \left[ \frac{1}{n^2} \log \vol(\Gamma_{4R}^{(n)}(\mathbf{x} \mid \mathbf{Y}^{(n)} \rightsquigarrow \mathbf{y}; \Phi, \varepsilon)) + 2m \log n \right].
		\]
	\end{proposition}
	
	\begin{proof}
		Fix a finite set of restricted chronological formulas $\Phi$ and $\varepsilon > 0$.  Since $t^2$ can be approximated by a polynomial in $(t + i)^{-1}$ uniformly on $[-4R,4R]$, there exists a restricted chronological formula $\psi$ such that
		\[
		\left|\psi^{\cM}(\mathbf{x}) - \frac{1}{2} \norm{\mathbf{x}}_2^2 \right| \leq \varepsilon \text{ for } \mathbf{x} \in (D_{0,4R}^{\cM})^m
		\]
		and
		\[
		\left| \Lambda_\psi^{(n)}(\mathbf{X}) - \frac{1}{2} \norm{\mathbf{X}}_2^2 \right| \leq \varepsilon \text{ for } \mathbf{X} \in (D_{4R}^{\mathbb{M}_n})^m.
		\]
		In particular, for $\mathbf{X} \in \Gamma_{4R}^{(n)}(\mathbf{x} \mid \mathbf{Y}^{(n)} \rightsquigarrow \mathbf{y}; \Phi \cup \{\psi\}, \varepsilon)$, we have
		\[
		\frac{1}{2} |\norm{\mathbf{X}}_2^2 - \norm{\mathbf{x}}_2^2| \leq 3 \varepsilon.
		\]
		Recall that the density for the Gaussian measure $\sigma^{(n)}$ on $\mathbb{M}_n^m$ is
		\[
		\frac{1}{(\sqrt{2 \pi n^2})^{2m}} \exp(-n^2 \norm{\mathbf{X}}_2^2 / 2) = \frac{1}{(2\pi n^2)^m} \exp(-n^2 \norm{\mathbf{X}}_2^2 / 2).
		\]
		Therefore,
		\begin{align*}
			& (2 \pi n^2)^{-m} \exp(-n^2 \norm{\mathbf{x}}_2^2 / 2 - 3 n^2 \varepsilon) \vol(\Gamma_{4R}^{(n)}(\mathbf{x} \mid \mathbf{Y}^{(n)} \rightsquigarrow \mathbf{y}; \Phi \cup \{\psi\}, \varepsilon)) \\
			\leq & 
			\sigma^{(n)}(\Gamma_{4R}^{(n)}(\mathbf{x} \mid \mathbf{Y}^{(n)} \rightsquigarrow \mathbf{y}; \Phi \cup \{\psi\}, \varepsilon)) \\
			\leq &
			(2 \pi n^2)^{-m} \exp(-n^2 \norm{\mathbf{x}}_2^2 / 2 + 3 n^2 \varepsilon) \vol(\Gamma_{4R}^{(n)}(\mathbf{x} \mid \mathbf{Y}^{(n)} \rightsquigarrow \mathbf{y}; \Phi \cup \{\psi\}, \varepsilon)),
		\end{align*}
		and so
		\begin{multline*}
			\biggl| \frac{\norm{\mathbf{x}}_2^2}{2} + m \log(2\pi) + 2m \log n + \frac{1}{n^2} \log \vol(\Gamma_{4R}^{(n)}(\mathbf{x} \mid \mathbf{Y}^{(n)} \rightsquigarrow \mathbf{y}; \Phi \cup \{\psi\}, \varepsilon)) \\ - \frac{1}{n^2} \log \sigma^{(n)}(\Gamma_{4R}^{(n)}(\mathbf{x} \mid \mathbf{Y}^{(n)} \rightsquigarrow \mathbf{y}; \Phi \cup \{\psi\}, \varepsilon)) \biggr| \leq 3 \varepsilon.
		\end{multline*}
		The conclusion now follows from \eqref{eq: Gaussian entropy with bounded microstates} since $\Phi$ and $\varepsilon$ were arbitrary and the measures of the microstate spaces are decrease when $\varepsilon$ decreases and the set $\Phi$ increases.
	\end{proof}

	The next proposition is extremely important for the proofs of our main results.  It says that $\chi_{\chron}^{\cU}(\mathbf{x} \mid \mathbf{y})$ is the largest possible limit of normalized classical entropies for random matrix models for $\mathbf{x}$ compatible with fixed deterministic models for $\mathbf{y}$.  The statement here is a direct adaptation of \cite[Theorem 4.10]{JekelPi2024}.
	
	\begin{proposition} \label{prop: classical-free variational principle}
		Let $m \in \N$ and $m' \in \N \cup \{0,\infty\}$.  Let $\cM$ satisfy $\rT_{\stat}$ and let $(\mathbf{x},\mathbf{y}) \in \cM^{m+m'}$.  Let $\mathbf{Y}^{(n)}$ be a deterministic matrix tuple with $\sup_n \norm{Y_j^{(n)}} < \infty$.  Let $\mathbf{X}^{(n)}$ be random variable in $\mathbb{M}_n^m$ satisfying the following conditions:
		\begin{enumerate}[(1)]
			\item For each restricted chronological formula $\varphi$, we have $\displaystyle \lim_{n \to \cU} \Lambda_{\varphi}^{(n)}(\mathbf{X}^{(n)},\mathbf{Y}^{(n)}) = \varphi^{\cM}(\mathbf{x},\mathbf{y})$.
			\item There exists $R > 0$ such that $\displaystyle \lim_{n \to \cU} \norm{X_j^{(n)}} \leq R$ in probability.
			\item There exist $C, K > 0$ such that for all $\delta > 0$ and $n \in \N$, we have $\displaystyle \mathbb{P}(\norm{\mathbf{X}^{(n)}}_2 \geq C + \delta) \leq e^{-Kn^2 \delta^2}$.
		\end{enumerate}
		Let $h^{(n)}$ be the normalized classical entropy
		\[
		h^{(n)}(\mathbf{X}^{(n)}) = \frac{1}{n^2} h(\mathbf{X}^{(n)}) + 2m \log n.
		\]
		Then
		\begin{equation} \label{eq: classical entropy inequality}
			\lim_{n \to \cU} h^{(n)}(\mathbf{X}^{(n)}) \leq \chi_{\chron}^{\cU}(\mathbf{x} \mid \mathbf{y}).
		\end{equation} 
		On the other hand, suppose that $\mathbf{Y}^{(n)}$ satisfies $\lim_{n \to \cU} \Lambda_{\varphi}^{(n)}(\mathbf{Y}^{(n)}) = \varphi^{\cM}(\mathbf{y})$ for all chronological formulas $\varphi$.  Then there exist random matrix models $\mathbf{X}^{(n)}$ such that (1) holds, $\norm{X_j^{(n)}} \leq R$ for some constant $R$, and 
		\begin{equation} \label{eq: classical entropy equality}
			\lim_{n \to \cU} h^{(n)}(\mathbf{X}^{(n)}) = \chi_{\chron}^{\cU}(\mathbf{x} \mid \mathbf{y}).
		\end{equation}
	\end{proposition}
	
	We point out that hypothesis (3) holds for Ginibre random matrices by the Herbst concentration inequality \eqref{eq: Herbst concentration inequality}.  The proof of \cite[Theorem 4.10]{JekelPi2024} works almost verbatim, and so we will not repeat the argument here.  We merely comment on the few differences:
	\begin{itemize}
		\item Because we are using non-self-adjoint matrices, the real dimension is twice as large, so the normalizing constant in $h^{(n)}$ is $2m \log n$ rather than $m \log n$.
		\item In \cite[Theorem 4.10]{JekelPi2024}, the entropy was defined for a fixed choice of matrix approximations $\mathbf{Y}^{(n)}$ for $\mathbf{y}$.  In the setting here, we have already proved that the value is the same for all choices of $\mathbf{Y}^{(n)}$.
		\item The proof of \cite[Theorem 4.10]{JekelPi2024} used microstate spaces associated to general open neighborhoods in the space of laws.  For this proposition, one uses instead the spaces $\Gamma_R^{(n)}(\mathbf{x} \mid \mathbf{Y}^{(n)} \rightsquigarrow \mathbf{y}; \Phi, \varepsilon)$.
	\end{itemize}

	\subsection{Semiconvex and semiconcave definable predicates} \label{subsec: semiconvex and semiconcave}

	\begin{definition}
		Let $\phi$ be a function on an inner-product space $H$ and let $c > 0$.  We say that $f: H \to \R$ is \emph{$c$-semiconvex} if $\phi(x) + \frac{1}{2} c \norm{x}^2$ is convex, or equivalently (after some computation), for $x_0, x_1 \in H$ and $\lambda \in [0,1]$,
		\[
		\phi((1 - \lambda)x_0 + \lambda x_1) \leq (1 - \lambda) \phi(x_0) + \lambda \phi(x_1) + \frac{c \lambda(1 - \lambda)}{2} \norm{x_1 - x_0}^2.
		\]
		Similarly, we say that $\phi$ is $c$-semiconcave if $\phi(x) - \frac{c}{2} \norm{x}^2$ is concave, or equivalently,
		\[
		\phi((1 - \lambda)x_0 + \lambda x_1) \geq (1 - \lambda) \phi(x_0) + \lambda \phi(x_1) - \frac{c \lambda(1 - \lambda)}{2} \norm{x_1 - x_0}^2.
		\]
	\end{definition}
	
	The following is a direct adaptation of \cite[Proposition 5.2]{JekelTypeCoupling}.  Note that in the proposition, we view $L^2(M_0)$ as a real inner product space, using $\re \ip{\cdot,\cdot}_{\tr^{\cM}}$, and differentiability is thus understood with respect to this real inner product.
	
	\begin{proposition} \label{prop: semiconvex semiconcave gradient}
		Let $\mathrm{T}$ be a consistent $\cL_{\chron}$-theory containing $\rT_{\proc}$.  Let $m \in \N$ and $m' \in \N \cup \{0,\infty\}$.  Let $\phi(\mathbf{x},\mathbf{y})$ be a chronologically definable predicate relative to $\rT$.  Suppose that  for some $c > 0$, for all $\cM$ satisfying $\rT$, $\phi^{\cM}(\mathbf{x},\mathbf{y})$ is $c$-semiconvex and $c$-semiconcave as a function of $\mathbf{x}$ in $M_0^m$.
		\begin{enumerate}[(1)]
			\item For each $\cM$ satisfying $\mathrm{T}$, $\phi^{\cM}$ is a differentiable function of $\mathbf{x}$ on $M_0$, and more precisely, for each $\mathbf{x} \in M_0^m$ and $\mathbf{y} \in M_0^{m'}$, there exists a unique $\nabla_{\mathbf{x}} \phi^{\cM}(\mathbf{x},\mathbf{y}) \in L^2(M_0)^m$ such that
			\[
			\left|\phi^{\cM}(\mathbf{x}',\mathbf{y}) - \phi^{\cM}(\mathbf{x},\mathbf{y}) - \re \ip{\mathbf{x}' - \mathbf{x}, \nabla_{\mathbf{x}} \phi^{\cM}(\mathbf{x},\mathbf{y})} \right| \leq \frac{c}{2} \norm{\mathbf{x}' - \mathbf{x}}_{L^2(M_0)}^2.
			\]
			\item The function
			\[
			\iota^{\cM}(\mathbf{x}, \mathbf{y},\mathbf{z}) = \re \ip{\mathbf{z},\nabla \phi^{\cM}(\mathbf{x},\mathbf{y})}_{L^2(\cM)^n}
			\]
			is a chronologically definable predicate relative to $\mathrm{T}$.
			\item Suppose that for each $\mathbf{r} \in (0,\infty)^m$ and $\mathbf{r}' \in (0,\infty)^{m'}$, there exists $\mathbf{R} \in (0,\infty)^m$ such that $\nabla_{\mathbf{y}} \phi$ maps $D_{0,(\mathbf{r},\mathbf{r}')}^{\cM}$ into $D_{0,\mathbf{R}}^{\cM}$.  Then $\nabla_{\mathbf{x}} \phi$ is a definable function.
		\end{enumerate}
	\end{proposition}
	
	Although the statement in \cite{JekelTypeCoupling} is given for the language of tracial von Neumann algebras $\cL_{\tr}$, the proof also clearly works for chronologically definable predicates on $M_0$ studied in \S \ref{subsec: M zero definable things}, i.e.\ using $\cL_{\chron,0}$ rather than $\cL_{\tr}$.  Indeed, the arguments for (1) and (2) only use the theory of convex functions on inner product spaces.\footnote{In (1), we omit the claim from \cite[Proposition 5.2]{JekelTypeCoupling} that the vector is in $L^2$ of the definable closure, which comes out of \cite[Proposition 4.5]{JekelTypeCoupling}, to minimize the effort needed to check the claims.  However, this is automatic from (3) anyway.}  The proof of Claim (3) only uses the sup/inf operations, which in this paper will be taken over balls in the $0$-level of the filtration $M_0$.
	
	The following is a direct adaptation of \cite[Theorem 1.5, Theorem 5.10]{JekelTypeCoupling}.
	
	\begin{theorem}
		Let $\mathrm{T}$ be a consistent $\cL_{\chron}$-theory containing $\rT_{\proc}$.  Let $m \in \N$ and $m' \in \N \cup \{0,\infty\}$.  Let $\phi(\mathbf{x},\mathbf{y})$ be a chronologically definable predicate relative to $\rT$.  Let $\mathbf{r} \in (0,\infty)^m$ and $\mathbf{r}' \in (0,\infty)^{m'}$, and let $\varepsilon$.  Then there exists a definable predicate $\psi$ relative to $\rT$ such that $\psi$ is semiconvex and semiconcave, for $\cM$ satisfying $\rT$,
		\[
		|\varphi^{\cM}(\mathbf{x},\mathbf{y}) - \psi^{\cM}(\mathbf{x},\mathbf{y})| \leq \varepsilon \quad \text{ for } \mathbf{x} \in D_{0,\mathbf{r}}^{\cM}, \mathbf{y} \in D_{0,\mathbf{r}'}^{\cM}.
		\]
		Moreover, $\psi$ can be chosen to satisfy Proposition \ref{prop: semiconvex semiconcave gradient}, so that $\nabla_{\mathbf{x}} \psi$ is a chronologically definable function.
	\end{theorem}
	
	As with the previous statement, the proof from \cite{JekelTypeCoupling} works without any changes.  The chronologically definable predicate $\psi$ can be chosen as
	\begin{equation} \label{eq: regularization}
		\psi^{\cM}(\mathbf{x},\mathbf{y}) = \sup_{\mathbf{w} \in D_{\mathbf{r}}} \inf_{\mathbf{z} \in \mathcal{D}_{\mathbf{R}}} \left[ \phi^{\cM}(\mathbf{z},\mathbf{y}) + \frac{1}{4t} d(\mathbf{w},\mathbf{z})^2 - \frac{1}{2t} d(\mathbf{x},\mathbf{w})^2 \right],  \quad \text{ for } \mathbf{x} \in D_{0,\mathbf{r}}, \mathbf{y} \in D_{0,\mathbf{r}'},
	\end{equation}
	for appropriate choice of $t > 0$, and where $\mathbf{R}$ is any tuple with $R_j \geq r_j$.
	
	The later results of \cite{JekelTypeCoupling} made heavy use of the fact that the boundedness in \ref{prop: semiconvex semiconcave gradient} is automatic when $m' = 0$ if we are working in von Neumann algebras with trivial center \cite[Corollary 5.6]{JekelTypeCoupling}.  This is based on the fact that $\nabla \varphi$ is equivariant under unitary conjugation in $\mathbf{x}$, if there is no dependence on $\mathbf{y}$, hence we cannot invoke this argument if $m' > 0$.  Instead, we will work with definable predicates of the form $\psi(\mathbf{x},\mathbf{y}) + a q(\mathbf{x})$ where $q(\mathbf{x}) = \frac{1}{2} \norm{\mathbf{x}}_2^2$ and $\psi$ is $\norm{\cdot}_1$-Lipschitz in $\mathbf{x}$, which guarantees useful bounds on the gradient as follows.
	
	\begin{lemma} \label{lem: bounded gradient}
		Let $m \in \N$ and $m' \in \N \cup \{0,\infty\}$.  Let $\varphi = \psi + aq$ be a chronologically definable predicate with respect to $\rT_{\stat}$, such that $\psi$ is semiconvex, semiconcave and $L$-Lipschitz with respect to $\norm{\cdot}_1$ in first $m$ variables.  Then for $\cM$ satisfying $\rT_{\stat}$ and $(\mathbf{x},\mathbf{y}) \in \cM^{m+m'}$, we have
		\[
		\norm{\nabla_{m,m'} \varphi^{\cM}(\mathbf{x},\mathbf{y}) - a \mathbf{x}}_\infty \leq L.
		\]
		Hence, by Proposition \ref{prop: semiconvex semiconcave gradient}, $\nabla_{m,m'} \varphi$ is a chronologically definable function.
	\end{lemma}
	
	\begin{proof}
		By Proposition \ref{prop: semiconvex semiconcave gradient}, $\varphi$ and $\psi$ have well-defined gradients in $L^2(\cM)^m$.  Since $\psi$ is $\norm{\cdot}_1$-Lipschitz,
		\[
		|\ip{\nabla_{m,m'} \psi^{\cM}(\mathbf{x},\mathbf{y}), \mathbf{x}'}| \leq L \norm{\mathbf{x}'}_1,
		\]
		and hence by noncommutative $L^p$ duality, $\norm{\nabla_{m,m'} \psi^{\cM}(\mathbf{x},\mathbf{y})}_\infty \leq L$.  The conclusions follow immediately using Proposition \ref{prop: semiconvex semiconcave gradient} (3).
	\end{proof}
	
	\begin{lemma} \label{lem: Lipschitz plus quadratic}
		Let $\varphi$ be a chronologically definable predicate over $\rT_{\stat}$ of the form
		\[
		\varphi(\mathbf{x},\mathbf{y}) = \psi(\mathbf{x},\mathbf{y}) + a q(\mathbf{x})
		\]
		where $a > 0$ and $q(\mathbf{x},\mathbf{y}) = q(\mathbf{x}) = \frac{1}{2} \norm{\mathbf{x}}_2^2$ and $\psi$ is $L$-Lipschitz in each coordinate with respect to $\norm{\cdot}_1$.  Let $\cM$ satisfy $\rT_{\stat}$.
		\begin{enumerate}[(1)]
			\item For $\mathbf{x} \in M_0^m$ with $\norm{x_j} \leq r$, letting $r' = a^{-1}(L + r)$, we have
			\[
			\sup_{\mathbf{x}' \in M_0^m} \left[ \re \ip{\mathbf{x}, \mathbf{x}'} - \varphi^{\cM}(\mathbf{x}',\mathbf{y}) \right] = \sup_{\mathbf{x}' \in (D_{0,r'}^{\cM})^m} \left[ \re \ip{\mathbf{x}, \mathbf{x}'} - \varphi^{\cM}(\mathbf{x}',\mathbf{y}) \right].
			\]
			\item The Legendre transform $\varphi^{\star}$ with respect to the first $m$ variables, given by
			\begin{equation} \label{eq: definition of Legendre transform}
				\varphi^{\star}(\mathbf{x},\mathbf{y}) := \sup_{\mathbf{x}' \in M_0^m} \left[ \re \ip{\mathbf{x}, \mathbf{x}'} - \varphi^{\cM}(\mathbf{x}',\mathbf{y}) \right],
			\end{equation}
			is a chronologically definable predicate with respect to $\rT_{\stat}$.
			\item $\varphi^{\star} - a^{-1} q$ is $a^{-1}L$-Lipschitz with respect to $\norm{\cdot}_1$ in the first $m$ variables.
			\item $\varphi^{\star}$ is convex in the first $m$ variables.
			\item If $\varphi$ is $c$-semiconcave in the first $m$ variables , then $\varphi^{\star} - c^{-1}q$ is convex in the first $m$ variables.
			\item If $\varphi$ is convex in the first $m$ variables, then $(\varphi^{\star})^{\star} = \varphi$.
		\end{enumerate}
	\end{lemma}
	
	\begin{proof}
		(1) For $j = 1$, \dots, $m$, let $x_j''$ be the closest point to $x_j'$ in the $r$-ball per Lemma \ref{lem: projection onto R ball}.  Then
		\begin{align*}
			\re \ip{\mathbf{x}, \mathbf{x}''} &\geq \re \ip{\mathbf{x}, \mathbf{x}'} - \sum_{j=1}^m \norm{x_j} \norm{x_j' - x_j''}_1 \\
			&\geq \ip{\mathbf{x}, \mathbf{x}'} - r \sum_{j=1}^m \norm{x_j' - x_j''}_1 \\
			\psi^{\cM}(\mathbf{x}'',\mathbf{y}) &\leq \psi^{\cM}(\mathbf{x}',\mathbf{y}) + L \sum_{j=1}^m \norm{x_j' - x_j''}_1 \\
			\frac{a}{2} \norm{x''}_2 &\leq \frac{a}{2} \norm{x'}_2^2 - ar' \sum_{j=1}^m \norm{x_j' - x_j''}_1,
		\end{align*}
		where the last line follows from Lemma \ref{lem: projection onto R ball} (3).  Since $r + L = cr'$, subtracting the second and third inequalities from the first yields that
		\begin{align*}
			& \re \ip{\mathbf{x}, \mathbf{x}''} - \psi^{\cM}(\mathbf{x}'',\mathbf{y}) - \frac{a}{2} \norm{x''}_2^2 \\
			\geq & \re \ip{\mathbf{x}, \mathbf{x}'} - \psi^{\cM}(\mathbf{x}',\mathbf{y}) - \frac{a}{2} \norm{x'}_2^2,
		\end{align*}
		which proves the claim.
		
		(2) By claim (1), $\varphi^{\star}$ agrees with a chronologically definable predicate on $\norm{x_j} \leq r$, for each $r > 0$.  Therefore, $\varphi^{\star}$ itself is a chronologically definable predicate.
		
		(3) Observe that
		\begin{align*}
			\varphi^{\star}(\mathbf{x},\mathbf{y}) - \frac{1}{2c} \norm{\mathbf{x}}_2^2 &= \sup_{\mathbf{x}' \in M_0^m} \biggl[ \re \ip{\mathbf{x}', \mathbf{x}} - \frac{c}{2} \norm{\mathbf{x}'}_2^2 - \frac{1}{2c} \norm{\mathbf{x}}_2^2 - \psi^{\cM}(\mathbf{x}',\mathbf{y}) \biggr] \\
			&= \sup_{\mathbf{x}' \in M_0^m} \biggl[ \frac{a}{2} \norm{a^{-1}\mathbf{x} - \mathbf{x}'}_2^2 - \psi^{\cM}(\mathbf{x}',\mathbf{y}) \biggr] \\
			&= \sup_{\mathbf{x}' \in M_0^m} \biggl[ \frac{a}{2} \norm{\mathbf{x}'}_2^2 - \psi^{\cM}(a^{-1} \mathbf{x} - \mathbf{x}',\mathbf{y}) \biggr].
		\end{align*}
		The right-hand side is the supremum of a family of functions that are $a^{-1}L$-Lipschitz in $x_j$ with respect to $\norm{\cdot}_1$, which proves the claim.
		
		(4) is a standard fact which follows because $\varphi^{\star}$ by definition is the supremum of a family of affine functions in $\mathbf{x}$.
		
		(5) is standard and be proved as follows:
		\begin{align*}
			(\varphi^{\star})^{\cM}(\mathbf{x},\mathbf{y}) - \frac{1}{2c} \norm{\mathbf{x}}_2^2 &= \sup_{\mathbf{x}' \in M_0^m} \left[ \re \ip{\mathbf{x},\mathbf{x}'} - \frac{1}{2c} \norm{\mathbf{x}}_2^2 - \varphi^{\cM}(\mathbf{x}',\mathbf{y}) \right] \\
			&= \sup_{\mathbf{x}' \in M_0^m} \left[ -\frac{c}{2} \norm{c^{-1} \mathbf{x} - \mathbf{x}'}_2^2 - \left(\varphi^{\cM}(\mathbf{x}',\mathbf{y}) - \frac{c}{2} \norm{\mathbf{x}'}_2^2  \right) \right] \\
			&= \sup_{\mathbf{x}' \in M_0^m} \left[ -\frac{c}{2} \norm{\mathbf{x}'}_2^2 - \left(\varphi^{\cM}(c^{-1} \mathbf{x} - \mathbf{x}',\mathbf{y}) - \frac{c}{2} \norm{c^{-1} \mathbf{x} - \mathbf{x}'}_2^2  \right) \right],
		\end{align*}
		which is the supremum of a family of convex functions of $\mathbf{x}$ since $-(\varphi - cq)$ is convex.
		
		(6) is also a standard fact from convex analysis.  See also \cite[\S 4.2]{GJNS2021}.
	\end{proof}
	
	\begin{remark}
		The Legendre transform is denoted $\mathcal{L} \varphi$ in \cite{GJNS2021,JekelTypeCoupling}, but we use $\varphi^{\star}$ here to ease the load on the letter $\cL$.  Meanwhile, we use the five-pointed star $\star$ to avoid confusion with the adjoint of an operator or the dual of a Banach space.
	\end{remark}
	
	\begin{lemma} \label{lem: Lipschitz plus quadratic 2}
		Let $\varphi$ be a chronologically definable predicate over $\rT_{\stat}$ of the form
		\[
		\varphi(\mathbf{x},\mathbf{y}) = \psi(\mathbf{x},\mathbf{y}) + a q(\mathbf{x})
		\]
		where $a \geq 0$ and $q(\mathbf{x},\mathbf{y}) = q(\mathbf{x}) = \frac{1}{2} \norm{\mathbf{x}}_2^2$ and $\psi$ is $L$-Lipschitz in each coordinate with respect to $\norm{\cdot}_1$.  Let $\cM$ satisfy $\rT_{\stat}$.  Consider the inf-convolution
		\begin{equation} \label{eq: inf convolution}
			(cq \square \varphi)^{\cM}(\mathbf{x},\mathbf{y}) = \inf_{\mathbf{v} \in M_0^m} \left[ \frac{c}{2} \norm{\mathbf{v} - \mathbf{x}}_2^2 + \varphi^{\cM}(\mathbf{v},\mathbf{y}) \right]
		\end{equation}
		Then
		\begin{enumerate}[(1)]
			\item If $\norm{x_j} \leq r$, then to evaluate the infimum in \eqref{eq: inf convolution}, it suffices to consider $\mathbf{v} \in D_{0,r'}^{\cM}$ where $r' = (c + a)^{-1}(L + cr)$.
			\item $cq \square \varphi$ is a chronologically definable predicate with respect to $\rT_{\stat}$.
			\item $(cq \square \varphi) - (c - (a+c)^{-1}c^2) q$ is $(c+a)^{-1}cL$-Lipschitz with respect to $\norm{\cdot}_1$ in the first $m$ variables.
			\item $cq \square \varphi$ is $c$-semiconcave.
			\item If $\varphi$ is convex, then $cq \square \varphi$ is convex.
			\item $(cq \square \varphi)^{\star} = c^{-1}q + \varphi^{\star}$.
		\end{enumerate}
	\end{lemma}
	
	\begin{proof}
		(1) Observe that
		\begin{align*}
			(cq \square \varphi)^{\cM}(\mathbf{x},\mathbf{y}) &= \inf_{\mathbf{v} \in M_0^m} \left[ \frac{c}{2} \norm{\mathbf{x} - \mathbf{v}}_2^2 + \varphi(\mathbf{v}) \right] \\
			&= \frac{c}{2} \norm{\mathbf{x}}_2^2 - \sup_{\mathbf{v}} \left[ \re \ip{c\mathbf{x},\mathbf{v}} - \varphi^{\cM}(\mathbf{v}) - \frac{c}{2} \norm{\mathbf{v}}_2^2 \right] \\
			&= cq^{\cM}(\mathbf{x}) - ((\varphi + cq)^{\star})^{\cM}(c\mathbf{x},\mathbf{y}).
		\end{align*}
		By Lemma \ref{lem: Lipschitz plus quadratic}, the supremum in the definition of $((\varphi+cq)^{\star})^{\cM}(c\mathbf{x},\mathbf{y}) = ((\psi + (a+c)q)^{\star})^{\cM}$ is witnessed by values of $\mathbf{v}$ with $\norm{v_j} \leq (a+c)^{-1}(L + cr)$.
		
		(2) This follows from Lemma \ref{lem: Lipschitz plus quadratic}
		
		(3) By Lemma \ref{lem: Lipschitz plus quadratic}, $((\varphi+cq)^{\star})^{\cM}(\mathbf{x},\mathbf{y}) - (a+c)^{-1} q^{\cM}(\mathbf{x})$ is $(a+c)^{-1}L$-Lipschitz with respect to $\norm{\cdot}_1$ in the first $m$ variables, and hence $((\varphi+cq)^{\star})^{\cM}(c\mathbf{x},\mathbf{y}) - (a+c)^{-1}c^2 q^{\cM}(\mathbf{x})$ is $(a+c)^{-1} cL$-Lipschitz.  Subtracting this from $cq$ yields the desired result.
		
		(4) - (6) are standard properties.  Proof of (4) and (5) can be found in \cite[Fact 2.20]{Jekel2025InfoGeom}.  See also \cite[\S 4.2]{GJNS2021}.
	\end{proof}
	
	\begin{lemma} \label{lem: expression as inf convolution}
		Let $\varphi$ be a chronologically definable predicate over $\rT_{\stat}$ of the form
		\[
		\varphi(\mathbf{x},\mathbf{y}) = \psi(\mathbf{x},\mathbf{y}) + a q(\mathbf{x})
		\]
		where $a > 0$ and $q(\mathbf{x},\mathbf{y}) = q(\mathbf{x}) = \frac{1}{2} \norm{\mathbf{x}}_2^2$ and $\psi$ is $L$-Lipschitz in each coordinate with respect to $\norm{\cdot}_1$.  Assume that $c > a$ and $\varphi$ is convex and $c$-semiconcave in the first $m$ variables.  Let $r > 0$.  Let $r_1 = c^{-1}L + (1 - ac^{-1})r$ and $r_2 = (a + c)^{-1}a(1+c)L + ar$.  Then there exists a chronologically definable predicate $\omega$ such that for $\cM$ satisfying $\rT$ and $\mathbf{x} \in (D_{0,r}^{\cM})^m$ and $\mathbf{y} \in M_0^{m'}$,
		\[
		\varphi^{\cM}(\mathbf{x},\mathbf{y}) = \inf_{\mathbf{v} \in (D_{0,r_1}^{\cM})^m} \left[ \frac{c}{2} \norm{\mathbf{v} - \mathbf{x}}_2^2 + \sup_{\mathbf{w} \in (D_{0,r_2}^{\cM})} \left[ \re \ip{\mathbf{v}, \mathbf{w}} - \omega^{\cM}(\mathbf{w},\mathbf{y}) \right] \right].
		\]
	\end{lemma}
	
	\begin{proof}
		By Lemma \ref{lem: Lipschitz plus quadratic}, $\varphi^{\star}$ is a chronologically definable predicate with respect to $\rT_{\stat}$ and $\varphi^{\star} - c^{-1}q$ is convex and $\varphi^{\star} - a^{-1}q$ is $a^{-1}L$-Lipschitz with respect to $\norm{\cdot}_1$ in the first $m$ variables.  Let $\omega = \varphi^{\star} - c^{-1} q$.  Hence, $\omega - (a^{-1} - c^{-1}) q$ is $a^{-1}L$-Lipschitz with respect to $\norm{\cdot}_1$ in the first $m$ variables (here $a^{-1} - c^{-1} > 0$ by assumption).
		
		Therefore, by Lemma \ref{lem: Lipschitz plus quadratic} again $\omega^{\star}$ is a chronologically definable predicate and $\omega^{\star} - (a^{-1} - c^{-1})^{-1} q$ is $(a^{-1} - c^{-1})^{-1} a^{-1}L = (1 - ac^{-1})^{-1} L$-Lipschitz in the first $m$ variables.  By Lemma \ref{lem: Lipschitz plus quadratic 2}, $cq \square \omega^{\star}$ is a chronologically definable predicate and is convex, and
		\[
		(cq \square \omega^{\star})^{\star} = c^{-1}q + (\omega^{\star})^{\star} = c^{-1}q + \omega = \varphi^{\star}.
		\]
		Since $cq \square \omega^{\star}$ and $\varphi$ are both convex, by Lemma \ref{lem: Lipschitz plus quadratic} (6), we have $cq \square \omega^{\star} = \varphi$, or
		\[
		\varphi(\mathbf{x},\mathbf{y}) = \inf_{\mathbf{v} \in M_0^m} \left[ \frac{c}{2} \norm{\mathbf{v} - \mathbf{x}}_2^2 + \sup_{\mathbf{w} \in M_0^m} \left[ \re \ip{\mathbf{v}, \mathbf{w}} - \omega^{\cM}(\mathbf{w},\mathbf{y}) \right] \right].
		\]
		If $\mathbf{x}$ satisfies $\norm{x_j} \leq r$, then by Lemma \ref{lem: Lipschitz plus quadratic 2} applied to $\omega^{\star}$, the infimum is witnessed by values of $\mathbf{v}$ with
		\begin{align*}
			\norm{v_j} &\leq (c + (a^{-1} - c^{-1})^{-1})^{-1}((a^{-1} - c^{-1})^{-1} a^{-1}L + cr) \\
			&= (c(a^{-1} - c^{-1}) + 1)^{-1}  (a^{-1}L + (a^{-1} - c^{-1})cr) \\
			&= (ca^{-1})^{-1} (a^{-1}L + (a^{-1} - c^{-1})cr) \\
			&= c^{-1} L + (1 - ac^{-1})r \\
			&= r_1.
		\end{align*}
		Now assume that $\mathbf{v}$ satisfies $\norm{v_j} \leq r_1$, and consider the inner supremum over $\mathbf{w}$.  By applying Lemma \ref{lem: Lipschitz plus quadratic} to $\omega$, the supremum is witnessed by $\mathbf{w}$ with
		\begin{align*}
			\norm{w_j} &\leq (a^{-1} - c^{-1})^{-1}(L + r_1) \\
			&=  (a^{-1} - c^{-1})^{-1}(L + c^{-1} L + (1 - ac^{-1})r) \\
			&= (a + c)^{-1}a(1+c)L + ar \\
			&= r_2. \qedhere
		\end{align*}
	\end{proof}

	\subsection{Finite-dimensional approximations of definable predicates} \label{subsec: more FD approximations}
	
	In Theorem \ref{thm: infinitesimal change of variables}, we will perform a change of variables using the gradient of a semiconvex and semiconcave chronologically definable predicate.  This requires studying an analogous transformation on the finite-dimensional matrix models, and thus we need to construct suitable approximations $\Lambda_{\varphi}^{(n)}$ for chronologically definable predicates $\varphi$.  Note that in Proposition \ref{prop: approx by pointwise function} such finitary approximations are only constructed when $\varphi$ is a restricted chronological formula.  We also only initially defined the center-valued evaluation $\varphi_E^{\cM}$ in the random matrix ultraproduct $\cM$ when $\varphi$ is a chronological formula; however, using Corollary \ref{cor: center-valued definable predicates}, the center-valued evaluation $\varphi_E^{\cM}(\mathbf{x})$ makes sense in the matrix ultraproduct \eqref{eq: random matrix quotient} when $\varphi$ is a definable predicate with respect to some $\rT$ containing $\rT_{\operatorname{stat}}$.
	
	The first lemma constructs finite-dimensional approximations which are bounded.
	
	\begin{lemma} \label{lem: FD approximation 1}
		Let $m \in \N \cup \{0,\infty\}$.  Let $\varphi$ be a chronologically definable predicate with respect to $\rT_{\operatorname{stat}}$ in $m$ variables.  Fix $\mathbf{r} \in (0,\infty)^m$.  Then there exist functions $\Lambda_\varphi^{(n)}: \mathbb{M}_n^m \to \R$ such that $|\Lambda_{\varphi}^{(n)}|$ is globally bounded by some constant $M$ and for all $\mathbf{X}^{(n)} \in (L^\infty(\Omega,\mathcal{G}_0,\mathbb{P}) \otimes \mathbb{M}_n)^m$ with $\norm{X_j^{(n)}} \leq r_j$, we have
		\[
		\varphi_E^{\cM}([\mathbf{X}^{(n)}]_{n \in \N}) = [\Lambda_{\varphi}^{(n)}(\mathbf{X}^{(n)})]_{n \in \N} 
		\]
		where $\cM$ is the random matrix ultraproduct \eqref{eq: random matrix ultraproduct}.
	\end{lemma}
	
	\begin{proof}
		Since restricted chronological formulas are dense in the space of chronologically definable predicates, there exists a sequence $\psi_n$ of restricted chronological formulas such that
		\[
		\norm{\psi_k - \varphi}_{0,\mathbf{r},\rT} \leq 2^{-k-1},
		\]
		so in particular
		\[
		\norm{\psi_k - \psi_{k+1}}_{0,\mathbf{r},\rT} \leq 2^{-k}.
		\]
		Let $f_k: \R \to \R$ be given by $f_k(x) = \min(\max(x,-2^{-k}),2^k)$.  Thus, $|f_k| \leq 2^{-k}$ and the sum
		\[
		\psi = \psi_1 + \sum_{k=1}^\infty f_k(\psi_{k+1} - \psi_k)
		\]
		converges uniformly.  We also that $\psi^{\cM} = \varphi^{\cM}$ on $D_{0,\mathbf{r}}^{\cM}$ for $\cM$ satisfying $\rT_{\stat}$.  Hence, by Corollary \ref{cor: center-valued definable predicates}, for the random matrix ultraproduct $\cM$, we also have $\varphi_E^{\cM} = \psi_E^{\cM}$ on $D_{0,\mathbf{r}}^{\cM}$.
		
		Let $\Lambda_{\psi_k}^{(n)}$ be the sequence of approximations given by Proposition \ref{prop: approx by pointwise function}.  Let
		\[
		\Lambda_{\varphi}^{(n)} = \Lambda_{\psi}^{(n)} = \Lambda_{\psi_1}^{(n)} + \sum_{k=1}^\infty f_k(\Lambda_{\psi_{k+1}}^{(n)} - \Lambda_{\psi_k}^{(n)}).
		\]
		Let $\cM$ be the random matrix ultraproduct.  Using Proposition \ref{prop: approx by pointwise function} and the fact that each summand is bounded by $2^{-k}$, we have that for a sequence $\mathbf{X}^{(n)}$ with $\sup_n \norm{X_j^{(n)}} \leq r_j$,
		\[
		\varphi_E^{\cM}([\mathbf{X}^{(n)}]_{n \in \N}) = \psi_E^{\cM}([\mathbf{X}^{(n)}]_{n \in \N})  = [\Lambda_\psi^{(n)}(\mathbf{X}^{(n)})]_{n \in \N}.
		\]
		Note that $\Lambda_{\psi}^{(n)}$ is bounded by some constant, since $\Lambda_{\psi_1}^{(n)}$ is bounded by Proposition \ref{prop: approx by pointwise function} and $|f_k(\Lambda_{\psi_{k+1}}^{(n)} - \Lambda_{\psi_k}^{(n)})| \leq 2^{-k}$.
	\end{proof}
	
	The next lemma constructs finite-dimensional approximations that preserve semiconvexity and semiconcavity.
	
	\begin{lemma} \label{lem: FD approximation of SCSC}
		Let $m \in \N$.  Let $c > 0$.  Let $\varphi$ be a  chronologically definable predicate with respect to $\rT_{\stat}$ in $m$ variables which is $c$-semiconcave and $c$-semiconvex in the first $m$ variables, and assume that $\varphi - aq$ is $L$-Lipschitz with respect to $\norm{\cdot}_1$ in the first $m$ variables, where $|a| < c$.  Let $r > 0$.  Then there exist functions $\Lambda_\varphi^{(n)}: \mathbb{M}_n^{m+m'} \to \R$ such that $\Lambda_\varphi^{(n)}$ is $c$-semiconcave and $c$-semiconcave, and for all $\mathbf{X}^{(n)} \in L^\infty(\Omega,\mathcal{G}_0,\mathbb{P})$ with operator norm bounded by $r$ and $\mathbf{Y}^{(n)}$ with $\sup_n \norm{Y_j^{(n)}} < \infty$, we have
		\begin{equation} \label{eq: FD approximation of SCSC}
			\varphi_E^{\cM}([\mathbf{X}^{(n)}]_{n \in \N},[\mathbf{Y}^{(n)}]_{n \in \N}) = [\Lambda_{\varphi}^{(n)}(\mathbf{X}^{(n)},\mathbf{Y}^{(n)})]_{n \in \N},
		\end{equation}
		where $\cM$ is the random matrix ultraproduct \eqref{eq: random matrix ultraproduct}.  Moreover, $\Lambda_{\varphi}^{(n)}(\mathbf{X},\mathbf{Y}) - (c/2) \norm{\mathbf{X}}_2^2$ is $L'$-Lipschitz in $\mathbf{X}$ with respect to $\norm{\cdot}_1$, for some constant $L'$ independent of $N$, and hence also $\nabla \Lambda_{\varphi}^{(n)}(\mathbf{X},\mathbf{Y}) - c \mathbf{X}$ is bounded in $\norm{\cdot}_\infty$.
	\end{lemma}
	
	\begin{proof}
		Note that $\varphi + cq$ is convex and $2c$-semiconcave in the first $m$ variables, and $\varphi + cq - (a+c)q$ is $L$-Lipschitz in the first $m$ variables.  Moreover, $0 < a + c < 2c$. By Lemma \ref{lem: expression as inf convolution}, there exists $r_1, r_2 > 0$ and a $\rT_{\stat}$ definable predicate $\omega$ such that for $\mathbf{x} \in (D_{0,r}^{\cM})^m$ and $\mathbf{y} \in \cM^{m'}$, we have
		\[
		(\varphi + cq)^{\cM}(\mathbf{x},\mathbf{y}) = \inf_{\mathbf{v} \in (D_{0,r_1}^{\cM})^m} \left[ c \norm{\mathbf{v} - \mathbf{x}}_2^2 + \sup_{\mathbf{w} \in (D_{0,r_2}^{\cM})} \left[ \re \ip{\mathbf{v}, \mathbf{w}} - \omega^{\cM}(\mathbf{w},\mathbf{y}) \right] \right].
		\]
		Let $\psi^{\cM}(\mathbf{x},\mathbf{y})$ be the right-hand side of this formula for general $\mathbf{x}$ not necessarily in the ball of radius $r$.  Let $\Lambda_\omega^{(n)}$ be a sequence of finite-dimensional approximations for $\omega$ given by Lemma \ref{lem: FD approximation 1}.  Define
		\[
		\Lambda_{\psi}^{(n)}(\mathbf{X},\mathbf{Y}) = \inf_{\mathbf{V} \in (D_{r_1}^{\mathbb{M}_n})^m} \left[ c \norm{\mathbf{V} - \mathbf{X}}_2^2 + \sup_{\mathbf{W} \in (D_{r_2}^{\mathbb{M}_n})^m} \left[ \re \ip{\mathbf{V}, \mathbf{W}} - \omega^{\cM}(\mathbf{W},\mathbf{Y}) \right] \right].
		\]
		By the proof of Proposition \ref{prop: approx by pointwise function} (inductive case 2), we have
		\[
		[\Lambda_\psi^{(n)}(\mathbf{X}^{(n)},\mathbf{Y}^{(n)})]_{n \in \N} = \psi_E^{\cM}([\mathbf{X}^{(n)}]_{n \in \N},[\mathbf{Y}^{(n)}]_{n \in \N}).
		\]
		Hence, letting $\Lambda_\varphi^{(n)}(\mathbf{X},\mathbf{Y}) = \Lambda_{\psi}^{(n)}(\mathbf{X},\mathbf{Y}) - (c/2) \norm{\mathbf{X}}_2^2$, we have \eqref{eq: FD approximation of SCSC} whenever $\norm{\mathbf{X}_j^{(n)}} \leq r$.
		
		To prove the claim about Lipschitzness, we observe that
		\[
		\Lambda_{\psi}^{(n)}(\mathbf{X},\mathbf{Y}) - c \mathbf{X}_2^2 = \inf_{\mathbf{V} \in (D_{r_1}^{\mathbb{M}_n})^m} \left[ 2c \re \ip{\mathbf{V},\mathbf{X}} + c \norm{\mathbf{V}}_2^2 + \sup_{\mathbf{W} \in (D_{r_2}^{\mathbb{M}_n})^m} \left[ \re \ip{\mathbf{V}, \mathbf{W}} - \omega^{\cM}(\mathbf{W},\mathbf{Y}) \right] \right],
		\]
		which is the infimum of a family of $r_1$-Lipschitz functions of $\mathbf{X}$.
	\end{proof}
	
	The next lemma shows that the gradient of the finite-dimensional approximation is a finite-dimensional approximation for the gradient.
	
	\begin{lemma} \label{lem: FD approximation of gradient}
		Let $m \in \N$ and $m' \in \N \cup \{0,\infty\}$.  Let $\varphi(\mathbf{x},\mathbf{y})$ be a chronologically definable predicate in $m+m'$ variables with respect to $\rT_{\proc}$ which is $c$-semiconvex and $c$-semiconcave in the first $m$ variables.  Let $\Lambda_{\varphi}^{(n)}: (\mathbb{M}_n)^{m+m'} \to \R$ be a $c$-semiconvex and $c$-semiconcave function satisfying \eqref{eq: FD approximation of SCSC}.  Assume that for each $\mathbf{r} \in (0,\infty)^{m+m'}$ there exists $\mathbf{R} \in (0,\infty)^m$ such that
		\begin{itemize}
			\item For $\cM$ satisfying $\rT_{\proc}$, if $(\mathbf{x},\mathbf{y}) \in D_{0,\mathbf{r}}^{\cM}$, then $\nabla_{m,m'} \varphi^{\cM}(\mathbf{x},\mathbf{y}) \in D_{0,\mathbf{R}}^{\cM}$.
			\item If $(\mathbf{X}, \mathbf{Y}) \in D_{\mathbf{r}}^{\mathbb{M}_n}$, then $\nabla_{m,m'} \Lambda_{\varphi}^{(n)}(\mathbf{X},\mathbf{Y}) \in D_{0,\mathbf{R}}^{\cM}$.
		\end{itemize}
		Let $\cM$ be the filtration on the random matrix ultraproduct \eqref{eq: random matrix ultraproduct}, let $\cV$ be a character on $Z(M_\infty)$, and let $\pi: M_\infty \to Q_\infty$ be the quotient map.  Then for any random matrix tuples $(\mathbf{X}^{(n)}, \mathbf{Y}^{(n)})$ taking values in $D_{\mathbf{r}}^{\mathbb{M}_n}$, we have
		\begin{equation} \label{eq: FD approximation of gradient}
			\nabla_{m,m'} \varphi^{\cQ}(\pi([\mathbf{X}^{(n)}]_{n \in \N}),\pi([\mathbf{Y}^{(n)}]_{n \in \N})) = \pi([\nabla_{m,m'} \Lambda_{\varphi}^{(n)}(\mathbf{X}^{(n)},\mathbf{Y}^{(n)})]_{n \in \N}).
		\end{equation}
	\end{lemma}
	
	\begin{proof}
		Consider another matrix $m$-tuple $(\mathbf{H}^{(n)})_{n \in \N}$ which is bounded in operator norm.  Since $\Lambda_{\varphi}^{(n)}$ is $c$-semiconvex and $c$-semiconcave in the first $m$ arguments,
		\begin{multline*}
			\biggl| \Lambda_{\varphi}^{(n)}([\mathbf{X}^{(n)} + \mathbf{H}^{(n)}]_{n \in \N}, [\mathbf{Y}^{(n)}]_{n \in \N}) - \Lambda_{\varphi}^{(n)}([\mathbf{X}^{(n)}]_{n \in \N}, [\mathbf{Y}^{(n)}]_{n \in \N}) \\
			- \re \ip*{\nabla_{m,m'} \Lambda_{\varphi}^{(n)}(\mathbf{X}^{(n)},\mathbf{Y}^{(n)}), \mathbf{H}^{(n)}}_{\tr_n} \biggr| \leq \frac{c}{2} \norm{\mathbf{H}^{(n)}}_2^2.
		\end{multline*}
		Let $\mathbf{x} = [\mathbf{X}^{(n)}]_{n \in \N}$ and $\mathbf{y} = [\mathbf{Y}^{(n)}]_{n \in \N}$ and $\mathbf{h} = [\mathbf{H}^{(n)}]_{n \in \N}$ and $\mathbf{p} = [\nabla_{m,m'} \Lambda_{\varphi}^{(n)}(\mathbf{X}^{(n)},\mathbf{Y}^{(n)})]_{n \in \N}$ as tuples in $\cM$.  Then the following inequality holds in $Z(\cM)$:
		\[
		\biggl| \varphi_E^{\cM}(\mathbf{x} + \mathbf{h},\mathbf{y}) - \varphi_E^{\cM}(\mathbf{x},\mathbf{y}) \\
		- \sum_{j=1}^m \re E[p_j^* h_j] \biggr| \leq \frac{c}{2} \sum_{j=1}^m E [|h_j|^2].
		\]
		Then using Proposition \ref{prop: Los theorem},
		\[
		\Bigl|\varphi^{\cQ}(\pi(\mathbf{x}) + \pi(\mathbf{h}), \pi(\mathbf{y})) - \varphi^{\cQ}(\mathbf{x}) - \re \ip{\pi(\mathbf{p}),\pi(\mathbf{h})}_{\tau_{\cQ}} \Bigr| \leq \frac{c}{2} \norm{\pi(\mathbf{h})}_2^2.
		\]
		Since $\pi$ is surjective, this implies that $\nabla \varphi^{\cQ}(\mathbf{x},\mathbf{y}) = \pi(\mathbf{p})$, which is exactly \eqref{eq: FD approximation of gradient}.
	\end{proof}
	
	\begin{lemma} \label{lem: FD approximation of gradient 2}
		Let $m \in \N$ and $m' \in \N \cup \{0,\infty\}$.  Let $\varphi(\mathbf{x},\mathbf{y})$ be a chronologically definable predicate in $m+m'$ variables with respect to $\rT_{\stat}$ satisfying the hypotheses of the previous lemma.  Let $\psi$ be another chronologically definable predicate in $m+m'$ variables.  Let $\Lambda_{\varphi}^{(n)}: (\mathbb{M}_n)^{m+m'} \to \R$ be a $c$-semiconvex and $c$-semiconcave function satisfying \eqref{eq: FD approximation of SCSC}, and let $\Lambda_\psi^{(n)}$ satisfy \eqref{eq: FD approximation of SCSC} for $\psi$.  Define the chronologically definable predicate
		\[
		\omega(\mathbf{x},\mathbf{y}) = \psi(\nabla_{m,m'} \varphi(\mathbf{x},\mathbf{y}),\mathbf{y}).
		\]
		Then for any $\mathbf{x} = [\mathbf{X}^{(n)}]_{n \in \N}$ and $\mathbf{y} = [\mathbf{Y}^{(n)}]_{n \in \N}$ in the matrix ultraproduct $M_0$ given by \eqref{eq: random matrix ultraproduct 2},
		\begin{equation} \label{eq: FD approximation composition}
			\omega_E^{\cM}([\mathbf{X}^{(n)}]_{n \in \N},[\mathbf{Y}^{(n)}]_{n \in \N}) = [\Lambda_\psi^{(n)}(\nabla \Lambda_\varphi^{(n)}(\mathbf{X}^{(n)},\mathbf{Y}^{(n)}),\mathbf{Y}^{(n)})]_{n \in \N}.
		\end{equation}
	\end{lemma}
	
	\begin{proof}
		Let $\cV$ be a character on $Z(\cM)$ and let $\pi: \cM \to \cQ$ be the corresponding quotient.  Using Theorem \ref{prop: Los theorem} and Lemma \ref{lem: FD approximation of gradient},
		\begin{align*}
			\cV \circ \omega_E^{\cM}(\mathbf{x},\mathbf{y}) &= \omega^{\cQ}(\pi(\mathbf{x}),\pi(\mathbf{y})) \\
			&= \psi^{\cQ}(\nabla_{m,m'} \varphi^{\cQ}(\mathbf{x},\mathbf{y}),\mathbf{y}) \\
			&= \psi^{\cQ}(\pi([\nabla_{m,m'} \Lambda_{\varphi}^{(n)}(\mathbf{X}^{(n)},\mathbf{Y}^{(n)})]_{n \in \N}),\pi([\mathbf{Y}^{(n)}]_{n \in \N})) \\
			&= \cV \circ \psi_E^{\cM}([\nabla_{m,m'} \Lambda_{\varphi}^{(n)}(\mathbf{X}^{(n)},\mathbf{Y}^{(n)})]_{n \in \N},[\mathbf{Y}^{(n)}]_{n \in \N}) \\
			&= \cV( [\Lambda_\psi^{(n)}(\nabla_{m,m'} \Lambda_{\varphi}^{(n)}(\mathbf{X}^{(n)},\mathbf{Y}^{(n)}), \mathbf{Y}^{(n)})]_{n \in \N} ).
		\end{align*}
		Since the character $\cV$ was arbitrary, we have \eqref{eq: FD approximation composition}.
	\end{proof}

	\subsection{The Laplacian} \label{subsec: Laplacian}
	
	\begin{definition} \label{def: heat semigroup}
		Let $m \in \N$ and $m' \in \N \cup \{0,\infty\}$.  We define an operation $P_{m,m',t_0}$ on formulas in free variables $\mathbf{x} = (x_j)_{j=1}^m$ and $\mathbf{y} = (y_j)_{j=1}^{m'}$ by
		\[
		(P_{t_0} \varphi)(\mathbf{x},\mathbf{y}) = (S_{t_0} \varphi)(\mathbf{x} + (z_{0,t_0,j})_{j=1}^m, \mathbf{y}).
		\]
		In other words, we perform the same transformations as $S_{t_0}$ and also replace $x_j$ with $x_j + z_{t_0,j}$.  Note that we have assumed that the variables bound to some quantifier have been disjointified from the free variables, so in particular, the bound variables are left unchanged by our substitutions.  When $m$ and $m'$ are clear from context, we will often abbreviate $P_{m,m',t}$ to simply $P_t$.
	\end{definition}
	
	\begin{remark}
		The operation $P_t$ leaves the space of chronological formulas invariant.  Moreover, for a theory $\rT$ containing $\rT_{\operatorname{stat}}$, we have that $\norm{P_t \varphi}_{0,r,\rT} \leq \norm{\varphi}_{t,r+6\sqrt{t},\rT} = \norm{\varphi}_{0,r+6\sqrt{t},\rT}$.  Hence, by passing to the completion, $T_t$ is well-defined on the space of chronologically definable predicates with respect to $\rT$.
	\end{remark}

	\begin{lemma}
		$\rT_{\proc}$ implies that $P_{t_0} P_{t_1} \varphi = P_{t_0+t_1} \varphi$ for all formulas.  Similarly, if $\rT$ contains $\rT_{\operatorname{stat}}$, then this holds when $\varphi$ is a chronologically definable predicate.
	\end{lemma}
	
	\begin{proof}
		Each occurrence of $x_j$ in $\varphi$ is replaced by $x_j + z_{0,t_1,j}$ in $P_{t_1} \varphi$.  Then $x_j + z_{0,t_1,j}$ is replaced by $x_j + z_{0,t_0,j} + z_{t_0,t_0+t_1,j}$ in $P_{t_0} P_{t_1} \varphi$.  Since $\mathrm{T}_{\operatorname{proc}}^0$ requires that $z_{0,t_0,j} + z_{t_0,t_0+t_1,j} = z_{0,t_0+t_1,j}$, the overall effect is the same as if we replaced $x_j$ by $x_j + z_{0,t_0+t_1,j}$ in $P_{t_0} P_{t_1} \varphi$ which is the same as what happens in $P_{t_0+t_1} \varphi$.  The transformation of the constant symbols $z_{t,j}$ and the quantifiers proceeds in the same way as for $P_{t_0} P_{t_1}$.
		
		The second claim follows by applying the first claim to chronological formulas and then passing to the completion.
	\end{proof}

	\begin{lemma} \label{lem: gradient of heat semigroup}
		Let $\varphi$ be a $\rT_{\stat}$-chronologically definable predicate in $m+m'$ variables which is $c$-semiconvex and $c$-semiconcave in the first $m$ variables.  Then $P_{m,m',t} \varphi$ is also a chronologically definable predicate that is $c$-semiconvex and $c$-semiconcave in the first $m$ variables.
		
		Let $\cM$ be an $\cL_{\proc}$-structure satisfying $\rT_{\operatorname{stat}}$.  Let $\nabla_{m,m'} \varphi$ denote the gradient with respect to the first $m$ variables of $\varphi$ as a function on $M_0$, with respect to the real inner product $\re \ip{\cdot,\cdot}_\tau$.  Then
		\begin{equation} \label{eq: gradient of heat semigroup}
			\nabla_{m,m'} (P_t \varphi)^{\cM}(\mathbf{x},\mathbf{y}) = E_{M_0}[(\nabla_{m,m'} \varphi^{S_t \cM})(\mathbf{x} + (z_{0,t,j})_{j=1}^m,\mathbf{y})]
		\end{equation}
		for $(\mathbf{x},\mathbf{y}) \in M_0^{m+m'}$.  Moreover,
		\begin{equation} \label{eq: gradient of heat semigroup 2}
			\norm{ \nabla_{m,m'} P_{m,m',t} \varphi^{\cM}(\mathbf{x},\mathbf{y}) - \nabla_{m,m'} \varphi^{\cM}(\mathbf{x},\mathbf{y})}_2 \leq c \norm{ (z_{0,t,j})_{j=1}^m}_2 \leq c\sqrt{6mt}.
		\end{equation}
	\end{lemma}
	
	\begin{proof}
		The claim about semiconvexity and semiconcavity is immediate from the definition.
		
		Let $(\mathbf{x},\mathbf{y}) \in M_0^{m+m'}$ and $\mathbf{h} \in M_0^m$.  Let $\delta > 0$.  By stationarity, we have
		\begin{align*}
			(P_t \varphi)^{\cM}(\mathbf{x} + \delta \mathbf{h}, \mathbf{y}) &= S_t \varphi(\mathbf{x} + \delta \mathbf{h} + (z_{0,t,j})_{j=1}^m, \mathbf{y}) \\
			&= \varphi^{S_t \cM}(\mathbf{x} + \delta \mathbf{h} + (z_{0,t,j})_{j=1}^m, \mathbf{y}).
		\end{align*}
		Differentiating at $\delta = 0$, we obtain
		\begin{align*}
			\frac{d}{d\delta} \Bigr|_{\delta = 0} (P_t \varphi)^{\cM}(\mathbf{x} + \delta \mathbf{h}, \mathbf{y}) &= \re \ip*{ \nabla_{m,m'} \varphi^{S_t \cM}(\mathbf{x} + (z_{0,t,j})_{j=1}^m, \mathbf{y}), \mathbf{h} }_{\cM_t} \\
			&= \re \ip*{ E_{\cM_)}[\nabla_{m,m'} \varphi^{S_t \cM}(\mathbf{x} + (z_{0,t,j})_{j=1}^m, \mathbf{y})], \mathbf{h} }_{L^2(M_0)} 
		\end{align*}
		which implies \eqref{eq: gradient of heat semigroup}.  We also note that by stationarity,
		\[
		\varphi^{\cM}(\mathbf{x} + \delta \mathbf{h}, \mathbf{y}) = (S_t \varphi)^{\cM}(\mathbf{x} + \delta \mathbf{h}, \mathbf{y}) = \varphi^{S_t \cM}(\mathbf{x} + \delta \mathbf{h}, \mathbf{y}),
		\]
		and hence by differentiating,
		\[
		\re \ip{ \nabla_{m,m'} \varphi^{\cM}(\mathbf{x}, \mathbf{y}), \mathbf{h} } = \re \ip{ \nabla_{m,m'} \varphi^{S_t \cM}(\mathbf{x}, \mathbf{y}), \mathbf{h} },
		\]
		whence
		\[
		\nabla_{m,m'} \varphi^{\cM}(\mathbf{x},\mathbf{y}) = \nabla_{m,m'} \varphi^{S_t \cM}(\mathbf{x}, \mathbf{y}).
		\]
		Since $\varphi$ is $c$-semiconvex and $c$-semiconcave, its gradient is $c$-Lipschitz, and hence
		\begin{align*}
			\lVert \nabla_{m,m'} P_{m,m',t} \varphi^{\cM}(\mathbf{x},\mathbf{y}) & - \nabla_{m,m'} \varphi^{\cM}(\mathbf{x},\mathbf{y}) \rVert_2 \\
			&=  \norm*{ E_{M_0} \left[ \nabla_{m,m'} \varphi^{S_t \cM}(\mathbf{x} + (z_{0,t,j})_{j=1}^m, \mathbf{y}) - \nabla_{m,m'} \varphi^{S_t \cM}(\mathbf{x} + (z_{0,t,j})_{j=1}^m, \mathbf{y}) \right] }_2 \\
			&\leq \norm*{ \nabla_{m,m'} \varphi^{S_t \cM}(\mathbf{x} + (z_{0,t,j})_{j=1}^m, \mathbf{y}) - \nabla_{m,m'} \varphi^{S_t \cM}(\mathbf{x} + (z_{0,t,j})_{j=1}^m, \mathbf{y}) }_2 \\
			&\leq c \norm{ (z_{0,t,j})_{j=1}^m}_2 \\
			&\leq c \sqrt{6mt},
		\end{align*}
		which shows \eqref{eq: gradient of heat semigroup 2}.
	\end{proof}

	\subsection{Infinitesimal change of variables} \label{subsec: change of variables}
	
	For a chronologically definable predicate $\varphi$ with respect to $\rT_{\stat}$, let
	\begin{equation} \label{eq: heat averaged predicate}
		\psi_t^{\cM}(\mathbf{x},\mathbf{y}) = \frac{1}{t} \int_0^t P_s \varphi(\mathbf{x},\mathbf{y})\,ds.
	\end{equation}
	To see that this is well-defined, we claim that $s \mapsto P_s \varphi$ is continuous with respect to each norm $\norm{\cdot}_{0,(\mathbf{r},\mathbf{r}'),\rT}$.  Because $P_t$ is a semigroup, it suffices to prove the claim when $s = 0$.  Indeed, given $\varepsilon > 0$, there exists $\delta > 0$ such that if $\cM$ satisfies $\rT$, then
	\[
	x, x' \in D_{0,\mathbf{r}+3\sqrt{s}}^{\cM} \text{ and } \mathbf{y} \in D_{0,\mathbf{r}'}^{\cM} \text{ and } \max_j d(x_j,x_j') < \delta \implies |\varphi(\mathbf{x},\mathbf{y}) - \varphi(\mathbf{x}',\mathbf{y})| \leq \varepsilon.
	\]
	Note that $S_s \cM$ also satisfies $\rT_{\stat}$, and therefore, if $6 \sqrt{s} < \delta$, then 
	\[
	|\varphi^{S_s \cM}(\mathbf{x} + (z_{0,s,j})_{j=1}^m, \mathbf{y}) - \varphi^{S_s \cM}(\mathbf{x},\mathbf{y})| \leq \varepsilon.
	\]
	Therefore, the integral \eqref{eq: heat averaged predicate} makes sense as a Riemann integral in the Fr{\'e}chet space of chronologically definable predicates.
	
	\begin{lemma}
		With the setup above, for any $\cQ$ satisfying $\rT_{\stat}$, we have
		\[
		\norm{\nabla_{m,m'} \psi_t^{\cQ}(\mathbf{x},\mathbf{y}) - \nabla_{m,m'} \varphi^{\cM}(\mathbf{x},\mathbf{y})}_2 \leq \frac{2}{\sqrt{3}} c(mt)^{1/2}.
		\]
	\end{lemma}
	
	\begin{proof}
		Now because $\varphi$ is $c$-semiconvex and $c$-semiconcave, Lemma \ref{lem: gradient of heat semigroup} implies that the same is true for $P_s \varphi$, and also $s \mapsto \nabla_{m,m'} P_s \varphi$ is continuous with respect to the relevant Fr{\'e}chet norms.  It follows that
		\[
		\nabla_{m,m'} \psi_t = \frac{1}{t} \int_0^t \nabla_{m,m'} P_s \varphi\,ds.
		\]
		Furthermore, using \eqref{eq: gradient of heat semigroup 2},
		\begin{align*}
			\norm{\nabla \psi_t^{\cQ}(\mathbf{x},\mathbf{y}) - \nabla \varphi^{\cM}(\mathbf{x},\mathbf{y})}_2 &= \norm*{\frac{1}{t} \int_0^t [\nabla P_s \varphi^{\cQ}(\mathbf{x},\mathbf{y}) - \nabla \varphi^{\cM}(\mathbf{x},\mathbf{y})]\,ds }_2 \\
			&\leq \frac{1}{t} \int_0^t \norm*{\nabla P_s \varphi^{\cQ}(\mathbf{x},\mathbf{y}) - \nabla \varphi^{\cM}(\mathbf{x},\mathbf{y})}_2 \,ds \\
			&\leq \frac{1}{t} \int_0^t \sqrt{6ms} \,ds \\
			&= \frac{2}{\sqrt{3}} c(mt)^{1/2}.
		\end{align*}
	\end{proof}
	
	Our next goal is to define matrix approximations $\Lambda_{\psi_t}^{(n)}$ for $\psi_t$ and compute their Laplacian.  The main point of the construction is that
	\[
	\Delta \int_0^t e^{s \Delta} u\,ds = e^{t\Delta} u - u,
	\]
	and so the Laplacian will be well-controlled.  We first need to establish notation for the Hessian and Laplacian on $(\mathbb{M}_n)_{\sa}^m$, taking into account the normalization of the traces / inner products.
	
	Let $H_{m,m'}^{(n)}$ denote the Hessian operator with respect to the inner product $\ip{\cdot,\cdot}_{\tr_n}$ on $(\mathbb{M}_n)_{\sa}^{m+m'}$ with respect to the first $m$ variables.  In other words, for real-valued functions on $(\mathbb{M}_n)_{\sa}^{m+m'}$ that are smooth in the first $m$ variables, $H_{m,m'}^{(n)} u(\mathbf{X},\mathbf{Y})$ is the linear transformation on $(\mathbb{M}_n)_{\sa}^m$ satisfying
	\[
	u(\mathbf{X}+\mathbf{V},\mathbf{Y}) = u(\mathbf{X},\mathbf{Y}) + \ip{\mathbf{V}, \nabla u(\mathbf{X},\mathbf{Y})}_{\tr_n} + \frac{1}{2} \ip{\mathbf{V},H_{m,m'}^{(n)}(\mathbf{X},\mathbf{Y}) [\mathbf{V}]}_{\tr_n} + o(\norm{\mathbf{V}}_2^2).
	\]
	Let $\Delta_{m,m'}^{(n)}$ be the normalized Laplacian on $(\mathbb{M}_n)_{\sa}^{m+m'}$ with respect to the first $m$ variables, given by
	\[
	\Delta_{m,m'}^{(n)} u(\mathbf{X},\mathbf{Y}) = \frac{1}{n^2} \Tr[H_{m,m'}^{(n)} u(\mathbf{X},\mathbf{Y})] = 2 \lim_{\varepsilon \to 0} \frac{1}{\varepsilon} [u(\mathbf{X} + (Z_{0,\varepsilon,j}^{(n)})_{j=1}^m, \mathbf{Y}) - u(\mathbf{X},\mathbf{Y})] 
	\]
	for smooth $u$.  For smooth $u$ with (for instance) polynomial growth at $\infty$, let
	\[
	P_t u(\mathbf{X},\mathbf{Y}) = \mathbb{E}[u(\mathbf{X} + (Z_{0,t,j}^{(n)})_{j=1}^m, \mathbf{Y})],
	\]
	where $\mathbf{X}$ and $\mathbf{Y}$ are deterministic matrix tuples.  Then
	\begin{equation} \label{eq: Laplacian integration identity}
		\frac{1}{2} \Delta_{m,m'} \int_0^t P_s u \,ds = \frac{1}{2} \int_0^t \Delta_{m,m'} P_s u\,ds = \int_0^t \frac{d}{ds} P_s u\,ds = P_t u - u.
	\end{equation}
	It is also easy to see that the equality extends to merely \emph{continuous} functions with polynomial growth because the right-hand side is well-defined for such functions.
	
	We now proceed with the construction of the finite-dimensional approximations of our functions.
	
	\begin{lemma} \label{lem: FD approximation and heat averaging}
		Let $\varphi$ and $\Lambda_{\varphi}^{(n)}$ be as in Lemma \ref{lem: FD approximation of SCSC}.  For self-adjoint matrix tuples, let
		\begin{equation} \label{eq: definition of Lambda psi t}
			\Lambda_{\psi_t}^{(n)}(\mathbf{X},\mathbf{Y}) = \int_0^1 \mathbb{E} \Lambda_{\varphi}^{(n)}(\mathbf{X} + [F_{\sqrt{t}}(Z_{0,s,j}^{(n)})]_{j=1}^m, \mathbf{Y})\,ds,
		\end{equation}
		and let
		\[
		\Lambda_{P_t \varphi}^{(n)}(\mathbf{X},\mathbf{Y}) = \mathbb{E} \Lambda_{\varphi}^{(n)}(\mathbf{X} + [F_{\sqrt{t}}(Z_{0,t,j}^{(n)})]_{j=1}^m, \mathbf{Y}).
		\]
		Then $\Lambda_{\psi_t^{(n)}}$ is $c$-semiconvex and $c$-semiconcave and
		\begin{equation} \label{eq: Laplacian of FD approximation}
			\frac{1}{2} \Delta_{m,m'} \Lambda_{\psi_t}^{(n)} = \frac{1}{t}\left[\Lambda_{P_t \varphi}^{(n)}(\mathbf{X},\mathbf{Y}) - \Lambda_{\varphi}^{(n)}(\mathbf{X},\mathbf{Y}) \right].
		\end{equation}
		Moreover, for self-adjoint tuples $\mathbf{X}^{(n)}, \mathbf{Y}^{(n)}$ from $L^\infty(\Omega,\mathcal{F}_0,\mathbb{P}) \otimes \mathbb{M}_n$ where each coordinate is bounded in operator norm, letting $\mathbf{x}$ and $\mathbf{y}$ be the corresponding elements in the matrix ultraproduct $\cM$ from \eqref{eq: random matrix ultraproduct}, we have
		\begin{align}
			(\psi_t)_E^{\cM}(\mathbf{x},\mathbf{y}) &= [\Lambda_{\psi_t}^{(n)}(\mathbf{X}^{(n)},\mathbf{Y}^{(n)})]_{n \in \N} \label{eq: FD approximation for psi t} \\
			(P_t \varphi)_E^{\cM}(\mathbf{x},\mathbf{y}) &= [\Lambda_{P_t \varphi}^{(n)}(\mathbf{X}^{(n)},\mathbf{Y}^{(n)})]_{n \in \N}. \label{eq: FD approximation for P t phi}
		\end{align}
	\end{lemma}
	
	\begin{proof}
		The semiconvexity and semiconcavity of $\Lambda_{\psi_t}^{(n)}$ is clear because it is an average of $c$-semiconvex and $c$-semiconcave functions.
		
		We obtain \eqref{eq: Laplacian of FD approximation} directly from applying \eqref{eq: Laplacian integration identity} to the definition of $\Lambda_{\psi_t}^{(n)}$ in \eqref{eq: definition of Lambda psi t}.
		
		Next, take random matrix tuples $\mathbf{X}^{(n)}$ and $\mathbf{Y}^{(n)}$ satisfying the hypotheses.  As in the proof of Proposition \ref{prop: approx by pointwise function}, since $\Lambda$ is bounded and $\norm{\cdot}_1$-Lipschitz on each operator norm ball, the Poincar{\'e} inequality \eqref{eq: Poincare inequality} implies that
		\[
		[\Lambda_{\varphi}^{(n)}(\mathbf{X}^{(n)} + [F_{\sqrt{t}}(Z_{0,s,j}^{(n)})]_{j=1}^m, \mathbf{Y})]_{n \in \N} = [\mathbb{E} [\Lambda_{\varphi}^{(n)}(\mathbf{X}^{(n)} + [F_{\sqrt{t}}(Z_{0,s,j}^{(n)})]_{j=1}^m, \mathbf{Y}) \mid \mathcal{F}_0] ]_{n \in \N}.
		\]
		Next, note that for $s \leq t$, we have
		\[
		\norm{F_{3\sqrt{t}}(Z_{0,s}^{(n)}) - F_{3\sqrt{s}}(Z_{0,s}^{(n)})}_{L^2} \leq \norm{Z_{0,s}^{(n)} - F_{3\sqrt{s}}(Z_{0,s}^{(n)})}_{L^2} \to 0
		\]
		by \eqref{eq: truncated GUE matrix 2}.  Therefore,
		\begin{align*}
			[\mathbb{E} [\Lambda_{\varphi}^{(n)}(\mathbf{X}^{(n)} + [F_{3\sqrt{t}}(Z_{0,s,j}^{(n)})]_{j=1}^m, \mathbf{Y}) \mid \mathcal{F}_0] ]_{n \in \N} &= 
			[\Lambda_{\varphi}^{(n)}(\mathbf{X} + [F_{3\sqrt{t}}(Z_{0,s,j}^{(n)})]_{j=1}^m, \mathbf{Y})]_{n \in \N} \\
			&= [\Lambda_{\varphi}^{(n)}(\mathbf{X} + [F_{3\sqrt{s}}(Z_{0,s,j}^{(n)})]_{j=1}^m, \mathbf{Y})]_{n \in \N} \\
			&= (S_s\varphi)_E^{\cM}(\mathbf{x} + (z_{0,s,j})_{j=1}^m, \mathbf{y}) \\
			&= (P_s \varphi)_E^{\cM}(\mathbf{x},\mathbf{y}).
		\end{align*}
		This also shows \eqref{eq: FD approximation for P t phi} by substituting $s = t$.  Now recall that
		\[
		\Lambda_{\psi_t}^{(n)}(\mathbf{X}^{(n)},\mathbf{X}^{(n)}) = \frac{1}{t} \int_0^t \Lambda_{\varphi}^{(n)}(\mathbf{X}^{(n)} + [F_{3\sqrt{t}}(Z_{0,s,j}^{(n)})]_{j=1}^m, \mathbf{Y}) \mid \mathcal{G}_0]\,ds.
		\]
		Using the continuity estimates for the integrand as a function on an operator norm ball, which are uniform in $n$ and $s$, one can exchange the Riemann integral and the ultraproduct, so that
		\begin{align*}
			(\psi_t)_E^{\cM})(\mathbf{x},\mathbf{y}) &= \frac{1}{t} \int_0^t P_s \varphi(\mathbf{x},\mathbf{y})\,ds \\
			&= \frac{1}{t} \int_0^t [\Lambda_{\varphi}^{(n)}(\mathbf{X}^{(n)} + [F_{3\sqrt{t}}(Z_{0,s,j}^{(n)})]_{j=1}^m, \mathbf{Y}) \mid \mathcal{G}_0]]_{n \in \N}\,ds \\
			&= [\Lambda_{\psi_t}^{(n)}(\mathbf{X}^{(n)},\mathbf{Y}^{(n)})]_{n \in \N}.
		\end{align*}
		This proves \eqref{eq: FD approximation for psi t}.
	\end{proof}
	
	\begin{theorem} \label{thm: infinitesimal change of variables}
		Let $m \in \N$ and $m' \in \N \cup \{0,\infty\}$.  Let $|a| < c$ be real numbers.  Let $\varphi$ be a chronologically definable predicate in $m + m'$ variables with respect to $\rT_{\stat}$ which is $c$-semiconvex and $c$-semiconcave in the first $m$ variables, and such that $\varphi - aq$ is Lipschitz in $\norm{\cdot}_1$ with respect to the first $m$ variables.  Let $(\mathbf{x},\mathbf{y})$ be a self-adjoint $m+m'$-tuple from $Q_0$ in the matrix ultraproduct quotient structure $\cQ$ from \eqref{eq: random matrix quotient}. Then for $\varepsilon \in [0,1/2c]$,
		\[
		\chi_{\chron}^{\cU}(\mathbf{x} + \varepsilon \nabla_{m,m'} \varphi^{\cQ}(\mathbf{x}, \mathbf{y}) \mid \mathbf{y}) \geq \chi_{\chron}^{\cU}(\mathbf{x} \mid \mathbf{y}) + \varepsilon \overline{\Delta}_{m,m'} \varphi^{\cQ}(\mathbf{x},\mathbf{y}) - 4m(c \varepsilon)^2.
		\]
	\end{theorem}
	
	\begin{proof}
		By Proposition \ref{prop: classical-free variational principle}, fix random matrix models $(\mathbf{X}^{(n)},\mathbf{Y}^{(n)})$ such that each coordinate is bounded in operator norm, $\mathbf{Y}^{(n)}$ is deterministic, and for restricted chronological formulas $\omega$,
		\[
		\lim_{n \to \cU} \Lambda_\omega^{(n)}(\mathbf{X}^{(n)},\mathbf{Y}^{(n)}) = \omega^{\cQ}(\mathbf{x},\mathbf{y}) \text{ in probability},
		\]
		and the classical entropy of $\mathbf{X}^{(n)}$ converges to $\chi_{\chron}^{\cU}(\mathbf{x} \mid \mathbf{y})$.  Let $\mathbf{x}_0 = [\mathbf{X}^{(n)}]_{n \in \N}$ and $\mathbf{y}_0 = [\mathbf{Y}^{(n)}]_{n \in \N}$ be the corresponding self-adjoint tuples in $\cM$.  Note that $(\pi(\mathbf{x}_0),\pi(\mathbf{y}_0))$ has the same chronological type as the original $(\mathbf{x},\mathbf{y})$.
		
		Let $\Lambda_{\varphi}^{(n)}$ be a sequence of finite-dimensional approximations for $\varphi$ given by Lemma \ref{lem: FD approximation of SCSC}, which are $c$-semiconvex and $c$-semiconcave.  Let $\psi_t$ be given by \eqref{eq: heat averaged predicate}.  Let $\Lambda_{\psi_t}^{(n)}$ and $\Lambda_{P_t \varphi}^{(n)}$ be the finite-dimensional approximations given by Lemma \ref{lem: FD approximation and heat averaging}.  By the same lemma,
		\[
		[\Lambda_{\psi_t}^{(n)}(\mathbf{X}^{(n)},\mathbf{Y}^{(n)})]_{n \in \N} = (\psi_t)_E^{\cM}(\mathbf{x}_0,\mathbf{y}_0) = 1_{Z(\cM)} \psi_t^{\cQ}(\mathbf{x},\mathbf{y}).
		\]
		The same applies to $q + \varepsilon \psi_t$ where $q(\mathbf{x},\mathbf{y}) = \norm{\mathbf{x}}_2^2$.  We therefore also have by Lemma \ref{lem: FD approximation of gradient 2} that for restricted chronological formulas $\omega$,
		\begin{align*}
			[\Lambda_\omega^{(n)}(\mathbf{X}^{(n)} & + \varepsilon \nabla_{m,m'} \Lambda_{\psi_t}^{(n)}(\mathbf{X}^{(n)},\mathbf{Y}^{(n)}),\mathbf{Y}^{(n)})]_{n \in \N} \\
			&= [\omega \circ (\id_m + \varepsilon \nabla_{m,m'}\varphi, \id_{m'})]_E^{\cM} (\mathbf{x}_0 + \varepsilon \nabla_{m,m'} \psi_t^{\cQ}(\mathbf{x}_0,\mathbf{y}_0),\mathbf{y}_0) \\
			&= 1_{Z(\cM)} \omega^{\cQ}( \nabla_{m,m'}\varphi^{\cQ}(\mathbf{x},\mathbf{y}),\mathbf{y}),
		\end{align*}
		where $\id_{m'}$ denotes the identity function in $m'$-many variables.  Hence, by Proposition \ref{prop: convergence of type},
		\[
		\lim_{n \to \cU} \Lambda_\omega^{(n)}(\nabla_{m,m'} \Lambda_{\psi_t}^{(n)}(\mathbf{X}^{(n)},\mathbf{Y}^{(n)}), \mathbf{Y}^{(n)}) = \omega^{\cQ}(\mathbf{x} + \varepsilon \nabla_{m,m'}\varphi^{\cQ}(\mathbf{x},\mathbf{y}),\mathbf{y}) \text{ in probability.}
		\]
		Therefore, by Proposition \ref{prop: classical-free variational principle},
		\[
		\chi_{\chron}^{\cU}(\nabla_{m,m'} \psi^{\cQ}(\mathbf{x},\mathbf{y}) \mid \mathbf{y}) \geq \lim_{n \to \cU} h^{(n)}(\mathbf{X}^{(n)}),
		\]
		where $h^{(n)}$ is the normalized classical entropy.
		
		We now apply the change of variables for classical entropy.  Since $c \varepsilon \leq 1/2$, we have that $\Lambda_{\psi_t}^{(n)}$ is twice differentiable almost everywhere and
		\[
		\frac{1}{2} \leq I + \varepsilon H_{m,m'}^{(n)} \Lambda_{\psi_t}^{(n)}(\mathbf{X},\mathbf{Y}) \leq \frac{3}{2}.
		\]
		By the change of variables for classical entropy
		\[
		h^{(n)}(\mathbf{X}^{(n)} + \varepsilon \nabla_{m,m'} \Lambda_{\psi_t}^{(n)}(\mathbf{X}^{(n)},\mathbf{Y}^{(n)})) = h^{(n)}(\mathbf{X}^{(n)}) + \frac{1}{n^2} \mathbb{E} \Tr \log(I + \varepsilon H_{m,m'}^{(n)} \Lambda_{\psi_t}^{(n)}(\mathbf{X}^{(n)},\mathbf{Y}^{(n)})).
		\]
		We invoke Taylor expansion for the logartithm function at $1$ to deduce that for $z \in [-1/2,1/2]$,
		\[
		|\log(1 + z) - z| \leq \frac{1}{2!} |z|^2 \sup_{w \in [-1/2,1/2]} \left| \frac{-1}{(1+w)^2} \right| = 2 |z|^2.
		\]
		Hence,
		\begin{align*}
			\frac{1}{n^2} \mathbb{E} \Tr \log(I + \varepsilon H_{m,m'}^{(n)} \Lambda_{\psi_t}^{(n)}(\mathbf{X}^{(n)},\mathbf{Y}^{(n)})) &\geq \frac{1}{n^2} \mathbb{E} \Tr(\varepsilon H_{m,m'}^{(n)} \Lambda_{\psi_t}^{(n)}(\mathbf{X}^{(n)},\mathbf{Y}^{(n)})) - 4m(c \varepsilon)^2 \\
			&= \mathbb{E} \Delta_{m,m'}^{(n)} \Lambda_{\psi_t}^{(n)}(\mathbf{X}^{(n)},\mathbf{Y}^{(n)})) - 4 m(c \varepsilon)^2.
		\end{align*}
		Now we plug in \eqref{eq: Laplacian of FD approximation} to obtain
		\[
		h^{(n)}(\mathbf{X}^{(n)} + \varepsilon \nabla_{m,m'} \Lambda_{\psi_t}^{(n)}(\mathbf{X}^{(n)},\mathbf{Y}^{(n)})) \geq h^{(n)}(\mathbf{X}^{(n)}) + \frac{2 \varepsilon}{t} \mathbb{E} \left[ \Lambda_{P_t \varphi}^{(n)}(\mathbf{X}^{(n)},\mathbf{Y}^{(n)}) - \Lambda_{\varphi}^{(n)}(\mathbf{X}^{(n)},\mathbf{Y}^{(n)}) \right] - 4m(c \varepsilon)^2.
		\]
		Taking the limit as $n \to \cU$, and using Proposition \ref{prop: classical-free variational principle} again,
		\begin{align*}
			\chi_{\chron}^{\cU}(\mathbf{x} + \varepsilon \nabla_{m,m'} \psi_t^{\cQ}(\mathbf{x},\mathbf{y}) \mid \mathbf{y}) &\geq \lim_{n \to \cU} h^{(n)}(\mathbf{X}^{(n)} + \varepsilon \nabla_{m,m'} \Lambda_{\psi_t}^{(n)}(\mathbf{X}^{(n)},\mathbf{Y}^{(n)})) \\
			&\geq \lim_{n \to \cU} h^{(n)}(\mathbf{X}^{(n)}) + \lim_{n \to \cU} \frac{\varepsilon}{t} \mathbb{E} \left[ \Lambda_{P_t \varphi}^{(n)}(\mathbf{X}^{(n)},\mathbf{Y}^{(n)}) - \Lambda_{\varphi}^{(n)}(\mathbf{X}^{(n)},\mathbf{Y}^{(n)}) \right] - 4m(c \varepsilon)^2 \\
			&= \chi_{\chron}^{\cU}(\mathbf{x} \mid \mathbf{y}) + \frac{2\varepsilon}{t} [P_t \varphi^{\cQ}(\mathbf{x},\mathbf{y}) - \varphi^{\cQ}(\mathbf{x},\mathbf{y})] - 4m(c \varepsilon)^2.
		\end{align*}
		Then we take the $\limsup$ as $t \to 0^+$.  Recall that $\nabla_{m,m'} \psi_t^{\cQ}(\mathbf{x}, \mathbf{y}) \to \nabla_{m,m'} \varphi^{\cQ}(\mathbf{x},\mathbf{y})$ as $t \to 0$ by \eqref{eq: gradient of heat semigroup 2}.  Note that $\nabla_{m,m'} \psi_t(\mathbf{x})$ is uniformly bounded in operator norm, and $\chi_{\chron}^{\cU}$ is upper semicontinuous on the type space $\mathbb{S}_{\mathbf{r}}$ for each $\mathbf{r} \in (0,\infty)^{m+m'}$ as a consequence of Corollary \ref{cor: semicontinuity of entropy}.  Therefore,
		\begin{align*}
			\chi_{\chron}^{\cU}(\mathbf{x} + \varepsilon \nabla_{m,m'} \varphi^{\cQ}(\mathbf{x},\mathbf{y}) \mid \mathbf{y}) &\geq \limsup_{t \to 0^+} \chi_{\chron}^{\cU}(\mathbf{x} + \varepsilon \nabla_{m,m'} \psi_t^{\cQ}(\mathbf{x},\mathbf{y}) \mid \mathbf{y}) \\
			&\geq \chi_{\chron}^{\cU}(\mathbf{x} \mid \mathbf{y}) + \varepsilon \limsup_{t \to 0^+} \frac{2}{t} [P_t \varphi^{\cQ}(\mathbf{x},\mathbf{y}) - \varphi^{\cQ}(\mathbf{x},\mathbf{y})] - 4m(c \varepsilon)^2 \\
			&= \chi_{\chron}^{\cU}(\mathbf{x} \mid \mathbf{y}) + \varepsilon \overline{\Delta}_{m,m'} \psi^{\cQ}(\mathbf{x},\mathbf{y}) - 4m(c \varepsilon)^2,
		\end{align*}
		which is what we wanted to prove.
	\end{proof}
	
	\section{Entropy and Wasserstein geodesics} \label{sec: info geom}
	
	In this section, we will study the Wasserstein distance on spaces of chronological types, prove an analog of Monge--Kantorovich duality for the Wasserstein distance, and prove geodesic concavity and the evolution variational inequality for $\chi_{\chron}^{\cU}$.  In \cite{Jekel2025InfoGeom}, we studied the behavior of entropy for types in the language of tracial von Neumann algebras, and showed the entropy in that setting is locally Lipschitz along Wasserstein geodesics.  The setting of chronological types in this paper makes the study of entropy more tractable and allows us to obtain geodesic concavity.  In fact, this works more generally for the conditional entropy $\chi_{\chron}^{\cU}(\mathbf{x} \mid \mathbf{y})$ in $\mathbf{x}$.
	
	This is defined by working over a space of chronological types for $\mathbf{x}$ relative to a fixed $\mathbf{y}$.  The \emph{type of $\mathbf{x}$ over $\mathbf{y}$} sends a formula $\varphi$ in $m + m'$ variables to $\varphi^{\cM}(\mathbf{x},\mathbf{y})$.  We remark that formulas in $\mathbf{x}$ and $\mathbf{y}$ can be viewed as formulas in $\mathbf{x}$ in an expanded language with constant symbols added for $\mathbf{y}$, and this is the usual perspective taken in model theory.
	
	Before proving geodesic concavity, we first describe the Wasserstein distance for the types relative to $\mathbf{y}$, and generalize the Monge--Kantorovich duality from \cite{JekelTypeCoupling} to this setting.  The proof of concavity combines this optimal transport machinery with the general infinitesimal change of variables for entropy in Theorem \ref{thm: infinitesimal change of variables}.
	
	\subsection{The space of chronological types over $\mathbf{y}$} \label{subsec: conditional types and distance}
	
	Although it is more proper to associate a type space to a theory $\rT$ rather than a particular structure (hence the notation $\mathbb{S}_{\chron,m}(\rT)$ above), it will be convenient here to work in a (sufficiently large) fixed structure throughout.  This approach will probably be more intuitive to non-model theorists than fixing another theory $\rT$ in the language $\cL_{\chron,\mathbf{y}}$ expanded by constant symbols.
	
	Hence, throughout the discussion, we will work with a fixed $\cL_{\proc}$-structure $\cM$ which is \emph{countably saturated}.  Countable saturation is a condition that guarantees that if some relations in a countable family of variables can be approximately satisfied in $\cM$, then they can be exactly satisfied in $\cM$.  We remark that ultraproducts usually have countable saturation, which can be proved using a diagonalization argument.  See \cite[\S 2.6]{JekelTypeCoupling}.
	
	\begin{definition}
		An $\cL_{\proc}$-structure $\cM$ is said to be \emph{countably saturated} if the following holds:  Let $m, m' \in \N \cup \{\infty\}$.  Let $\Phi$ be a collection of $\cL_{\proc}$-formulas in $m$ variables.  Let $\mathbf{r} \in (0,\infty)^m$ and let $\mathbf{y} \in M_0^{m'}$.  Suppose that for every finite $F \subseteq \Phi$ and $\varepsilon > 0$, there exists $\mathbf{x} \in \prod_{j=1}^m D_{0,r_j}^{\cM}$ such that
		\[
		\max_{\varphi \in F} |\varphi(\mathbf{x},\mathbf{y})| < \varepsilon.
		\]
		Then there exists $\mathbf{x} \in \prod_{j=1}^m D_{0,r_j}^{\cM}$ such that
		\[
		\forall \varphi \in \Phi, \varphi^{\cM}(\mathbf{x},\mathbf{y}) = 0.
		\]
	\end{definition}
	
	\begin{remark}
		Countable saturation is related to compactness because it is saying that if the collection of sets $\{ \mathbf{x}: |\varphi(\mathbf{x},\mathbf{y})| \leq \varepsilon \}$ for $\varphi \in \Phi$ and $\varepsilon > 0$ have the finite intersection property, then their total intersection is nonempty.  See also \cite[Proposition 2.28]{JekelTypeCoupling}.
	\end{remark}
	
	\begin{definition} \label{def: conditional chronological type}
		Fix a countably saturated $\cL_{\proc}$-structure $\cM$ satisfying $\rT_{\stat}$.  Let $m \in \N$ and $m' \in \N \cup \{0,\infty\}$.  Let $\mathbf{y} \in M_0^{m'}$.  We define the space of \emph{chronological types in $\cM$ over $\mathbf{y}$} as
		\[
		\mathbb{S}_{\chron,m}(\cM / \mathbf{y}) = \{ \tp_{\chron}(\mathbf{x},\mathbf{y}): \mathbf{x} \in M_0^m \}.
		\]
		We equip $\mathbb{S}_{\chron,m}(\cM / \mathbf{y})$ with the subspace topology inherited from $\mathbb{S}_{\chron,m+m'}(\Th(\cM))$.  Similarly, for $\mathbf{r} \in (0,\infty)^m$, let
		\[
		\mathbb{S}_{\chron,\mathbf{r}}(\cM / \mathbf{y}) =  \{ \tp(\mathbf{x},\mathbf{y}): \mathbf{x} \in \prod_{j=1}^m D_{0,r_j}^{\cM} \}.
		\]
	\end{definition}
	
	\begin{lemma} \label{lem: conditional chronological type compactness}
		With the setup of Definition \ref{def: conditional chronological type}, the space $\mathbb{S}_{\chron,m,\mathbf{r}}(\cM / \mathbf{y})$ is compact.
	\end{lemma}
	
	\begin{proof}
		The topology on $\mathbb{S}_{\chron,m,\mathbf{r}}(\cM / \mathbf{y})$ is the weak-$*$ topology induced by the evaluation of formulas.  For each formula $\varphi$, there is some bound $M_\varphi$ such that $|\varphi^{\cM}(\mathbf{x}, \mathbf{y})| \leq M_\varphi$ for $\mathbf{x} \in \prod_{j=1}^m D_{0,r_j}^{\cM}$.  Hence, $\mathbb{S}_{\chron,m,\mathbf{r}}(\cM / \mathbf{y})$ can be identified with subspace of $\prod_{\varphi \in \mathcal{F}_{\chron,m+m'}} [-M_\varphi,M_\varphi]$ with the product topology, which is compact by Tychonoff's theorem.  So it suffices to show that $\mathbb{S}_{\chron,m,\mathbf{r}}(\cM / \mathbf{y})$ is closed.  Choose a point $\mu$ in the closure; then for any finite collection $F \subseteq \mathcal{F}_{\chron,m+m'}$, there exists $\mathbf{x} \in \prod_{j=1}^m D_{0,r_j}^{\cM}$ with $|\varphi^{\cM}(\mathbf{x},\mathbf{y}) - \mu_{\varphi}| < \varepsilon$ for $\varphi \in F$.  By applying the definition of countable saturation to $\Phi = \{ \varphi - (\mu,\varphi): \varphi \in \mathcal{F}_{\chron,m+m'}\}$, there exists some $\mathbf{x} \in \prod_{j=1}^m D_{0,r_j}^{\cM}$ with $\varphi^{\cM}(\mathbf{x},\mathbf{y}) = \mu_\varphi$ for all $\varphi$, so $\mathbb{S}_{\chron,m,\mathbf{r}}(\cM / \mathbf{y})$ is closed as desired.
	\end{proof}
	
	\begin{definition}[Wasserstein distance and optimal couplings] \label{def: Wasserstein distance}
		Fix a countably saturated $\cL_{\proc}$-structure $\cM$ satisfying $\rT_{\stat}$.  Let $m \in \N$ and $m' \in \M \cup \{0,\infty\}$.  Let $\mathbf{y} \in M_0^{m'}$.  For types $\mu_0, \mu_1 \in \mathbb{S}_{\chron,m,\mathbf{r}}(\cM / \mathbf{y})$, define
		\begin{equation} \label{eq: definition of Wasserstein}
			d_W(\mu_0,\mu_1) = \inf \{ \norm{\mathbf{x}_0 - \mathbf{x}_1}_2: \mathbf{x}_0, \mathbf{x}_1 \in M_0^m, \tp_{\chron}(\mathbf{x}_j, \mathbf{y}) = \mu_j \}
		\end{equation}
		and
		\begin{equation} \label{eq: definition of coupling constant}
			C(\mu_0,\mu_1) = \sup \{ \re \ip{\mathbf{x}_0, \mathbf{x}_1}: \mathbf{x}_0, \mathbf{x}_1 \in M_0^m, \tp_{\chron}(\mathbf{x}_j, \mathbf{y}) = \mu_j \}.
		\end{equation}
		We say that $(\overline{x}_0,\overline{x}_1)$ is an \emph{optimal coupling} $(\mu_0,\mu_1)$ if it achieves the infimum in \eqref{eq: definition of Wasserstein} or equivalently achieves the supremum in \eqref{eq: definition of coupling constant}.
	\end{definition}
	
	Just as for all the other versions of $L^2$-Wasserstein distance, the equivalence above follows because $\norm{\mathbf{x}_0 - \mathbf{x}_1}_2^2 = \norm{\mathbf{x}_0}_2^2 + \norm{\mathbf{x}_1}_2^2 - 2 \re \ip{\mathbf{x}_0, \mathbf{x}_1}$, and $\norm{\mathbf{x}_0}_2^2$ and $\norm{\mathbf{x}_1}_2^2$ are uniquely determined by the types $\mu_0$ and $\mu_1$.  We remind the reader that the Wasserstein distance does not necessarily generate the weak-$*$ topology on $\mathbb{S}_{\chron,m,\mathbf{r}}(\cM / \mathbf{y})$.  While we will not discuss the chronological case in detail here, see \cite[\S 5.5-5.6]{GJNS2021} for the quantifier-free case and \cite[Proposition 2.4.9]{AGKE2022} for the case of full types in $\mathcal{L}_{\tr}$.
	
	We remark next that optimal couplings always exist, which is a nice example application of countable saturation.
	
	\begin{lemma} \label{lem: existence of optimal couplings}
		Fix a countably saturated $\cL_{\proc}$-structure $\cM$ satisfying $\rT_{\stat}$.  Let $m \in \N$ and $m' \in \M \cup \{0,\infty\}$.  Let $\mathbf{y} \in M_0^{m'}$.  For types $\mu_0, \mu_1 \in \mathbb{S}_{\chron,m,\mathbf{r}}(\cM / \mathbf{y})$.  For every $\mathbf{x}_0 \in M_0^m$ with $\tp_{\chron}^{\cM}(\mathbf{x}_0,\mathbf{y}) = \mu_0$, there exists $\mathbf{x}_1 \in M_0^m$ such that $(\mathbf{x}_0,\mathbf{x}_1)$ is an optimal coupling of $(\mu_0,\mu_1)$.
	\end{lemma}
	
	\begin{proof}
		First, we show that some optimal coupling exists without the assumption that $\mathbf{x}_0$ is fixed.  Let $\mathbf{r}'$ such that $\mathbf{y} \in D_{0,\mathbf{r}'}^{\cM}$.  Since $\mathbb{S}_{\chron,(\mathbf{r},\mathbf{r}')}$ is metrizable in the weak-$*$ topology by Lemma \ref{lem: separability}, there exists some chronologically definable predicates $\eta_0$, $\eta_1 \geq 0$ such that for $\mathbf{x} \in D_{0,\mathbf{r}}^{\cM}$ we have $\eta_j^{\cM}(\mathbf{x},\mathbf{y}) = 0$ if and only if $\tp_{\chron}^{\cM}(\mathbf{x}, \mathbf{y}) = \mu_j$.  We claim that
		\[
		C(\mu_0,\mu_1) = \inf_{\varepsilon > 0} \sup_{\mathbf{x}_0,\mathbf{x}_1 \in D_{0,\mathbf{r}}^{\cM}}
		\re \ip{\mathbf{x}_0,\mathbf{x}_1} - \varepsilon^{-1} \eta_0^{\cM}(\mathbf{x}_0,\mathbf{y}) -  \varepsilon^{-1} \eta_1^{\cM}(\mathbf{x}_1,\mathbf{y}).
		\]
		Indeed, let $C$ be the value of the right-hand side.  By substituting a coupling $(\mathbf{x}_0,\mathbf{x}_1)$ of $(\mu_0,\mu_1)$ as a candidate for the sup, we see that $C \geq C(\mu_0,\mu_1)$.  For the converse, note that for every $\delta, \varepsilon > 0$, there exist $\mathbf{x}_0$, $\mathbf{x}_1 \in D_{0,\mathbf{r}}^{\cM}$ such that
		\[
		\max(C - \re \ip{\mathbf{x}_0,\mathbf{x}_1} + \varepsilon^{-1} \eta_0^{\cM}(\mathbf{x}_0,\mathbf{y}) + \varepsilon^{-1} \eta_1^{\cM}(\mathbf{x}_1,\mathbf{y}),0) \leq \delta.
		\]
		By countable saturation, there exist $\mathbf{x}_0$, $\mathbf{x}_1$ such that this holds for all $\delta$, $\varepsilon$.  Since $|\re \ip{\mathbf{x}_0,\mathbf{x}_1}|$ is bounded by a constant depending on $\mathbf{r}$, this implies that $\eta_j^{\cM}(\mathbf{x}_j,\mathbf{y}) = 0$, hence $\tp_{\chron}^{\cM}(\mathbf{x}_j,\mathbf{y}) = \mu_j$ and $\re \ip{\mathbf{x}_0,\mathbf{x}_1} \geq C$.  Therefore, $C(\mu_0,\mu_1) \geq C$.  Moreover, the pair $(\mathbf{x}_0,\mathbf{x}_1)$ in the foregoing argument satisfies $\re \ip{\mathbf{x}_0,\mathbf{x}_1} = C = C(\mu_0,\mu_1)$ and is thus an optimal coupling.
		
		Now suppose that $\mathbf{x}_0$ is fixed with $\tp_{\chron}^{\cM}(\mathbf{x}_0,\mathbf{y}) = \mu_0$.  Let $(\mathbf{x}_0',\mathbf{x}_1')$ be some optimal coupling of $(\mu_0,\mu_1)$.  Since $\mathbf{x}_0$ and $\mathbf{x}_0'$ have the same chronological type,
		\[
		\sup_{\mathbf{x}_1 \in D_{0,\mathbf{r}}^{\cM}}
		\re \ip{\mathbf{x}_0,\mathbf{x}_1} -  \varepsilon^{-1} \eta_1^{\cM}(\mathbf{x}_1,\mathbf{y}) = \sup_{\mathbf{x}_1 \in D_{0,\mathbf{r}}^{\cM}}
		\re \ip{\mathbf{x}_0',\mathbf{x}_1} -  \varepsilon^{-1} \eta_1^{\cM}(\mathbf{x}_1,\mathbf{y}) \geq C(\mu_0,\mu_1).
		\]
		So for $\delta, \varepsilon > 0$, there exists $\mathbf{x}_1$ with $\re \ip{\mathbf{x}_0,\mathbf{x}_1} - \varepsilon^{-1} \eta_1^{\cM}(\mathbf{x}_1,\mathbf{y}) \geq C(\mu_0,\mu_1) - \delta$.  Using countable saturation as before, we find some $\mathbf{x}_1$ which satisfies this for all $\delta, \varepsilon$, and hence $(\mathbf{x}_0,\mathbf{x}_1)$ is an optimal coupling of $(\mu_0,\mu_1)$.
	\end{proof}

	\subsection{Monge--Kantorovich duality} \label{subsec: MK duality}
	
	Now we can prove the analog of Monge--Kantorovich duality for the setting of chronological types over a given $\mathbf{y}$.  The theorem and proof are an adaptation of \cite[Theorem 1.1]{JekelTypeCoupling}, but parts of the argument are simpler thanks to working with functions which are quadratic plus $\norm{\cdot}_1$-Lipschitz.  Compare also \cite[Proposition 1.3]{GJNS2021} for the setting of quantifier-free types over $\cL_{\tr}$ (noncommutative laws).
	
	\begin{theorem} \label{thm: MK duality}
		Fix a countably saturated $\cL_{\proc}$-structure $\cM$ satisfying $\rT_{\stat}$.  Let $m \in \N$ and $m' \in \N \cup \{0,\infty\}$.  Let $\mathbf{y}_0 \in M_0^{m'}$.  For any types $\mu_0, \mu_1 \in \mathbb{S}_{\chron,m,\mathbf{r}}(\cM / \mathbf{y}_0)$, there exist chronologically definable predicates $\varphi_0$, $\varphi_1$ in $m+m'$ variables with respect to $\rT_{\stat}$ such that for all $\cN$ satisfying $\rT_{\stat}$,
		\begin{equation} \label{eq: admissibility}
			\varphi_0^{\cN}(\mathbf{x}_0,\mathbf{y}) + \varphi_1^{\cN}(\mathbf{x}_1,\mathbf{y}) \geq \re \ip{\mathbf{x_0}\mathbf{x}_1} \text{ for all } \mathbf{x}_0, \mathbf{x}_1 \in \cN_0^m,
		\end{equation}
		and equality holds when $(\mathbf{x}_0,\mathbf{x}_1)$ is an optimal coupling in $\cM$ of $(\mu_0,\mu_1)$ relative to $\mathbf{y}_0$.  Moreover, $\varphi_j$ can be chosen to be convex in first $m$ variables and of the form $\varphi_j = \psi_j + q$ where $\psi_j$ is Lipschitz with respect to $\norm{\cdot}_1$ in the first $m$ variables.
	\end{theorem}
	
	\begin{proof}
		Fix a chronologically definable predicate $\eta_1 \geq 0$ in $m + m'$ variables such that if $\mathbf{x} \in D_{0,\mathbf{r}}^{\cM}$ and $\eta_1(\mathbf{x},\mathbf{y}) = 0$, then $\tp_{\chron}^{\cM}(\mathbf{x},\mathbf{y}) = \mu_1$.  For $\varepsilon > 0$, let
		\[
		\theta_{\varepsilon}(\mathbf{x},\mathbf{y}) = \sup_{ \mathbf{x}_1 \in D_{0,\mathbf{r}}^{\cM}} \left[ \ip{\mathbf{x},\mathbf{x}_1} - \varepsilon^{-1} \eta_1(\mathbf{x}_1,\mathbf{y}) \right].
		\]
		Observe that for $\mathbf{x} \in M_0^m$,
		\[
		\lim_{\varepsilon \to 0^+} \theta_{\varepsilon}^{\cM}(\mathbf{x},\mathbf{y}) = C(\tp_{\chron}(\mathbf{x},\mathbf{y}, \mu_1).
		\]
		This follows from the same argument as Lemma \ref{lem: existence of optimal couplings}; see also \cite[proof of Proposition 4.1]{JekelTypeCoupling} for a parallel argument worked out in more detail.  Note $\theta_{\varepsilon}$ decreases as $\varepsilon$ decreases.
		
		Fix $\mathbf{x}_0 \in M_0^m$ with $\tp_{\chron}(\mathbf{x}_0,\mathbf{y}) = \mu_0$.  Fix $\varepsilon_k$ for $k \in \N$ such that
		\[
		\theta_{\varepsilon_k}^{\cM}(\mathbf{x},\mathbf{y}) - C(\mu_0,\mu_1) \leq 2^{-k-1},
		\]
		and so also
		\[
		0 \leq \theta_{\varepsilon_k}^{\cM}(\mathbf{x},\mathbf{y}) - \theta_{\varepsilon_{k+1}}^{\cM}(\mathbf{x},\mathbf{y}) \leq 2^{-k}.
		\]
		Let $F_{2^{-k}}$ be the cutoff function as in Lemma \ref{lem: projection onto R ball}, and let
		\[
		\theta = \theta_{\varepsilon_1} + \sum_{k=1}^\infty F_{2^{-k}}(\theta_{\varepsilon_{k+1}} - \theta_{\varepsilon_k}).
		\]
		Since $|F_{2^{-k}}| \leq 2^{-k}$, the series converges uniformly and thus defines a chronologically definable predicate with respect to $\rT_{\stat}$.  We also note since
		\[
		\theta_{\varepsilon_{k+1}} - \theta_{\varepsilon_k} \leq F_{2^{-k}}(\theta_{\varepsilon_{k+1}} - \theta_{\varepsilon_k}) \leq 0,
		\]
		we have that for $\mathbf{x} \in M_0^m$,
		\[
		\theta^{\cM}(\mathbf{x},\mathbf{y}) \geq \lim_{\varepsilon \to 0^+} \theta_{\varepsilon}^{\cM}(\mathbf{x},\mathbf{y}) = C(\tp_{\chron}^{\cM}(\mathbf{x},\mathbf{y}), \mu_1).
		\]
		Moreover, at $\mathbf{x}_0$, equality is achieved since $|\theta_{\varepsilon_{k+1}}^{\cM}(\mathbf{x}_0,\mathbf{y}) - \theta_{\varepsilon_k}(\mathbf{x}_0,\mathbf{y})| \leq 2^{-k}$.
		
		In order to define $\varphi_0$ and $\varphi_1$, we need one more ingredient.  Define the formula
		\[
		\delta(\mathbf{x},\mathbf{y}) := \inf_{\mathbf{v} \in D_{0,\mathbf{r}}} q(\mathbf{x}) = \inf_{\mathbf{v} \in D_{0,\mathbf{r}}} \frac{1}{2} \norm{\mathbf{x} - \mathbf{v}}_2^2.
		\]
		It is routine to show that $\delta$ is convex, and $\delta^{\cN}(\mathbf{x},\mathbf{y}) = 0$ whenever $\mathbf{x} \in D_{0,\mathbf{r}}^{\cN}$ for some $\cN$ satisfying $\rT_{\stat}$.  Moreover,
		\[
		\delta^{\cN}(\mathbf{x},\mathbf{y}) - q(\mathbf{x}) = \inf_{\mathbf{v} \in D_{0,\mathbf{r}}^{\cN}} \left[ \re \ip{\mathbf{x},\mathbf{v}} - \frac{1}{2} \norm{\mathbf{v}}_2^2 \right].
		\]
		Since $\mathbf{v}$ is bounded in operator norm, this is a $\norm{\cdot}_1$-Lipschitz function of $\mathbf{x}$.  Now define
		\[
		\varphi_1(\mathbf{x},\mathbf{y}) = \delta(\mathbf{x},\mathbf{y}) + \sup_{\mathbf{v} \in D_{0,\mathbf{r}}} \left[ \re \ip{\mathbf{x},\mathbf{v}} - \theta(\mathbf{v},\mathbf{y}) \right].
		\]
		Note that $\varphi_1^{\cN}$ is convex in $\mathbf{x}$ for every $\cN$ satisfying $\rT_{\stat}$ because $\delta$ is convex and the second term is the supremum of a family of affine functions.  Moreover, $\sup_{\mathbf{v} \in D_{0,\mathbf{r}}} \left[ \re \ip{\mathbf{x},\mathbf{v}} - \theta(\mathbf{v},\mathbf{y}) \right]$ is also $\norm{\cdot}_1$-Lipschitz in $\mathbf{x}$ since there is a bound on the operator norm of $\mathbf{v}$.  Thus, overall, $\varphi_1$ is convex and $\varphi_1 - q$ is Lipschitz in $\norm{\cdot}_1$ with respect to $\mathbf{x}$.
		
		Let $\varphi_0$ be the Legendre transform $(\varphi_1)^{\star}$.  By Lemma \ref{lem: Lipschitz plus quadratic}, $\varphi_0$ is a chronologically definable predicate with respect to $\rT_{\stat}$, it is convex in the first $m$ variables, and $\varphi - q$ is $\norm{\cdot}_1$-Lipschitz in the first $m$ variables.  Since $\varphi_0$ was chosen as the Legendre transform of $\varphi_1$, we have \eqref{eq: admissibility}.  So it remains to show that equality holds when $(\mathbf{x}_0,\mathbf{x}_1)$ is an optimal coupling of $(\mu_0,\mu_1)$ relative to $\mathbf{y}_0$ in $\cM$.  In this case, note that
		\begin{align*}
			\varphi_0^{\cM}(\mathbf{x}_0,\mathbf{y}_0) &= \sup_{\mathbf{w} \in M_0^m} \inf_{\mathbf{v} \in D_{0,\mathbf{r}}^{\cM}} \left[ \re \ip{\mathbf{x}_1,\mathbf{w}} - \delta^{\cM}(\mathbf{w},\mathbf{y}) + \re \ip{\mathbf{w},\mathbf{v}} + \theta^{\cM}(\mathbf{v},\mathbf{y}_0) \right] \\
			&\leq \sup_{\mathbf{w} \in M_0^m} \left[ \re \ip{\mathbf{x}_0,\mathbf{w}} + 0 - \re \ip{\mathbf{w},\mathbf{x}_0} + \theta^{\cM}(\mathbf{x}_0,\mathbf{y}_0) \right] \\
			&\leq \theta^{\cM}(\mathbf{x}_0,\mathbf{y}_0).
		\end{align*}
		Moreover,
		\begin{align*}
			\varphi_1^{\cM}(\mathbf{x}_1,\mathbf{y}) &= 0 +\sup_{\mathbf{v} \in D_{0,\mathbf{r}}^{\cM}} \left[ \ip{\mathbf{x}_1,\mathbf{v}} - \theta^{\cM}(\mathbf{v},\mathbf{y}_0) \right] \\
			&\leq \sup_{\mathbf{v} \in D_{0,\mathbf{r}}^{\cM}} \left[ \ip{\mathbf{x}_1,\mathbf{v}} - C(\tp_{\chron}^{\cM}(\mathbf{v},\mathbf{y}_0), \mu_1) \right] \\
			&\leq 0.
		\end{align*}
		Therefore,
		\begin{align*}
			\varphi_0^{\cM}(\mathbf{x}_0,\mathbf{y}) + \varphi_1^{\cM}(\mathbf{x}_1,\mathbf{y}_0) &\leq \theta^{\cM}(\mathbf{x}_0,\mathbf{y}_0) + 0 \\
			&\leq C(\tp_{\chron}^{\cM}(\mathbf{x}_0,\mathbf{y}_0),\mu_1) \\
			&= \re \ip{\mathbf{x}_0,\mathbf{x}_1},
		\end{align*}
		where the last line follows because $(\mathbf{x}_0,\mathbf{x}_1)$ is an optimal coupling of $(\mu_0,\mu_1)$.
	\end{proof}
	
	Given an optimal coupling $(\mathbf{x}_0,\mathbf{x}_1)$ for $(\mu_0,\mu_1)$ relative to $\mathbf{y}$, the \emph{displacement interpolation} is the family of variables $\mathbf{x}_t = (1-t) \mathbf{x}_0 + t\mathbf{x}_1$.  Just as in the classical setting, if we let $\mu_t = \tp_{\chron}^{\cM}(\mathbf{x}_t,\mathbf{y}_0)$, we have
	\begin{align*}
		d_W(\mu_0,\mu_1)
		&\leq d_W(\mu_0,\mu_t) + d_W(\mu_t,\mu_1) \\
		&\leq \norm{\mathbf{x}_0 - \mathbf{x}_t}_2 + \norm{\mathbf{x}_t - \mathbf{x}_1}_2 \\
		&= t\norm{\mathbf{x}_0 - \mathbf{x}_1}_2 + t \norm{\mathbf{x}_0 - \mathbf{x}_1}_2 \\
		&= \norm{\mathbf{x}_0 - \mathbf{x}_1}_2 \\
		&= d_W(\mu_0,\mu_1),
	\end{align*}
	hence all the inequalities are equalities.  Thus, the map $t \mapsto \mu_t$ is a geodesic in the metric sense, and that $(\mathbf{x}_s,\mathbf{x}_t)$ are an optimal coupling of $(\mu_s,\mu_t)$ for $s, t \in [0,1]$.  In order to study the behavior of the entropy $\chi_{\chron}^{\cU}(\mathbf{x}_t \mid \mathbf{y})$ along this geodesic, we will describe how to obtain a pair of convex functions witnessing Monge--Kantorovich duality for $(\mathbf{x}_s,\mathbf{x}_t)$ from a given pair of convex functions associated to $(\mathbf{x}_0,\mathbf{x}_1)$ using \cite[Proposition 2.21]{Jekel2025InfoGeom}.
	
	\begin{proposition} \label{prop: displacement interpolation}
		Fix a countably saturated $\cL_{\proc}$-structure $\cM$ satisfying $\rT_{\stat}$.  Let $m \in \N$ and $m' \in \N \cup \{0,\infty\}$.  Let $\mathbf{y}_0 \in M_0^{m'}$.  Fix $\mu_0, \mu_1 \in \mathbb{S}_{\chron,m,\mathbf{r}}(\cM / \mathbf{y}_0)$, let $(\mathbf{x}_0,\mathbf{x}_1)$ be an optimal coupling, and $\mathbf{x}_t = (1-t) \mathbf{x}_0 + t \mathbf{x}_1$ for $t \in [0,1]$.
		
		Let $\varphi_0$ and $\varphi_1$ be definable predicates with respect to $\rT_{\stat}$ satisfying the conclusions of Theorem \ref{thm: MK duality}.  For $0 \leq s \leq t \leq 1$ and $\cN$ satisfying $\rT_{\stat}$, let
		\begin{align*}
			\varphi_{t,s}^{\cN}(\mathbf{x},\mathbf{y}) &= \inf_{\mathbf{x}' \in \cN^m} \left[ \frac{t}{2s} \norm{\mathbf{x}}_2^2 - \frac{t-s}{s} \re \ip{\mathbf{x},\mathbf{x}'} + \frac{(t-s)(1-s)}{2s} \norm{\mathbf{x}'}_2^2 + (t-s) \varphi_0^{\cN}(\mathbf{x}',\mathbf{y}) \right] \text{ when } s > 0, \\
			\varphi_{t,0}^{\cN}(\mathbf{x},\mathbf{y}) &= \frac{1-t}{2} \norm{\mathbf{x}}_2^2 + t \varphi_0^{\cM}(\mathbf{x},\mathbf{y}),
		\end{align*}
		and
		\begin{align*}
			\varphi_{s,t}^{\cN}(\mathbf{x},\mathbf{y}) &= \inf_{\mathbf{x}' \in \cN^m} \left[ \frac{1-s}{2(1-t)} \norm{\mathbf{x}}_2^2 - \frac{t-s}{1-t} \ip{\mathbf{x},\mathbf{x}'} + \frac{(t-s)t}{2(1-t)} \norm{\mathbf{x}'}_2^2 + (t-s) \varphi_1^{\cN}(\mathbf{x}',\mathbf{y}) \right] \text{ when } t < 1 \\
			\varphi_{s,1}^{\cN}(\mathbf{x}) &= \frac{s}{2} \norm{\mathbf{x}}_2^2 + (1-s) \varphi_1^{\cN}(\mathbf{x},\mathbf{y}).
		\end{align*}
		Then for $0 \leq s \leq t \leq 1$,
		\begin{enumerate}[(1)]
			\item $\varphi_{t,s}$ and $\varphi_{s,t}$ are chronologically definable predicates with respect to $\rT_{\stat}$, and $\varphi_{t,s} - q$ and $\varphi_{s,t} - q$ are $\norm{\cdot}_1$-Lipschitz in the first $m$ variables.
			\item $\varphi_{t,s}$ is $t/s$-semiconcave for $s > 0$ and $(1-t) / (1-s)$-strongly convex for $t < 1$.
			\item $\varphi_{s,t}$ is $(1-s)/(1-t)$-semiconcave for $t < 1$ and $s/t$-strongly convex for $s > 0$.
			\item $\varphi_{t,s}^{\cN}(\mathbf{x},\mathbf{y}) + \varphi_{s,t}^{\cN}(\mathbf{x}',\mathbf{y}) \geq \re \ip{\mathbf{x},\mathbf{x}'}$ for all $\cN$, $\mathbf{x}$, and $\mathbf{x}'$.
			\item $\varphi_{t,s}^{\cM}(\mathbf{x}_s,\mathbf{y}_0) + \varphi_{s,t}^{\cM}(\mathbf{x}_t,\mathbf{y}_0) = \re \ip{\mathbf{x}_s,\mathbf{x}_t}$.
			\item For $s > 0$, $\varphi_{t,s}$ is differentiable in the first $m$ variables, $\nabla_{m,m'} \varphi_{t,s}$ is a chronologically definable function, and $\nabla_{m,m'} \varphi_{t,s}^{\cM}(\mathbf{x}_s,\mathbf{y}_0) = \mathbf{x}_t$.
			\item For $t < 1$, $\varphi_{s,t}$ is differentiable in the first $m$ variables, $\nabla_{m,m'} \varphi_{s,t}$ is a chronologically definable function, and $\nabla_{m,m'} \varphi_{s,t}^{\cM}(\mathbf{x}_t,\mathbf{y}_0) = \mathbf{x}_s$.
		\end{enumerate}
	\end{proposition}
	
	\begin{proof}
		For claim (1), observe that for $s > 0$,
		\[
		(\varphi_{t,s} - q)^{\cN}(\mathbf{x},\mathbf{y}) = (t-s) \inf_{\mathbf{x}' \in \cN^m} \left[ \frac{1}{s} \norm{\mathbf{x} - \mathbf{x}'}_2^2  + (\varphi_0 - q)^{\cN}(\mathbf{x}',\mathbf{y}) \right].
		\]
		Hence, by applying Lemma \ref{lem: Lipschitz plus quadratic 2} (1) - (3) to $\varphi_0 - q$ with $a = 0$ and $c = 1/s$, we conclude $\varphi_{t,s} - q$ is a chronologically definable predicate which is $\norm{\cdot}_1$-Lipschitz with respect to the first $m$ variables.  The conclusion for $\varphi_{t,0}$ is immediate.  The cases for $\varphi_{s,t}$ are symmetrical.
		
		Claims (2) - (5) follow directly from \cite[Proposition 2.21]{Jekel2025InfoGeom}.
		
		For claim (6), differentiability and definability of the gradient follow from Proposition \ref{prop: semiconvex semiconcave gradient} and Lemma \ref{lem: bounded gradient}.  Using claims (4) and (5), for $\mathbf{x} \in \cM^m$, we have
		\begin{align*}
			\varphi_{t,s}^{\cM}(\mathbf{x},\mathbf{y}_0) - \varphi_{t,s}^{\cM}(\mathbf{x}_s,\mathbf{y}_0) &\geq [\re \ip{\mathbf{x},\mathbf{x}_t} - \varphi_{s,t}^{\cM}(\mathbf{x}_t,\mathbf{y})] - [\re \ip{\mathbf{x}_s,\mathbf{x}_t} - \varphi_{s,t}^{\cM}(\mathbf{x}_t,\mathbf{y})] \\
			&= \re \ip{\mathbf{x} - \mathbf{x}_s, \mathbf{x}_t},
		\end{align*}
		which forces $\nabla_{m,m'} \varphi_{t,s}^{\cM}(\mathbf{x}_s,\mathbf{y}_0) = \mathbf{x}_t$.
		
		Claim (7) is symmetrical to (6).
	\end{proof}

	The next result is a generalization of \cite[Lemma 4.9]{Jekel2025InfoGeom}, which concerns finite-dimensional approximations of optimal couplings.  We show that if $(\mathbf{x}_0,\mathbf{x}_1)$ is an optimal coupling of $(\mu_0,\mu_1)$ relative to $\mathbf{y}_0$, and if $(\mathbf{X}_0^{(n)},\mathbf{Y}_0^{(n)})$ is a sequence of random matrix models for $(\mathbf{x}_0,\mathbf{y}_0)$, then there exists a compatible random matrix model $\mathbf{X}_1^{(n)}$ for $\mathbf{x}_1$.  This statement is crucial for our proof of geodesic concavity and the evolution variational inequality.  We remark that \cite[\S 4.4]{Jekel2025InfoGeom} shows that the analogous statement is false for noncommutative laws with the plain free entropy $\chi$ and the Biane--Voiculescu Wasserstein distance, which also implies the failure of Monge--Kantorovich duality using quantifier-free definable predicates.
	
	\begin{theorem} \label{thm: optimal coupling lifts}
		Let $\cM$ be the random matrix ultraproduct as in \eqref{eq: random matrix ultraproduct} and $\cQ$ a quotient as in \eqref{eq: random matrix quotient} with respect to some character $\cV$ on $Z(\cM)$.  Let $m \in \N$ and $m' \in \N \cup \{0,\infty\}$.  Let $\mathbf{y}_0 \in M_0^{m'}$.  Fix $\mu_0, \mu_1 \in \mathbb{S}_{\chron,m,\mathbf{r}}(\cM / \mathbf{y}_0)$, and let $(\mathbf{x}_0,\mathbf{x}_1)$ be an optimal coupling.
		
		Let $\mathbf{Y}_0^{(n)}$ be a deterministic matrix tuple with $\sup_n \norm{Y_{0,j}^{(n)}}_\infty < \infty$ for each $j$, and let $\mathbf{X}_0^{(n)}$ be random matrices with $\norm{X_{0,j}^{(n)}}_\infty \leq r_j$ for all $n$ and $j$, such that
		\begin{equation} \label{eq: convergence of type at time zero}
			\lim_{n \to \cU} \Lambda_{\varphi}^{(n)}(\mathbf{X}_0^{(n)},\mathbf{Y}_0^{(n)}) = \varphi^{\cQ}(\mathbf{x}_0,\mathbf{y}_0)
		\end{equation}
		for all restricted chronological formulas in $m + m'$ variables.
		
		Then there exist random matrix tuples $\mathbf{X}_1^{(n)}$ such that
		\begin{enumerate}[(1)]
			\item $\norm{\mathbf{X}_{1,j}^{(n)}}_\infty \leq r_j$ for all $n$ and $j$;
			\item $\mathbf{X}_1^{(n)}$ is a function $\mathbf{X}_0^{(n)}$ for each $n$;
			\item for all restricted chronological formulas in $2m + m'$ many variables,
			\begin{equation} \label{eq: lifting convergence}
				\lim_{n \to \cU} \Lambda_{\psi}^{(n)}(\mathbf{X}_0^{(n)},\mathbf{X}_1^{(n)},\mathbf{Y}_0^{(n)}) = \psi^{\cQ}(\mathbf{x}_0,\mathbf{x}_1,\mathbf{y}_0);
			\end{equation}
			\item $\mathbf{X}_0^{(n)}$ and $\mathbf{X}_1^{(n)}$ are optimally coupled with respect to the classical $L^2$ Wasserstein distance for probability distributions on $(\mathbb{M}_n)^m$.
		\end{enumerate}
	\end{theorem}
	
	\begin{proof}
		Let $\mathbf{x}_t = (1-t) \mathbf{x}_0 + t \mathbf{x}_1$, and let $\mu_t = \tp_{\chron}^{\cQ}(\mathbf{x}_t,\mathbf{y})$.  We will first construct a random matrix model $\mathbf{X}_t^{(n)}$ for $\mathbf{x}_t$ when $t < 1$, and obtain a matrix model for $\mathbf{x}_1$ at the end using a diagonalization argument.
		
		Let $\varphi_0$ and $\varphi_1$ be as in Theorem \ref{thm: MK duality}, and $\varphi_{s,t}$ and $\varphi_{t,s}$ be as in Proposition \ref{prop: displacement interpolation}. Let $r_j'$ be a bound for $\norm{y_{0,j}}_\infty$ and $\norm{Y_{0,j}^{(n)}}_\infty$.  Fix a nonnegative chronologically definable predicate $\eta_t$ in $m + m'$ variables such that for $\cN$ satisfying $\rT_{\stat}$ and $(\mathbf{x},\mathbf{y}) \in D_{0,(\mathbf{r},\mathbf{r}')}^{\cM}$, we have
		\[
		\eta_t^{\cN}(\mathbf{x},\mathbf{y}) = 0 \iff \tp_{\chron}^{\cN}(\mathbf{x},\mathbf{y}) = \mu_t.
		\]
		Moreover, define
		\[
		\omega^{\cN}(\mathbf{x},\mathbf{y}) = \sup_{\mathbf{x}' \in D_{0,\mathbf{r}}^{\cN}} \left[ \re \ip{\mathbf{x},\mathbf{x}'} - \varphi_{0,t}^{\cN}(\mathbf{x}',\mathbf{y}) - \eta_t^{\cN}(\mathbf{x}',\mathbf{y}) \right].
		\]
		Let $\Lambda_{\varphi_{0,t}+\eta_t}^{(n)}$ be finite-dimensional approximations for $\eta_t + \varphi_{0,t}$ given by Lemma \ref{lem: FD approximation 1} for the given radius bounds $(\mathbf{r},\mathbf{r}')$, and set
		\[
		\Lambda_\omega^{(n)}(\mathbf{X},\mathbf{Y}) = \sup_{\mathbf{X}' \in D_{\mathbf{r}}^{\mathbb{M}_n}} \left[ \re \ip{\mathbf{X},\mathbf{X}'} - \Lambda_{\varphi_{0,t} + \eta_t}^{(n)} (\mathbf{X}',\mathbf{Y}) \right].
		\]
		Thus, $\Lambda_\omega^{(n)}$ is a sequence of finite-dimensional approximations for $\omega$ in the sense of Lemma \ref{lem: FD approximation 1} (see again the proof of Proposition \ref{prop: approx by pointwise function}).  Note that by the construction in Lemma \ref{lem: FD approximation 1} $\Lambda_{\varphi_{0,t} + \eta_t}^{(n)}$ and $\Lambda_\omega^{(n)}$ are continuous functions.  Hence, the supremum in the definition of $\Lambda_\omega^{(n)}$ is achieved.  Now for $\mathbf{X} \in D_{\mathbf{r}}^{\mathbb{M}_n}$, let
		\[
		A^{(n)}(\mathbf{X}) = \{ \mathbf{X}' \in D_{\mathbf{r}}^{\mathbb{M}_n}: \re \ip{\mathbf{X},\mathbf{X}'} - \Lambda_{\varphi_{0,t} + \eta_t}^{(n)} (\mathbf{X}',\mathbf{Y}_0^{(n)}) = \Lambda_\omega^{(n)}(\mathbf{X},\mathbf{Y}_0^{(n)}) \},
		\]
		where $\mathbf{Y}_0^{(n)}$ are the given deterministic tuples.  We want to make a Borel-measurable selection of some $\mathbf{X}' \in A^{(n)}(\mathbf{X})$ for each $\mathbf{X} \in D_{\mathbf{r}}^{\mathbb{M}_n}$ using the Kuratowski--Ryll-Nardzewski selection theorem \cite{KRN1965}.  The hypotheses are satisfied:
		\begin{enumerate}[(1)]
			\item $A^{(n)}(\mathbf{X})$ is a compact set for each $\mathbf{X}$, by continuity of $\Lambda_{\varphi_{0,t} + \eta_t}^{(n)}$.
			\item We need to show that for open $E \subseteq D_{\mathbf{r}}^{\mathbb{M}_n}$, the set
			\[
			\{ \mathbf{X} \in D_{\mathbf{r}}^{\mathbb{M}_n}: A^{(n)}(\mathbf{X}) \cap E \neq \varnothing \},
			\]
			is Borel-measurable.  Since any open set in Euclidean space is a countable union of closed sets, it suffices to prove that the above set is Borel when $E$ is closed, rather than open.  Let
			\[
			K = \{ (\mathbf{X},\mathbf{X}') \in (D_{\mathbf{r}}^{\mathbb{M}_n})^2: \re \ip{\mathbf{X},\mathbf{X}'} - \Lambda_{\varphi_{0,t} + \eta_t}^{(n)} (\mathbf{X}',\mathbf{Y}_0^{(n)}) = \Lambda_\omega^{(n)}(\mathbf{X},\mathbf{Y}_0^{(n)}) \},
			\]
			which is a closed set by continuity of $\Lambda_\omega^{(n)}$ and $\Lambda_{\varphi_{0,t} + \eta_t}^{(n)}$.  Then
			\[
			\{ \mathbf{X} \in D_{\mathbf{r}}^{\mathbb{M}_n}: A^{(n)}(\mathbf{X}) \cap E \neq \varnothing \} = \pi_1(K \cap (D_{\mathbf{r}}^{\mathbb{M}_n} \times E)),
			\]
			where $\pi_1: D_{\mathbf{r}}^{\mathbb{M}_n} \times D_{\mathbf{r}}^{\mathbb{M}_n} \to D_{\mathbf{r}}^{\mathbb{M}_n}$ is the projection onto the first coordinate.  The image is compact by continuity, so that $\{ \mathbf{X} \in D_{\mathbf{r}}^{\mathbb{M}_n}: A^{(n)}(\mathbf{X}) \cap E \neq \varnothing \}$ is Borel as desired.
		\end{enumerate}
		Thus, by the Kuratowski--Ryll-Nardzewski measurable selection theorem, there exists a Borel-measurable $F_t^{(n)}: D_{\mathbf{r}}^{\mathbb{M}_n} \to D_{\mathbf{r}}^{\mathbb{M}_n}$ such that $F_t^{(n)}(\mathbf{X}) \in A^{(n)}(\mathbf{X})$ for all $\mathbf{X} \in D_{\mathbf{r}}^{\mathbb{M}_n}$.
		
		Let
		\[
		\mathbf{X}_t^{(n)} = F_t^{(n)}(\mathbf{X}_0^{(n)}).
		\]
		Note that $\Lambda_\omega^{(n)}$ is a convex function of $\mathbf{X}$ by construction.  For each $\mathbf{X}$, because $F_t^{(n)}(\mathbf{X})$ achieves the supremum in the definition of $\Lambda_\omega^{(n)}$, we know that $F_t^{(n)}(\mathbf{X})$ is in $\underline{\nabla}_{m,m'} \Lambda_\omega^{(n)}(\mathbf{X},\mathbf{Y}_0^{(n)})$, meaning that it is a subgradient vector for $\Lambda_\omega^{(n)}(\cdot, \mathbf{Y}_0^{(n)})$ at the point $\mathbf{X}$.  Therefore, by the classical Monge--Kantorovich duality, $(\mathbf{X}_0^{(n)},\mathbf{X}_t^{(n)})$ are optimally coupled.
		
		Now we have to show that $(\mathbf{X}_0^{(n)},\mathbf{X}_t^{(n)},\mathbf{Y}_0^{(n)})$ satisfy the analog of \eqref{eq: lifting convergence} with respect to $(\mathbf{x}_0,\mathbf{x}_t,\mathbf{y})$.  Let $\tilde{\cV}$ be an arbitrary character on $Z(\cM)$, not necessarily equal to $\cV$, and let $\tilde{\pi}: \cM \to \tilde{\cQ}$ be the corresponding quotient.  Let
		\[
		(\tilde{\mathbf{x}}_0, \tilde{\mathbf{x}}_t, \tilde{\mathbf{y}}_0) = (\tilde{\pi}([\mathbf{X}_0^{(n)}]_{n \in \N}),\tilde{\pi}([\mathbf{X}_t^{(n)}]_{n \in \N}),\tilde{\pi}([\mathbf{Y}_0^{(n)}]_{n \in \N})).
		\]
		By the defining property of the finite-dimensional approximations from Lemma \ref{lem: FD approximation 1} and by Proposition \ref{prop: Los theorem},
		\begin{align*}
			\omega^{\tilde{\cQ}}(\tilde{\mathbf{x}}_0,\tilde{\mathbf{y}}_0) &= \tilde{\cV}([\Lambda_\omega^{(n)}(\mathbf{X}_0^{(n)},\mathbf{Y}_0^{(n)})]_{n \in \N}) \\
			&= \tilde{\cV}([\re \ip{\mathbf{X}_0^{(n)}, \mathbf{X}_t^{(n)}} - \Lambda_{\varphi_{0,t} + \eta_t}^{(n)}(\mathbf{X}_t^{(n)},\mathbf{Y}_0^{(n)})]_{n \in \N}) \\
			&= \re \ip{\tilde{\mathbf{x}}_0,\tilde{\mathbf{x}}_t} - (\varphi_{0,t} + \eta_t)^{\tilde{\cQ}}(\tilde{\mathbf{x}}_t,\tilde{\mathbf{y}}_0).
		\end{align*}
		We also have $\tp_{\chron}^{\tilde{\cQ}}(\tilde{\mathbf{x}}_0,\tilde{\mathbf{y}}_0) = \tp_{\chron}^{\cQ}(\mathbf{x}_0,\mathbf{y}_0)$, and hence
		\begin{align*}
			\omega^{\tilde{\cQ}}(\tilde{\mathbf{x}}_0,\tilde{\mathbf{y}}_0) &= \omega^{\cQ}(\mathbf{x}_0,\mathbf{y}_0) \\
			&= \sup_{\mathbf{x}' \in D_{0,\mathbf{r}}^{\cQ}} \left[ \re \ip{\mathbf{x}_0,\mathbf{x}'} - \varphi_{0,t}^{\cQ}(\mathbf{x}',\mathbf{y}) - \eta_t^{\cN}(\mathbf{x}',\mathbf{y}) \right] \\
			&\geq \re \ip{\mathbf{x}_0,\mathbf{x}_t} - \varphi_{0,t}^{\cQ}(\mathbf{x}_t,\mathbf{y}) - \eta_t^{\cN}(\mathbf{x}_t,\mathbf{y}) \\
			&= \varphi_{t,0}^{\cQ}(\mathbf{x}_0,\mathbf{y}_0) \\
			&= \varphi_{t,0}^{\tilde{\cQ}}(\tilde{\mathbf{x}}_0,\tilde{\mathbf{y}}_0).
		\end{align*}
		Therefore,
		\begin{align*}
			\re \ip{\tilde{\mathbf{x}}_0, \tilde{\mathbf{x}}_t} &\leq \varphi_{t,0}^{\tilde{\cQ}}(\tilde{\mathbf{x}}_0,\tilde{\mathbf{y}}_0) + \varphi_{0,t}^{\tilde{\cQ}}(\tilde{\mathbf{x}}_t,\tilde{\mathbf{y}}_0) \\
			&\leq \omega^{\tilde{\cQ}}(\tilde{\mathbf{x}}_0, \tilde{\mathbf{y}}_0) + \varphi_{0,t}^{\tilde{\cQ}}(\tilde{\mathbf{x}}_t,\tilde{\mathbf{y}}_0) \\
			&= \re \ip{\tilde{\mathbf{x}}_0, \tilde{\mathbf{x}}_t} - \eta_t^{\tilde{\cQ}}(\tilde{\mathbf{x}}_t, \tilde{\mathbf{y}}_0).
		\end{align*}
		Therefore, we have $\eta_t(\tilde{\mathbf{x}}_t,\tilde{\mathbf{y}}_0) = 0$, which implies that $\tp_{\chron}^{\tilde{\cQ}}(\tilde{\mathbf{x}}_t,\tilde{\mathbf{y}}_0) = \mu_t$.  We also deduce that $\re \ip{\tilde{\mathbf{x}}_0, \tilde{\mathbf{x}}_t} = \varphi_{t,0}^{\tilde{\cQ}}(\tilde{\mathbf{x}}_0,\tilde{\mathbf{y}}_0) + \varphi_{0,t}^{\tilde{\cQ}}(\tilde{\mathbf{x}}_t,\tilde{\mathbf{y}}_0)$, which implies that $(\tilde{\mathbf{x}}_0,\tilde{\mathbf{x}}_t)$ are an optimal coupling of $(\mu_0,\mu_1)$.  Since $\varphi_{0,t}(\cdot,\tilde{\mathbf{y}}_0)$ is differentiable at $\tilde{\mathbf{x}}_t$, this means that $\tilde{\mathbf{x}}_0 = \nabla_{m,m'} \varphi_{0,t}^{\tilde{\cQ}}(\tilde{\mathbf{x}}_t,\tilde{\mathbf{y}}_0)$.  Therefore,
		\[
		\tp_{\chron}^{\tilde{\cQ}}(\tilde{\mathbf{x}}_0,\tilde{\mathbf{x}}_t,\tilde{\mathbf{y}}_0) = \tp_{\chron}^{\cQ}(\mathbf{x}_0,\mathbf{x}_t,\mathbf{y}_0).
		\]
		Since the character $\tilde{\cV}$ was arbitrary, we conclude that \eqref{eq: lifting convergence} holds.
		
		We have therefore proved the desired properties for $\mathbf{X}_t^{(n)}$ when $t < 1$.  We proceed to construct some $\mathbf{X}_1^{(n)}$ satisfying the conclusions of the theorem.  Fix a sequence $(\psi_k)_{k \in \N}$ of a restricted chronological formulas which is dense in the space of chronologically definable predicates with respect to $\rT_{\stat}$, which can be done because of Lemma \ref{lem: separability}.  Using continuity of $\psi_j$, fix $t_k$ such that
		\begin{equation} \label{eq: approximation of time 1}
			\max_{j \leq k} |\psi_j^{\cQ}(\mathbf{x}_0,\mathbf{x}_t,\mathbf{y}_0) - \psi_j^{\cQ}(\mathbf{x}_0,\mathbf{x}_1,\mathbf{y}_0)| \leq \frac{1}{k}.
		\end{equation}
		Moreover, let $I_0 = \N$ and for $k \geq 1$,
		\begin{equation} \label{eq: definition of Ik}
			I_k = \left\{n \geq k: \mathbb{P}\left(\max_{j \leq k} |\Lambda_{\psi_j}^{(n)}(\mathbf{X}_0^{(n)},\mathbf{X}_{t_k}^{(n)},\mathbf{Y}_0^{(n)}) - \psi_j^{\cQ}(\mathbf{x}_0,\mathbf{x}_t,\mathbf{y}_0))| \leq \frac{1}{k}\right) \geq 1 - \frac{1}{k} \right\}.
		\end{equation}
		By \eqref{eq: lifting convergence} for $\mathbf{X}_t^{(n)}$, the set $I_k$ is an element of the ultrafilter $\cU$.  Also since $I_k \subseteq [k,\infty)$, we have $\bigcap_{k \in \N} I_k = 0$.  Let $J_k = \bigcap_{k' \leq k} I_{k'}$, so that $J_k \in \cU$ and $J_{k+1} \subseteq J_k$ and $\bigcap_{k \in \N} J_k = 0$.  Thus, $\N$ is the disjoint union of $J_k \setminus J_{k+1}$ for $k \in \N_0$.  Set
		\[
		\mathbf{X}_1^{(n)} = \mathbf{X}_{t_k}^{(n)} \text{ for } n \in J_k \setminus J_{k+1}.
		\]
		Then by \eqref{eq: approximation of time 1} and \eqref{eq: definition of Ik} and the triangle inequality
		\[
		n \in J_k \setminus J_{k+1} \implies \mathbb{P}\left(\max_{j \leq k} |\Lambda_{\psi_j}^{(n)}(\mathbf{X}_0^{(n)},\mathbf{X}_1^{(n)},\mathbf{Y}_0^{(n)}) - \psi_j^{\cQ}(\mathbf{x}_0,\mathbf{x}_1,\mathbf{y}_0))| \leq \frac{2}{k}\right) \geq 1 - \frac{1}{k}.
		\]
		Note that if this condition holds for $k$, it also holds for any $k' \geq k$.  Therefore, we have
		\[
		n \in J_k \implies \mathbb{P}\left(\max_{j \leq k} |\Lambda_{\psi_j}^{(n)}(\mathbf{X}_0^{(n)},\mathbf{X}_1^{(n)},\mathbf{Y}_0^{(n)}) - \psi_j^{\cQ}(\mathbf{x}_0,\mathbf{x}_1,\mathbf{y}_0))| \leq \frac{2}{k}\right) \geq 1 - \frac{1}{k}.
		\]
		Since $J_k \in \cU$, this means that
		\[
		\lim_{n \to \cU} \Lambda_{\psi_j}^{(n)}(\mathbf{X}_0^{(n)},\mathbf{X}_1^{(n)},\mathbf{Y}_0^{(n)}) = \psi_j^{\cQ}(\mathbf{x}_0,\mathbf{x}_1,\mathbf{y}_0) \text{ in probability}
		\]
		for all $j \in \N$.  This establishes \eqref{eq: lifting convergence} for $\mathbf{X}_1^{(n)}$.
	\end{proof}

	\subsection{Geodesic concavity of entropy} \label{subsec: geodesic concavity}
	
	\begin{lemma} \label{lem: displacement Laplacian}
		Consider the same setup as Proposition \ref{prop: displacement interpolation} and for $s \in (0,1)$, let
		\[
		\omega_s = \frac{1}{1-s} (\varphi_{1,s} - q).
		\]
		Then
		\begin{enumerate}[(1)]
			\item $\omega_s$ is $1/s$-semiconcave and $1/(1-s)$-semiconvex in the first $m$ variables.
			\item $\omega_s$ is $\norm{\cdot}_1$-Lipschitz in the first $m$ variables.
			\item For $t \in [0,1]$, we have $\mathbf{x}_t = \mathbf{x}_s + (t-s) \nabla_{m,m'} \omega^{\cM}(\mathbf{x}_s,\mathbf{y}_0)$.
			\item $s \mapsto \overline{\Delta}_{m,m'} \omega^{\cM}(\mathbf{x}_s,\mathbf{y}_0)$ and $s \mapsto \underline{\Delta}_{m,m'} \omega^{\cM}(\mathbf{x}_s,\mathbf{y}_0)$ are decreasing functions.
		\end{enumerate}
	\end{lemma}
	
	\begin{proof}
		(1) By Proposition \ref{prop: displacement interpolation} (1), $\varphi_{1,s}$ is $1/s$-semiconcave.  Hence, $\varphi_{1,s} - q$ is $1/s - 1 = (1-s)/s$-semiconcave and so $\omega_s$ is $1/s$-semiconcave.  Moreover, since $\varphi_{1,s}$ is convex, $\omega_s$ is $1/(1-s)$-semiconvex.
		
		(2) This is immediate from Proposition \ref{prop: displacement interpolation} (1).
		
		(3) By Proposition \ref{prop: displacement interpolation} (6), $\mathbf{x}_1 = \nabla_{m,m'} \varphi_{1,s}^{\cM}(\mathbf{x}_s,\mathbf{y}_0)$.  Hence,
		\begin{align*}
			\nabla_{m,m'} \omega^{\cM}(\mathbf{x}_s,\mathbf{y}_0) &= \frac{1}{1-s} (\nabla_{m,m'} \varphi_{1,s}^{\cM}(\mathbf{x}_s,\mathbf{y}_0) - \mathbf{x}_s) \\
			&= \frac{1}{1-s} (\mathbf{x}_1 - \mathbf{x}_s) \\
			&= \mathbf{x}_1 - \mathbf{x}_0.
		\end{align*}
		Since $\mathbf{x}_t = \mathbf{x}_s + (t-s) (\mathbf{x}_1 - \mathbf{x}_0)$, the claim follows.
		
		(4) First, we recall the relationship between $\varphi_{1,s}$ and $\varphi_{1,t}$ (compare \cite[Proposition 4.12(5)]{GJNS2021}).  By definition of $\varphi_{1,s}$,
		\begin{align*}
			\varphi_{1,s}^{\cM}(\mathbf{x},\mathbf{y}) &= \inf_{\mathbf{x}' \in M_0^m} \left[ \frac{1}{2s} \norm{\mathbf{x}}_2^2 - \frac{1-s}{s} \re \ip{\mathbf{x},\mathbf{x}'} + \frac{(1-s)^2}{s} \norm{\mathbf{x}'}_2^2 + (1-s) \varphi_0(\mathbf{x}',\mathbf{y}) \right] \\
			&= \inf_{\mathbf{x}' \in M_0^m} \left[ \frac{1}{2s} \norm{\mathbf{x}}_2^2 - \frac{1}{s} \re \ip{\mathbf{x},\mathbf{x}'} + \frac{1}{s} \norm{\mathbf{x}'}_2^2 + (1-s) \varphi_0((1-s)^{-1}\mathbf{x}',\mathbf{y}) \right] \\
			&= \inf_{\mathbf{x}' \in M_0^m} \left[ \frac{1}{2s} \norm{\mathbf{x} - \mathbf{x}'}_2^2 + (1-s) \varphi_0((1-s)^{-1}\mathbf{x}',\mathbf{y}) \right].
		\end{align*}
		Hence, letting $\varphi_0((1-s)^{-1} \cdot, \cdot)$ denote the function $\varphi_0^{\cM}((1-s)^{-1} \mathbf{x}, \mathbf{y})$, we have
		\[
		\varphi_{1,s} = [s^{-1} q] \square [(1-s) \varphi_0((1-s)^{-1} \cdot, \cdot)].
		\]
		Therefore,
		\[
		\varphi_{1,s}^{\star} = sq + [(1-s) \varphi_0((1-s)^{-1} \cdot, \cdot)]^{\star} = sq + (1-s) \varphi_0^{\star},
		\]
		where the last inequality follows from a direct change of variables in the definition of the Legendre transform.  Thus, for $0 \leq s \leq t$,
		\[
		\varphi_{1,t}^{\star} = tq + (1-t) \varphi_0^{\star} = \frac{t-s}{1-s} q + \frac{1-t}{1-s} (sq + (1-s) \varphi_0^{\star}) = \frac{t-s}{1-s} q + \frac{1-t}{1-s} \varphi_{1,s}^{\star},
		\]
		and hence by Lemma \ref{lem: Lipschitz plus quadratic 2},
		\[
		\varphi_{1,t} = \frac{1-s}{t-s} q \square \frac{1-t}{1-s} \varphi_{1,s}\left( \frac{1-s}{1-t} \cdot, \cdot \right).
		\]
		We note that
		\[
		\varphi_{1,t}(\mathbf{x}_t,\mathbf{y}_0) = \inf_{\mathbf{x}' \in \cM^m}  \left[ \frac{1-s}{2(t-s)} \norm{\mathbf{x}_t - \mathbf{x}'}_2^2 + \frac{1-t}{1-s} \varphi_{1,s}\left(\frac{1-s}{1-t} \mathbf{x}',\mathbf{y}_0 \right) \right].
		\]
		The function to be minimized is differentiable in $\mathbf{x}'$ since $\varphi_{1,s}$ is differentiable, and the gradient is
		\[
		\frac{1-s}{t-s} (\mathbf{x}' - \mathbf{x}_t) + \nabla_{m,m'} \varphi^{\cM}\left( \frac{1-s}{1-t} \mathbf{x}', \mathbf{y}_0 \right).
		\]
		When $\mathbf{x}' = \frac{1-t}{1-s} \mathbf{x}_s$, the gradient becomes
		\[
		\frac{1-t}{t-s} \mathbf{x}_s - \frac{1-s}{t-s} \mathbf{x}_t + \mathbf{x}_1 = 0,
		\]
		and since the function to be minimized is convex, this point is the unique minimizer, that is,
		\[
		\varphi_{1,t}^{\cM}(\mathbf{x}_t,\mathbf{y}_0) = \frac{1-s}{2(t-s)} \norm*{ \frac{1-t}{1-s} \mathbf{x}_s - \mathbf{x}_t}_2^2 + \frac{1-t}{1-s} \varphi^{\cM}(\mathbf{x}_s,\mathbf{y}_0).
		\]
		To estimate the Laplacian, we need to estimate $\varphi_{1,t}^{S_\varepsilon \cM}(\mathbf{x}_t + \mathbf{z}_{\varepsilon})$ where $\mathbf{z}_{\varepsilon} = (z_{0,\varepsilon,j})_{j=1}^m$.  Using the candidate point $\frac{1-t}{1-s} (\mathbf{x}_s + \mathbf{z}_{\varepsilon})$, we obtain
		\begin{align*}
			\varphi_{1,t}^{S_{\varepsilon} \cM}(\mathbf{x}_t + \mathbf{z}_{\varepsilon}, \mathbf{y}_0) &\leq \frac{1-s}{2(t-s)} \norm*{ \frac{1-t}{1-s} (\mathbf{x}_s + \mathbf{z}_{\varepsilon}) - (\mathbf{x}_t + \mathbf{z}_{\varepsilon})}_2^2 + \frac{1-t}{1-s} \varphi^{\cM}(\mathbf{x}_s + \mathbf{z}_{\varepsilon},\mathbf{y}_0) \\
			&= \frac{1-s}{2(t-s)} \left[ \norm*{ \frac{1-t}{1-s} \mathbf{x}_s - \mathbf{x}_t}_2^2 + \norm*{ \frac{s-t}{1-s} \mathbf{z}_{\varepsilon}}_2^2 \right] + \frac{1-t}{1-s} \varphi^{\cM}(\mathbf{x}_s + \mathbf{z}_{\varepsilon},\mathbf{y}_0) \\
			&= \frac{1-s}{2(t-s)} \norm*{ \frac{1-t}{1-s} \mathbf{x}_s - \mathbf{x}_t}_2^2 + \frac{(t-s)m \varepsilon}{(1-s)} + \frac{1-t}{1-s} \varphi^{\cM}(\mathbf{x}_s + \mathbf{z}_{\varepsilon},\mathbf{y}_0),
		\end{align*}
		where we have applied that $\mathbf{z}_{\varepsilon}$ is orthogonal to $M_0$ and $\norm{\mathbf{z}_{\varepsilon}}_2^2 = 2m \varepsilon$.  Subtracting the previous equation for $\varphi_{1,t}^{\cM}(\mathbf{x}_t,\mathbf{y}_0)$ yields
		\[ 
		\varphi_{1,t}^{S_{\varepsilon} \cM}(\mathbf{x}_t + \mathbf{z}_{\varepsilon}, \mathbf{y}_0) -\varphi_{1,t}^{\cM}(\mathbf{x}_t,\mathbf{y}_0) \leq \frac{(t-s)m \varepsilon}{(1-s)} + \frac{1-t}{1-s} \left[ \varphi^{\cM}(\mathbf{x}_s + \mathbf{z}_{\varepsilon},\mathbf{y}_0) - \varphi^{\cM}(\mathbf{x}_s,\mathbf{y}_0) \right].
		\]
		Now dividing by $\varepsilon$, multiplying by $2$, and taking the $\limsup$ as $\varepsilon \to 0$ yields
		\[
		\overline{\Delta}_{m,m'} \varphi_{1,t}^{\cM}(\mathbf{x}_t,\mathbf{y}_0) \leq \frac{2(t-s)m}{(1-s)} + \overline{\Delta}_{m,m'} \varphi_{1,s}^{\cM}(\mathbf{x}_s,\mathbf{y}_0).
		\]
		Rearranging yields
		\[
		\frac{1}{1-t} \left[ \overline{\Delta}_{m,m'} \varphi_{1,t}^{\cM}(\mathbf{x}_t,\mathbf{y}) - 2m \right] \leq \frac{1}{1-s} \left[ \overline{\Delta}_{m,m'} \varphi_{1,s}^{\cM}(\mathbf{x}_s,\mathbf{y}) - 2m \right].
		\]
		Since $\overline{\Delta}_{m,m'} q = 2m = \underline{\Delta}_{m,m'} q$, this means precisely that $\overline{\Delta}_{m,m'} \omega_t^{\cM}(\mathbf{x}_t,\mathbf{y}_0) \leq \overline{\Delta}_{m,m'} \omega_s^{\cM}(\mathbf{x}_s,\mathbf{y}_0)$.  The same also holds for $\underline{\Delta}_{m,m'}$ simply by using the $\liminf$ rather than $\limsup$ in the argument.
	\end{proof}
	
	\begin{theorem} \label{thm: geodesic concavity}
		Consider the same setup as Proposition \ref{prop: displacement interpolation} and let $\omega_s$ be as in Lemma \ref{lem: displacement Laplacian}.
		\begin{enumerate}[(1)]
			\item The map $t \mapsto \chi_{\chron}^{\cU}(\mathbf{x}_t \mid \mathbf{y}_0)$ is concave on $[0,1]$.
			\item For almost every $t \in (0,1)$, we have
			\[
			\frac{d}{dt} \chi_{\chron}^{\cU}(\mathbf{x}_t,\mathbf{y}_0)  = \overline{\Delta}_{m,m'} \omega^{\cM}(\mathbf{x}_t) = \underline{\Delta}_{m,m'} \omega^{\cM}(\mathbf{x}_t).
			\]
			\item For $0 \leq s < t \leq 1$, we have
			\[
			2m \log \frac{1-t}{1-s} \leq \chi_{\chron}^{\cU}(\mathbf{x}_t,\mathbf{y}_0) - \chi_{\chron}^{\cU}(\mathbf{x}_s,\mathbf{y}_0) \leq 2m \log \frac{t}{s}.
			\]
		\end{enumerate}
	\end{theorem}
	
	\begin{proof}
		We will postpone the proof of Claim (1) until the end, since the endpoints $0$ and $1$ require additional argument.
		
		(2) Suppose that $0 < s < t < 1$ and $t - s \leq \frac{1}{2} \min(s,1-s,t,1-t)$.  Recall from Lemma \ref{lem: displacement Laplacian} that $\omega_s$ is $1/s$-semiconcave and $1/(1-s)$-semiconvex.  Let $c_s = \max(1/s,1/(1-s))$.  Applying Theorem \ref{thm: infinitesimal change of variables} with $\varepsilon = t - s$, we obtain that
		\begin{equation} \label{eq: entropy derivative lower bound}
			\chi_{\chron}^{\cU}(\mathbf{x}_t \mid \mathbf{y}_0) \geq \chi_{\chron}^{\cU}(\mathbf{x}_s \mid \mathbf{y}_0) + (t-s) \overline{\Delta}_{m,m'} \omega_s^{\cM}(\mathbf{x}_s,\mathbf{y}_0) - 2mc_s^2(t-s)^2.
		\end{equation}
		Since $\omega_s^{\cM}$ is $1/(1-s)$-semiconvex, we have $\Delta_{m,m'} \omega_s^{\cM} \geq -2m/(1-s)$, and so
		\begin{equation} \label{eq: entropy derivative lower bound 2}
			\chi_{\chron}^{\cU}(\mathbf{x}_t \mid \mathbf{y}_0) \geq \chi_{\chron}^{\cU}(\mathbf{x}_s \mid \mathbf{y}_0) -\frac{2m(t-s)}{1-s} - 2mc_s^2(t-s)^2.
		\end{equation}
		Symmetrically, we can apply Theorem \ref{thm: infinitesimal change of variables} to the function $-\omega_t$ with $c_t = \max(1/t,1/(1-t))$ and $\varepsilon = t-s$ to obtain
		\begin{equation} \label{eq: entropy derivative upper bound}
			\chi_{\chron}^{\cU}(\mathbf{x}_s \mid \mathbf{y}_0) \geq \chi_{\chron}^{\cU}(\mathbf{x}_t \mid \mathbf{y}_0) + (s-t) \underline{\Delta}_{m,m'} \omega_t^{\cM}(\mathbf{x}_s,\mathbf{y}_0) - 2mc_t^2(t-s)^2,
		\end{equation}
		and since $\omega_t^{\cM}$ is $1/t$-semiconcave,
		\begin{equation} \label{eq: entropy derivative upper bound 2}
			\chi_{\chron}^{\cU}(\mathbf{x}_s \mid \mathbf{y}_0) \geq \chi_{\chron}^{\cU}(\mathbf{x}_t \mid \mathbf{y}_0) + \frac{2m(s-t)}{t} - 2mc_t^2(t-s)^2.
		\end{equation}
		Therefore, when $t - s \leq \frac{1}{2} \min(s,1-s,t,1-t)$, we have
		\begin{equation} \label{eq: entropy Lipschitz estimate}
			|\chi_{\chron}^{\cU}(\mathbf{x}_t \mid \mathbf{y}_0) - \chi_{\chron}^{\cU}(\mathbf{x}_s \mid \mathbf{y}_0)| \leq 2m \max\left(\frac{1}{s},\frac{1}{1-s},\frac{1}{t},\frac{1}{1-t} \right) ( |s - t| + |s - t|^2).
		\end{equation}
		Therefore, $t \mapsto \chi_{\chron}^{\cU}(\mathbf{x}_t \mid \mathbf{y}_0)$ is locally Lipschitz and so differentiable almost everywhere.  Note that \eqref{eq: entropy derivative lower bound} gives a lower bound for the right-hand derivative and \eqref{eq: entropy derivative upper bound} gives an upper bound for the left-hand derivative.  So at every point $t \in (0,1)$ where $\chi_{\chron}^{\cU}(\mathbf{x}_t \mid \mathbf{y}_0)$ is differentiable, we have
		\[
		\overline{\Delta}_{m,m'} \omega_t^{\cM}(\mathbf{x}_t,\mathbf{y}_0) \leq \frac{d}{dt} \chi_{\chron}^{\cU}(\mathbf{x}_t \mid \mathbf{y}_0) \leq \underline{\Delta}_{m,m'} \omega_t^{\cM}(\mathbf{x}_t,\mathbf{y}_0) \leq \overline{\Delta}_{m,m'} \omega_t^{\cM}(\mathbf{x}_t,\mathbf{y}_0),
		\]
		which implies that all three quantities are equal as asserted in claim (2).
		
		(3) First suppose that $0 < s < t < 1$.  We apply \eqref{eq: entropy derivative lower bound 2} and \eqref{eq: entropy derivative upper bound 2} to obtain that for almost every $t \in (0,1)$,
		\[
		-\frac{2m}{1 - t} \leq \frac{d}{dt} \chi_{\chron}^{\cU}(\mathbf{x}_t \mid \mathbf{y}_0) \leq \frac{2m}{t}.
		\]
		Therefore, for $s \leq t$,
		\[
		\chi_{\chron}^{\cU}(\mathbf{x}_t \mid \mathbf{y}_0) - \chi_{\chron}^{\cU}(\mathbf{x}_s \mid \mathbf{y}_0) \leq 2m \int_s^t \frac{1}{u}\,du = 2m \log \frac{t}{s},
		\]
		and symmetrically
		\[
		\chi_{\chron}^{\cU}(\mathbf{x}_t \mid \mathbf{y}_0) - \chi_{\chron}^{\cU}(\mathbf{x}_s \mid \mathbf{y}_0) \geq -2m \int_s^t \frac{1}{1-u}\,du = 2m \log \frac{1-t}{1-s}.
		\]
		Now consider the case when $s = 0$.  Then $\log (t/s) = \infty$ so the right-hand inequality in (3) is trivial.  By Proposition \ref{prop: classical-free variational principle}, let $\mathbf{Y}^{(n)}$ be a deterministic matrix tuple and $\mathbf{X}_0^{(n)}$ a random matrix tuple such that \eqref{eq: convergence of type at time zero} holds and
		\[
		\lim_{n \to \cU} h^{(n)}(\mathbf{X}_0^{(n)}) = \chi_{\chron}^{\cU}(\mathbf{x}_0 \mid \mathbf{y}_0).
		\]
		Let $\mathbf{X}_1^{(n)}$ be the compatible random matrix models for $\mathbf{x}_1$ given by Theorem \ref{thm: optimal coupling lifts}.  Let $\mathbf{X}_t^{(n)} = (1-t)\mathbf{X}_0^{(n)} + t \mathbf{X}_1^{(n)}$.  Since $(\mathbf{X}_0^{(n)}, \mathbf{X}_1^{(n)})$ are classically optimally coupled, the classical Monge--Kantorovich duality and the classical analog of Proposition \ref{prop: displacement interpolation} implies that there is a convex and $1/(1-t)$-semiconcave function $\varphi_{0,t}^{(n)}$ such that $\mathbf{X}_0^{(n)} = \nabla \varphi_{0,t}^{(n)}(\mathbf{X}_t^{(n)})$ almost everywhere in $\mathbb{M}_n^m$.  In particular, $\mathbf{X}_0^{(n)}$ is a $1/(1-t)$-Lipschitz function of $\mathbf{X}_t^{(n)}$, and recall by Theorem \ref{thm: optimal coupling lifts} that $\mathbf{X}_1^{(n)}$ is also a Borel function of $\mathbf{X}_0^{(n)}$.  Hence, by classical change of variables for entropy, we have
		\begin{align*}
			h^{(n)}(\mathbf{X}_0^{(n)}) &= h^{(n)}(\mathbf{X}_1^{(n)}) + \mathbb{E} \left[ \frac{1}{n^2} \Tr(\log H \varphi_{0,t}^{(n)}(\mathbf{X}_t^{(n)}) \right] \\
			&\leq h^{(n)}(\mathbf{X}_1^{(n)}) + 2m \log \frac{1}{1-t}.
		\end{align*}
		Hence, by Proposition \ref{prop: classical-free variational principle},
		\begin{align*}
			\chi_{\chron}^{\cU}(\mathbf{x}_t \mid \mathbf{y}_0) &\geq \lim_{n \to \cU} h^{(n)}(\mathbf{X}_t^{(n)})  \\
			&\geq \lim_{n \to \cU} h^{(n)}(\mathbf{X}_0^{(n)}) + 2m \log(1-t) \\
			&= \chi_{\chron}^{\cU}(\mathbf{x}_t \mid \mathbf{y}_0) + 2m \log(1-t).
		\end{align*}
		Hence, (3) is established when $s = 0$.  The case when $t = 1$ is symmetrical; one simply switches the roles of $\mathbf{x}_t$ and $\mathbf{x}_{1-t}$ for $t \in [0,1]$.
		
		(1) Note that concavity of $t \mapsto \chi_{\chron}^{\cU}(\mathbf{x}_t,\mathbf{y}_0)$ on the open interval $(0,1)$ follows because $t \mapsto \overline{\Delta}_{m,m'} \omega_t^{\cM}(\mathbf{x}_t,\mathbf{y}_0)$ is decreasing by Lemma \ref{lem: displacement Laplacian} (4), and this agrees with the derivative of $\chi_{\chron}^{\cU}(\mathbf{x}_t,\mathbf{y}_0)$ for almost every $t$ by claim (2).  To extend the concavity to the closed interval $[0,1]$, it suffices to show that the value at each endpoint is not any larger than the $\liminf$ as we approach the endpoint.  Note that by claim (3) for $s = 0$,
		\[
		\liminf_{t \to 0^+} \chi_{\chron}^{\cU}(\mathbf{x}_t \mid \mathbf{y}_0) \geq \chi_{\chron}^{\cU}(\mathbf{x}_0 \mid \mathbf{y}),
		\]
		and symmetrically applying claim (3) with $t = 1$,
		\[
		\liminf_{s \to 1^-} \chi_{\chron}^{\cU}(\mathbf{x}_s \mid \mathbf{y}) \geq \chi_{\chron}^{\cU}(\mathbf{x}_1 \mid \mathbf{y}). \qedhere
		\]
	\end{proof}
	
	\subsection{Evolution variational inequality} \label{subsec: EVI}
	
	In this section, we show that the heat evolution in $\mathbb{S}_{\chron,m,\mathbb{r}}(\cQ / \mathbf{y}_0)$ is the gradient flow for entropy in a suitable sense.
	
	Suppose that $\mathbf{x}_0 \in \cQ_0^m$ with $\nu_0 = \tp_{\chron}^{\cQ}(\mathbf{x}_0,\mathbf{y}) \in \mathbb{S}_{\chron,m,\mathbb{r}}(\cQ / \mathbf{y}_0)$.  Let $\mathbf{z}_t = (z_{0,t,j})_{j=1}^m$ be the canonical Brownian motion in $\cQ^m$, and let $\nu_t$ be the chronological type of $(\mathbf{x}_0 + \mathbf{z}_t,\mathbf{y}_0)$ in the shifted structure $S_t \cQ$.  Note that for a chronological formula $\varphi$,
	\[
	(\nu_t,\varphi) = \varphi^{S_t\cQ}(\mathbf{x}_0+\mathbf{z}_t,\mathbf{y}_0) = (P_t \varphi)^{\cQ}(\mathbf{x}_0,\mathbf{y}_0) = (\nu_0,P_t \varphi),
	\]
	where $P_t$ is the heat semigroup with respect to the first $m$ variables.  This shows that $\nu_t$ is a continuous function of $\nu_0$ with respect to the logic topology, and in particular, it does not depend on the choice of representative $(\mathbf{x}_0,\mathbf{y}_0)$ of the type $\nu_0$.  We will thus also denote $\nu_t$ as $P_t^\dagger \nu_0$.  Since $P_t$ is a semigroup acting on chronological formulas, we see that $P_t^{\dagger}$ is a semigroup acting on chronological types over $\mathbf{y}_0$.

	\begin{theorem} \label{thm: EVI}
		Let $\cM$ and $\cQ$ be as in \eqref{eq: random matrix ultraproduct} and \eqref{eq: random matrix quotient}.  Let $m \in \N$ and $m' \in \N \cup \{0,\infty\}$.  Fix $\mathbf{y}_0 \in Q_0^{m'}$ and $\mathbf{x}_0$, $\mathbf{x}_1 \in Q_0^{m}$.  Let
		\[
		\nu_0 = \tp_{\chron}^{\cQ}(\mathbf{x}_0,\mathbf{y}_0), \quad
		\nu_t = \tp_{\chron}^{S_t \cQ}(\mathbf{x}_0+\mathbf{z}_t,\mathbf{y}), \quad 
		\sigma = \tp_{\chron}^{\cQ}(\mathbf{x}_1,\mathbf{y}_0)
		\]
		where $t \geq 0$ and $\mathbf{z}_t$ is the first $m$ coordinates of the canonical Brownian motion.  Then we have the evolution variational inequality
		\[
		\frac{1}{2} \left[ d_W(\nu_t,\sigma)^2 - d_W(\nu_s,\sigma)^2 \right] \leq (t-s) \left[ \chi_{\chron}^{\cU}(\mathbf{x}_0+ \mathbf{z}_t \mid \mathbf{y}_0) - \chi_{\chron}^{\cU}(\mathbf{x}_1 \mid \mathbf{y}_0) \right] \text{ for } 0 \leq s < t < \infty.
		\]
	\end{theorem}
	
	\begin{proof}
		By the semigroup property of $P_t^{\dagger}$, it suffices to prove the inequality in the case that $s = 0$.  Let $(\mathbf{X}_1^{(n)},\mathbf{Y}_0^{(n)})$ be random matrix approximations for $(\mathbf{x}_1,\mathbf{y}_0)$ given by Proposition \ref{prop: classical-free variational principle}, so that $\mathbf{Y}_0^{(n)}$ is deterministic, $\mathbf{X}_1^{(n)}$ is random,
		\[
		\lim_{n \to \cU} \Lambda_\varphi^{(n)}(\mathbf{X}_1^{(n)},\mathbf{Y}_0^{(n)}) = \varphi^{\cQ}(\mathbf{x}_0,\mathbf{y}_0) \text{ in probability}
		\]
		for restricted chronological formulas $\varphi$, and
		\[
		\lim_{n \to \cU} h^{(n)}(\mathbf{X}_1^{(n)}) = \chi_{\chron}^{\cU}(\mathbf{x}_1 \mid \mathbf{y}).
		\]
		By Lemma \ref{lem: existence of optimal couplings}, assume without loss of generality that $(\mathbf{x}_0,\mathbf{x}_1)$ are an optimal coupling of $(\nu_0,\sigma)$.  Now let $\mathbf{X}_0^{(n)}$ be random matrix models for $\tp_{\chron}^{\cQ}(\mathbf{x}_0,\mathbf{y}_0)$ given by Theorem \ref{thm: optimal coupling lifts}, so that $(\mathbf{X}_0^{(n)},\mathbf{X}_1^{(n)})$ is classically optimally coupled.  Let $\mathbf{Z}_t^{(n)}$ be the first $m$ coordinates of the matrix Brownian motion.  Let $\nu_t^{(n)}$ be the probability distribution of $\mathbf{X}_0^{(n)} + \mathbf{Z}_t^{(n)}$ and let $\sigma^{(n)}$ be the probability distribution of $\mathbf{X}_1^{(n)}$.  The classical heat evolution on $\mathbb{M}_n^m$ and the classical entropy satisfy the EVI with constant $0$ by \cite[Theorem 3.16]{MurSav2020gradient}, and hence also the integral form EVI$_0$'' from \cite[Theorem 3.3]{MurSav2020gradient}.  As explained above, we use a dimensionally normalized entropy $h^{(n)}$, where the normalization is compatible with the normalization of the Brownian motion $\mathbf{Z}_t^{(n)}$ and the corresponding heat flow.  Thus, we have
		\begin{equation} \label{eq: classical EVI on matrices}
			\frac{1}{2} \left[ d_W(\nu_t^{(n)},\sigma^{(n)})^2 - d_W(\nu_0^{(n)},\sigma^{(n)})^2 \right] \leq t \left[ h^{(n)}(\nu_t) - h^{(n)}(\sigma^{(n)}) \right].
		\end{equation}
		Let $(\mathbf{V}_t^{(n)},\mathbf{W}_t^{(n)})$ be a classical optimal coupling of $(\nu_t^{(n)},\sigma^{(n)})$ in $L^2(\Omega,\mathcal{F}_0,\mathbb{P}; \mathbb{M}_n^m)$.  Let $r$ be an upper bound for the operator norm of $\mathbf{X}_0^{(n)}$.  Let
		\begin{align*}
			\mathbf{v}_t &= \pi([F_{r + 3 \sqrt{t}}(\mathbf{V}_t^{(n)})]_{n \in \N}) \in \cQ_0^m \\
			\mathbf{w}_t &= \pi([\mathbf{W}_t^{(n)}]_{n \in \N}) \in \cQ_0^m.
		\end{align*}
		Then using the Lipschitz property of $F_{r+3\sqrt{t}}$,
		\begin{align}
			d_W(\nu_t,\sigma) &\leq \norm{\mathbf{v}_t - \mathbf{w}_t}_2^2 \nonumber \\
			&= \lim_{n \to \cU} \norm{F_{r+3\sqrt{t}}(\mathbf{V}_t^{(n)}) - \mathbf{W}_t^{(n)}}_{L^2} \nonumber \\
			&\leq \lim_{n \to \cU} \norm{\mathbf{V}_t^{(n)} - \mathbf{W}_t^{(n)}}_{L^2} \nonumber \\
			&= \lim_{n \to \cU} d_W(\nu_t^{(n)},\sigma^{(n)}). \label{eq: EVI upper bound for Wasserstein}
		\end{align}
		By the our choice of $\mathbf{X}_0^{(n)}$, we have
		\begin{equation} \label{eq: EVI equality for Wasserstein}
			d_W(\nu_0,\sigma) = \norm{\mathbf{x}_0 - \mathbf{x}_1}_2 = \lim_{n \to \cU} \norm{\mathbf{X}_0^{(n)} - \mathbf{X}_1^{(n)}}_{L^2} = \lim_{n \to \cU} d_W(\nu_0^{(n)},\sigma^{(n)}).
		\end{equation}
		By Proposition \ref{prop: classical-free variational principle},
		\begin{equation} \label{eq: EVI lower bound for entropy}
			\chi_{\chron}^{\cU}(\mathbf{x}_0 + \mathbf{z}_t \mid \mathbf{y}) \geq \lim_{n \to \cU} h^{(n)}(\nu_t^{(n)}).
		\end{equation}
		Note that the hypotheses of Proposition \ref{prop: classical-free variational principle} are satisfied because using concentration of measure
		\begin{align*}
			\lim_{n \to \cU}
			\Lambda_{\varphi}^{(n)}(\mathbf{X}_0^{(n)} + \mathbf{Z}_t^{(n)},\mathbf{Y}_0^{(n)}) &= \lim_{n \to \cU} \Lambda_{\varphi}^{(n)}(\mathbf{X}_0^{(n)} + F_{3\sqrt{t}} (\mathbf{Z}_t^{(n)}),\mathbf{Y}_0^{(n)}) \\
			&= \lim_{n \to \cU} \mathbb{E} \Lambda_{\varphi}^{(n)}(\mathbf{X}_0^{(n)} + F_{3\sqrt{t}} (\mathbf{Z}_t^{(n)}),\mathbf{Y}_0^{(n)}) \\
			&= P_t \varphi^{\cM}(\mathbf{x}_0, \mathbf{y}_0) \\
			&= \varphi^{S_t \cM}(\mathbf{x}_0 + \mathbf{z}_t,\mathbf{y}_0).
		\end{align*}
		Moreover, the operator norm and tail bound hypotheses hold for $\mathbf{X}_0^{(n)} + \mathbf{Z}_t^{(n)}$ since $\mathbf{X}_0^{(n)}$ is bounded in operator norm and the Ginibre random matrices satisfy \eqref{eq: Herbst concentration inequality} and \eqref{eq: operator norm bound for GUE}.  By the choice of $\mathbf{X}_1^{(n)}$,
		\begin{equation} \label{eq: EVI equality for entropy}
			\chi_{\chron}^{\cU}(\mathbf{x}_1 \mid \mathbf{y}_0) = \lim_{n \to \cU} h^{(n)}(\sigma^{(n)})
		\end{equation}
		Hence, plugging in \eqref{eq: EVI upper bound for Wasserstein}, \eqref{eq: EVI equality for Wasserstein}, \eqref{eq: EVI lower bound for entropy}, and \eqref{eq: EVI equality for entropy} into \eqref{eq: classical EVI on matrices}, we obtain
		\[
		\frac{1}{2} \left[ d_W(\nu_t,\sigma)^2 - d_W(\nu_0,\sigma) \right] \leq t \left[ \chi_{\chron}^{\cU}(\mathbf{x}_0+ \mathbf{z}_t \mid \mathbf{y}_0) - \chi_{\chron}^{\cU}(\mathbf{x}_1 \mid \mathbf{y}_0) \right],
		\]
		which completes the proof.
	\end{proof}

	\bibliographystyle{plain}
	\bibliography{typeLDP}
	
\end{document}